\documentclass{memo-l}

\usepackage[english]{babel}
\usepackage{amsmath,amsthm,hyperref,amssymb,graphics,amsfonts,enumerate,bbm,graphicx,float,color,chngcntr,pdfsync,caption}
\counterwithin{section}{chapter}
\restylefloat{table}
\theoremstyle{plain}
\allowdisplaybreaks

\newtheorem{lemma}{Lemma}[chapter]
\newtheorem{proposition}[lemma]{Proposition}
\newtheorem{theorem}[lemma]{Theorem}
\newtheorem{cor}[lemma]{Corollary}

\theoremstyle{definition}
\newtheorem{definition}[lemma]{Definition}

\theoremstyle{remark}
\newtheorem{assumption}[lemma]{Assumption}
\newtheorem{convention}[lemma]{Convention}
\newtheorem{example}[lemma]{Example}
\newtheorem{remark}[lemma]{Remark}

\newcommand{\bet}{\left |}
\newcommand{\rag}{\right |}
\newcommand{\equa}{\begin{eqnarray*}}
\newcommand{\tion}{\end{eqnarray*}}
\newcommand{\Om}{\Omega}
\newcommand{\om}{\omega}
\newcommand{\vare}{\varepsilon}
\newcommand{\vph}{\varphi}
\newcommand{\half}{{\frac{1}{2}}}
\newcommand{\probsp}{(\Om ,\cF,\P)}
\newcommand{\A}{{\mathbb{A}}}
\newcommand{\E}{{\mathbb{E}}}
\newcommand{\F}{{\mathbb{F}}}
\renewcommand{\H}{{\mathbb{H}}}
\newcommand{\R}{{\mathbb{R}}}
\newcommand{\Q}{{\mathbb{Q}}}
\newcommand{\B}{{\mathbb{B}}}
\newcommand{\D}{{\mathbb{D}}}
\newcommand{\G}{{\mathbb{G}}}
\renewcommand{\P}{{\mathbb{P}}}
\newcommand{\N}{{\mathbb{N}}}
\newcommand{\cA}{{\mathcal{A}}}
\newcommand{\cB}{{\mathcal{B}}}
\newcommand{\cD}{{\mathfrak{D}}}
\newcommand{\cF}{{\mathcal{F}}}
\newcommand{\cH}{{\mathcal{H}}}
\newcommand{\cL}{{\mathcal{L}}}
\newcommand{\cG}{{\mathcal{G}}}
\newcommand{\cP}{{\mathcal{P}}}
\newcommand{\cM}{{\mathcal{M}}}
\newcommand{\cN}{{\mathcal{N}}}
\newcommand{\cR}{{\mathcal{R}}}

\newcommand{\sptext}[3]{\hspace{#1 em}\mbox{#2}\hspace{#3 em}}
\newcommand{\U}{\mathcal{B}(C(M))}
\newcommand{\C}{\mathcal{C}}
\renewcommand{\r}[1]{\widehat{#1}}
\newcommand{\ws}{{\Delta}}
\newcommand{\lip} [2]{\left \| #1 \right \|_{L_\infty([0,T];L_{#2}(\Omega))}}
\newcommand{\lips}[2]{\left \| #1 \right \|_{L_{#2}^\ast(\Omega;L_2([0,T]))}}
\newcommand{\rh}{\mathcal{RH}}
\newcommand{\bmo}{{\rm BMO}}
\newcommand{\bh}{\rm{cExp}}
\newcommand{\Lip}{{\rm Lip}}

\newcommand{\sli}{{\rm sl}}
\newcommand{\myspan}{{\rm span}}
\newcommand{\he}{{\bf h}}
\newcommand{\var}{{\rm var}}
\newcommand{\esssup}{{\rm esssup}}
\newcommand{\indexnorm}{\|}

\makeindex

\newcommand{\st}[1]{{#1}}

\begin{document}


\frontmatter

\title[Decoupling on the Wiener space]
      {Decoupling on the Wiener Space, Related Besov Spaces, and Applications to BSDEs}

\author{Stefan Geiss}
\address{University of Jyv\"askyl\"a, Department of Mathematics and Statistics, P.O.Box 35, FI-40014 University of Jyv\"askyl\"a,
         Finland}
\email{stefan.geiss@jyu.fi}

\author{Juha Ylinen}
\address{University of Jyv\"askyl\"a, Department of Mathematics and Statistics, P.O.Box 35, FI-40014 University of Jyv\"askyl\"a,
         Finland}
\email{juha.m.ylinen@jyu.fi}
\thanks{The authors were supported by the project 
        "Stochastic and Harmonic Analysis, interactions, and applications", No. 133914,
         of the Academy of Finland.
         The second author was partly supported by the Vilho, Yrj\"o and Kalle V\"ais\"al\"a foundation of the Finnish Academy of Science and Letters.}

\date{}



\dedicatory{To my parents (Stefan Geiss).}

\begin{abstract}
\st{
We introduce a decoupling method on the Wiener space to define a wide class of an\-iso\-tro\-pic
Besov spaces. The decoupling method is based on a general distributional approach and not restricted 
to the Wiener space.
\smallskip

The class of Besov spaces we introduce contains the traditional isotropic Besov spaces obtained 
by the real interpolation method, but also new spaces that are designed to investigate 
backwards stochastic differential equations (BSDEs).
As examples we discuss the Besov regularity (in the sense of our spaces) of forward diffusions and local times.
It is shown that among our newly introduced Besov spaces there are spaces that characterize quantitative properties of 
directional derivatives in the Malliavin sense without computing or accessing these Malliavin derivatives explicitly. 
\smallskip

Regarding BSDEs, we deduce regularity properties of the solution processes from the Besov regularity of the initial data, 
in particular upper bounds for their $L_p$-variation, where the generator might be of quadratic type
and where no structural assumptions, for example in terms of a forward diffusion, are assumed.
As an example we treat sub-quadratic BSDEs with unbounded terminal conditions.

\smallskip

Among other tools, we use methods from harmonic analysis. As a by-product, we improve the asymptotic behaviour 
of the multiplicative constant in a generalized Fefferman inequality and verify the optimality of the bound we established.
}
\end{abstract}

\maketitle

\tableofcontents

\mainmatter


\chapter{Introduction}
\label{chapter:intro}


\section{Background}
A backward stochastic differential equation (BSDE) is an equation of type
\begin{equation}\label{eqn:BSDE_intro}
   Y_t = \xi + \int_t^T f(s,Y_s,Z_s) ds - \int_t^T Z_s dW_s,
\end{equation}
where $T>0$ is a fixed \st{finite} time horizon, $W=(W_t)_{t\in [0,T]}$ is a $d$-dimensional Brownian motion,
$\xi: \Omega \to \R$ \st{is} a given $\cF_T$-measurable {\em terminal condition}, and
\[ f: [0,T]\times \Omega\times \R \times \R^d \to \R \]
\st{is} a given predictable random generator which might be non-Mar\-kovian. 
Given the data 
$(\xi,f)$, one looks for adapted solution processes $(Y,Z)$. 
Backward stochastic differential equations have a wide range of applications, 
for example in stochastic control and, more generally, in stochastic modeling.
In the case of a Markovian generator, where the randomness comes from a forward diffusion,
there is an important and extremely useful connection to non-linear partial differential 
equations of parabolic type, the so-called (non-linear) Feynman-Kac theory. 
Two seminal papers in this theory were the work of Bismut \cite{Bismut:73}, and Pardoux and 
Peng \cite{Pardoux:Peng:90}.
\medskip

The simulation of BSDEs is an important topic and subject to active research. To setup simulation schemes one 
needs an approximation theory for BSDEs, for example to find optimal time-grids or to obtain upper and lower 
rates for the speed of convergence of these schemes measured in an appropriate way.
To investigate these approximation properties it is more or less mandatory to understand
the variational properties of the solution $(Y,Z)$, i.e. the behavior of
\begin{equation}\label{eqn:variation_YZ}
   \| Y_t-Y_s\|_p 
   \sptext{1}{and (say)}{1}
   \left \| \left ( \int_s^t | Z_r|^2 dr \right )^\frac{1}{2} \right \|_p 
\end{equation}
for all $0\le s < t \le T$ and an appropriate range of $p\in (0,\infty)$,
\st{where $\|\xi\|_p := \| \xi\|_{L_p(\Omega)}= (\E |\xi|^p)^\frac{1}{p}$
for a random variable $\xi:\Omega\to \R$.}

   
\section{Outline of the main ideas}
\label{sec:outline}

In these notes we develop an approach to estimate the variations from \eqref{eqn:variation_YZ}
in terms of the regularity of the data $(\xi,f)$, where the regularity is 
a fractional smoothness expressed in terms of Besov spaces. 
Our approach is based on an anisotropic decoupling of the Wiener space. Recently this  
decoupling was already successfully used in \cite{GGG:12,CGeiss:Steinicke:16} and constitutes one 
of the few approaches to estimate variational properties of non-Markovian backwards equations using 
only knowledge of the initial data. Let us explain the basic line of ideas to motivate the structure of these notes.
\medskip
   
If the generator in our BSDE vanishes, i.e. $f\equiv 0$, then one has that
\[ Y_t = \E (\xi | \cF_t). \]
Therefore, in the case $f\not \equiv 0$ the map 
\[ G_t^f: \xi \to Y_t \]
can be interpreted as some kind of generalized {\em non-linear conditional expectation} along the generator $f$
(see \cite{Peng:97, Coquet:Hu:Memin:Peng:02,Peng:05} for the notion of $g$-expectation and nonlinear expectations).
\st{It turns out that our notion of regularity is} 
stable with respect to this non-linear map $G_t^f$. Moreover, since 
\[  \| Y_t-Y_s\|_p 
    \le \   \| Y_t -\E(Y_t|\cF_s) \|_p + \|  \E(Y_t|\cF_s) - Y_s \|_p 
   \]
and since $\|  \E(Y_t|\cF_s) - Y_s \|_p$ can be handled by 'standard' methods,
the main question consists in investigating the behavior of
$\| Y_t - \E(Y_t|\cF_s)\|_p$ for $s\uparrow t$. 
It turns out that
this behaviour corresponds to a notion of fractional smoothness in $L_p$
of the random variable $Y_t$. The crucial point here is that
\begin{equation}\label{eqn:decoupling_vs_conditional_expectation}
 \| Y_t - \E(Y_t|\cF_s)\|_p \sim \| Y_t - Y_t^{(s,t]}\|_p 
\end{equation}
for $p \in [1,\infty)$, where $Y_t^{(s,t]}$ is a decoupled version of $Y_t$
in the sense explained below. Therefore we proceed as follows: 
\begin{enumerate}[(a)]
\item In Chapter \ref{chapter:general_factorization} we introduce 
      a factorization \st{and a method to transfer stochastic processes from one stochastic basis to another one
      while keeping distributional and measurability properties.}
\item \st{In} Chapter \ref{chapter:transference_sde} we apply \st{the methods from 
      Chapter \ref{chapter:general_factorization} to the Wiener space, in particular}
      to stochastic differential equations \st{driven by the Brownian motion.} 
\item \st{In} Chapter \ref{chapter:Besov_spaces} the decoupling and the corresponding 
      Besov spaces on the Wiener space are introduced and investigated.
\item \st{In} Chapter \ref{chapter:BMO} we provide some tools about BMO spaces and reverse 
      H\"older inequalities \st{and apply them to} non-Lipschitz BSDEs.
\item \st{In} Chapter \ref{chapter:BSDE} we \st{apply further} the results
      of Chapters \ref{chapter:general_factorization}, \ref{chapter:transference_sde},
      \ref{chapter:Besov_spaces}, and \ref{chapter:BMO} to BSDEs.
\end{enumerate}
\medskip

We proceed with some exemplary ideas and results obtained in this article:
\medskip

\underline{\bf Chapters \ref{chapter:general_factorization} - \ref{chapter:Besov_spaces}:}
The decoupling to obtain $F^{(a,b]}$ from a random variable $F:\Omega \to\R$ on the Wiener space is done as follows: 
We start with a Wiener space built on a $d$-dimensional Brownian motion $W=(W_t)_{t\in [0,T]}$. 
Then we take a copy of this Wiener space, denote the corresponding Brownian motion by 
$W'=(W_t')_{t\in [0,T]}$, and form the canonical product space carrying the $2d$-dimensional Brownian motion 
$((W_t,W'_t))_{t\in [0,T]}$. But the pair $(W,W')$ of Brownian motions is not the one we are interested in in the sequel.
Instead, we take (for example) an interval $(a,b]\subset (0,T]$ and consider the mixed
Brownian motion $W^{(a,b]}=(W_t^{(a,b]})_{t\in [0,T]}$ where the increments on the interval 
$(a,b]$ from $W$ are replaced by the increments of the independent copy 
$W'=(W_t')_{t\in [0,T]}$, i.e. we define
\[ W_t^{(a,b]} = \left \{ \begin{array}{rcl}
                        W_t &:& 0\le t \le a \\
                        W_a+W'_t - W'_a &:& a\le t \le b \\
                        W_a + (W'_b-W'_a)+(W_t-W_b) &:& b\le t \le T
                        \end{array} \right . .\]

In other words, the Gaussian structure on $(a,b]$ is replaced by an independent copy:

\hspace*{6em}
\includegraphics[scale=.2]{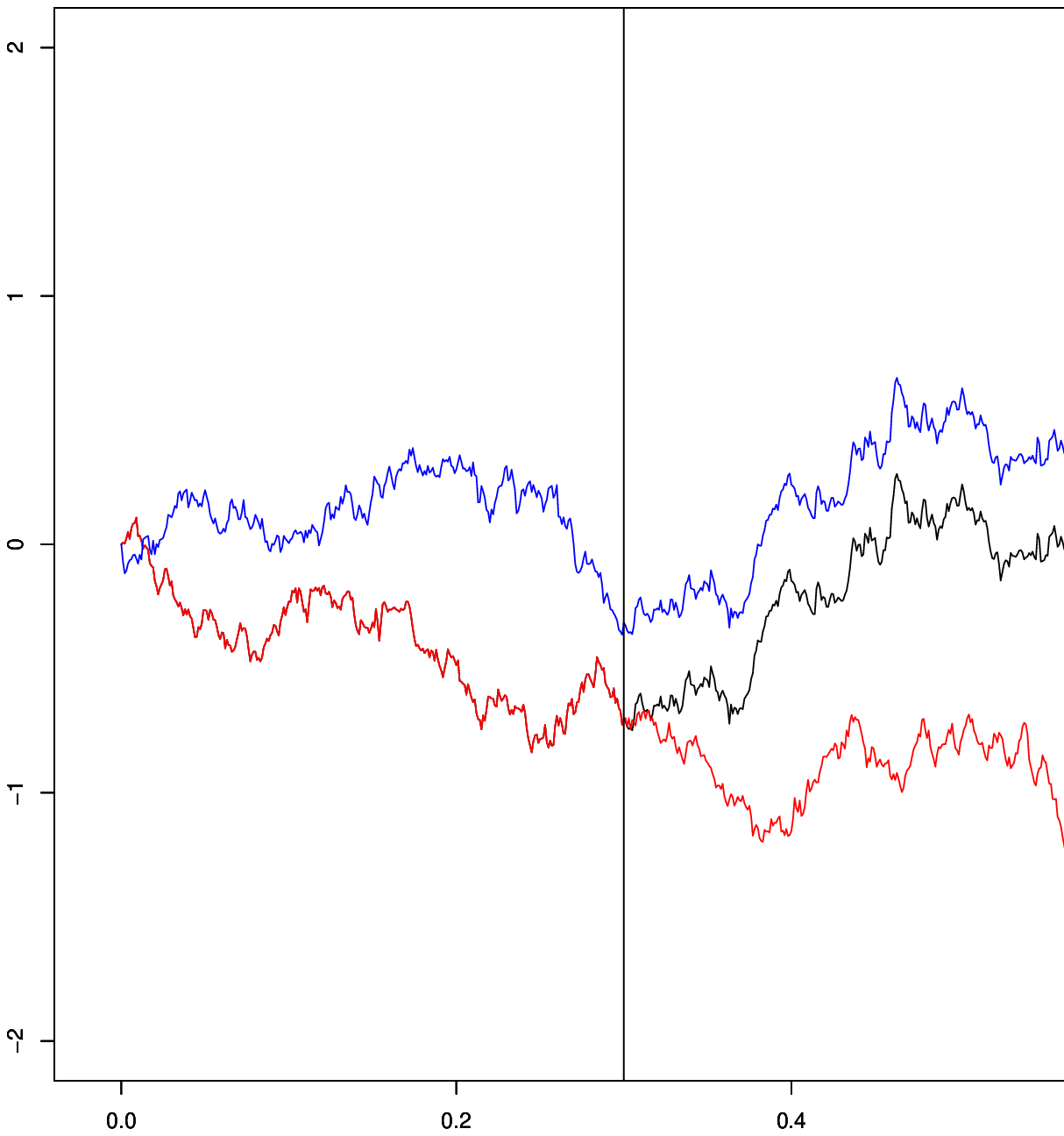}

Now the random variables $F$ from the original Wiener space built on $W$ are extended to the product space 
carrying $(W,W')$ and are transformed by a functional mapping
$F\to F^{(a,b]}$ along the same map as $W\to W^{(a,b]}$ is transformed. 
\medskip

After we have introduced the decoupling method, our
next step consists in observing that one can define anisotropic Besov spaces by imposing 
H\"older type conditions on a random variable $\xi\in L_p$ like 
\begin{equation}\label{eqn:anisotropic_decoupling}
 \| \xi - \xi^{(a,b]}\|_p\le c \,\alpha(a,b)
\end{equation}
for all $0\le a < b \le T$ and an appropriate weight function $\alpha(\cdot,\cdot)$. These anisotropic 
Besov spaces are part of a wider class of spaces containing the traditional Besov spaces obtained by the 
real interpolation method. 
\st{To explain the diction {\em anisotro\-pic}, let us assume $d=1$ and let us formally 
write $\xi=f(W)$ for an appropriate functional $f:C[0,T] \to \R$. 
If $\emptyset \not = (a,b]\not = (0,T]$, then  in \eqref{eqn:anisotropic_decoupling} we compare 
$f(W)$ with $f(W^{(a,b]})$ and note that there is {\em no} constant $c\in [0,1]$ 
such that
\[ \E W_s W_t^{(a,b]} = c \, \E W_s W_t 
   \sptext{1}{for all}{1} s,t\in [0,T]. \]
Now let $\theta \in (0,1)$ and define  
the Brownian motion $W^\theta$ by 
\[ W^\theta_t := \sqrt{1-\theta^2} W_t + \theta W_t'
   \sptext{1}{for}{1}t\in [0,T]. \]
Here the Brownian motion $W^\theta$ (partially) decouples $W$ uniformly in time, not only on 
$(a,b]$. This means, that we have an {\em isotropic} decoupling. In contrast to 
\eqref{eqn:anisotropic_decoupling}, the expression
$\|f(W) - f(W^\theta)\|_p$ compares $f(W)$ and $f(W^\theta)$, where
\[ \E W_s W_t^\theta = \sqrt{1-\theta^2} \, \E W_s W_t
   \sptext{1}{for all}{1} s,t\in [0,T]. \]
The reader is also referred to Remark \ref{remark:reason_for_an-isotropic_BSDEs} below 
for a more detailed example of being {\em anisotropic}.
}

\medskip

To explain a prototype of \st{our} Besov spaces, let us assume 
$p,r\in [2,\infty)$ and $\xi \in L_p$. In  Chapter \ref{chapter:Besov_spaces} we use
\st{inequality \eqref{eqn:anisotropic_decoupling} with}
\[ \alpha_r(a,b):= \sqrt[r]{b-a} \]
to define
$\xi\in \B^{\Phi_r}_p$ provided that
\[    \| \xi \|_{\B_p^{\Phi_r}}^p 
   := \E |\xi|^p +   \| \xi \|_{\Phi_r,p}^p  < \infty
   \sptext{1}{with}{1}
     \| \xi \|_{\Phi_r,p}^p
  := \sup_{0\le a < b \le T} \left | \frac{\| \xi - \xi^{(a,b]}\|_p}{\sqrt[r]{b-a}}\right |^p. \]
The case $r=2$ is treated by Theorem \ref{theorem:Phi_2} and includes the following situation, where $\D_{1,2}$ stands for the
Malliavin Sobolev space and $D\xi$ for the Malliavin derivative:
\bigskip

\begin{theorem}
\label{theorem:Phi_2_intro}
One has $\B_2^{\Phi_2} \subseteq \D_{1,2}$. Moreover, for $p\in [2,\infty)$ and $\xi\in \D_{1,2}\cap L_p$ it holds
      \[ \| \xi\|_{\Phi_2,p} \sim_c
         \sup_{0\le a < b \le T} \left \| \left ( \frac{1}{b-a} \int_a^b |D_s\xi|^2 ds \right )^\frac{1}{2} \right \|_p
      , \]
where $c >0$ depends on $p$ only. In particular, for $p=2$ we have that
\[ \| \xi\|_{\Phi_2,2} \sim_c
         {\rm esssup}_{s\in [0,T]} \| D_s\xi \|_2. \]
\end{theorem}
\bigskip

The impact of Theorem \ref{theorem:Phi_2_intro} (Theorem \ref{theorem:Phi_2})
is at least twofold: Firstly, we can access the Malliavin derivative by the spaces  
$\B^{\Phi_2}_p$ without using the derivative explicitly. Secondly, the above 
theorem can be localized by replacing $\| \xi \|_{\B_p^{\Phi_2}}^p$ with
\[    \E |\xi|^p + \sup_{A\le a < b \le B} \left | \frac{\| \xi - \xi^{(a,b]}\|_p}{\sqrt{b-a}}\right |^p
      \]
for some $0\le A<B \le T$. Here  $\xi$ does not need to belong to $\D_{1,2}$ anymore.
\medskip

The case $r=4$ turns out to be relevant for the local time of a Brownian motion, 
for example represented by  
\[ L_t^\alpha = \lim_{\varepsilon \downarrow 0} \frac{1}{2\varepsilon} \int_0^t \chi_{(\alpha-\varepsilon,\alpha+\varepsilon)}(W_s) ds 
   \mbox{ a.s.} \]
\st{We} prove in Corollary \ref{cor:smoothness_local_time} that for all $\alpha\in \R$ and $p\in (1,\infty)$ one has that
\[ L_T^\alpha \in \B_p^{\Phi_4}\setminus \left [ \bigcup_{r\in [2,4)} \B_p^{\Phi_r} \right ]. \]
\medskip

\st{{\sc Background and related results}:
Our method includes with Theorem \ref{theorem:real_interpolation_Gaussian_Hilbert} a characterization by decoupling of 
the real interpolation spaces $(L_p,\D_{1,p})_{\theta,q}$ for the full range of interpolation parameters
$(\theta,q)\in  (0,1)\times [1,\infty]$, where $p\in [2,\infty)$. This  directly extends \cite[Theorem 3.1]{Geiss:Toivola:14} 
to the case that the supporting Hilbert space of the Gaussian structure of the abstract Wiener space is infinite dimensional.
In \cite[Remark on p. 428]{Hirsch:99} a different characterization by decoupling was given in the case $p=q$, i.e. 
for  $(L_p,\D_{1,p})_{\theta,p}$. The case $q\not = p$ is  of natural interest on its own, but
 the full range 
of parameters $(\theta,q)\in (0,1)\times [1,\infty]$ is also crucial 
for the understanding of certain phenomena in applications.
\smallskip

The idea to use decoupling to understand better Malliavin Sobolev spaces was used before: 
The natural question, whether Malliavin Sobolev spaces are stable under Lipschitz mapping has been raised by Watanabe in 
\cite{Watanabe:93} and answered by Hirsch \cite{Hirsch:99} by describing $(L_p,\D_{1,p})_{\theta,p}$ by decoupling.
Roughly speaking, any representation by decoupling is stable under Lipschitz mappings, so our Besov spaces $\B_p^\Phi$ are 
stable. Therefore Theorem \ref{theorem:real_interpolation_Gaussian_Hilbert} below
verifies as a by-product that the spaces $(L_p,\D_{1,p})_{\theta,q}$ are stable under Lipschitz mappings for all
$(\theta,q)\in (0,1)\times [1,\infty]$ and $p\in [2,\infty)$.
}

\bigskip
\underline{\bf Chapter \ref{chapter:BMO}:}
Given a continuous BMO-martingale $M$ and its Dol{\'e}an-Dade exponential $\mathcal{E}(M)$,
we introduce the sliceable numbers $\sli_N(M)$, that measure the distance of $M$ to
$\H_\infty$, \st{in Definition \ref{definition:sliceable}. Here $\H_\infty$ stands for 
the space of all continuous mean zero martingales $N$ with 
$\| N\|_{\H_\infty}:= \| \langle N \rangle \|_\infty<\infty$ (see Definition \ref{definition:H_infty})}.
Denoting by 
$\rh_\beta (\mathcal{E}(M))$ the constant in the reverse H\"older inequality 
for $\mathcal{E}(M)$ with the exponent $\beta$, we prove in Theorem \ref{theorem:scliceable_rh}:
\bigskip

\begin{theorem}\label{theorem:scliceable_rh_intro}
Let $\Phi:(1,\infty)\to (0,\infty)$ be a non-increasing function and let
\[ \Psi: \Big \{(\gamma,\beta)\in [0,\infty)\times (1,\infty): 0\le\gamma < \Phi(\beta)<\infty  \Big \}\to [0,\infty) \]
be right-continuous in its first argument and such that
\[ \Psi(\gamma_1,\beta) \le \Psi(\gamma_2,\beta) 
   \sptext{1}{for}{1}
   0\le \gamma_1 \le \gamma_2 < \Phi(\beta), \]
with the property that
$\| M \|_{\bmo}  < \Phi(\beta)$ implies
$\rh_\beta(\mathcal{E}(M)) \le \Psi(\| M \|_{\bmo},\beta)$. 
Then, for $\sli_N(M) < \Phi(\beta)$ we have that
\[ \rh_\beta (\mathcal{E}(M)) \le \big [\Psi(\sli_N(M),\beta)\big ]^N.\]
\end{theorem}
\bigskip

The point of this observation is that we get explicit exponents $\beta$ and explicit
bounds for $\rh_\beta (\mathcal{E}(M))$ in terms of the  sliceable numbers $(\sli_N(M))_{N\ge 1}$. This is applied to BMO-martingales obtained
by the fractional gradient $|Z|^\theta$ of our BSDE where $\theta\in [0,1]$ is the
parameter from \eqref{eqn:generator_intro} below that describes the degree of the BSDEs of 
not being Lipschitz in the $Z$-component ($\theta=0$ corresponds to the Lipschitz case, 
$\theta=1$ to the quadratic case).
\smallskip

Another contribution concerns the generalized Fefferman inequality \cite[Lemma 1.6]{Delbaen:Tang:10} (see also \cite[Theorem 1.1]{Banuelos:Bennett:88}). 
We prove with Theorem \ref{thm:generalized_fefferman_measure} a more abstract version using adapted random measures
that yields in 
Corollary \ref{cor:AB_generalized_Fefferman_inequality_new_constant} to 
\[       \left \| \int_0^T | A_t B_t | dt \right \|_p 
     \le \sqrt{2p} \| A \|_{\H_p(S_2)} \| B \|_{\bmo(S_2)}
\]
which improves the asymptotic behavior of the constant from $p$ in  \cite{Delbaen:Tang:10}  to $\sqrt{p}$.
\st{We also verify that the asymptotic order $\sqrt{p}$ as $p\to \infty$ is optimal.}
\bigskip

\underline{\bf Chapter \ref{chapter:BSDE}}
The decoupling method for BSDEs originates from \cite{GGG:12}, where the terminal condition did depend on finitely
many increments of a forward diffusion and the generator was Markovian and Lipschitz.
The aim of this part of the notes is the further development of this method.
\st{Motivated by the equivalence} \eqref{eqn:decoupling_vs_conditional_expectation} we first decouple the BSDE \eqref{eqn:BSDE_intro}
in order to get \st{a} new BSDE
\[ Y_t^{(a,b]} = \xi^{(a,b]} + \int_t^T f^{(a,b]}(s,Y_s^{(a,b]},Z_s^{(a,b]}) ds 
                 - \int_t^T Z_s^{(a,b]} dW_s^{(a,b]} \]
and aim to use {\em a priori estimates} for BSDEs to estimate 
$\left \| \sup_{s\in [t,T]} | Y_s^{(a,b]} - Y_s| \right \|_p$ and  
$\left \| \left ( \int_t^T |Z_s^{(a,b]} - Z_s |^2 ds \right )^\frac{1}{2} \right \|_p$ 
from above by moments of
\[ \xi - \xi^{(a,b]}
   \sptext{1}{and}{1} f - f^{(a,b]}. \]
Here we consider generators $f:[0,T]\times \Omega\times\R\times\R^{d}\to\R$ such that 
      $(t,\om)\mapsto f(t,\om,y,z)$ is predictable for all $(y,z)$ and there are
      $L_Y,L_Z\ge 0$ and $\theta\in [0,1]$ such that
\begin{equation}\label{eqn:generator_intro}
            |f(t,\om,y_0,z_0)-f(t,\om,y_1,z_1)| 
         \le L_Y |y_0-y_1| + L_Z [1+|z_0|+|z_1|]^\theta |z_0-z_1|
\end{equation}
for all $(t,\om,y_0,y_1,z_0,z_1)$. Here 
$\theta=0$ represents the Lipschitz case,
$\theta=1$ the quadratic case, and
$\theta\in (0,1)$ the sub-quadratic case.
The basic stability result is Theorem \ref{theorem:comparison_psi_phi}, a special case is:
\bigskip

\begin{theorem}
\label{theorem:comparison_psi_phi_intro}
Assume for the BSDE
\[ Y_t = \xi + \int_t^T f(s,Y_s,Z_s)ds - \int_t^T Z_s dW_s, \qquad t \in [0,T], \]
conditions {\rm (B1)-(B4)} of Chapter \ref{chapter:BSDE} for $\theta\in [0,1]$ and that there is a non-increasing 
sequence $\st{(s_N)_{N\ge 1}} \subseteq [0,\infty)$  which dominates the sliceable numbers of the fractional gradient,
i.e. $\sli_N^{S_2} (|Z|^\theta) \le s_N$ \st{for all $N\ge 1$}.
Suppose that conditions {\rm (B5)-(B6)} of Chapter \ref{chapter:BSDE} are satisfied for $p \in [2,\infty)$
where in the case $\lim_N s_N > 0$ we additionally assume that 
$p>p_0(L_Z,\lim_N s_N)$.
Then, one has for all $t \in [0,T]$ and  $0\le a < b \le T$ that
\begin{multline*}
        \left \| \sup_{s\in [t,T]} | Y_s^{(a,b]} - Y_s| \right \|_p 
     +  \left \| \left ( \int_t^T |Z_s^{(a,b]} - Z_s |^2 ds \right )^\frac{1}{2} 
                \right \|_p \\
\le c
      \left [ \|\xi^{(a,b]}-\xi \|_p 
    + \left \| \int_t^T | f^{(a,b]}(s,Y_s,Z_s) - f(s,Y_s,Z_s)| ds 
            \right \|_p \right ].
\end{multline*}
\end{theorem}
\bigskip

In order to apply Theorem \ref{theorem:comparison_psi_phi_intro} (Theorem \ref{theorem:comparison_psi_phi}), 
and because of general interest, we discuss classes of quadratic and sub-quadratic BSDEs such that the assumptions of 
Theorems \ref{theorem:comparison_psi_phi_intro} and \ref{theorem:comparison_psi_phi} are satisfied
in Section \ref{sec:classes_quadratic_subquadratic_BSDEs}.
In case of sub-quadratic BSDEs we use the following definition:
\smallskip

\begin{definition}
\hspace*{0em}
\begin{enumerate}
\item We say that a random variable $\xi$ belongs to $\bh$ provided that there are 
      $(\eta,\mu) \in (0,1)\times (0,\infty)$ such that
      \[ 
             |\xi|_{\bh(\eta,\mu)} 
         :=  \sup_{t\in [0,T)} (T-t)^{\frac{1}{\eta}-1} \left \| \E(e^{\mu |\xi|}|\cF_t) \right \|_\infty
             < \infty.
      \]
\item For a c\`adl\`ag process $Y=(Y_t)_{t\in [0,T]}$ we say that $Y\in \bh$ provided that there are 
      $(\eta,\mu) \in (0,1)\times (0,\infty)$ such that
      \[ 
             |Y|_{\bh(\eta,\mu)} 
         :=  \sup_{t\in [0,T)} (T-t)^{\frac{1}{\eta}-1} \left \| \E(e^{\mu \sup_{s\in [t,T]} |Y_s|}|\cF_t) \right \|_\infty
             < \infty.
      \]
\end{enumerate}
\end{definition}

In Theorem \ref{theorem:existence_uniqueness_bh_subquadratic} we prove the following statement:
\bigskip

\begin{theorem}
\label{theorem:existence_uniqueness_bh_subquadratic_intro}
Assume \eqref{eqn:generator_intro} for some $\theta\in (0,1)$, $\sup_{(t,\omega)\in [0,T]\times \Omega} |f(t,\omega,0,0)|< \infty$, 
and that $\xi \in \bh$. Then there is a unique solution $(Y,Z)$ to the
{\rm BSDE} \eqref{eqn:BSDE_intro} in the class where 
$(Y,|Z|)\in \bh \times \H_2(S_2)$
\footnote{The spaces are given in Definitions \ref{definition:BH} and \ref{definition:H_p(S_2)} below.}.
Moreover, for this solution we have that 
\[ |Z|^\eta\in \bmo(S_2)\sptext{1}{for all}{1} \eta\in (0,1). \]
\end{theorem}
\bigskip

Theorem \ref{theorem:existence_uniqueness_bh_subquadratic_intro} enables us to apply 
Theorem \ref{theorem:comparison_psi_phi_intro}, so that a combination with 
Theorem \ref{theorem:Phi_2_intro} gives in Corollary \ref{cor:theorem:comparison_psi_phi_3}:
\bigskip

\begin{cor}\label{cor:theorem:comparison_psi_phi_3_intro}
Assume \eqref{eqn:generator_intro} for some $\theta\in (0,1)$, $\sup_{(s,\omega)\in [0,T]\times \Omega} |f(s,\omega,0,0)|< \infty$,
$\xi \in \bh$, and that $(Y,Z)$ is the unique solution to the {\rm BSDE} \eqref{eqn:BSDE_intro} in the sense of 
Theorem \ref{theorem:existence_uniqueness_bh_subquadratic_intro}.
Fix $t\in [0,T]$.
Then we have
\begin{multline}\label{eqn:cor:theorem:comparison_psi_phi_3_intro}
        \esssup_{s\in [0,t]} \| D_s Y_t \|_2 \\
    \le c \sup_{(a,b]\subseteq (0,t]} \frac{1}{\sqrt{b-a}} 
        \left [  \| \xi - \xi^{(a,b]} \|_2  
               + \left \| \int_t^T \sup_{y,z} | f(s,y,z) - f^{(a,b]}(s,y,z)| ds 
                 \right \|_2 \right ]
\end{multline}
with the convention that the finiteness of the right-hand side first implies $Y_t\in \D_{1,2}$ and 
then inequality \eqref{eqn:cor:theorem:comparison_psi_phi_3_intro}.
\end{cor}
\smallskip
The assertion of Corollary \ref{cor:theorem:comparison_psi_phi_3_intro} says that we only need to control 
directional derivatives of the initial data $(\xi,f)$ on the interval $(0,t]$ (because the perturbations of the original
Brownian motion  $W$ are only performed on $(a,b]\subseteq (0,t]$) to obtain smoothness of $Y_t$ and that the behaviour 
of $(\xi,f)$  regarding perturbations on $(t,T]$ does not have any impact - in a sense, we have a smoothing  effect.
\medskip

Finally, let us turn to the $L_p$-variation of a solution $(Y,Z)$ to our BSDE. Our idea is to use adapted time-nets obtained 
by a quantile method. This idea is made precise by the following two definitions:

\begin{definition}
Let $p\in [1,\infty)$, 
$A=(A_t)_{t\in [0,T]}$ be a measurable c\`adl\`ag process
$A:[0,T]\times \Omega \to \R$, and $C=(C_t)_{t\in [0,T]}$ be a measurable
process $C:[0,T]\times \Omega \to \R^d$, where $\R^d$ is equipped with the
\st{Euclidean} norm. For a deterministic time-net $\tau=(t_i)_{i=0}^n$ with $0=t_0\le t_1 \le \cdots \le t_n=T$ we let 
\[ \var_p([A,C]|\tau):= 
   \sup_{i=1,...,n}  \left \| \sup_{t_{i-1}\le s \le t \le t_i} |A_t - A_s| \right \|_p +
   \sup_{i=1,...,n} \left \| \left ( \int_{t_{i-1}}^{t_i} |C_r|^2 dr \right )^\frac{1}{2} \right \|_p. \]
\end{definition}

\begin{definition}   
Letting $\Lambda:[0,T]\to (0,\infty)$ be \st{integrable} and $n\ge 1$, the time-net
$\tau_n^\Lambda$ consists of $0=t_0 < \cdots < t_n=T$ such that, for all $i=1,...,n$,
\[ \int_{t_{i-1}}^{t_i} \Lambda(r) dr = \frac{1}{n} \int_0^T \Lambda(r) dr. \]
\end{definition}

Now we obtain as part of \st{Corollary \ref{cor:thm:L_p-variation:cor_2_new}} the following result:

\bigskip
\begin{theorem}
\label{thm:thm:L_p-variation:cor_2_new_intro}
Assume \eqref{eqn:generator_intro} for some $\theta\in (0,1)$,
$\sup_{(t,\omega)\in [0,T]\times \Omega} |f(t,\omega,0,0)|< \infty$,
$\gamma \in [2,\infty)$, $\xi\in \bh$, and that
\[ \| \xi - \xi^{(a,b]} \|_2 + \left \| \int_a^T  \sup_{(y,z) \in \R^{d+1}} 
   |f(r,y,z)-f^{(a,b]}(r,y,z)| dr \right \|_2 
   \le \left ( \int_a^b \Gamma(r)dr \right )^\frac{1}{\gamma} \] 
for some integrable Borel function $\Gamma :[0,T]\to [0,\infty)$. Define the weight function 
\[ \Lambda(r) :=  1+ \|f(r,0,0)\|_2 + \Gamma(r). \]
Then one has that
\[ \sup_{n\ge 1} \sqrt[\gamma]{n} \var_2([Y,Z]|\tau_n^{\Lambda}) < \infty \]
where the solution is taken from Theorem \ref{theorem:existence_uniqueness_bh_subquadratic_intro}.
\end{theorem}
\medskip

Theorem \ref{thm:thm:L_p-variation:cor_2_new_intro} allows us to control \st{the} $L_2$-variation of non-Markovian BSDEs by adapted 
time-nets where only the information of the initial data $(\xi,f)$ is used.
\medskip

\st{{\sc Background and related results}:
Because of applications in stochastic modeling and due to the connections to non-linear PDEs, 
the simulation of BSDEs is of particular importance and subject to active research (see for example 
\cite{Zhang:04,
      Bouchard:Touzi:04,
      Gobet:Lemor:Warin:2005,
      Bouchard:Elie:Touzi:09,
      Hu:Nualart:Song:11,
      Lionnet:Reis:Szpruch:15}
in the Lipschitz case,
\cite{Imke:Reis:Zhang:10,
      Richou:11,
      Chassagneux:Richou:16} in the quadratic case,
and \cite{EHJK:17} for an overview about various numerical methods related to BSDEs). 
To setup simulation schemes, one typically considers a time discretization. First, one fixes a deterministic 
time-grid $\tau = (t_i)_{i=0}^n$, where $0 = t_0 < t_1 < \dots < t_n = T$, and a simulation scheme based 
on this grid is considered. For the $Y$-process this means that one finds random variables $(Y_{t_i}^{\tau})_{i=0}^n$ 
that are sampled and provide an approximation of the random variables $(Y_{t_i})_{i=0}^n$. To study how accurate this approximation is, 
one option is to consider the $L_p$-\emph{simulation error}
\[ {\bf err}_p(\tau) := \sup_{0 \le i \le n} \|Y_{t_i}-Y_{t_i}^\tau\|_p \]
for certain $p \in [2,\infty)$. For any feasible simulation scheme, the simulation error should go to zero 
as the mesh-size of the grid goes to zero. Preferably there is even a rate of convergence, which could mean that 
there exists a $c_p>0$, independent of the particular grid $\tau$, such that
\begin{equation}\label{eqn:hope}
{\bf err}_p(\tau)
\le c_p \left(\max_{i=1,\dots,n} |t_i-t_{i-1}|\right)^{\frac{1}{2}}.
\end{equation}      
To obtain the estimate \eqref{eqn:hope}, it turns out to be more or less mandatory to have a path regularity of the 
exact solution itself. The preferred estimate would be to have some $d_p>0$ such that
\begin{equation}\label{eqn:goal}
    \|Y_{t}-Y_{s}\|_p
\le d_p (t-s)^{\frac{1}{2}}
\end{equation}
for any $0 \le s < t \le T$, or a variant of this inequality.
It is known that upper bounds for the variation $\|Y_t - Y_s\|_p$ also relate to 
differential properties of the initial data and how these properties transfer to the
solution processes. Let us review parts of the corresponding literature:
\medskip

\begin{enumerate}[(1)]
\item {\bf Initial data in $\D_{1,2}$ or of Lipschitz type}

Regularity and differential properties: Regarding Lipschitz BSDEs (the generator is Lipschitz in $z$)  we refer for
differential properties of $Y$ and the representation of $Z$ by the Malliavin derivative of $Y$
to \cite{ElKaroui:97,Hu:Nualart:Song:11,Lionnet:Reis:Szpruch:15,CGeiss:Steinicke:16b} and the references 
therein. The notion of an $L_\infty$-Lipschitz functional of a forward 
diffusion has been used in \cite{Zhang:04} and \cite{Bender:Denk:07}. 
For quadratic BSDEs (the generator only satisfies certain local Lipschitz conditions in $z$)
general regularity results are given in \cite{Ankirchner:Imkeller:DosReis:07}.  These general results
were applied to Markovian decoupled FBSDEs 
(in particular, the randomness of the data $(\xi,f)$ of the BSDE is induced by a forward process 
$(X_t)_{t \in [0,T]}$) in \cite{Ankirchner:Imkeller:DosReis:07} as well. 
In \cite{Cheridito:Nam:14} the existence and uniqueness of solutions to quadratic BSDEs is studied 
when the terminal condition $\xi$ has a uniformly bounded Malliavin derivative, i.e. $|D_\cdot \xi(\cdot)|\le c$ a.e.
which relates to our spaces $\B_p^{\Phi_2}$ by Theorem \ref{theorem:Phi_2}. The existence of solutions to some multidimensional quadratic BSDEs, 
examining as a special case sub-quadratic BSDEs, is considered in \cite{Cheridito:Nam:14_b} under the assumption that 
the terminal condition is bounded.
Continuing with decoupled Markovian FBSDEs under 
certain Lipschitz assumptions on the terminal condition, variational estimates for $Z$ can be found in  \cite{Richou:11} and uniqueness 
and existence results under conditions on the forward diffusion and the final time horizon $T$ are obtained in \cite{Richou:2012}.

Variational properties of $Y$: Typically estimates of type $\|Y_t - Y_s \|_p \le c_p \sqrt{t-s}$ (and related estimates for the $Z$-process)
are obtained for decoupled Markovian FBSDE. So, with terminal values of type $\xi=g(X_T)$ 
this kind of results can be found in 
\cite[Lemma 3.2]{Bouchard:Touzi:04},
\cite[Theorem 5.5]{Imke:Reis:10} (\cite[Lemma 5.1]{Imke:Reis:10} gives an estimate for $\|D_v Y_t - D_u Y_t\|_p$),
\cite[Theorem 4.4]{Imke:Reis:Zhang:10},
and \cite[Proposition 3.1]{Chassagneux:Richou:16}.
The setting is more general in \cite[Lemma 2.3]{Zhang:04}, as there the terminal condition is a path-dependent functional of a forward diffusion.
A fully random setting is used in \cite[Corollary 2.7]{Hu:Nualart:Song:11}.
\medskip

\item {\bf Markovian decoupled  FBSDEs with fractional singularities} of different types at the finite time horizon $T$:
To handle approximation problems for stochastic integrals with a singularity at time of maturity,
special non-equidistant time-nets have been used in \cite{Geiss:02} and \cite{Geiss:Geiss:04}. In the context
of BSDEs this idea and these time-nets have been exploited in \cite{Gobet:Makhlouf:10} and \cite{Turkedjiev:15}.
\medskip

\item {\bf Irregular path-dependent terminal conditions}:
Terminal conditions that depend on finitely many time instances of a forward diffusion and have 
there local fractional singularities have been considered in \cite{GGG:12}, the results extend those 
from \cite{Gobet:Makhlouf:10}. 
\end{enumerate}
\bigskip

Now, let us indicate our contribution related to BSDEs:
\medskip

\begin{enumerate}[(1)]
\item We improve the comparison theorem \cite[Theorem 5.1]{Ankirchner:Imkeller:DosReis:07}
      in Lemma \ref{lemma:briand:elie_new} below where we use a generalization of Fefferman's
      inequality (see Remark \ref{remark:lemma:briand:elie}).
      \medskip

\item Our decoupling method can be directly applied to the above mentioned $L_\infty$-Lipschitz functionals of forward 
      diffusions, as used in \cite{Zhang:04} and \cite{Bender:Denk:07}: Assuming such a functional $g(X)$, 
      that depends on finitely many instances of a forward diffusion $X=(X_t)_{t\in [0,T]}$, 
      we directly get the estimate
      \[ |g(X)-g(X^\varphi)| \le L \sup_{t\in [0,T]} |X_t-X_t^\varphi|. \]
      Therefore decoupling properties of $X$ directly transfer to $\xi=g(X)$ and we may use the results of 
      Section \ref{subsec:forward_diffusions} of these notes.
      \medskip

\item \label{item:1_contributions_Besov_space_before} 
      The spaces to describe the fractional smoothness of the terminal condition in \cite{Gobet:Makhlouf:10} and \cite{GGG:12}
      coincide with $\B_p^\Phi$ with $\Phi= \Phi_{r_1,...,r_L}^{(\theta_1,\infty),...,(\theta_L,\infty)}$ 
      (in \cite{Gobet:Makhlouf:10} with $L=1$) from Definition \ref{definition:anisotropic_Besov_conditional_expectations} below. 
      So the present article generalizes  results from \cite{GGG:12} to the fully path-dependent case where no structural
      assumptions on the terminal condition nor the generator are imposed. 
      \medskip

\item In Section \ref{sec:classes_quadratic_subquadratic_BSDEs} we investigate the uniqueness and distributional properties 
      of the $(Y,Z)$-processes of quadratic and sub-quadratic BSDEs that are not necessarily Markovian and that might have an 
      unbounded terminal condition.
      \medskip

\item In Section \ref{sec:formulation_anisotropic_Besov} we prove that regularity properties of a BSDE in terms 
      of $\B_p^\Phi$ for the terminal condition $\xi$, and a similar one for the generator $f$, are transferred to the solution processes  $(Y,Z)$
      without structural assumptions on $(\xi,f)$. For the particular case described in item
      \eqref{item:1_contributions_Besov_space_before}, this was partially  done in the presence of a forward diffusion 
      in \cite{GGG:12}.
      \medskip

\item Section \ref{sec:L_p-variation}: In  the literature usually estimates of the form  \eqref{eqn:goal}, that
      means estimates with {\it the order $\frac{1}{2}$}, are shown.
      This is due to Lipschitz or uniform $\D_{1,2}$ assumptions and appears in  
      \cite[Theorem 2.6, Corollary 2.7]{Hu:Nualart:Song:11}. There regularity results for $Y$ and $Z$ of the form  \eqref{eqn:goal}
      for non-Markovian Lipschitz BSDEs were proven under Lipschitz assumptions for the generator and under assumptions on the 
      Malliavin derivatives up to the second order of $\xi$ and $f$. In \cite[Theorem 2.3]{Hu:Nualart:Song:11} a condition 
      $M^{2,q}$ is used to investigate the variation of the $Y$-process of the solution to a BSDE with a random linear generator. 
      The structure of this BSDE yields to an explicit representation of the $Y$ process. The condition $M^{2,q}$ relates to our 
      $\B_p^{\Phi_2}$ spaces via Theorem \ref{theorem:Phi_2}. Translated to our setting, the condition $M^{2,q}$ is a condition 
      on the predictable projection of $(D_t\xi)_{t\in [0,T]}$, whereas our condition is a condition on $(D_t\xi)_{t\in [0,T]}$ 
      itself -- however, the condition in  \cite{Hu:Nualart:Song:11} is not a condition on $\xi$, but on $\xi\rho_T$, where $\rho_T$ is 
      a stochastic exponential.

      Parts of our contribution are: for the regularity of $Y$ we do not need to require assumptions on the differentiability of
      $\xi$ (for example), secondly we can also treat cases where we have rates in \eqref{eqn:goal} weaker than $\frac{1}{2}$.

\end{enumerate}

}


\section{Notation}
\label{sec:notation}

\st{The spaces $\R^n$ are equipped with the Euclidean norm
$|x|=(\sum_{j=1}^n |x_j|^2)^\frac{1}{2}$ so that 
$[\R^n,|\cdot|]$ becomes a Hilbert space.}
Given a metric space $M$, we let $C(M)$ be the space of all continuous real valued mappings on $M$. 
For a probability space $(\Omega,\cF,\P)$ the space of all random variables $X:\Omega \to \R$, i.e. Borel
measurable maps, is denoted by $\cL_0(\Omega,\cF,\P)$ 
\index{space! $\cL_0(\Omega,\cF,\P)$}
and equipped with the pseudo-metric  
\begin{equation}\label{eqn:pseudo-metric}
d_\Omega(X,Y) := \int_\Omega \frac{|X(\omega)-Y(\omega)|}{1+|X(\omega)-Y(\omega)|} d\P(\omega).
\end{equation}
The space $\cL_p(\Omega,\cF,\P)$, $p\in  (0,\infty)$, \index{space! $\cL_p(\Omega,\cF,\P)$, $p\in (0,\infty]$} consists  of all 
random variables $X:\Omega \to \R$ on $(\Omega,\cF,\P)$ such that
$\|X\|_p := \left ( \int_\Omega|X(\omega)|^p d\P(\omega)\right )^{1/p}<\infty$. 
As usual, for $p=\infty$ we let
$\|X\|_\infty := \esssup_{\omega\in \Omega} |X(\omega)|<\infty$ which yields to the space $\cL_\infty(\Omega,\cF,\P)$. 
By identifying two random variables $X$ and $Y$ on $(\Omega,\cF,\P)$ when $X=Y$ $\P$-a.s., we 
obtain equivalence classes, denoted by $[X]$, the quasi-normed spaces 
$(L_p(\Omega,\cF,\P),\|\cdot\|_p)$ for $p\in  (0,\infty]$, 
\index{space!$L_p(\Omega,\cF,\P)$, $p\in  (0,\infty]$}
and the complete metric space 
$(L_0(\Omega,\cF,\P),d_\Omega)$ with
\index{space!$L_0(\Omega,\cF,\P)$}
\begin{equation}\label{eqn:metric}
d_\Omega([X],[Y]):= d_\Omega(X,Y). 
\end{equation}
In Chapters \ref{chapter:general_factorization} and \ref{chapter:transference_sde} we carefully distinguish 
between equivalence classes and random variables, in the later chapters we follow the standard way to identify
equivalence classes and random variables if there is no risk of confusion.
For  two real valued random variables $X$ and $Y$ or $\R^n$-valued random vectors 
$(X_1,...,X_n)$ and $(Y_1,...,Y_n)$ the \st{notations}  $X\stackrel{d}{=}Y$ and 
$(X_1,...,X_n) \stackrel{d}{=}(Y_1,...,Y_n)$ \st{mean} equality in distribution. 
\index{$X\stackrel{d}{=}Y$}
\index{$(X_1,...,X_n) \stackrel{d}{=}(Y_1,...,Y_n)$}
We shall use the Burkholder-Davis-Gundy inequalities for continuous local martingales \cite[IV.4.1]{Revuz:Yor:99}
with $\beta_p\ge 1$ as constant, i.e. 
given $p \in \st{(0,\infty)}$ and a \st{continuous} real-valued martingale $(M_t)_{t\in [0,T]}$
vanishing at zero, we have
\begin{equation}\label{eqn:BDG}
   \frac{1}{\beta_p} \|\langle M \rangle_T^\frac{1}{2} \|_p \le \| \sup_{t\in [0,T]} |M_t| \|_p
        \le \beta_p  \|\langle M \rangle_T^\frac{1}{2} \|_p
\end{equation}
where $\beta_p \ge 1$ \index{constant! $\beta_p$} is an absolute constant and $\langle M \rangle_T$ \st{is} the quadratic variation of $M$ at time $T$.
We do not need the particular behaviour of the constants $\beta_p$, so that we use for the upper and lower bound
the same constant.
As conventions we use $0^0:=1$ and 
\[ A \sim_c B 
   \sptext{1}{for}{1}
   \frac{1}{c} A \le B \le c A \]
when $A,B\ge 0$ and $c\ge 1$.
\index{$A \sim_c B$} 
\st{Finally, for a set $S$ and $A\subseteq S$ we define the indicator function $\chi_A:S\to \R$ as 
\[ \chi_A(s) := \begin{cases}
                1 & : s\in A\\
                0 & : s\not \in A. \\
                \end{cases}. \]}
\index{$\chi_A$}


\chapter{A General Factorization}
\label{chapter:general_factorization}

There exist several factorization techniques for random variables and stochastic processes that
have the idea to factor a random variable or process through a canonical space that carries
the typical information about the problem one is interested in. We will use this idea as an
intermediate step  to decouple in Chapter \ref{chapter:Besov_spaces} the Wiener space and to generate
anisotropic Besov spaces.
For the Wiener space there are two natural choices as a canonical space: The function 
space of continuous functions that yields to the Wiener measure and the sequence space $\R^\N$ with
$\N=\{0,1,2,....\}$ that yields to an infinite product of standard Gaussian measures. We use the 
second approach as in \cite{Malliavin:etal:95} and \cite{Imkeller:09}, and extend this approach so that
no particular distribution (like the Gaussian distribution) is needed and so that it includes the 
handling of the stochastic processes we work with later. 
\st{The second approach is convenient for us because we need to consider, from the very beginning, 
only sequences of real valued random variables, and furthermore, it might be generalized to other
canonical spaces than spaces of continuous functions.}
\medskip

Our factorization procedure yields to the operators \st{$\C^M$ that are defined in two steps.
First, we introduce the operators $\C$ acting on random variables, then we extend them to the operators $\C^M$} acting on random
continuous functions defined on \st{complete metric spaces, that are locally $\sigma$-compact}.


\section{The operators $\C$ and $\C^M$}
\label{sec:C_T_and_properties}

We shall work with two probability spaces $(\Omega^i,\cF^i,\P^i)$, $i=0,1$, and
random variables $(\xi_k^i)_{k\in I}$, $\xi^i_k:\Omega^i\to\R$, where $I=\{ 0,\ldots, K\}$
or $I=\{ 0,1,2,\ldots \}$, and assume that

\index{coditions!(C1), (C2),\ldots}
\begin{enumerate}[(C1)]
\item $\cF^{\xi,i} :=\sigma ( \xi_k^i : k\in I )$,
\item $\cF^i=\cF^{\xi,i} \vee \cN^i$, where $\cN^i:= \{ A^i \in \cF^i: \P^i(A^i)=0\}$,
\item $(\xi_k^0)_{k\in I}$ and $(\xi_k^1)_{k\in I}$ have the same finite-dimensional
      distributions.
\end{enumerate}

If we omit the superscript $i$ in  $\Om^i,\cF^i,\P^i,(\xi_k^i)_{k\in I}$, or $\cF^{\xi,i}$,
then we consider one of the both probability spaces together with the corresponding
random variables and operators introduced later.
Let $\mathcal{B}(\R^I)$ be the $\sigma$-algebra generated by the cylinder sets
on $\R^I$, and let $\P^c$ \st{be} the law of the map
\[ J_0: \Omega \to \R^I 
   \sptext{1}{with}{1}
   J_0(\omega) := (\xi_k(\omega))_{k\in I}. \]
By the assumption (C3) the measure $\P^c$ is the same for both cases $i=0,1$.
Moreover, let us assume another probability space $(R,\cR,\rho)$, \st{and define
\[ J: R\times \Omega \to R\times \R^I
\sptext{1}{with}{1}
   J(r,\omega):= (r,J_0(\omega)).
\]}
For the construction of the operator $\C$ we start with two lemmas:

\begin{lemma}\label{lemma:Alex_process}
For any $\cR\otimes\cF$-measurable random variable $X:R\times \Omega\to\R$ there is an
$\cR\otimes \cF^\xi$-measurable random variable $X^\xi:R\times \Omega\to\R$ with
$(\rho\otimes\P)\left(  X=X^\xi \right)=1$.
\end{lemma}

\begin{proof}
We show that the $\rho\otimes \P$-completion of $\st{\cR} \otimes \cF^\xi$ contains $\cR\otimes \cF$. 
It is sufficient to prove that $A\times B \in \overline{\cR\otimes \cF^\xi}^{\rho\otimes\P}$ 
for $A\in  \cR$ and $B\in\cF$. We find a $B^\xi\in \cF^\xi$ such that $\P(B\Delta B^\xi)=0$. Hence
$(A\times B)\Delta (A\times B^\xi)= A \times (B\Delta B^\xi)$ is of $\rho \otimes \P$-measure zero.
Because of $A\times B^\xi\in \cR\otimes \cF^\xi$ we can conclude the proof.
\end{proof}
\medskip

\begin{lemma}
\label{lemma:Functional_representation}
The following assertions hold true:
\begin{enumerate}[{\rm (1)}]
\item For each $\cR\otimes \cF^\xi$-measurable random variable $X:R\times \Omega \to \R$ there exists a random
      variable $\r X : R \times \R^I  \to\R$ such that
      \[ X:(R\times \Omega) \xrightarrow{J} (R\times \R^I) \xrightarrow{\r X} \R. \]
\item For $\cR\otimes \cF^{\xi,0}$-measurable random variables $X,X':R \times \Omega^0\to\R$ with
      $(\rho \otimes \P^0)(X=X')=1$ one has $(\rho \otimes \P^1)(\r {X} \circ J^1=\r {X'} \circ J^1)=1$
      where the factorizations $X=\r X \circ J^0$ and $X'=\r {X'} \circ J^0$
      are obtained by part {\rm (1)}.
\end{enumerate}
\end{lemma}
\medskip

\begin{proof}
(1) The map $J$ generates the $\sigma$-algebra $\cR\otimes \cF^\xi$.
Hence we apply the functional representation from the Factorization
Lemma \cite[p. 62]{Bauer:01} and (1) follows.
(2) The assumption implies by a change of variables $(\rho\otimes \P^c) (\r X= \r{X'}) = 1$, and by
another change of variables the conclusion of assertion (2).
\end{proof}

The above lemma enables us to introduce the operator $\C$
that maps an equivalence class $[X]$ from $L_0(R\times \Omega^0)$ to the equivalence class 
$[\r X \circ J^1]$ in $L_0(R\times \Omega^1)$
so that $[X]$ and $[\r X \circ J^1]$ have the same law.
\medskip

\begin{definition}
\label{definition:C_S}
\index{operator!$\C$}
\index{operator!$\C_0$}
\hspace*{0em}
\begin{enumerate}[{\rm (1)}]
\item We define the map
      $\C: L_0(R\times \Omega^0) \to L_0(R\times \Omega^1)$ by
      \[   \C(X) = \C([X]) := [\r X \circ J^1],  \]
      where $X\in [X]$ is an  $\cR\otimes \cF^{\xi,0}$-measurable representative of $[X]$.
\item We define the map
      $\C_0: L_0(\Omega^0) \to L_0(\Omega^1)$ by
      \[   \C_0(X) = \st{\C_0}([X]) := [\r X \circ J^1_0],  \]
      where $X\in [X]$ is an  $\cF^{\xi,0}$-measurable representative of $[X]$.
\end{enumerate}
\end{definition}
\medskip

Part (2) of Definition \ref{definition:C_S} corresponds to the case where $R=\{r_0\}$ is a singleton.
\st{We gave a separate definition since $\C_0$ will play a particular role later on. 
Basic properties of $\C_0$ and $\C$ are summarized in Proposition \ref{proposition:properties_C_T} below.
For its formulation we need a class of functionals $\Phi:\cL_0(R) \times\cdots\times \cL_0(R) \to \R$ 
that, for example in the case $(R,\cR,\rho)=([0,1],\cB([0,1]),\lambda)$ with $\lambda$ being the Lebesgue measure, 
excludes Dirac functionals $\Phi (f):= f(r_0)$, where $r_0\in [0,1]$ is fixed.}
\medskip

\begin{definition}
\label{definition:consistent}
\index{functional!consistent}
A functional $\Phi: (\cL_0(R))^n \to \R$ is called {\em consistent} provided that for all
probability spaces $(A,\cA,\Q)$ and jointly measurable $X_1,...,X_n:R\times A\to \R$
the map $F_X:A\to\R$ with
\[ F_X(\om)=\Phi(X_1(\cdot,\omega),...,X_n(\cdot,\omega))\]
is measurable and $\Q(F_X=F_{X'})=1$ if $(\rho\st{\otimes} \Q) (X_i\not =X_i')=0$ for $i=1,...,n$.
\end{definition}

\begin{proposition}
\label{proposition:properties_C_T}
For $X,X_1,\dots,X_n \in  \cL_0(R\times \Omega^0)$ and $Y_i\in \C(X_i)$, \st{$i=1,\ldots,n$,}
the following holds true:
\begin{enumerate}[{\rm (1)}]
\item $\xi_k^1 \in \C_0(\xi_k^0)$ for $k\in I$.
\item $\C$ is a linear isometry and bijection.
\item $(Y_1,\dots,Y_n) \stackrel{d}{=} (X_1,\dots,X_n)$.
\item For a Borel function $g:\R^n \to \R$ one has
      \[ g(Y_1,\dots,Y_n)\in \C(g(X_1,\dots,X_n)). \]
\item If $\Phi: \cL_0(R) \times \cdots \times \cL_0(R) \to \R$ is consistent, then
      \[ \Phi(Y_1,...,Y_n) \in \C_0 (\Phi(X_1,...,X_n)). \]
\item If $X$ is $\cR\otimes \cF^{\xi,0}$-measurable, then
      there is an $\cR\otimes \cF^{\xi,1}$-measurable $Y \in \C(X)$ such that
      for all $ r\in R$ one has
      \[ Y(r,\cdot) \in \C_0(X(r,\cdot)). \]
\item For $Y \in \cL_0(R\times \Omega^1)$ one has $Y \in \C(X)$ if and only if there is a null-set
      $\cN \subseteq R$ such that for all $r\in R\setminus \cN$ one has
       \[ Y(r,\cdot) \in \C_0(X(r,\cdot)). \]
\end{enumerate}
\end{proposition}
\medskip

\begin{proof}
(1) follows from the definition of $\C_0$.
\medskip

(2) {\sc Linearity:}
Let $a,b \in \R$ and $X,Y \in \cL_0(R\times \Omega^0)$, and take
$\cR\otimes \cF^{\xi,0}$-measurable representatives $X^\xi\in [X]$ and $Y^\xi\in [Y]$. Then
$aX^\xi + bY^\xi\in a[X]+b[Y]$.
From Lemma \ref{lemma:Functional_representation} we get that
\[ X^\xi(\eta)  =  \r {X^\xi} \circ J^0(\eta)
   \sptext{1}{and}{1}
   Y^\xi(\eta)  =  \r {Y^\xi} \circ J^0(\eta), \]
for all $\eta \in R\times \Omega^0$. Defining point-wise
 \[ T := a \r {X^\xi} + b \r {Y^\xi}, \]
we get that $T:R\times \R^I \to \R$ is measurable and
\[ T(J^0(\eta)) = a X^\xi(\eta) +b Y^\xi(\eta)
   \sptext{1}{for all}{1}
   \eta \in R\times \Omega^0 \]
so that $T(J^0)\in [aX^\xi+bY^\xi]$.
By definition of $\C$,
\[ T(J^1) =  a \r {X^\xi} \circ J^1 + b  \r {Y^\xi} \circ J^1
   \in \C(aX + bY), \]
but is also an element of $a\C(X) + b\C(Y)$.
\medskip

{\sc Isometry:}
Because the laws of $J^0$ and $J^1$ coincide, it follows that $X$
and the representatives of $\C(X)$ have the same distribution. As $d(X,X') = d(X-X',0)$
the property that $\C$ is an isometry follows immediately.
\medskip

{\sc Bijection:}
Since $\C$ is an isometry, it is an injection. Now let $Y \in \cL_0(R\times \Omega^1)$
and take $Y^\xi$ to be \st{an} $\st{\cR\otimes} \cF^{\xi,1}$-measurable representative of $[Y]$. Then
there is a measurable $\widehat{Y^\xi}:R\times \R^I\to \R$ such that
\[ Y^\xi(\eta) = \r {Y^\xi} \circ J^1(\eta)
    \sptext{1}{for all}{1}
    \eta \in R\times \Omega^1. \]
Now $\eta \mapsto \r {Y^\xi} \circ J^0(\eta)$ is $\st{\cR}\otimes \cF^{\xi,0}$-measurable and
 \[ \C([\r {Y^\xi} \circ J^0]) = [Y]. \]

(3) The characteristic functions of $(X_1,...,X_n)$ and $(Y_1,...,Y_n)$ coincide, because
for all $(t_1,...,t_n) \in \R^n$ and $Y_k\in \C\left( X_k\right)$ we have
 \[
	 \int_{R\times \Omega^1} e^{i\sum_{k=1}^n t_k Y_k} d(\rho\otimes \P^1)
      =  \int_{R\times \Omega^0} e^{i\sum_{k=1}^n t_k X_k} d(\rho\otimes \P^0) \]
where we used (2) and that $\C$ keeps the distribution invariant.
\medskip

(4) We choose $X_1^\xi,\dots,X_n^\xi$ to be $\cR\otimes \cF^{\xi,0}$-measurable representatives of the
 classes $[X_1],...,[X_n]$, so that
 \[ X_i^\xi(\eta) = \r {X_i^\xi} \circ J^0(\eta) \]
 for  $i=1,\dots,n$ and all $\eta \in R\times \Omega^0$. Next we define the measurable
 functional $T_Z:R\times \R^I \to \R$ as
 \[ T_Z(\zeta) := g(\r{X_1^\xi}(\zeta),\dots,\r{X_n^\xi}(\zeta)) \]
 so that
 $T_Z\circ J^0 = g(X_1^\xi,\dots,X_n^\xi)$.
 By definition of $\C$ we get that
 \[ \C(g(X_1^\xi,\dots,X_n^\xi)) = [T_Z\circ J^1]. \]
 On the other side, by definition of $T_Z$ we have that
 \[ T_Z\circ J^1 = g(\r {X_1^\xi} \circ J^1,\dots,\r {X_n^\xi} \circ J^1), \]
 which is $\rho\otimes \P^1$-a.s. the same as $g(Y_1,\dots,Y_n)$, where $Y_i\in\C(X_i)$. This concludes the proof.
 \medskip

(5) We choose $\cR\otimes \cF^{\xi,0}$-measurable representatives $X_i^\xi\in [X_i]$, define
$Y_i^\xi:= \r{X_i^\xi} \circ J^1$, and get
\equa
       F_{Y^\xi}(\om^1)
& = & \Phi(\r{X_1^\xi}(\cdot,J^1_0(\omega^1)),...,\r{X_N^\xi} (\cdot,J^1_0(\omega^1))), \\
       F_{X^\xi}(\om^0)
& = & \Phi(\r{X_1^\xi}(\cdot,J^0_0(\omega^0)),...,\r{X_N^\xi} (\cdot,J^0_0(\omega^0))).
\tion
\st{Defining
$\Psi: \R^I \to \R$ by} $\Psi(\zeta) :=  \Phi(\r{X_1^\xi}(\cdot,\zeta),...,\r{X_N^\xi} (\cdot,\zeta))$,
our assumptions yields  to a measurable map and
$F_{X^\xi} = \Psi \circ J^0_0$ and
$F_{Y^\xi} = \Psi \circ J^1_0$.
Consequently, $F_{Y^\xi} \in \C_0 ( F_{X^\xi} )$. Finally, our assumption yields that
$F_{X^\xi}$ and $F_X$ belong to the same equivalence class, and $F_{Y^\xi}$ and $F_Y$ belong
to the same equivalence class, so that the proof is complete.
\medskip

(6) We have that $X = \r X \circ J^0$ for some $\r X$,
which implies \st{$X(r) = \widehat{X}(r,J^0_0)$ for all $r\in R$}, and define
$Y :=  \r X \circ J^1$.
By construction this implies that \st{$Y(r) = \r X(r,J^1_0)$
for all $r\in R$}.
\medskip

(7) \st{Choose $X^\xi\in [X]$ to be $\cR\otimes \cF^{\xi,0}$-measurable and
$Y^\xi:= \r {X^\xi} \circ J^1$ so that
\[ Y^\xi\in \C(X) \sptext{1}{and}{1}
   Y^\xi(r) \in \C_0(X^\xi(r)) \]
for all $r\in R$. Moreover, $\P^0 (X^\xi(r) = X(r)) = 1$ for $r\in R\setminus \cN'$ where
$\cN' \subseteq R$ is a null-set,
so that
\[ \C_0(X^\xi(r))=\C_0(X(r)) \]
for all $r\in R \setminus \cN'$. Hence, $Y^\xi(r) \in \C_0(X(r))$ for all $r \in R\setminus \cN'$.
The claim now follows from the fact, that for $Y \in \cL_0(R\times \Omega^1)$ we have that $Y \in \C(X)$ 
if and only if $\P^1(Y(r)=Y^\xi(r))=1$ for all $r \in R \setminus \cN''$, where $\cN'' \subseteq R$ is a null-set.}
\end{proof}

We extend our definition of $\C$ to decouple later random generators of BSDEs.
Let $M$ be a \st{\rm complete metric space that is locally $\sigma$-compact,}\index{locally $\sigma$-compact} i.e. there exist compact 
subsets $\emptyset \not = K_1 \subseteq K_2 \subseteq \dots $, such that $\overline{\mathring K}_n=K_n$ and
$M=\cup_{n=1}^\infty \mathring{K}_n$. \st{For the following we recall that $C(M)$ is the 
space of continuous $\R$-valued functions on $M$.}

\begin{definition}
\index{space!$\cL_0(A;C(M))$}
\index{space!$L_0(A;C(M))$}
Given a measurable space $(A,\cA)$, we let $f\in \cL_0(A;C(M))$ if and only if
$f:A\times M \to \R$ is a Carath\'eodory function, i.e. $f$ satisfies that
\begin{enumerate}[(a)]
\item $\alpha \to f(\alpha,x)$ is measurable for all $x\in M$,
\item $x\to f(\alpha,x)$ is continuous for all $\alpha\in A$.
\end{enumerate}
If $(A,\cA)$ is equipped with a probability measure $\Q$, then 
the space $L_0(A;C(M))$ is the space of equivalence classes with
$f\sim g$ if $\Q(f(x)=g(x), x\in M )=1$.
\end{definition}

\begin{remark}
Equivalently, a Carath\'eodory function  is a measurable function 
$f:A\to C(M)$, when $C(M)$ is equipped with the
smallest $\sigma$-algebra $\U$ such that for all $x\in M$ the maps
$\delta_x: C(M)\to\R$ with $\delta_x(f):=f(x)$ are Borel-measurable.
\end{remark}

The next lemma extends the operator $\C$ to $C(M)$-valued
random variables.

\smallskip

\begin{lemma}\label{lemma:unique_extension_X_i}
For $f \in \cL_0(R\times \Omega^0;C(M))$ there is a $g\in \cL_0(R\times \Omega^1;C(M))$ with
$g(x) \in \C (f(x))$ for all $x \in M$. If $g_1$ and $g_2$ satisfy this property,
then $g_1=g_2$ $(\rho \otimes \P^1)$-a.s.
\end{lemma}

\begin{proof}
Proposition \ref{proposition:properties_C_T} implies that $(f(x))_{x \in M}$ and
$(h(x))_{x \in M}$ have the same finite-dimensional distributions for $h(x) \in \C(f(x))$, so
that the result follows from Proposition \ref{proposition:cont_modification}.
\end{proof}
\smallskip
Now we are ready to introduce the extension $\C^M$ of $\C$ that maps equivalence classes from 
$L_0(R\times \Omega^0;C(M))$ to $L_0(R\times \Omega^1;C(M))$ while keeping the distributional properties of the
equivalence classes.
\medskip

\begin{definition}\label{definition:cont_extension}
\index{operator!$\C^M$}
We let
\[ \C^M: L_0(R\times \Omega^0;C(M)) \to L_0(R\times \Omega^1;C(M)) \]
such that $\C^M([f])$ is the unique equivalence-class whose representatives $g$
satisfy $g(x)\in \C(f(x))$ for all $x\in M$. Moreover, \st{we define} 
$\C^M(f) :=\C^M([f])$ \st{for \linebreak $f \in \cL_0(R\times \Omega^0;C(M))$.}
\end{definition}


\section{The operators $\C$ and $\C^M$ for stochastic processes}

\st{In this section we specialize to stochastic processes $X:[0,T]\times\Om \to \R$, where 
$T \in (0,\infty)$ is fixed. This means, that we complement some results from Section \ref{sec:C_T_and_properties} 
in the case $(R,\cR,\rho)=([0,T],\cB([0,T]),\lambda/T)$ where $\lambda$ 
is the Lebesgue measure. Here we distinguish more clearly between the operators 
$\C$ and $\C_0$ from Definition \ref{definition:C_S}.} We will use the following notation:

\renewcommand{\arraystretch}{1.2}
\st{\begin{center}
\begin{tabular}{|lcl|lcl|}\hline
 $\Omega_0$     &:= &  $\Omega$                                   & $\Om_T$        &:= & $[0,T]\times\Om$                             \\
 $\Sigma_0^\xi$ &:= &  $\cF^\xi$                                  & $\Sigma_T^\xi$ &:= & $\mathcal{B}([0,T])\otimes\cF^\xi$           \\
 $\Sigma_0$     &:= &  $\cF$                                      & $\Sigma_T$     &:= & $\mathcal{B}([0,T])\otimes\cF$                \\
 $\P_0$         &:= &  $\P$                                       & $\P_T$         &:= & $(\lambda\times\P)/T$                         \\
 $\C_0$         &from& \hspace{-1em} Definition \ref{definition:C_S} & $\C_T$         &:= & $\C$                       \\
                 \hline
\end{tabular}
\end{center}
\index{operator!$\C_T$}
\index{space!$(\Omega_T,\Sigma_T,\P_T)$}

\begin{remark}
One might also consider the infinite time interval $[0,\infty)$ by the choice $(R,\cR,\rho)=([0,\infty),\cB([0,\infty)),\mu)$, where 
(for example) $\mu$ is a probability measure with the same null-sets as the Lebesgue measure.
\end{remark}
}
First we show how continuity and measurability properties are transferred by
the operators $\C_0$ and $\C_T$. \st{Here we use the following convention:}

\begin{convention}\label{convention}
\st{Let $S\in \{0,T\}$ and assume} a sub-$\sigma$-algebra $\cG_S \subseteq \Sigma_S$. \st{We} will interpret
$L_0(\Om_S,\cG_S,\P_S)$ as the space of
equivalence classes  $[X]\in L_0(\Om_S,\Sigma_S,\P_S)$ that
contain a  $\cG_S$-mea\-su\-rable representative.
Similarly, $L_0(\Om_S,\cG_S,\P_S;C(M))$ is the space of
equivalence classes  $[X]\in L_0(\Om_S,\Sigma_S,\P_S;C(M))$ that
contain a \linebreak $\st{(\cG_S,\U)}$-measurable representative.
\end{convention}
\medskip

\begin{proposition}\label{prop:properties_C_T_measurability}
For $i=0,1$ assume right-continuous filtrations $\G^i=(\cG^i_t)_{t \in [0,T]}$ with
$\cG^i_t \subseteq \cF^i$ such that $\cG_0^i$ contains all null-sets of $\cF^i$
and
\[ \C_0(L_0(\Om^0,\cG^0_t)) \subseteq L_0(\Om^1,\cG^1_t)
   \sptext{1}{for  all}{1}
    t \in [0,T]. \]
Then the following assertions are true:
\begin{enumerate}[{\rm (1)}]
\item If $X$ is path-wise continuous and $\G^0$-adapted, then there exists a path-wise continuous
      $\G^1$-adapted process $Y\in \cL_0(\Omega^1_T)$ with
      \[ Y(t)\in \C_0(X(t)) \sptext{1}{for all}{1} t \in [0,T]. \]
\item One has $\C_T(L_0(\Omega_T^0,\cP_T^0)) \subseteq L_0(\Omega_T^1,\cP_T^1)$, where $\cP^i_T$ are the
      predictable $\sigma$-alge\-bras generated by the $\G^i$-adapted processes with paths
      that are left-continuous and have limits from the right.
\end{enumerate}
\end{proposition}
\medskip

\begin{proof}
(1) Taking $\beta(t) \in  \C_0(X(t))$ to be $\cG_t^1$-measurable, Proposition \ref{proposition:properties_C_T}(3)
implies that $(\beta(t))_{t\in [0,T]}$ and $(X(t))_{t\in [0,T]}$ have the same finite-dimen\-sional
distributions. 
\st{For $M=[0,T]$ we can use in the proof of Proposition \ref{proposition:cont_modification} the sets
$K_1=K_2=\cdots = M$ and $D_0=A=[0,T] \cap \Q$. Furthermore, in the proof of 
Proposition \ref{proposition:cont_modification}} we note that $Y(t)$ is defined
as the a.s.-limit of $\beta_{t_n}$, where we may take \st{now} $t_n\uparrow t$.
By our assumption $\beta_{t_n}\in \cL_0(\Om^1,\cG_{t_n}^1)$, so that $Y(t)$ is $\cG^1_t$-measurable.
The facts that $(Y(t))_{t \in [0,T]}$ is continuous and a modification of $(\beta_t)_{t \in [0,T]}$
were proven in Proposition \ref{proposition:cont_modification}.
\medskip

(2) Applying \cite[p. 133, step (b) of the proof of Lemma 2.4]{Karatzas:Shreve:91} we can approximate
any predictable process $X\in \cL_0(\Omega_T^0,\cP_T^0)$ by a sequence of continuous adapted processes
$X^n\in \cL_0(\Omega_T^0,\cP_T^0)$ with $d_T^0(X^n,X)\to_n 0$ \st{(first we approximate $X$ by bounded 
processes by truncation, then we use  \cite{Karatzas:Shreve:91})}. Applying \st{part} (1), we find continuous adapted
processes $Y^n$ such that
$\lim_n d^1_T (Y^n,Y)\to_n 0$ for $Y\in \C_T(X)$. Because of
$Y^n \in \cL_0(\Omega_T^1,\cP_T^1)$ we can choose
$Y\in \cL_0(\Omega_T^1,\cP_T^1)$ as well.
\end{proof}
\medskip

The next \st{proposition} is needed later for technical reason\st{s}:
\smallskip

\begin{proposition}\label{prop:properties_CTM}
The following assertions hold true:
\begin{enumerate}[{\rm (1)}]
\item For $M:=\R^d$, $f \in \cL_0(\Om_T^0;C(M))$, $X_1,...,X_d \in \cL_0(\Om_T^0)$,
      $g \in \C_T^{M}(f)$, and $Y_i\in \C_T(X_i)$, \st{$i=1,\ldots,d$,} one has that
      \[ g(\cdot,Y(\cdot)) \in  \C_T \big ( f(\cdot,X(\cdot))\big ). \]
\item Let $S\in \{0,T\}$ \st{and let $M$ be a complete metric space that is locally $\sigma$-compact}. 
      If one has that  $\C_S(L_0(\Om_S^0,\cG^0_S)) \subseteq L_0(\Om_S^1,\cG^1_S)$
      for $\sigma$-algebras $\cG_S^i\subseteq \Sigma_S^i$,
      then
      \[ \C_S^M(L_0(\Om_S^0,\cG^0_S;C(M))) \subseteq L_0(\Om_S^1,\cG^1_S;C(M)). \]
\end{enumerate}
\end{proposition}
\medskip

\begin{proof}
(1a) First note that Lemma \ref{lemma:joint_measurability}  implies
\[ (f(t,X(t)))_{t\in [0,T]} \in \cL_0(\Om_T^0)
   \sptext{1}{and}{1}
   (g(t,Y(t)))_{t\in [0,T]} \in \cL_0(\Om_T^1).\]

(1b) Define for $i=1,...,d$, $n \ge 1, a_k \in \R$ and a Borel-measurable partition $\bigcup_{k=0}^n B_k = \R$
with $B_k\not = \emptyset$ the processes
\equa
A_i(t) &:=& \sum_{k=0}^n a_k 1_{B_k}(X_i(t))
    \sptext{1}{and}{1}
    A(t) = (A_1(t),...,A_d(t)), \\
D_i(t) &:=& \sum_{k=0}^n a_k 1_{B_k}(Y_i(t))
    \sptext{1}{and}{1}
    D(t) = (D_1(t),...,D_d(t)).
\tion
By Proposition \ref{proposition:properties_C_T} we conclude
\equa
&   & \C_T\left ( (f(t,A(t)))_{t\in [0,T]} \right ) \\
& = & \sum_{n_1,\dots,n_d=0}^n 
       \C_T\left ( \Big (f(t,a_{n_1},\dots,a_{n_d})
                               1_{B_{n_1}\times\dots\times B_{n_d}}(X_1(t),\dots,X_d(t))
                       \Big )_{t\in [0,T]} \right ) \\
& = & \sum_{n_1,\dots,n_d=0}^n
            \C_T\left ( \Big (f(t,a_{n_1},\dots,a_{n_d}) \Big )_{t\in [0,T]} \right ) \\
&   & \hspace*{8em}
      \C_T\left ( \Big (1_{B_{n_1}\times\dots\times B_{n_d}}(X_1(t),\dots,X_d(t)) \Big )_{t\in [0,T]}  \right) \\
& \ni & \sum_{n_1,\dots,n_d=0}^n
            (g(t,a_{n_1},\dots,a_{n_d}))_{t\in [0,T]} 
                (1_{B_{n_1}\times\dots\times B_{n_d}}(Y_1(t),\dots,Y_d(t)))_{t\in [0,T]} \\
& = & (g(t,D(t)))_{t\in [0,T]},
\tion
where the multiplication of equivalence classes is defined as usual.
\medskip

(1c) For $L_n(x) := \sum_{k=-4^n}^{4^n-1} \frac{k}{2^n} 1_{[\frac{k}{2^n},\frac{k+1}{2^n})}(x)$
with $x\in\R$ we let
\[ A_i^n(t) :=  L_n(X_i(t))
   \sptext{1}{and}{1}
   D_i^n(t) := L_n(Y_i(t)) \]
so that $d^0_T(A_i^n,X_i)\to_n 0$ for $i=1,\dots,d$.
Proposition \ref{proposition:properties_C_T} yields
$D_i^n \in \C_T(A_i^n)$
and
$d^1_T(D_i^n,Y_i)= d^1_T(\C_T(A_i^n),\C_T(X_i))\to_n 0$.
Because of step (b) and because $\C_T$ is an isometry\st{,} we obtain the estimates
\equa
&   &   d^1_T\left(\C_T((f(t,X(t)))_{t\in [0,T]})  , [(g(t,Y(t)))_{t\in [0,T]}]      \right) \\
&\le&   d^1_T\left(\C_T((f(t,X(t)))_{t\in [0,T]})  , \C_T((f(t,A^n(t)))_{t\in [0,T]}) \right) \\
&   & + d^1_T\left(\C_T((f(t,A^n(t)))_{t\in [0,T]}), [(g(t,D^n(t)))_{t\in [0,T]}]      \right) \\
&   & + d^1_T\left([(g(t,D^n(t)))_{t\in [0,T]}]     , [(g(t,Y(t)))_{t\in [0,T]}]      \right) \\
& = &   d^0_T\left((f(t,X(t)))_{t\in [0,T]}       , (f(t,A^n(t)))_{t\in [0,T]}      \right) \\
&   & \hspace*{8em}      + d^1_T\left((g(t,D^n(t)))_{t\in [0,T]}     , (g(t,Y(t)))_{t\in [0,T]}      \right).
\tion
Because $f(t,A^n(t)) \to_n f(t,X(t))$ for all $(t,\omega)\in \Omega_T^0$\st{,} we have that
\[ d^0_T\left((f(t,X(t)))_{t\in [0,T]}      , (f(t,A^n(t)))_{t\in [0,T]}      \right)\to_n 0. \]
For the last expression we use that $D_i^n \to_n Y_i$ in probability implies the convergence 
$(g(t,D^n(t)))_{t\in [0,T]} \to_n (g(t,Y(t)))_{t\in [0,T]}$ in probability as well.
\bigskip

(2) From  Proposition \ref{proposition:cont_modification} it follows that
the equivalence-class $\C_S^M(f)$ contains a $(\cG^1_S,\U)$-measurable representative.
\end{proof}
\medskip

We conclude with some comments on part (5) of Proposition \ref{proposition:properties_C_T}:
\medskip

\begin{remark}\label{remark:consistent_new}
\hspace{0em} \medskip
\begin{enumerate}[{\rm (1)}]
\item Let $L_0([0,T])$ be equipped with the Borel $\sigma$-algebra, obtained 
      by the metric of
      type (\ref{eqn:metric}) from Section \ref{sec:notation}, and assume
      a $\bigotimes_1^n \cB(L_0([0,T]))$\st{-measurable} $\Psi:(L_0([0,T]))^n \to \R$ such that
      \[ \Phi(f_1,...,f_n) = \Psi([f_1],...,[f_n])
         \sptext{.7}{for}{.7} f_1,...,f_n\in \cL_0([0,T]). \]
      Then $\Phi$ is consistent.

      In fact, the space $L_0([0,T])$ is separable so that its Borel $\sigma$-algebra is generated by the
open balls. We equip $\cL_0([0,T])$ with the smallest $\sigma$-algebra $\mathcal{B}(\cL_0([0,T]))$ such that
$q:\cL_0([0,T])\to L_0([0,T])$ with $q(f):= [f]$ is
measurable. A measurable process $X:[0,T]\times A\to \R$ generates a canonical map
$\hat{X} : A \to \cL_0([0,T])$ that is measurable because
\[ \left \{ \omega \in A : \int_0^T \frac{|X(t,\omega)-f(t)|}{1+|X(t,\omega)-f(t)|} dt
    < \vare \right \} \in \cA \]
for all $\vare>0$ and $f\in \cL_0([0,T])$.  Hence we can finish the proof as the composition of two measurable
maps is measurable.
\bigskip
\item 
For a measurable $\phi : [0,T]\times \R^n \to \R$ and
$g=(g_1,...,g_n) \in (\cL_0([0,T]))^n$ we obtain a  consistent functional
by
\[ \Phi(g):=
   \int_0^T \phi(t, g(t))
                   \chi_{\{ \int_0^T |\phi(t, g(t))| dt < \infty \} } dt. \]
Applying Proposition \ref{proposition:properties_C_T}(5) to the function $\phi(t,x):= |x|^p\wedge L$ with
$L,p\in (0,\infty)$ and $x\in \R^n$, we get that
\[ \int_0^T (|Y(t)|^p \wedge L) dt \in \C_0 \left ( \int_0^T (|X(t)|^p \wedge L) dt \right ) \]
\st{for $X(t)=(X_1(t),...,X_n(t))$ and $Y(t)=(Y_1(t),...,Y_n(t))$, where $X_i \in \cL_0(\Om_T^0)$ and $Y_i \in \C_T(X_i)$ for $i=1,\dots,n$}.
Assuming that  $\int_0^T |X(t,\omega)|^p dt  < \infty$ for all $\omega\in\Omega^0$, we have
\[ \lim_{N \to \infty}  \left ( \int_0^T (|X(t,\omega) |^p \wedge N) dt \right )
  = \int_0^T |X(t,\omega)|^p dt \]
and that $( \int_0^T (|Y(t)|^p \wedge N) dt )_{N\ge 1}$ is a Cauchy sequence in probabi\-lity.
As this sequence converges for all $\omega\in\Omega^1$ (possibly to infinity) we get that
\bigskip

\begin{enumerate}[(a)]
\item $\P^1(\{\omega\in \Omega^1:\int_0^T |Y(t,\omega)|^p dt  < \infty\})=1$,
\item $\int_0^T |Y(t)|^p \chi_{\{ \int_0^T |Y(s)|^p ds < \infty \} } dt
       \in \C_0 \left ( \int_0^T |X(t)|^p  dt \right )$.
\end{enumerate}

\end{enumerate}
\end{remark}


\chapter{Transference of SDEs}
\label{chapter:transference_sde}

In this chapter we apply the method from Chapter \ref{chapter:general_factorization} to the Wiener 
space. The main technical result is Theorem \ref{theorem:change_multi_SDE} below and gives 
a functional map to move a BSDE from one stochastic basis to another one. For this we do not need
any uniqueness of the solution of the BSDE that is moved.
By using an independent copy of the Wiener space we generate in Chapter \ref{chapter:BSDE} below a 
twisted copy of our BSDE by this procedure. The comparison of the original BSDE with the twisted copy will yield 
to the notion of anisotropic smoothness.
Theorem \ref{theorem:change_multi_SDE} might also be exploited to map a BSDE to the canonical path-space 
of continuous functions \st{or from the canonical path-space back to some other space}.


\section{Setting}
\label{sec:setting_transference_sde}

For $i=0,1$ assume complete probability spaces $(\Omega^i,\cF^i,\P^i)$ hosting $d$-dimensional Brownian 
motions 
\[ W^i = (W_t^i)_{t\in [0,T]} = ((W_{t,1}^i,...,W_{t,d}^i)^\top)_{t\in [0,T]},\]
where all paths are assumed to be continuous and $W_0^i\equiv 0$. 
Taking the transposed vector means also that the Brownian motion is considered as column vector.
Define the filtrations $\F^i=(\cF_t^i)_{t\in [0,T]}$ by
$\cF^i_t:=\sigma(W_s^i: s\in [0,t])\vee \cN^i$ with $\cN^i$ being the $\P^i$-null-sets. Replacing $\cF^i$ by $\cF_T^i$ 
we will assume that $\cF^i=\cF_T^i$. Furthermore, we equip \st{$L_2([0,T];\R^d)$} with the orthonormal basis 
$(h_k \otimes e_i)_{k=0,i=1}^{\infty,d}$, where $(h_k)_{k=0}^\infty$ are the $L_2(\st{[0,T]})$-normalized Haar functions 
\footnote{\st{The Haar functions are based on the dyadic intervals $(T\frac{l-1}{2^L},T\frac{l}{2^L}]$ with $L=0,1,2,\ldots$ and $l=1,\ldots,2^L$).}}
and $e_1,...,e_d$ \st{are} the unit vectors of \st{$\R^d$}. 
The corresponding systems $(\xi_k^i)_{k\in I}$ of random variables from Section \ref{sec:C_T_and_properties}
are given by 
\begin{equation}\label{eqn:B}
   \cB^i:=\{ g_{k,j}^i: k\ge 0, j=1,...,d \}
   \sptext{.7}{with}{.7}
   g_{k,j}^i := \int_0^T h_k(t) dW_{t,j}^i,
\end{equation}
where we take as the representative the finite differences of the
$j$-th coordinate of $W^i$ generated by the Haar function $h_k$. Because all paths of $W^i$ are continuous
we have
\[   \sigma (W_{t,j}^i    : t\in [0,T] ; j=1,...,d ) 
   = \sigma (g_{k,j}^i: k=0,1,2,... \mbox{ and } j=1,...,d ). \]
The predictable $\sigma$-algebras on $(\Omega^i,\cF^i,\P^i,\F^i)$ are denoted by $\cP^i$.


\section{Results}

Before we state the main result we need two lemmas.

\begin{lemma}\label{lemma:transference_BM_and adaptedness}
One has $W_{t,j}^1\in \C_0(W_{t,j}^0)$ for $j=1,...,d$ and $t\in [0,T]$ so
that $\C_0(L_0(\Om^0,\cF^0_t,\P^0)) \subseteq L_0(\Om^1,\cF^1_t,\P^1)$ for $t\in [0,T]$.
\end{lemma}

\begin{proof}
The construction and \st{Proposition \ref{proposition:properties_C_T}(1)} imply
$W_{t,j}^1 \in \C_0(W_{t,j}^0)$ whenever $t=Tk/2^n$ with $n=0,1,2,...$ and $k=0,...,2^n$.
For a $t\in (0,T)$ not of this form we find dyadic $t_n\in [0,T]$ with
$t_n\to t$. Hence $W_{t_n,j}^i\to W_{t,j}^i$ for $i=0,1$ in probability
and \st{Proposition \ref{proposition:properties_C_T}(2)} yields $W_{t,j}^1 \in \C_0(W_{t,j}^{0})$.
The second part of the statement is a consequence of the first one.
\end{proof}
\smallskip

\begin{lemma}\label{lemma:change_Ito}
Assume that $K^0 \in \cL_2(\Om_T^0,\cP^0)$. Then, for all $j=1,...,d$,
\[ \int_0^T K_t^1 dW_{t,j}^1 \in \C_0 \left( \int_0^T K_t^0 dW_{t,j}^0 \right), \]
where $K^1\in \C_T(K^0)$ is any $\mathcal{P}^1$-measurable
representative.
\end{lemma}
\medskip

\begin{proof}
Let $L \ge 1$, $0=t_0^L < \dots < t_L^L = T$, and $(\varphi_l^{0,L})_{l =1,\dots,L}$
such that $\varphi_l^{0,L}\in \cL_2(\Om^0,\cF_{t_{l-1}^L}^0)$,
and $K^{0,L}_t:= \sum_{l=1}^L \varphi_l^{0,L} 1_{(t_{l-1}^L,t_{l}^L]}(t)$ such that
 \[ \E^0 \int_0^T |K_t^0-K_t^{0,L}|^2 dt \to 0
    \sptext{1}{as}{1} L\to\infty, \]
see \cite[Lemma 3.2.4]{Karatzas:Shreve:91}.
Using \st{Proposition \ref{proposition:properties_C_T}} and Lemma \ref{lemma:transference_BM_and adaptedness}, letting
$\varphi_l^{1,L} \in \C_0(\varphi^{0,L}_l)$
and $K^{1,L}_t:= \sum_{l=1}^L \varphi_l^{1,L} 1_{(t_{l-1}^L,t_{l}^L]}(t)$, we get
\equa
      \C_0\left( \int_0^T K_t^0 dW_{t,j}^0 \right)
& = & \lim_{L \to \infty} \C_0\left( \sum_{l=1}^L \varphi^{0,L}_l(W_{t_l^L,j}^0-W_{t_{l-1}^L,j}^0) \right) \\
&\ni& \lim_{L \to \infty} \sum_{l=1}^L \varphi^{1,L}_l (W_{t_l^L,j}^1-W_{t_{l-1}^L,j}^1) \\
& = &  \int_0^T K^1_t dW_{t,j}^1
\tion
where the limits are taken in $L_2(\Om^1)$ and $K^1$ is a
$\mathcal{P}^1$-measu\-rable process that satisfies
$\E^1 \int_0^T |K_t^1 - K_t^{1,L} |^2 dt \to_L 0$.
Because of \st{Proposition \ref{proposition:properties_C_T}(7)} we have $K^{1,L}\in \C_T(K^{0,L})$ so that 
$K^1\in \C_T(K^0)$ as well.
\end{proof}

For integers $N,d \ge 1$ and $(\Omega,\F,\P,W)$ being one of the quadruples 
$(\Omega^i,\F^i,\P^i,W^i)$ we consider
\begin{equation}\label{eqn:gBSDE}
 L_t = \xi  + \int_t^T f(s,K_s) ds
	    - \sum_{j=1}^d \int_t^T g_{j}(s,K_s)dW_{s,j}
\end{equation}
where
\begin{enumerate}[$(S1)$]
\item $\xi\in \cL_0(\Om)$,
\item $f$ and $g_{j}$
      are $(\mathcal{P},\mathcal{B}(C(\R^N)))$-measurable,
\item $L=(L_t)_{t\in [0,T]}$, $L_t:\Omega\to \R$, is continuous and
      $\F$-adapted,
\item $K=(K_t)_{t\in [0.T]}$, $K_t:\Omega\to \R^N$, is
      $\cP$-measurable,
\item $ \E\int_0^T
        \left [   | f(t,K_t) |
            + | g(t,K_t) |^2 \right ] dt < \infty$,
\item $(\xi,f,g,K,L,W)$ satisfies (\ref{eqn:gBSDE}) for
      $t \in [0,T]$  $\P$-a.s.
\end{enumerate}
\index{coditions!(S1), (S2),\ldots}
\smallskip

Our main technical result is:
\smallskip

\begin{theorem}
\label{theorem:change_multi_SDE}
Assume that
$(\xi^0,f^0,g^0,K^0,L^0,W^0)$ satisfies $(S1)$-$(S6)$.
Let $\xi^1 \in \C_0(\xi^0)$,
$f^1 \in \C_T^{\R^N}(f^0)$ and
$g_{j}^1 \in \C_T^{\R^{N}}(g^0_{j})$ be
$(\mathcal{P}^1,\mathcal{B}(C(\R^N)))$\st{-measurable},
$L^1\in\C_0^{[0,T]}(L^0)$ be $\F^1$-adapted
and $K^1_l \in \C_T(K^0_l)$ be $\cP^1$-measurable for $l=1,...,N$.
Then $(\xi^1,f^1,g^1,K^1,L^1,W^1)$ satisfies conditions $(S1)$-$(S6)$.
\end{theorem}
\smallskip

\begin{proof}
The existence of suitable measurable representatives can be deduced from a combination of 
Lemma \ref{lemma:transference_BM_and adaptedness} and
\st{Propositions \ref{prop:properties_C_T_measurability} and \ref{prop:properties_CTM}}.
Using \st{Proposition \ref{prop:properties_CTM}(1)} we have for
$\phi\in\{ f,g_j \}$ that
$(\phi^1(t,K^1_t))_{t\in [0,T]}\in \C_T \left ( (\phi^0(t,K^0_t))_{t\in [0,T]}\right )$.
Continuing with Remark \ref{remark:consistent_new}(2) yields that condition (S5) is satisfied for
$f^1(t,K^1_t)$ and $g^1(t,K^1_t)$. For a fixed $t\in [0,T]$ we have
\[      \C_0(L^0_t)
   =  \C_0(\xi^0)  + \C_0\left( \int_t^T f^0(s,K^0_s) ds \right)
     - \sum_{j=1}^d \C_0 \left( \int_t^T g_{j}^0(s,K^0_s) dW_{s,j}^{0} \right). \]
Using Remark \ref{remark:consistent_new}(2) with $\phi(t,x) = x$, we have
 \[  \int_t^T f^1 \big (s,K^1_s \big ) ds
     \in \C_0\left( \int_t^T f^0(s,K^0_s) ds \right)
 \]
for all $t \in [0,T]$.
Similarly Lemma \ref{lemma:change_Ito} gives, for $t \in [0,T]$,
 \[ \int_t^T g_{j}^1(s,K^1_s) dW_{s,j}^{1}
     \in \C_0 \left( \int_t^T g_{j}^0(s,K^0_s) dW_{s,j}^{0} \right). \qedhere
 \]
\end{proof}

Later, in our application we need that certain properties of the gene\-rator transfer. 
For this purpose we use the following 
\medskip

\begin{remark}
\label{rem:quantitative_bounds_transformed_generator}
Assume that $h^0:\Omega^0_T\to C(\R^N)$ is $(\cP^0,\cB(C(\R^N)))$-measurable and $h^1 \in \C^{\R^N}_T(h^0)$
is $(\cP^1,\cB(C(\R^N)))$-measurable. Then the following holds:
\begin{enumerate}[{(1)}]
\item $h^0(\cdot,\cdot,0) \stackrel{d}{=} h^1(\cdot,\cdot,0)$ with respect to
      $\lambda\times\P^0$ and $\lambda\times\P^1$.
\item Given a continuous $H:\R^N\times \R^N \to [0,\infty)$ such that, for all $(t,\omega^0,x_0,x_1)$,
      \[ |h^0(t,\omega^0,x_0) - h^0(t,\omega^0,x_1)| \le H(x_0,x_1), \]
     then we can choose $h^1$ such that,  for all $(t,\omega^1,x_0,x_1)$,
     \[ |h^1(t,\omega^1,x_0) - h^1(t,\omega^1,x_1)| \le H(x_0,x_1). \]
\end{enumerate}
\end{remark}

\begin{proof} (1) follows from Definition \ref{definition:cont_extension} and 
\st{Proposition \ref{proposition:properties_C_T}}.
\medskip

(2) Given $x \in \R^N$, we have by construction $h^1(x) \in \C_T(h^0(x))$, so that 
\[ (h^1(x_0),h^1(x_1))\stackrel{d}{=}(h^0(x_0),h^0(x_1))
   \sptext{1}{for all}{1} x_0,x_1\in \R^N. \]
This implies that 
\[    \|h^1(x_0)-h^1(x_1) \|_{L_\infty(\Omega^1_T)}
   =  \|h^0(x_0)-h^0(x_1) \|_{L_\infty(\Omega^0_T)} 
 \le  H(x_0,x_1). \]
\pagebreak
Hence, letting
\equa
&   & \Om_{T,0}^1 \\
&:= & \{ (t,\om^1) \in [0,T]\times\Omega^1:
      |h^1(t,\om^1,x_0)-h^1(t,\om^1,x_1)| 
      \le H(x_0,x_1) \\
&   & \hspace*{18em} \mbox{ for all } x_0,x_1\in\R^N \} \\
& = & \{ (t,\om^1) \in [0,T]\times\Omega^1:
      |h^1(t,\om^1,x_0)-h^1(t,\om^1,x_1)| 
      \le H(x_0,x_1) \\
&   & \hspace*{18em} \mbox{ for all } x_0,x_1\in\Q^N \},
\tion
we have that $\Om_{T,0}^1 \in \cP^1$ and $\P^1_T(\Om_{T,0}^1)=1$. Setting
\[ \tilde h^1 := \chi_{\Om_{T,0}^1} h^1 \in \C_T^{\R^N}(h^0), \]
we obtain a $(\cP^1,\cB(C(\R^N)))$-measurable map as desired.
\end{proof}


\chapter{Anisotropic Besov Spaces on the Wiener Space}
\label{chapter:Besov_spaces}

In this chapter we introduce anisotropic Besov spaces on the Wiener space by the decoupling
method from Chapter \ref{chapter:general_factorization}. The spaces are
designed such that non-linear conditional expectations, that are generated by BSDEs, map these
spaces into itself (see Chapter \ref{chapter:BSDE}). This fact will provide variational estimates
for solutions to BSDEs. Our approach to define
anisotropic Besov spaces is very flexible as it allows different types of spaces, including the
classical spaces obtained by the real interpolation method.


\section{Classical Besov spaces on the Wiener space} 

In this section we introduce the classical Besov spaces on the Wiener space obtained by the real
interpolation method. To do so we first recall the real interpolation method \st{and the concept 
of Banach space valued random variables.}

\subsection{Real interpolation method}
For detailed information about the real interpolation method the reader is referred (for example) to 
the monographs
\cite{Bennett:Sharpley:88},
\cite{Bergh:Loefstroem:76}, or
\cite{Triebel:78}.
To define the method in the general context, we say that two Banach spaces $(E_0,E_1)$ form a 
{\em compatible couple} 
\index{compatible couple of Banach spaces} 
provided that there is a \st{Banach space $X$} 
such that $E_0$ and $E_1$ are continuously embedded into $X$. \st{By this assumption we can define
$E_0 + E_1 := \{ x = x_0 + x_1, x_0\in E_0, x_1 \in E_1 \}$, where the sum is taken in $X$. Afterwards, 
$X$ can be taken to be $E_0+E_1$ if
\[ \| x \|_{E_0+E_1} := \inf \{ \|x_0\|_{E_0} + \|x_1 \|_{E_1} : x = x_0 + x_1, x_i\in E_i \}, \]
see \cite[Lemma 2.3.1]{Bergh:Loefstroem:76}.}
Assuming \st{additionally} that $E_1$ is continuously embedded into $E_0$,
which is our typical case later, we can take $X=E_0$ itself.

\begin{definition}
\index{$K$-functional}
Given a compatible couple $(E_0,E_1)$ of Banach spaces and $x\in E_0+E_1$ and $t>0$, we define the 
$K$-functional 
\index{$K$-functional}
\index{$K(x,t;E_0,E_1)$}
\[ K(x,t;E_0,E_1)
   := \inf \{ \| x_0\|_{E_0} + t \| x_1\|_{E_1} : x = x_0+ x_1, x_i \in E_i \}. \]
For $\theta\in (0,1)$ and $q\in [1,\infty]$ we let $(E_0,E_1)_{\theta,q}$ be the real interpolation 
space
\index{space!real interpolation space}
\index{space!$(E_0,E_1)_{\theta,q}$}
of all $x \in E_0+ E_1$ such that
\[         \| x\|_{(E_0,E_1)_{\theta,q}} :=
         \left \| t^{-\theta}  K(x,t;E_0,E_1) \right \|_{L_q\left ((0,\infty),\frac{dt}{t} \right )}< \infty. \]
\end{definition}
\medskip

\st{To explain the role of the parameters $(\theta,q)$ let us begin with some properties of the real 
interpolation method:}

\pagebreak

\begin{proposition}[\st{{\cite[Section 1.3.3]{Triebel:78}}}]
\label{proposition:properties_real_interpolation}
Let $(E_0,E_1)$ be a compatible couple of Banach spaces,  $\theta\in (0,1)$, and $q\in [1,\infty]$.
Then one has the following:
\begin{enumerate}
\item $(E_0,E_1)_{\theta,q} = (E_1,E_0)_{1-\theta,q}$ for  $\theta\in (0,1)$ and $q\in [1,\infty]$.
\item $(E_0,E_1)_{\theta,q_0} \subseteq (E_0,E_1)_{\theta,q_1}$ for
      $\theta\in (0,1)$ and $\st{1}\le q_0\le q_1 \le \infty$.
\item If $E_1$ is continuously embedded into $E_0$, then
       \[ (E_0,E_1)_{\theta_0,q_0} \subseteq (E_0,E_1)_{\theta_1,q_1}
          \sptext{1}{for}{1}
          0<\theta_1 < \theta_0 <1 \sptext{.7}{and}{.7} q_0,q_1 \in [1,\infty]. \]
\end{enumerate}
\end{proposition}
\medskip

\st{If we assume that $E_0=L_p$ and that $E_1 \subseteq L_p$ is a subspace that describes certain regularity properties
with $\|\cdot\|_p \le \|\cdot \|_{E_1}$, then we are in the position of Proposition \ref{proposition:properties_real_interpolation}(3).
The parameter $\theta$ becomes the main regularity parameter, and for a fixed $\theta$, the parameter $q$ becomes another regularity parameter, 
that can be interpreted as a fine-tuning parameter. The ordering in Proposition \ref{proposition:properties_real_interpolation}(3) 
is also called {\em lexicographical ordering}.
}
\smallskip

\st{\subsection{Banach space valued random variables}
Given a separable Banach space $X$ and a probability space $(\Omega,\cF,\P)$, a map $F:\Omega \to X$ is 
measurable if it is measurable with respect to $(\cF,\cB(X))$, where  $\cB(X)$ is the Borel $\sigma$-algebra
generated by the norm open sets in $X$. For $p\in (0,\infty]$ we define
\equa
\| F \|_{L_p^X}    &= & \| F \|_{L_p^X(\Omega)} := \big \| \| F \|_X \big \|_p,\\
 \cL_p^X(\Omega) &:= & \{ F:\Omega \to X \mbox{ measurable},  \| F \|_{L_p^X} < \infty \},
\tion
and let $L_p^X(\Omega)$ be the corresponding space of equivalence classes where we identify random variables $F,G:\Omega\to X$
whenever $\P(F=G)=1$.
\index{space!$\cL_p^X$}
\index{space!$L_p^X$}
}

\smallskip
\subsection{Besov spaces on the abstract Wiener space}
We assume a separable Hilbert space $H$, a complete probability space $(\Omega,\cF,\P)$,  and an iso-normal 
family of Gaussian random variables $(g_h)_{h\in H}$, $g_h: \Omega\to \R$, i.e.  
\[ \E g_h = 0 \sptext{1}{and}{1} 
   \E g_h g_k = \langle h,k \rangle
    \sptext{.7}{for all}{.7}
    h,k \in H. \] 
For $\alpha_1,...,\alpha_n\in \R$ and $h_1,...,h_n\in H$ this implies 
\[   \alpha_1 g_{h_1} + \cdots + \alpha_n g_{h_n} \
   = g_{\alpha_1 h_1 + \cdots \alpha_n h_n} \mbox{ a.s.}, \]
that means that $(g_h)_{h\in H}$ is a Gaussian process. W.l.o.g. we may assume that
$\cF$ is the completion of $\sigma( g_h : h\in H)$. 
Let $(\he_n)_{n=0}^\infty$ be the normalised {\em Hermite-polynomials}, 
\index{Hermite polynomial}
i.e. $\he_n: \R\to \R$ with
$\he_0\equiv 1$ and 
\[ \he_n(x) := (-1)^n \frac{1}{\sqrt{n!}} e^{\frac{x^2}{2}} \frac{d^n}{d x^n} e^{-\frac{x^2}{2}}
   \sptext{1}{for}{1} n\ge 1. \]
Letting $\gamma_N$ be the standard Gaussian measure on $\R^N$, the Hermite polynomials form an orthogonal basis
in $L_2(\R,\gamma_1)$. Now we are in a position to define the Wiener chaos:

\begin{definition}
\index{Wiener chaos}
Let $(e_k)_{k\in I}\subseteq H$ be an orthogonal basis of $H$. Given $n\ge 1$, the space 
(of equivalence classes)
\[ \cH_n := \overline{\myspan \left \{ \prod_{k\in I} \he_{n_k}(g_{e_k}): \sum_{k\in I} n_k = n \right \}}
   \subseteq L_2, \]
where the closure is taken in $L_2$,
is the $n$-th Wiener chaos. For $n=0$ we let
$\cH_0$ be space of all equivalence classes that contain a constant.
\end{definition}
\bigskip

The space $\cH_n$ does not depend on the choice of the orthogonal basis $(e_k)_{k\in I}\subseteq H$. Moreover, 
one has the fundamental {\em Wiener chaos expansion}
\index{Wiener chaos expansion}
\[ L_2(\Omega,\cF,\P) = \oplus_{n=0}^\infty \cH_n, \] 
in particular the spaces $\cH_n$ and $\cH_m$  are orthogonal for $n\not = m$.
Letting
\[ \st{P_n} : L_2 \to \cH_n \subseteq L_2 \]
be the orthogonal projection onto the $n$-th chaos, we define \st{the Hilbert space}
\[ \D_{1,2} := \left \{ \xi\in L_2: \| \xi\|_{\D_{1,2}}^2:=\sum_{n=0}^\infty (n + 1) \| \st{P_n} \xi \|_2^2 < \infty \right \}. \]
As Malliavin derivative we take
$D: \D_{1,2}\to \st{L_2^H}$ with
\index{Malliavin derivative}
\[   D \left ( \prod_{k\in I} \he_{n_k}(g_{e_k}) \right )
   := \sum_{l\in I} \prod_{k\not = l} \he_{n_k}(g_{e_k}) h'_{n_l}(g_{e_l}) e_l. \] 
\st{By definition the elements of $\D_{1,2}$ are equivalence classes from $L_2$,
$D$ is defined on equivalence classes and maps to equivalence classes in $L_2^H$. 
When needed, we interpret $DF$ as an element of $\cL_2^H(\Omega)$ or of $\cL_2^{\R^d}(\Omega\times [0,T])$ if 
$H=L_2^{\R^d}([0,T])$.} It is known that
\[ Df(g_{h_1},...,g_{h_n}) = \sum_{k=1}^n \frac{\partial f}{\partial x_k}(g_{h_1},...,g_{h_n}) h_k \]
for (say) $f\in C_b^\infty(\R^n)$ and $h_1,...,h_n\in H$.
If $p\in (2,\infty)$, then we let
\[ \D_{1,p} := \{ \st{\xi}\in \D_{1,2}: \|\st{\xi} \|_{\D_{1,p}}^p:=\| \st{\xi}\|_p^p+ \|D\st{\xi}\|^p_\st{L_p^H} < \infty \} 
   \index{space!$\D_{1,p}$}, \]
which is consistent with the case $p=2$. 
\st{The spaces $\D_{1,p}$ are known to be  Banach spaces (as $p\in[2,\infty)$ one can use the completeness of $\D_{1,2}$ and $L_p$, and 
Fatou's lemma).} Moreover, we set
\begin{equation}\label{eqn:classical_Besov_spaces}
   \B_{p,q}^\theta := (L_p,\D_{1,p})_{\theta,q}. 
   \index{space!$B_{p,q}^\theta$} 
\end{equation}
In the case $\dim(H)=\st{n}$ we identify $L_2$ with $L_2(\R^{\st{n}},\mathcal{B}(\R^\st{n}),\gamma_{\st{n}})$ and use the family
$ g_{(\xi_1,...,\xi_{\st{n}})} : \R^\st{n} \to \R$ given by
\[ g_{(\xi_1,...,\xi_\st{n})}(x_1,...,x_\st{n}) := \xi_1 x_1 + \cdots + \xi_\st{n} x_\st{n}. \]
We denote these particular Besov spaces by  $\B_{p,q}^\theta (\R^\st{n},\gamma_\st{n})$.
\index{space!$\B_{p,q}^\theta (\R^\st{n},\gamma_\st{n})$}
To motivate the decoupling method and the corresponding Besov spaces introduced in Sections
\ref{sec:Besov:setting} and \ref{sec:Besov:definition} below, \st{we describe the spaces 
$\B_{p,q}^\theta (\R^\st{n},\gamma_\st{n})$ by decoupling:}
\medskip

\begin{theorem}[{\cite[Theorem 3.1]{Geiss:Toivola:14}}] 
\label{thm:BesovSpacesNormEquiv}
Let $p\in [2,\infty)$, $\theta\in (0,1)$, $q\in [1,\infty]$\st{,} and $f\in L_p(\R^\st{n},\gamma_\st{n})$. Then
\begin{equation}\label{eqn:thm:BesovSpacesNormEquiv}
  \| f \|_{\B_{p,q}^\theta(\R^\st{n},\gamma_\st{n})} 
 \sim_{c_{(\ref{thm:BesovSpacesNormEquiv})}} 
    \| f \|_p +
        \left \| (1-t)^{-\frac{\theta}{2}} \left \| \st{f(g) - f(t g + \sqrt{1-t^2} g')} \right \|_p 
          \right \|_{L_q\left ([0,1),\frac{dt}{1-t} \right )} 
\end{equation}
where $c_{(\ref{thm:BesovSpacesNormEquiv})} \geq 1$ depends uniquely on $(p,\theta,q)$, and \st{$g$ and $g'$ 
are independent $\R^\st{n}$-valued random variables with law $\gamma_\st{n}$.}
\end{theorem}	
\medskip

\begin{proof}
\st{To derive our formulation from that one in \cite{Geiss:Toivola:14} we consider an $\st{n}$-dimensional Brownian motion
$(W_t)_{t\in [0,1]}$ with respect to a filtration  $(\cF_t)_{t\in [0,1]}$
and notice (cf. Lemma \ref{lemma:exchange-conditional_expectation_new} below) that}
\begin{equation}\label{eqn:Geiss-Toivola_by_decoupling}
   \|f(W_1) - \E(f(W_1)|\cF_t)      \|_p \sim_2 
   \|f(W_1) -    f(W_t+[W'_1-W'_t]) \|_p,
\end{equation}
where $(W'_t)_{t\in [0,1]}$ is an independent copy of  $(W_t)_{t\in [0,1]}$. \st{If we set,}
\[ W_1^{(t,1]}:= W_t+[W'_1-W'_t], \]
\st{then we obtain from \cite[Theorem 3.1]{Geiss:Toivola:14} and \eqref{eqn:Geiss-Toivola_by_decoupling} 
the equivalence}
\begin{equation}\label{eqn:Besov_decoupling}
  \| f \|_{\B_{p,q}^\theta(\R^\st{n},\gamma_\st{n})} 
 \sim
    \| f \|_p +
        \left \| (1-t)^{-\frac{\theta}{2}} \left \| f(W_1)- f(W_1^{(t,1]}) \right \|_p 
          \right \|_{L_q\left ([0,1),\frac{dt}{1-t} \right )}.
\end{equation}
\st{Finally, because $(W_1,W_1^{(t,1]})$ and $(g,t g + \sqrt{1-t^2} g')$ have the same law,
we conclude the proof.}
\end{proof}

\st{Theorem \ref{thm:BesovSpacesNormEquiv}} generalizes results from \cite{Geiss:Hujo:07}.
\st{In the formulation of Theorem \ref{thm:BesovSpacesNormEquiv} we use an \emph{isotropic} decoupling,
which means that the Gaussian structure $g$ is uniformly replaced by $t g + \sqrt{1-t^2} g'$.
Instead, the right-hand side of \eqref{eqn:Besov_decoupling} uses an {\em anisotropic} decoupling
in the larger Wiener space based on the Brownian motion $(W_s)_{s\in [0,1]}$ as the replacement of
$(W_s)_{s\in [0,1]}$ is $(W_s^{(t,1]})_{s\in [0,1]}$,
i.e. only part of the Gaussian structure is decoupled. This anisotropic decoupling in \eqref{eqn:Besov_decoupling}
is the key idea of \cite{GGG:12} to obtain estimates for the variation of BSDEs, an isotropic decoupling  
in the larger Wiener space could not be used in this context as explained in Remark \ref{remark:reason_for_an-isotropic_BSDEs} below.
}


\section{Setting}
\label{sec:Besov:setting} 

For $d\ge 1$ and $T>0$ we fix two standard $d$-dimensional Brownian motions $W=(W_t)_{t\in [0,T]}$ and
$W'=(W'_t)_{t\in [0,T]}$, where all paths are assumed to be continuous \st{with $W_0\equiv 0$ and 
$W'_0\equiv 0$, that are} defined on complete probability spaces
$(\Om,\cF,\P)$ and $(\Om',\cF',\P')$, where $\cF$ and $\cF'$ are the completions
of $\sigma(W_t:t\in [0,T])$ and  $\sigma(W'_t:t\in [0,T])$, respectively. We let
\[ \overline{\Omega} := \Omega\times\Omega', \quad
   \overline{\P} := \P\times \P', \quad
   \overline{\cF} := \overline{\cF\otimes\cF'}^{\overline{\P}} \]
and extend the Brownian motions $W$ and $W'$
canonically to $\Omega\times\Omega'$.
Given a measurable function $\vph:(0,T]\to [0,1]$, we let
\begin{equation}\label{eqn:W_vph}
 W^\vph_t: = \int_0^t [1-\vph(u)^2]^\frac{1}{2} dW_u + \int_0^t \vph(u) dW'_u
\end{equation}
and \st{again assume} continuity for all trajectories and that $W_0^\vph\equiv 0$.
For example, for $0\le a < b \le T$, this definition yields to
\[ W_t^{\chi_{(a,b]}} = \left \{ \begin{array}{rcl}
                        W_t &:& 0\le t \le a \\
                        W_a+W'_t - W'_a &:& a\le t \le b \\
                        W_a + (W'_b-W'_a)+(W_t-W_b) &:& b\le t \le T
                        \end{array} \right . \overline{\P}\mbox{-a.s.}. \]
The process $W^\vph$ is a standard Brownian motion and  $(\cF_t^\vph)_{t \in [0,T]}$ will
denote its $\overline{\P}$-augmented natural filtration, i.e.
\[ \cF_t^\vph := \sigma(W_s^\vph: s\in [0,t]) \vee \overline{\cN}, \]
where $\overline{\cN}$ are the $\overline{\P}$-null-sets from $\overline{\cF}$.
Identifying $a\in [0,1]$ with the function $\vph:(0,T]\to [0,1]$ that is constant $a$, we agree to take the versions
\[ W^0 = W \sptext{1}{and}{1}
   W^1 = W'. \]
To apply the results from Chapter \ref{chapter:transference_sde} we use the pairing between
\[ (\overline{\Omega},\cF^0,\overline{\P},\F^0,W^0,\cB^0)
   \sptext{1}{and}{1}
  (\overline{\Omega},\cF^\vph,\overline{\P},\F^\vph,W^\vph,\cB^\vph), \]
where
$\cF^\psi := \cF^\psi_T$,
$\F^\psi =(\cF_t^\psi)_{t\in [0,T]}$,
and $\cB^\psi$ is defined like in (\ref{eqn:B})
for $\psi\in \{ 0,\varphi \}$.
The corresponding operators $\C_S$ and $\C_S^M$ from Definitions \ref{definition:C_S} and \ref{definition:cont_extension}
are denoted by $\C_S(\varphi)$ and $\C_S^M(\varphi)$, respectively.

\begin{convention}
\label{convention:extension_new}
\hspace*{0em}
\begin{enumerate}[{\rm (1)}]
\item If needed, we extend a random variable $\xi:\Omega \to\R$ to 
      $\widetilde {\xi} :\overline{\Omega}\to \R$ by
      $\widetilde {\xi} (\omega,\omega') := \xi(\omega)$. The extension $\widetilde {\xi}$
      is measurable with respect to $\cF^0$. In this sense we can apply the operator
      \[ \C_0(\vph) : L_0(\overline{\Omega}, \cF^0) \to L_0(\overline{\Omega}, \cF^\vph) \]
      to $\xi$.
      To simplify the notation, $\widetilde{\xi}$ will be usually denoted by $\xi$ as well. 
\item \st{For a random variable $\xi:\Omega\to\R$} we denote by  $\xi^\vph$ the elements of $\C_0(\vph)(\xi)$ 
      and, for $0\le a < b \le T$, by $\xi^{(a,b]}$ the random variable $\xi^{\chi_{(a,b]}}$, i.e.
      \[ \xi^\vph \in \C_0(\vph)(\xi)
         \sptext{1}{and}{1}
         \xi^{(a,b]} := \xi^{\chi_{(a,b]}}. \]
      Because of Lemma \ref{lemma:transference_BM_and adaptedness} this notation is consistent with the definition from 
      {\rm (\ref{eqn:W_vph})}.
\end{enumerate}
\end{convention}


\section{Definition of anisotropic Besov spaces}
\label{sec:Besov:definition}

We start by defining the parameter space
\[ \cD := \{ \psi \in \cL_2((0,T]) : 0 \le \psi \le 1 \} \]
equipped with the pseudo-metric
\footnote{\st{Here we only have a pseudo-metric as we do not work with the equivalence classes.}}
\[ \delta(\vph,\psi) := \| \vph-\psi\|_{L_2((0,T])}. \]
\index{space!$(\cD,\delta)$}
To define our Besov spaces we need some preparations.
\medskip

\begin{lemma}\label{lemma:L_2-continuity}
\hspace*{0em}
\begin{enumerate}
\item For $\vph,\psi\in \cD$, $k\ge 0$, and $i\in \{1,...,d\}$ one
      has that
      \[ \E | g_{k,i}^\vph - g_{k,i}^\psi |^2
	\le 2 \| h_k \|_\infty^2 \int_0^T |\varphi(t)^2 - \psi(t)^2| dt. \]
\item \st{If} $\vph_n,\vph\in \cD$ are such that $\lim_n \delta(\vph_n,\vph)=0$, then
      \[ \lim_n \E | g_{k,i}^{\vph_n} - g_{k,i}^\vph |^2 = 0. \]
\end{enumerate}
\end{lemma}
\medskip

\begin{proof}
\st{(1)} Starting from the corresponding definitions we get
\equa
&   & \E | g_{k,i}^\vph - g_{k,i}^\psi |^2 \\
& = & \E \left |   \int_0^T h_k(t) \langle e_i , dW_t^\vph \rangle
                 - \int_0^T h_k(t) \langle e_i , dW_t^\psi \rangle \right |^2 \\
& = & \E \bigg |   \int_0^T h_k(t) \sqrt{1-\vph(t)^2} \langle e_i , dW_t \rangle
                 + \int_0^T h_k(t) \vph(t)            \langle e_i , dW'_t \rangle \\
&   &            - \int_0^T h_k(t) \sqrt{1-\psi(t)^2} \langle e_i , dW_t \rangle
                 - \int_0^T h_k(t) \psi(t)            \langle e_i , dW'_t \rangle       \bigg |^2 \\
& = &    \int_0^T h_k(t)^2 \left [ \sqrt{1-\vph(t)^2} - \sqrt{1-\psi(t)^2} \right ]^2 dt \\
&   & +  \int_0^T h_k(t)^2 \left [ \vph(t)            - \psi(t)            \right ]^2 dt \\
&\le&    \int_0^T h_k(t)^2 |\vph(t)^2 -\psi(t)^2 | dt
      +  \int_0^T h_k(t)^2 \left [ \vph(t)            - \psi(t)            \right ]^2 dt.
\tion
\st{(2)} If we assume that $\lim_n \delta(\vph_n,\vph)=0$, then $\vph_n\to \vph$ in probability
with respect to the normalized Lebesgue measure on $[0,T]$ and therefore
$|\vph_n^2 - \vph^2|\to 0$ in probability as well. The
boundedness $|\vph_n(t)|\le 1$ and $|\vph(t)|\le 1$ yields to
$\lim_n \int_0^T |\varphi(t)^2 - \varphi_n(t)^2| dt =0$ and we can apply part \st{(1)}.
\end{proof}

\begin{lemma}\label{lemma:exchange_continuity}
Let $p\in (0,\infty)$ and $\xi\in \cL_p\probsp$. Then $\delta(\vph_n,\vph)\to_n 0$ implies that
\[ \lim_n \| \xi^{\vph_n} - \xi^\vph \|_p = 0. \]
\end{lemma}

\begin{proof}
(a) Assume that $\xi$ is bounded. Given $\vare>0$ we find
$N\ge 1$, $f\in C_b(\R^N)$, and $(\gamma_i)_{i=1}^N\subset \cB^0$
such that $\| \xi - f(\gamma_1,...,\gamma_N) \|_p < \vare$. Then, by \st{Proposition \ref{proposition:properties_C_T}},
\equa
&   &  \frac{1}{c_p} \| \xi^{\vph_n} - \xi^{\vph} \|_p \\
&\le& \| \xi^{\vph_n} - f(\gamma_1^{\vph_n},...,\gamma_N^{\vph_n}) \|_p
      + \| f(\gamma_1^{\vph_n},...,\gamma_N^{\vph_n}) - f(\gamma_1^{\vph},...,\gamma_N^{\vph}) \|_p \\
&   & + \| f(\gamma_1^{\vph},...,\gamma_N^{\vph}) - \xi^{\vph} \|_p \\
&\le& 2 \vare + \| f(\gamma_1^{\vph_n},...,\gamma_N^{\vph_n}) - f(\gamma_1^{\vph},...,\gamma_N^{\vph}) \|_p.
\tion
We can conclude by
$\lim_n \| f(\gamma_1^{\vph_n},...,\gamma_N^{\vph_n}) - f(\gamma_1^{\vph},...,\gamma_N^{\vph}) \|_p =0$ which
follows by Lemma \ref{lemma:L_2-continuity}.
\medskip

(b) Assuming a general $\xi\in \cL_p$, we let $\xi^L := (-L)\vee \xi \wedge L$ for $L>0$ and obtain,
again by \st{Proposition \ref{proposition:properties_C_T}},
\equa
&   & \frac{1}{c_p} \| \xi^{\vph_n} - \xi^\vph \|_p \\
&\le& \| \xi^{\vph_n}  - (\xi^L)^{\vph_n} \|_p +  \| (\xi^{L})^{\vph_n} - (\xi^L)^\vph \|_p
                          +  \| (\xi^L)^\vph  - \xi^\vph \|_p \\
& = & 2 \| \xi - \xi^L \|_p +  \| (\xi^L)^{\vph_n} - (\xi^L)^\vph \|_p.
\tion
Given $\vare>0$ we find an $L>0$ such that  $2 \| \xi - \xi^L \|_p\le \vare $, so that
\[ \frac{1}{c_p} \limsup_n \| \xi^{\vph_n} - \xi^\vph \|_p
   \le \vare +  \lim_n \| (\xi^L)^{\vph_n} - (\xi^L)^\vph \|_p
   \le \vare. \]
Because $\vare>0$ was arbitrary, $\lim_n \| \xi^{\vph_n} - \xi^\vph \|_p=0$.
\end{proof}
As a trivial by-product we get that $\xi^\vph=\xi^\psi$ $\overline{\P}$-a.s. if $\vph=\psi$ a.e. Now it is convenient to turn $\cD$ into a complete separable metric space. 
\bigskip

\begin{definition}
\index{space!$(\Delta,\delta)$}
We define the metric space $(\ws,\delta)$ as the equivalence classes of 
the pseudo-metric space $(\cD,\delta)$ with
\[ \cD = \{ \psi \in \cL_2((0,T]) : 0 \le \psi \le 1 \} 
   \sptext{1}{and}{1}
   \delta(\vph,\psi) = \| \vph-\psi\|_{L_2((0,T])}. \]
\end{definition}
\bigskip

Fixing $p\in (0,\infty)$ and $\xi\in\cL_p\probsp$, we obtain a well-defined map
\[ F_{\xi,p}:\ws\to [0,\infty)
   \sptext{1}{by}{1}
   \vph\to \| \xi - \xi^\vph\|_p. \]
Directly from Lemma \ref{lemma:exchange_continuity} we get

\begin{lemma}
For $p\in (0,\infty)$ and $\xi\in\cL_p\probsp$ the map $F_{\xi,p}:\ws\to [0,\infty)$
is continuous.
\end{lemma}

\begin{proof}
For $p\in [1,\infty)$ and $\vph_n\to\vph$ we get that
\[
       \left |  \| \xi - \xi^{\vph_n} \|_p -  \| \xi - \xi^\vph \|_p \right |
  \le  \| \xi^{\vph_n} - \xi^{\vph} \|_p
  \to 0 \]
as $n\to \infty$. In the case $p\in (0,1)$ we use
\[     \bet \E \left [ |\xi^{\vph_n} - \xi|^p -  |\xi^{\vph} - \xi|^p \right ] \rag
   \le \E |\xi^{\vph_n} - \xi^{\vph}|^p.\qedhere \]

\end{proof}

\medskip

\begin{definition}\label{definition:admissible}
\index{functional!admissible}
\index{coditions!(A1), (A2),\ldots}
Let $C^+(\ws)$ be the space of all non-negative continuous functions $F:\ws\to [0,\infty)$.
A functional $\Phi: C^+(\ws)\to [0,\infty]$ is called {\em admissible} provided that
\begin{enumerate}[{\rm({A}1)}]
\item $\Phi (F + G) \le \Phi(F) + \Phi(G)$,
\item $\Phi(\lambda F) = \lambda\Phi(F)$ for $\lambda\ge 0$,
\item $\Phi(F) \le \Phi(G)$ for $0\le F\le G$,
\item $\Phi(F) \le \limsup_n \Phi(F_n)$ for $\sup_{\vph\in \ws}|F_n(\vph)-F(\vph)|\to_n 0$.
\end{enumerate}
\end{definition}

\begin{example}
\label{example:admissible_weight}
Let $A\subseteq \Delta$ be non-empty and let $\alpha:A\to (0,\infty)$ be an arbitrary weight. Then
the functional
\[ \Phi(F) := \sup_{\vph\in A} \frac{F(\vph)}{\alpha(\vph)} \]
is admissible. As (A1)-(A3) are obvious, we only check (A4).
From $F(\vph) \le \limsup_n F_n(\vph)$ we complete the proof by
\[ \frac{F(\vph)}{\alpha(\vph)} \le  \limsup_n \left [ \sup_{\psi\in A} \frac{F_n(\psi)}{\alpha(\psi)} \right ]. \]
\end{example}
 
\begin{definition}
\label{definition:generalised_Besov_space}
\index{space!$\B_p^\Phi$}
For $p\in (0,\infty)$, $\xi\in L_p(\Omega)$, and an admissible $\Phi:C^+(\ws)\to [0,\infty]$
we let $\xi \in \B_p^\Phi$ provided that $\Phi (\vph \to \| \xi - \xi^\vph\|_p) < \infty$ and set
\[    \| \xi \|_{\B_p^\Phi}
   := \left [ \E|\xi |^p +\| \xi\|_{\Phi,p}^p \right ]^\frac{1}{p}
  \sptext{1}{with}{1}
  \| \xi\|_{\Phi,p} :=  \Phi (\vph \to \| \xi - \xi^\vph\|_p ). \]
\end{definition}
\medskip

\begin{proposition}
For $p\in [1,\infty)$ the  space  $\B_p^\Phi$ is a Banach space.
\end{proposition}

\begin{proof}
The norm properties can be easily verified, we only verify the completeness.
Assume a Cauchy sequence $(\xi_n)_{n\ge 1}$, we obtain by the completeness of $L_p$ a
limit $\xi=\lim_n \xi_n$ in $L_p$. To show that the convergence takes place
in $\B_p^\Phi$, let $\vare>0$ and find $n_\vare\ge 1$ such that for all
$m,n\ge n_\vare$ we have that
$\|\xi_n - \xi_m \|_p^p + \Phi (F_{\xi_n-\xi_m,p})^p <\vare^p$ with
$F_{\xi,p}(\vph)= \| \xi-\xi^\vph\|_p$. For all $m,n\ge 1$ we have that
\equa
&   & |F_{\xi_n-\xi_m,p}(\vph)-F_{\xi_n-\xi,p}(\vph)| \\
& = & \bet \| \xi_n-\xi_m - (\xi_n-\xi_m)^\vph \|_p - \| \xi_n-\xi - (\xi_n-\xi)^\vph \|_p \rag \\
&\le& \| \xi -\xi_m - (\xi-\xi_m)^\vph \|_p  \\
&\le& 2 \| \xi -\xi_m \|_p,
\tion
so that assumption (A4) implies for $n\ge n_\vare$ that
\[     \|\xi_n - \xi \|_p^p + \Phi (F_{\xi_n-\xi,p})^p
   \le \lim_m   \|\xi_n - \xi_m \|_p^p + \limsup_m \Phi (F_{\xi_n-\xi_m,p})^p
   \le 2 \varepsilon^p. \qedhere\]
\end{proof}


\section{Connection to real interpolation}

Besov spaces (or fractional order Sobolev spaces) on the Wiener space were studied by various authors, 
see for example  \cite{Watanabe:93}, \cite{Hirsch:99}, \st{\cite[Chapter 8.6]{Bogachev:10},  and \cite{Geiss:Toivola:14}}. In this section we relate our 
\st{definition of Besov spaces to the}
classical Gaussian Besov spaces obtained by the real interpolation method.

\st{
\subsection{The isotropic case}
In our intuition a functional $\Phi:C^+(\Delta) \to [0,\infty]$ is {\em isotropic},
provided that $\Phi$ depends on the constant functions in $C^+(\Delta)$ only. 
Instead of giving a formal definition, we introduce a class of such functionals:}
\medskip

\st{
\begin{definition}
\label{definition:isotropic_Besov_conditional_expectations}
\index{functional!$\Phi^{(K,\mu,q)}$}
Let $\mu$ be a measure on $\cB([0,1])$, $q\in [1,\infty]$, and let $K:[0,1] \to [0,\infty)$ 
be measurable. Let $\vph_r:(0,T]\to \R$, $r\in [0,1]$, denote the constant function
$\vph_r \equiv r$. 
For $F\in C^+(\Delta)$ we define
\[    \Phi^{(K,\mu,q)}(F) 
   := \left \| K(\cdot)\, F(\vph_\cdot) \right \|_{L_q([0,1],\mu)}\in [0,\infty]. \]
\end{definition}
\bigskip
Recalling that 
$W^{\varphi_r}=\sqrt{1-r^2} W +r W'$, we use for any $\xi\in \cL_0(\Omega)$ the notation
$\xi( \sqrt{1-r^2} W +r W'):= \xi^{\varphi_r}$. For $q\in [1,\infty)$
Definition \ref{definition:isotropic_Besov_conditional_expectations} yields to
\[ \| \xi \|_{\B_p^{\Phi^{(K,\mu,q)}}} 
   = \left [ \| \xi\|_p^p 
     +  \left( \int_0^1 \left [ K(r)\| \xi(W)-\xi( \sqrt{1-r^2} W +r W') \|_p \right ]^q d\mu(r) \right )^\frac{p}{q} \right ]^\frac{1}{p}. \]
}
\smallskip

\begin{lemma}
\label{lemma:isotropic_example_is_admissible}
The functional $\Phi^{(\st{K,\mu},q)}$ satisfies 
the conditions {\rm (A1), (A2), (A3)},  and {\rm (A4)}.
\end{lemma}
\st{
\begin{proof}
From the definitions it follows that the map $r\mapsto F(\vph_r)$ is continuous so that
$\Phi^{(K,\mu,q)}(F)$ is well-defined. The assumptions (A1), (A2), and (A3) are immediate. To verify (A4),
we assume $F_n,F:\Delta\to [0,\infty)$ to be continuous with
\[  \sup_{\vph\in \ws}|F_n(\vph)-F(\vph)|\to_n 0. \]
Then (A4) follows from the Fatou property of $L_q([0,1],\mu)$, because
\equa
      \left \| K(\cdot) F(\vph_\cdot) \right \|_{L_q([0,1],\mu)}
& = & \left \| \lim_n K(\cdot) F_n(\vph_\cdot) \right \|_{L_q([0,1],\mu)}  \\
&\le& \liminf_n \left \| K(\cdot) F_n(\vph_\cdot) \right \|_{L_q([0,1],\mu)}.
\tion
\end{proof}

\begin{theorem}
\label{theorem:real_interpolation_Gaussian_Hilbert}
For $\theta\in (0,1)$, $q\in [1,\infty]$, and $p\in [2,\infty)$ one has that
\[        \| \xi \|_{\B_{p,q}^\theta}
   \sim_c \| \xi \|_{\B_p^{\Phi^{(K,\mu,q)}}}
   \]
with $d\mu (r) := \frac{r}{\sqrt{1-r^2}(1-\sqrt{1-r^2})} \chi_{(0,1)}(r) d r$ and 
$K(r) := (1-\sqrt{1-r^2})^{-\frac{\theta}{2}}\chi_{(0,1]}(r)$,
and where $c\ge 1$ depends uniquely on $(p,q,\theta)$.
\end{theorem}

\begin{proof}
After a change of variables the assertion is equivalent to 
\begin{equation}\label{eqn:Besov_decoupling_abstract}
  \| \xi \|_{\B_{p,q}^\theta} 
 \sim_c
    \| \xi \|_p +
        \left \| (1-t)^{-\frac{\theta}{2}} \left \| \xi(W) - \xi(tW+\sqrt{1-t^2}W') \right \|_p 
          \right \|_{L_q\left ([0,1),\frac{dt}{1-t} \right )},
\end{equation}
which is the general form of \eqref{eqn:thm:BesovSpacesNormEquiv}. Because the proof  of  \eqref{eqn:Besov_decoupling} in \cite{Geiss:Toivola:14}
relies on a finite-dimensional argument, we still need to verify \eqref{eqn:Besov_decoupling_abstract}. First we remark the crucial fact, that the multiplicative constant 
in  \eqref{eqn:thm:BesovSpacesNormEquiv} does not depend on the dimension $n$. 
We use the proof of Proposition \ref{prop:approximation-D1p} in the appendix with the supporting Hilbert space 
$H:=L_2^{\R^d}([0,T])$ and take the orthonormal basis from Section \ref{sec:setting_transference_sde}. We enumerate 
this tensor-basis and rename it to $(e_i)_{i=1}^\infty$. The $\sigma$-algebras $\cH_n$ 
are defined as in the proof of Proposition \ref{prop:approximation-D1p}.
We also set $\xi_n := \E(\xi| \cH_n)$ and observe the following:
\medskip

\begin{enumerate}
\item $|K(\xi,t;L_p,\D_{1,p})-K(\eta,t;L_p,\D_{1,p})| \le \| \xi-\eta \|_p$ for $\xi,\eta\in L_p$.
\item $\| (\xi_0)_n \|_p          \le \| \xi_0 \|_p$          for $\xi_0\in L_p$.
\item $\| (\xi_1)_n \|_{\D_{1,p}} \le \| \xi_1 \|_{\D_{1,p}}$ for $\xi_1\in \D_{1,p}$.
\end{enumerate}
\medskip

Assertions (2) and (3) give 
\[ K(\xi_n,t;L_p,\D_{1,p}) \le K(\xi_{n+1},t;L_p,\D_{1,p}) \le  K(\xi,t;L_p,\D_{1,p}). \]
Together with (1) we obtain 
\[  K(\xi_n,t;L_p,\D_{1,p}) \uparrow_n K(\xi,t;L_p,\D_{1,p}), \]
and finally $ \| \xi_n \|_{\B_{p,q}^\theta} \uparrow_n  \| \xi \|_{\B_{p,q}^\theta}$.
On the other side, for $t\in [0,1)$ one has
\begin{equation}\label{eqn:monotonicity}
       \| \xi_n    (W) -     \xi_n    (tW+\sqrt{1-t^2} W') \|_p
   \le \| \xi_{n+1}(W) - \xi_{n+1}(tW+\sqrt{1-t^2} W') \|_p 
\end{equation}
which can be verified as follows: By Doob's factorization theorem we may write
\[ \xi_n = f_n(g_{e_1},\ldots,g_{e_n}) \]
where $f_n:\R^n\to \R$ is a Borel function. 
Then we get (note that $(g_{e_k})_{k=1}^\infty$ are independent standard Gaussian random variables)
for an independent copy $(g'_{e_k})_{k=1}^\infty$ that
\equa
&   &  \| \xi_n(W) -     \xi_n    (tW+\sqrt{1-t^2} W') \|_p^p \\
& = &  \| f_n(g_{e_1},\ldots,g_{e_n}) -  f_n (t g_{e_1}+ \sqrt{1-t^2} g'_{e_1} ,\ldots,t g_{e_n} + \sqrt{1-t^2}g'_{e_n}) \|_p^p \\
& = &  \Big \|    \int_\R f_{n+1}(g_{e_1},\ldots,g_{e_n},\xi)d\gamma_1(\xi) 
         -  \int_\R \int_\R f_{n+1} (t g_{e_1}+ \sqrt{1-t^2} g'_{e_1} ,\ldots,\\
&   & \hspace*{9em} t g_{e_n} + \sqrt{1-t^2}g'_{e_n}, t \xi + \sqrt{1-t^2}\xi') 
                d\gamma_1(\xi) d\gamma_1(\xi') \Big \|_p^p \\
&\le& \int_\R \int_\R  \Big \| f_{n+1}(g_{e_1},\ldots,g_{e_n},\xi) 
         -  f_{n+1} (t g_{e_1}+ \sqrt{1-t^2} g'_{e_1} ,\ldots,\\
&   & \hspace*{9em} t g_{e_n} + \sqrt{1-t^2}g'_{e_n}, t \xi + \sqrt{1-t^2}\xi') 
      \Big \|_p^p d\gamma_1(\xi) d\gamma_1(\xi') \\
& = & \| \xi_{n+1}(W) -     \xi_{n+1}  (tW+\sqrt{1-t^2} W') \|_p^p.
\tion
This proves \eqref{eqn:monotonicity}. Moreover,
\equa
&   & \| [\xi_n(W) -     \xi_n    (tW+\sqrt{1-t^2} W')]- [\xi(W) -     \xi    (tW+\sqrt{1-t^2} W')]  \|_p \\
&\le& 2 \| \xi_n - \xi \|_p \to 0
\tion
as $n\to \infty$. Together with \eqref{eqn:monotonicity} this yields to
\[  \| \xi_n(W) -     \xi_n    (tW+\sqrt{1-t^2} W') \|_p \uparrow_n  \| \xi(W) -     \xi    (tW+\sqrt{1-t^2} W') \|_p \]
and
\begin{multline*}
   \left \| (1-t)^{-\frac{\theta}{2}} \left \| \xi_n(W) - \xi_n(tW+\sqrt{1-t^2}W') \right \|_p  \right \|_{L_q\left ([0,1),\frac{dt}{1-t} \right )} \\
   \uparrow_n 
       \left \| (1-t)^{-\frac{\theta}{2}} \left \| \xi(W) - \xi(tW+\sqrt{1-t^2}W') \right \|_p 
          \right \|_{L_q\left ([0,1),\frac{dt}{1-t} \right )}
\end{multline*}
which completes the proof because for $\xi_n$ the equivalence \eqref{eqn:Besov_decoupling_abstract}
was verified in  \eqref{eqn:thm:BesovSpacesNormEquiv}.
\end{proof}
\smallskip

\begin{remark}
There are other approaches to fractional smoothness on the Wiener space: One can use the Ornstein-Uhlenbeck 
semi-group  which also gives a link to Mehler's formula (see \cite{Hirsch:99}, \cite[Section 8.6]{Bogachev:10}, 
\cite[Section 6.7]{Bergh:Loefstroem:76}). For relations about this approach to 
Theorem \ref{theorem:real_interpolation_Gaussian_Hilbert} the reader is referred to \cite[Remark 3.5]{Geiss:Toivola:14}.
Another approach can be found in \cite[Theorem 13]{Hirsch:99}. It uses an isotropic decoupling as we do, is formulated 
by means of the trace interpolation method (cf. \cite[Section 1.8]{Triebel:78}), corresponds to the special choice  
$p=q$ in our setting, but yields to an alternative expression compared to Theorem \ref{theorem:real_interpolation_Gaussian_Hilbert}.
\end{remark}
}

\subsection{An anisotropic example}

\begin{definition}
\label{definition:anisotropic_Besov_conditional_expectations}
\index{functional!$\Phi_{r_1,...,r_L}^{(\theta_1,q_1),...,(\theta_L,q_L)}$}
For $0=r_0<r_1<\cdots r_L=T$, $\theta_l\in (0,1)$, $q_l\in [1,\infty]$, and
$F\in C^+(\ws)$ we let
\[        \Phi_{r_1,...,r_L}^{(\theta_1,q_1),...,(\theta_L,q_L)}(F)
   := \sup_{l=1,...,L} \left \| (r_l-t)^{-\theta_l/2} F(\chi_{(t,r_l]}) \right \|_{L_{q_l}([r_{l-1},r_l), \frac{dt}{r_l-t})}.
   \]
\end{definition}
\medskip

This functional is admissible:
\medskip

\begin{lemma}
The functional $\Phi_{r_1,...,r_L}^{(\theta_1,q_1),...,(\theta_L,q_L)}$ satisfies 
the conditions {\rm (A1), (A2), (A3)},  and {\rm (A4)}.
\end{lemma}

\begin{proof}
\st{Because the proof is a copy of the proof of Lemma \ref{lemma:isotropic_example_is_admissible}, 
we only check (A4). Assume $F_n,F\in C^+(\Delta)$} with
$\sup_{\vph\in \ws}|F_n(\vph)-F(\vph)|\to_n 0$. 
Then, by the Fatou property of the spaces $L_{q_l}$,
\equa
&   & \sup_{l=1,...,L} \left \| (r_l-t)^{-\theta_l/2} F(\chi_{(t,r_l]}) 
      \right \|_{L_{q_l}([r_{l-1},r_l), \frac{dt}{r_l-t})} \\
& = & \sup_{l=1,...,L} \left \| \lim_n \big [ (r_l-t)^{-\theta_l/2} F_n (\chi_{(t,r_l]}) \big ]
      \right \|_{L_{q_l}([r_{l-1},r_l), \frac{dt}{r_l-t})} \\
&\le&  \sup_{l=1,...,L} \liminf_n \left \| \big [ (r_l-t)^{-\theta_l/2} F_n (\chi_{(t,r_l]}) \big ]
      \right \|_{L_{q_l}([r_{l-1},r_l), \frac{dt}{r_l-t})} \\
&\le& \liminf_n 
      \sup_{l=1,...,L} \left \| \big [ (r_l-t)^{-\theta_l/2} F_n (\chi_{(t,r_l]}) \big ]
      \right \|_{L_{q_l}([r_{l-1},r_l), \frac{dt}{r_l-t})},
\tion
which proves (A4).
\end{proof}

\st{From} \eqref{eqn:Besov_decoupling} we obtain the following result 
about subspaces of $\B_p^{\Phi_{r_1,...,r_L}^{(\theta_1,q_1),...,(\theta_L,q_L)}}$
isomorphic to $\B_{p,q_l}^{\theta_l}(\R^d,\gamma_d)$:

\st{
\begin{proposition}
For a measurable function $f:\R^d \to\R$, 
$p\in [2,\infty)$, $0=r_0<r_1<\cdots r_L=T$, $\theta_l\in (0,1)$, $q_l\in [1,\infty]$, and
$l=1,\ldots,L$ we have
\[ f\in \B_{p,q_l}^{\theta_l}(\R^d,\gamma_d) 
   \sptext{1}{if and only if}{1}
   f\left ( \frac{W_{r_l}-W_{r_{l-1}}}{\sqrt{r_l-r_{l-1}}} \right ) 
   \in\B_p^{\Phi_{r_1,...,r_L}^{(\theta_1,q_1),...,(\theta_L,q_L)}}. \]
\end{proposition}
\bigskip

\begin{remark}
\label{remark:reason_for_an-isotropic_BSDEs}
Assume that  $p\in [2,\infty)$, $\theta_1,\ldots,\theta_L\in (0,1)$, $0=r_0 < r_1 < \cdots < r_L = T$,
and
\[ \xi \in \B_p^{\Phi_{r_1,...,r_L}^{(\theta_1,\infty),...,(\theta_L,\infty)}}. \]
If we let
\[ \cG_a^b := \sigma (W_t : t\in [0,a]) \vee \sigma (W_t - W_b:t\in [b,T])
   \sptext{1}{for}{1}
   0\le a \le b \le T, \]
then Lemma \ref{lemma:exchange-conditional_expectation_new} below implies that there is a constant $c\ge 0$ such that
\begin{equation} \label {eqn:pathwise_smoothness}
   \| \xi - \E(\xi|\cG_t^{r_l}) \|_p \le c (r_l-t)^\frac{\theta_l}{2}
   \sptext{1}{for}{1}
   t\in [r_{l-1},r_l)
   \sptext{.5}{and}{.5}
   l=1,\ldots,L.
\end{equation}
In other words, the conditional expectations $\E(\xi|\cG_t^{r_l})$ converge 
to $\xi$ in $L_p$  with the speed $(r_l-t)^\frac{\theta_l}{2}$ as
$t\uparrow r_l$. If $\theta_l<1$, then one can interpret
this as a singularity of order $1-\theta_l$ at $r_l$ because $(r_l-t)^\frac{1}{2}$ would be the speed for $\xi=W_T$. 
The concept from \eqref{eqn:pathwise_smoothness} was applied in \cite{GGG:12} in the context of BSDEs to obtain
path-dependent variational estimates. The setting of BSDEs, where we have a backward equation
with a pre-given terminal condition, did require the consideration of $\xi - \E(\xi|\cG_t^{r_l})$ rather 
than that one of $\E(\xi|\cF_{r_l})- \E(\xi|\cF_t)$, which could have been a first attempt.
The fact that in \cite{GGG:12} the $\theta_l$ are allowed to be different from each other is one reason to extend the isotropic
spaces $\B_{p,q}^\theta$ from \eqref{eqn:classical_Besov_spaces} to the spaces $\B_p^\Phi$ that might be 
anisotropic.
\end{remark}
}


\section{The space $\B_p^{\Phi_2}$}
\label{sec:Phi_2_Besov_spaces}

In this section we study the space $\B_p^{\Phi_2}$, where the functional
$\Phi_2:C^+(\ws) \to [0,\infty]$ is given by
\[ \Phi_2(F) := \sup_{0\le s < t \le T} \frac{F(\chi_{(s,t]})}{\sqrt{t-s}}. \]
To describe these spaces we let, for $p\in (0,\infty)$ and a measurable $\lambda:[0,T]\times \Omega \to \R^d$ with
$\E \int_0^T|\lambda_s|^2 ds <\infty$,
\index{space!$L_\infty([0,T] ;L_p(\Omega))$}
\index{space!$L_p^*   (\Omega;L_2([0,T]) )$}
\equa
 \lip {\lambda}{p} &:= &  {\rm esssup}_{s\in [0,T]} \||\lambda_s|\|_p, \\
 \lips{\lambda}{p} &:= &  \sup_{0\le a < b \le T} 
                          \left \| \left ( \frac{1}{b-a} \int_a^b |\lambda_s|^2 ds \right )^\frac{1}{2} \right \|_p.
\tion
To shorten the notation we also use $\|\lambda_s\|_p=\| |\lambda_s| \|_p$.
\st{We already introduced $L_q^X(\Omega)$ when $X$ is a separable Banach space. Above we 
use a different notation as we want to avoid a discussion about the separability of $L_p(\Omega)$,
which is not needed here.}
By the Lebesgue differentiation theorem (cf. Lemma \ref{lemma:Lebesgue_differentiation} below)
one has that
\equa
\lip {\lambda}{2} & = & \lips{\lambda}{2}, \\
\lips{\lambda}{p} &\le& \lip {\lambda}{p} \mbox{ for } 2\le p < \infty, \\
\lip {\lambda}{p} &\le& \lips{\lambda}{p} \,\, \mbox{ for } 0< p \le 2.
\tion

The next theorem, the main result of this section, is motivated as follows: If $\xi \in \B_2^{\Phi_2}$,
then $\xi \in \D_{1,2}$ and the quantity $\| \xi \|_{\Phi_2,p}$ enables us to access the Malliavin derivative
of $\xi$ without its explicit computation. As in Corollary \ref{cor:theorem:comparison_psi_phi_3_intro}
of Section \ref{sec:outline} announced, this can be exploited in the context of BSDEs to obtain the differentiability 
of the $Y$-process without differentiating the BSDE.
\bigskip

\begin{theorem}
\label{theorem:Phi_2}
One has that $\B_2^{\Phi_2}\subseteq \D_{1,2}$ and the following assertions hold true:
\begin{enumerate}[{\rm (1)}]
\item For $p\in [2,\infty)$ and $\xi\in \D_{1,2}\cap L_p$ one has
      \[ \| \xi\|_{\Phi_2,p} \sim_{c_{\eqref{theorem:Phi_2}(1),p}} \lips{D\xi}{p}, \]
      where $ c_{\eqref{theorem:Phi_2}(1),p} \ge 1$ depends on $p$ only.
\item For $p\in (1,2)$ and $\xi\in \D_{1,2}$ one has
      \[      \frac{1}{c_{\eqref{theorem:Phi_2}(2),p}} \lip{D\xi}{p}
          \le \| \xi\|_{\Phi_2,p}
          \le c_{\eqref{theorem:Phi_2}(2),p}  \lips{D\xi}{p}, \]
      where $ c_{\eqref{theorem:Phi_2}(2),p}\ge 1$ depends on $p$ only.
\item There is a $\xi\in \D_{1,2}$ such that for all $p\in [1,\infty)$ one has
      $ \xi \in L_p(\Omega)$,
      $D\xi \in L_p(\Omega;L_2([0,T]))$, and $\xi\not\in \B_p^{\Phi_2}$.
      \end{enumerate}
\end{theorem}
\bigskip

In the inequalities of the theorem above the expressions might be infinite.
For the case $p\in (1,2)$ the result is still incomplete. However, if one is interested in
good moment estimates, then the case $p\in [2,\infty)$ seems to be of more interest than the
case $p\in (1,2)$.
To prove Theorem \ref{theorem:Phi_2} we let
\[   \cG_a^b          :=   \sigma (W_t     : t\in [0,a]) \vee
                           \sigma (W_t-W_b : t\in [b,T] ) \]
for $0\le a \le b \le T$ considered as $\sigma$-algebra in $(\Omega,\cF,\P)$.
\index{$\cG_a^b$}
\medskip

\begin{lemma}\label{lemma:exchange-conditional_expectation_new}
For $p \in [1,\infty]$, $\xi=(\xi_1,\ldots,\xi_m)$ with
$\xi_1,\ldots,\xi_m \in \cL_p\probsp$,  a norm  $\|\cdot\|$ on $\R^m$,
and $0 \le s < t \le T$ one has
 \[
	\half \left \|\|\xi - \xi^{(s,t]}\| \right\|_p
\le	\left\|\| \xi- \E^{\cG_s^t} \xi \| \right\|_p
\le	\left\|\| \xi - \xi^{(s,t]} \| \right\|_p,
 \]
where in the first and last expression $\xi_1,\ldots,\xi_m$ are extended to $\overline{\Omega}$
according to Convention \ref{convention:extension_new} and the conditional expectation is taken
coordinate-wise.
\end{lemma}
\medskip

\begin{proof}
By $p\to \infty$ it is sufficient to show the assertion for $p\in [1,\infty)$.
Assuming $p\in [1,\infty)$ it is sufficient to consider $\xi=(\xi_1,\dots,\xi_m)$ of the form
\[ \xi = f(W_{t_1}-W_{t_0},...,W_{t_n}-W_{t_{n-1}}) \]
where $0\le t_0<\cdots t_n \le T$ and $f:\R^{nd}\to \R^m$ is continuous and bounded. W.l.o.g. we can assume
that $s$ and $t$ belong to the partition points. Then
 \[
          \big \|\| \xi - \xi^{(s,t]}\| \big \|_p
\le       \big \|\| \xi - \E^{\cG_s^t} \xi \| \big \|_p + \big \|\| \widetilde{\E^{\cG_s^t} \xi}-\xi^{(s,t]}\|\big \|_p
=        2\big \|\| \xi - \E^{\cG_s^t} \xi \| \big \|_p
 \]
and
\[      \big \| \| \xi - \E^{\cG_s^t} \xi \| \big \|_p
     =  \big \| \| \xi - \E(\xi^{(s,t]}|\sigma(W_r: r\in [0,T])) \| \big \|_p
    \le \big \|\| \xi - \xi^{{(s,t]}} \| \big \|_p. \qedhere\]
\end{proof}
\smallskip

\begin{proof}[Proof of Theorem \ref{theorem:Phi_2}]
(3) For $l\ge 1$ we take disjoint intervals $(s_l,t_l]\subseteq (0,T]$ with $t_l-s_l=T 4^{-l}$ and $t_l<s_{l+1}$.
Define
\[ A_l := l \cos(W_{s_l,1})(W_{t_l,1}-W_{s_l,1})
   \sptext{1}{and}{1}
   \xi := \sum_{l=1}^\infty A_l. \]
The sum converges in any $L_p$, $p\in [1,\infty)$, as
\[     \sum_{l=1}^\infty \| l \cos(W_{s_l,1})(W_{t_l,1}-W_{s_l,1}) \|_p
   \le c_p \sum_{l=1}^\infty l \sqrt{t_l-s_l}
    < \infty, \]
where $c_p:=\|g\|_p$ with $g\sim N(0,1)$. Moreover,
\[ D A_l = l \big [ \cos(W_{s_l,1}) \chi_{(s_l,t_l]} -  \sin (W_{s_l,1})(W_{t_l,1}-W_{s_l,1})  \chi_{(0,s]} \big ] \]
so that
\[     \| DA_l \|_{L_q^{L_2((0,T])}}
   \le l \big [ \sqrt{t_l-s_l} + \sqrt{T} c_q  \sqrt{t_l-s_l} \big ] \]
for $q\in [2,\infty)$. This implies $\xi\in \D_{1,2}$ and $D\xi \in \st{L_q^{L_2([0,T])}(\Omega)}$.
On the other hand,
\begin{multline*}
      \frac{\| A_L- A_L^{(s_L,t_L]} \|_p}{\sqrt{t_L-s_L}} \\
  =   L \frac{\| \cos (W_{s_L,1}) (W_{t_L,1}-W_{s_L,1}-(W')_{t_L,1}+(W')_{s_L,1}) \|_p}{\sqrt{t_L-s_L}} \\
 \ge  L \sqrt{2} c_p \| \cos(W_{s_L,1})\|_p
 \ge  L \kappa_p
\end{multline*}
where $\kappa_p:=\sqrt{2} c_p \inf_{s\in [0,T]}  \| \cos(W_{s,1}) \|_p>0$.
For each $L\ge 1$ this implies
\[      \frac{\|\xi-\xi^{(s_L,t_L]}\|_p}               {\sqrt{t_L-s_L}}
    \ge \frac{\|\sum_{l=1}^L (A_l-A_l^{(s_L,t_L]})\|_p}{\sqrt{t_L-s_L}}
     =  \frac{\| A_L-A_L^{(s_L,t_L]}\|_p}              {\sqrt{t_L-s_L}}
\ge \kappa_p L \]
and therefore $\xi\not\in \B^{\Phi_2}_p$.

(1) and (2) \underline{Step (a)}:  We prove $\B_2^{\Phi_2}\subseteq \D_{1,2}$. Let $0 \le a < b \le T$, and define for
$n \ge 1$ the set
\[ D_n(a,b) := \{ (t_1,...,t_n)\in (0,T]^n :
                  \mbox{ there is a } k \mbox{ such that } t_k \in (a,b] \}. \]
Assume $\xi \in L_2$ with chaos decomposition
\[ \xi = \sum_{n=0}^\infty I_n(f_n) \]
with symmetric $f_n : ((0,T]\times \{ 1,...,d \})^n \to \R$,
cf. \cite[Example 1.1.2]{Nualart:06}. By Lemma \ref{lemma:exchange-conditional_expectation_new}
the condition $\xi\in \B_2^{\Phi_2}$ is equivalent to the condition
\[  \sum_{n=1}^\infty n! \| f_n \chi_{D_n(a,b)} \|_{L_2^n}^2 \le c^2 (b-a) \]
for all $0\le a < b \le T$, where $L_2^n := L_2 (((0,T]\times \{ 1,...,d \})^n,\mu^n)$ with
$\mu:= \lambda \otimes \left ( \sum_{i=1}^d \delta_{\{ i \}}\right )$ and $\lambda$ being the Lebesgue measure.
For $L\ge 1$, $l=1,...,2^L$, and $n\ge 1$ let
\[ D_n^{l,L} := D_n\left ( T \frac{l-1}{2^L}, T \frac{l}{2^L} \right ), \]
so that
\[  \sum_{n=1}^\infty n! \| f_n \chi_{D_n^{l,L}} \|_{L_2^n}^2 \le c^2 2^{-L}. \]
Summing up over $l$ gives for all $N\ge 1$ that
\[  \sum_{n=1}^N n! \sum_{l=1}^{2^L} \| f_n \chi_{D_n^{l,L}} \|_{L_2^n}^2 \le c^2. \]
Let $\Delta_n^L$ be the union of all dyadic half-open cubes
\[ \left ( T \frac{l_1-1}{2^L}, T \frac{l_1}{2^L} \right ] \times \cdots \times
   \left ( T \frac{l_n-1}{2^L}, T \frac{l_n}{2^L} \right ] \]
with $l_1,...,l_n\in \{ 1,...,2^L \}$ pair-wise distinct. Then
\[ \Delta_n^L \subseteq \bigcup_{l=1}^{2^L} D_n^{l,L} \]
and
\[ {\rm card}\big \{ l\in \{1,...,2^L\} : (t_1,...,t_n) \in D_n^{l,L} \big \} = n
   \sptext{1}{for all}{1}
   (t_1,...,t_n)\in  \Delta_n^L.\]
Now we get that
\[     \sum_{n=1}^N n! n \| f_n \chi_{\Delta_n^L} \|_{L_2^n}^2
   \le \sum_{n=1}^N n! \sum_{l=1}^{2^L} \| f_n \chi_{D_n^{l,L}} \|_{L_2^n}^2
   \le c^2. \]
By $L\to\infty$ it follows that
\[     \sum_{n=1}^N n! n \| f_n \|_{L_2^n}^2 \le c^2. \]
Finally, $N\to \infty$ gives $\xi\in \D_{1,2}$.
\medskip

\underline{Step (b)}: Let $\xi \in \D_{1,2}$ with chaos expansion $\xi = \sum_{n=0}^\infty I_n (f_n)$ obtained by 
symmetric $f_n$ and fix $b\in (0,T]$. Consider the processes $(\mu_t^b(i))_{t\in [0,b]}$ from 
Lemma \ref{lemma:BDG-chaos}, so that for $p\in (1,\infty)$ and $a\in [0,b)$ we have that
\begin{equation}\label{eqn:BDG-chaos}
                  \left \| \xi - \xi^{(a,b]} \right \|_p
   \sim_2          \left \| \xi - \E\left (\xi| \cG_a^b \right ) \right \|_p
   \sim_{c_\eqref{lemma:BDG-chaos}} \left \| \left ( \int_a^b
                   |\mu_s^b|^2 ds \right )^\frac{1}{2}  \right \|_p\st{,}
\end{equation}
where Lemma \ref{lemma:exchange-conditional_expectation_new} is exploited in the first equivalence.
For $s\in [a,b]$ and $n\ge 0$ let $t_k^n := a + (k/2^n) (b-a)$ for $k=0,...,2^n$ and
\[ b_n(s) := \inf \left \{ t_k^n : s \le t_k^n,  k=0,...,2^n
                  \right \}. \]
Using (\ref{eqn:BDG-chaos}) and Lemma \ref{lemma:exchange-conditional_expectation_new} we get that
\equa
      c_{(\ref{lemma:BDG-chaos})} \sqrt{b-a}  \| \xi \|_{\Phi_2,p}
&\ge& c_{(\ref{lemma:BDG-chaos})} \left (   \sum_{k=1}^{2^n}
      \left \| \xi - \E\left (\xi| \cG_{t_{k-1}^n}^{t_k^n} \right )
      \right \|_p^2 \right )^\frac{1}{2} \\
&\ge& \left (   \sum_{k=1}^{2^n}
      \left \|  \left ( \int_{(t_{k-1}^n, t_k^n]}
      |\mu_s^{b_n(s)}|^2 ds \right )^\frac{1}{2} \right \|_p^2 \right )^\frac{1}{2}.
\tion
For $p\in [2,\infty)$ we continue by Fatou's lemma to
\equa
      c_{(\ref{lemma:BDG-chaos})} \sqrt{b-a}  \| \xi \|_{\Phi_2,p} 
&\ge& \liminf_n \left (   \sum_{k=1}^{2^n}
      \left \|  \left ( \int_{(t_{k-1}^n,t_k^n]}
      |\mu_s^{b_n(s)}|^2 ds \right )^\frac{1}{2} \right \|_p^2 \right )^\frac{1}{2} \\
&\ge& \liminf_n \left \|  \left ( \sum_{k=1}^{2^n} \int_{(t_{k-1}^n,t_k^n]}
      |\mu_s^{b_n(s)}|^2 ds \right )^\frac{1}{2} \right \|_p \\
& = & \liminf_n \left \|  \left ( \int_{(a,b]} |\mu_s^{b_n(s)}|^2 ds \right )^\frac{1}{2} \right \|_p \\
&\ge& \left \| \liminf_n \left ( \int_{(a,b]}  |\mu_s^{b_n(s)}|^2 ds \right )^\frac{1}{2} \right \|_p \\
&\ge& \left \| \left ( \int_{(a,b]} \liminf_n |\mu_s^{b_n(s)}|^2 ds \right )^\frac{1}{2} \right \|_p.
\tion
For $p\in (1,2)$ we get
\equa
&   & c_{(\ref{lemma:BDG-chaos})} \sqrt{b-a}  \| \xi \|_{\Phi_2,p} \\
&\ge& \liminf_n \left (   \sum_{k=1}^{2^n}
      \left \|  \left ( \int_{(t_{k-1}^n,t_k^n]}
      |\mu_s^{b_n(s)}|^2 ds \right )^\frac{1}{2} \right \|_p^2 \right )^\frac{1}{2} \\
&\ge& \liminf_n  \left ( \int_{(a,b]}
       \| \mu_s^{b_n(s)} \|_p^2 ds \right )^\frac{1}{2}  \\
&\ge&  \left ( \int_{(a,b]}
       \| \liminf_n |\mu_s^{b_n(s)}| \|_p^2 ds \right )^\frac{1}{2}.
\tion
Summarizing, this yields to
\begin{equation}\label{eqn:liminf_mu}
       \| \xi \|_{\Phi_2,p}
   \ge \frac{1}{c_{(\ref{lemma:BDG-chaos})}}
       \left \{ 
               \begin{array}{lcr}
                   \left \| \left ( \frac{1}{b-a} \int_{(a,b]} \liminf_n |\mu_s^{b_n(s)}|^2 ds \right )^\frac{1}{2} \right \|_p \!\!
               &:& \!\!  p \in [2,\infty) \\
                   \left ( \frac{1}{b-a} \int_{(a,b]} \| \liminf_n |\mu_s^{b_n(s)}| \|_p^2 ds \right )^\frac{1}{2} \!\!
               &:& \!\! p \in (1,2)
               \end{array}
       \right . .
\end{equation}
Now we observe that 
\equa
&   & \lim_n \int_{(a,b]} \E | \mu_s^{b_n(s)}(i) - D(s,i)\xi |^2 ds \\
& = & \lim_n \int_{(a,b]} \sum_{k=1}^\infty k^2 (k-1)! \| f_k((s,i),\cdot) (\chi_{((0,s] \cup (b_n(s),T])^{k-1}}-1) \|_{L_2^{k-1}}^2 ds \\
& = & 0
\tion
which follows by dominated convergence since 
\equa
&   & \int_{(a,b]} \sum_{k=1}^\infty k^2 (k-1)! \| f_k((s,i),\cdot) (\chi_{((0,s] \cup     
      (b_n(s),T])^{k-1}}-1) \|_{L_2^{k-1}}^2 ds \\
&\le& \int_{(0,1]} \sum_{k=1}^\infty k^2 (k-1)! \| f_k((s,i),\cdot) \|_{L_2^{k-1}}^2 ds \\
&\le& \| \xi \|^2_{\D_{1,2}}.
\tion
Hence there is a sub-sequence $(n_l)_{l=1}^\infty$ such that $\lim_l \mu_s^{b_{n_l}(s)} = D(s,\cdot)\xi$
$\lambda\otimes \P$ a.e. on $(a,b]\times \Omega$. Observing that (\ref{eqn:liminf_mu}) holds for the sub-sequence $(n_l)_{l=1}^\infty$ without modification as well, the desired lower bounds of 
$\| \xi \|_{\Phi_2,p}$ follow.

\underline{Step (c)}:
We verify the upper bounds of (1) and (2). Let us first assume that $\xi$ is smooth like
in Proposition \ref{prop:approximation-D1p}, i.e. by using the Haar system as orthogonal basis
we may assume that
\[ \xi = f(W_{t_1}-W_{t_0},...,W_{t_n}-W_{t_{n-1}}), \]
where $0=t_0<\cdots < t_n = T$ and $f\in C^\infty(\R^{nd})$ is bounded with bounded derivatives of all orders
(the bounds for the derivatives can depend on their order). By a possible redefinition of $f$ we can assume w.l.o.g. that $a=t_{k}<t_l=b$. We get
\[ D\xi = \sum_{i=1}^n \nabla_i  f(W_{t_1}-W_{t_0},...,W_{t_n}-W_{t_{n-1}})\chi_{(t_{i-1},t_i]}, \]
where $\nabla_i$ is the $d$-dimensional gradient acting on the $i$-block of variables. 
We fix $\xi_1,...,\xi_k,\xi_{l+1},...,\xi_n\in \R^d$ and let
\equa
f_\xi   (\eta_{k+1},...,\eta_l) &:= & f(\xi_1,...,\xi_k,\eta_{k+1},...,\eta_l,\xi_{l+1},...,\xi_n),\\
f_\xi^0 (\eta_{k+1},...,\eta_l) &:= & f_\xi\left (\eta_{k+1}\sqrt{\delta_{k+1}},...,
                                                  \eta_l    \sqrt{\delta_l    } \right )
\tion
for $\delta_i := t_i-t_{i-1}$. Moreover, we note that
\begin{multline*}
      \| \xi - \E (\xi|\cG_a^b) \|_p 
   =  \| f(W_{t_1}-W_{t_0},...,W_{t_n}-W_{t_{n-1}}) - \\
         \E_{k+1}^l f(W_{t_1}-W_{t_0},...,W_{t_n}-W_{t_{n-1}}) \|_p,
\end{multline*}
where $\E_{k+1}^l$ is the expected value with respect to the increments 
\[ (W_{t_{k+1}}-W_{t_k}, ...,W_{t_l}-W_{t_{l-1}}). \]
Applying Lemma \ref{lemma:PDE-Stein} yields to
\equa
&   &      \|       f_\xi (W) - \E  f_\xi (W) \|_p \\
& = & \bigg \|       f_\xi^0 \left (\frac{W_{t_{k+1}}- W_{t_k}}{\sqrt{\delta_{k+1}}},...,\frac{W_{t_l}
                                        - W_{t_{l-1}}}{\sqrt{\delta_l}}\right ) \\
&    &         - \E     f_\xi^0 \left (\frac{W_{t_{k+1}}- W_{t_k}}{\sqrt{\delta_{k+1}}},...,\frac{W_{t_l}
                                        - W_{t_{l-1}}}{\sqrt{\delta_l}}\right )  \bigg \|_p \\
&\le& c_{(\ref{lemma:PDE-Stein})}
      \left \| \left (  \sum_{i=1}^{l-k} \left |\nabla_i f_\xi^0 \left (\frac{W_{t_{k+1}}- W_{t_k}}
      {\sqrt{\delta_{k+1}}},...,\frac{W_{t_l}- W_{t_{l-1}}}{\sqrt{\delta_l}}\right ) \right     |^2  \right )^\frac{1}{2} \right \|_p \\
& = & c_{(\ref{lemma:PDE-Stein})}
      \bigg \| \bigg ( \sum_{i=k+1}^l  \delta_i 
      |\nabla_i f (\xi_1,...,\xi_k,W_{t_{k+1}}- W_{t_k},\\
&   & \hspace*{12em}...,W_{t_l}- W_{t_{l-1}},\xi_{l+1},...,\xi_n) |^2  \bigg )^\frac{1}{2} \bigg \|_p
\tion
and
\equa
&   & \| \xi - \E (\xi|\cG_a^b) \|_p \\
& = & \| f(W_{t_1}-W_{t_0},...,W_{t_n}-W_{t_{n-1}}) - \\
&   & \hspace*{6em}         \E_{k+1}^l f(W_{t_1}-W_{t_0},...,W_{t_n}-W_{t_{n-1}}) \|_p \\
&\le&  c_{(\ref{lemma:PDE-Stein})}
      \left \| \left ( \sum_{i=k+1}^l \delta_i 
      \left |\nabla_i f (W) \right |^2  \right )^\frac{1}{2} \right \|_p \\
& = & c_{(\ref{lemma:PDE-Stein})} 
      \left \| \left ( \sum_{i=k+1}^l \int_{(t_{i-1},t_i]} |\nabla_i f(W)|^2 ds \right )^\frac{1}{2} \right \|_p \\
& = & c_{(\ref{lemma:PDE-Stein})}  \left \| \left ( \int_{(a,b]} |D_s\xi|^2 ds \right )^\frac{1}{2} \right \|_p.
\tion
Now we assume the general case and let $q:= p \vee 2$. Our assumptions in (1) and (2) and under the assumption
that the right-hand sides in (1) and (2) are finite, we have that $\xi\in \D_{1,2}\cap L_q$ and $D\xi\in \st{L_q^{L_2([0,T])}(\Omega)}$. Using Proposition \ref{prop:approximation-D1p} we find smooth $\xi_n$ such that
\[ \xi_n\to \xi \sptext{1}{in}{1} L_q
   \sptext{2}{and}{2} 
   D\xi_n \to D \xi \sptext{1}{in}{1} \st{L_q^{L_2([0,T])}(\Omega)}. \]
Therefore by approximation,
\begin{equation}\label{eqn:variation_upper_bounded_integral_Ds}
  \| \xi - \E (\xi|\cG_a^b) \|_p 
   \le 
   c_{(\ref{lemma:PDE-Stein})}  \left \| \left ( \int_{(a,b]} |D_s\xi|^2 ds \right )^\frac{1}{2} \right \|_p
\end{equation}
under the assumptions (1) and (2). Dividing by $\sqrt{b-a}$ and taking the supremum over
$0\le a < b \le T$ gives the upper bound of $\| \xi \|_{\Phi_2,p}$.
\end{proof}


\section{An embedding theorem for functionals of bounded variation}
\label{sec:embedding_theorem}

We extend the approach from Section \ref{sec:Phi_2_Besov_spaces} to the functionals
$\Phi_r: C^+(\ws)\to [0,\infty]$, $r\in [2,\infty)$, given by
\begin{equation}\label{eqn:Phi_r}
\Phi_r(F) := \sup_{0\le s < t \le T} \frac{F(\chi_{(s,t]})}{(t-s)^\frac{1}{r}}.
\end{equation}
\index{space! $\B_p^{\Phi_r}$}
\begin{definition}
\index{bounded variation}
A Borel function $g:\R\to\R$ is of bounded variation provided that
\[ V(g) := \sup_{-\infty < x_0 < \cdots < x_n < \infty} 
           \sum_{k=1}^n |\st{g}(x_k)-\st{g}(x_{k-1})| < \infty. \]
\end{definition}
It follows from the definition that a function of bounded variation is bounded.
A typical example is $g=\chi_{[K,\infty)}$ where
$V(\chi_{[K,\infty)}) =1$. Now we get the following embedding:

\begin{theorem}
Let $r\in [2,\infty)$, $p\in [1,\infty)$, $\xi \in \B_p^{\Phi_r}$  and $g:\R\to \R$
be of bounded variation. Assume that the  law of $\xi$ has a bounded
density $\rho$. Then, for all $q\in [1,\infty)$, 
\[ g(\xi) \in \B_q^{\Phi_{\tilde r}}
   \sptext{1}{with}{1}
   \tilde r := \frac{p+1}{p} q r. \]
\end{theorem}

\begin{proof}
We use \cite[Theorem 2.4]{Avikainen:09} and get that 
\equa
      \left ( \E |g(\xi) - (g(\xi))^{(s,t]}|^q \right )^\frac{1}{q}
& = & \left ( \E |g(\xi) - (g(\xi^{(s,t]}))|^q \right )^\frac{1}{q} \\
&\le& 3^\frac{q+1}{q} \left (\sup_{x\in \R} \rho(x)\right )^{\frac{1}{q}\frac{p}{p+1}} V(g) 
      \| \xi - \xi^{(s,t]} \|_p^{\frac{1}{q}\frac{p}{p+1}}.
\tion
Dividing by $(t-s)^{\frac{1}{r}\frac{1}{q}\frac{p}{p+1}}$ gives the assertion.
\end{proof}
\bigskip

In view of Example \ref{example:BpPhi2_diffusion} the following limiting case is  
important:

\begin{cor}
If $r\in [2,\infty)$, $\xi \in \bigcap_{p\in [1,\infty)} \B_p^{\Phi_r}$ has a bounded
density, and if $g:\R\to \R$ is of bounded variation, then
\[ g(\xi) \in \bigcap_{q\in [1,\infty)} \bigcap_{\tilde r \in (qr,\infty)} \B_q^{\Phi_{\tilde r}}. \]
\end{cor}


\section{Examples}

\subsection{Forward diffusions}
\label{subsec:forward_diffusions}

The Malliavin differentiability of diffusions is well investigated,
see for example \cite{Nualart:06}. So the following is expected:

\begin{example}
\label{example:BpPhi2_diffusion}
Let 
\[ X_t = x_0 + \int_0^t \sigma(s,X_s) dW_s + \int_0^t b(s,X_s) ds \]
where $\sigma:[0,T]\times \R^d \to \R^d\times \R^d$ and $b:[0,T]\times \R^d\to \R^d$ are bounded and continuous, and satisfy 
\[ | \sigma(t,x)-\sigma(t,y)| + |b(t,x)-b(t,y)| \le L |x-y|
   \sptext{1}{for some}{1}
   L\ge 0. \]
By the proof of \cite[Theorem  3]{GGG:12} this implies for 
$p\in [2,\infty)$ that 
\[ \| X_T^\varphi - X_T \|_p \le c \left ( \int_0^T \varphi(r)^2 dr \right )^\frac{1}{2} \]
with $c=c(p,T,b,\sigma)>0$. In particular, for $X_T=(X_T^1,...,X_T^d)$,
\[ X_T^i \in \bigcap_{p\in (0,\infty)}  \B_p^{\Phi_2} \]
which follows by using $\varphi=\chi_{(s,t]}$ for $0\le s < t \le T$.
\end{example}

\subsection{Local time}

One can look at the fractional smoothness of local times 
$(L_t^\alpha)_{t\in (0,T],\alpha\in \R}$ of a one-dimensional Brownian motion
from different points of view: In 
\cite{Boufoussi:Roynette:93, Boufoussi:96} the smoothness with respect
to the state variable $\alpha$ is under consideration, whereas in 
\cite{Nualart:Vives:92,Airault:Ren:Zhang:00} (with a generalization 
in \cite{Xie:Zhang:07}) the smoothness in $\omega$ for fixed $(t,\alpha)$ is investigated
within the interpolation spaces generated by the Ornstein-Uhlenbeck operator.
Our result complements \cite[Theorem 1]{Airault:Ren:Zhang:00}.
 The smoothness obtained in \cite{Airault:Ren:Zhang:00} 
is strictly smaller than $1/2$. 
In Theorem \ref{thm:smoothness_N_L^L} and 
Corollary \ref{cor:smoothness_local_time} below we show that in the class
of Besov spaces $\B_p^\Phi$ the function $\Phi_r$ defined in 
\eqref{eqn:Phi_r} with $r=4$ is the correct one. Interpreting 
$\Phi_2$ as smoothness 1, the function $\Phi_4$ corresponds to the smoothness
$1/2$. Our approach is similar to \cite{Airault:Ren:Zhang:00}: First we 
investigate the functional $N_T^L$ and then the local
time itself  by Tanaka's formula.
\bigskip

\begin{theorem}\label{thm:smoothness_N_L^L}
\index{local time ! {$N_T^\alpha$}}
Let $d=1$, $\alpha \in \R$, and
\[ N_T^\alpha := \int_{(0,T]} \chi_{\{W_t > \alpha \}} dW_t. \]
Then, for all $p\in (1,\infty)$, one has that
\[ N_T^\alpha \in \B_p^{\Phi_4} \setminus \left [ \bigcup_{r\in [2,4)}\B_p^{\Phi_r} \right ]. \]
\end{theorem}
\medskip

\begin{remark}
\begin{enumerate}
\item The natural range for the parameter $r$ in $\Phi_r$ is $r\in [2,\infty)$ so that we used the condition 
      $r\in [2,4)$ instead of the equivalent one $r\in (0,4)$.
\item It follows that  $N_T^\alpha \in \B_p^{\Phi_4}$ for all $p \in (0,\infty)$, but for the part 
      $N_T^\alpha \not \in \bigcup_{r\in [2,4)}\B_p^{\Phi_r}$ our argument uses $p>1$.
\end{enumerate}
\end{remark}
\smallskip

\begin{proof}[Proof of Theorem \ref{thm:smoothness_N_L^L}]
(a) Denote $\xi = N_T^\alpha$. For the part $N_T^\alpha \in \B_p^{\Phi_4}$
we only need to consider the case $p\in [2,\infty)$ and let $0\le a < b \le T$. Then, a.s., 
\equa
        \xi - \xi^{(a,b]} 
& = &   \int_{(a,b]} \chi_{\{W_t > \alpha \}} dW_t 
      - \int_{(a,b]} \chi_{\{W_t^{(a,b]} > \alpha \}} dW_t^{(a,b]} \\
&   & + \int_{(b,T]} \left [ \chi_{\{W_t > \alpha \}} -  \chi_{\{W_t^{(a,b]} > \alpha \}}
                     \right ] dW_t 
\tion
where we use that  
$(\int_{(0,T]} \chi_{\{W_t > \alpha \}} dW_t)^{(a,b]}=  \int_{(0,T]} \chi_{\{W_t^{(a,b]} > \alpha \}} dW_t^{(a,b]}$ a.s.
which can be proved by approximating the stochastic integral by Riemann sums that converge in $L_2$ towards the original
integral and to apply the $\cdot^{(a,b]}$-operation to the Riemann sums.
Then, by the Burkholder-Davis-Gundy inequalities,
\equa
&   & \|  \xi - \xi^{(a,b]} \|_p \\
&\le& 2 \left \| \int_{(a,b]} \chi_{\{W_t > \alpha \}} dW_t \right \|_p 
       + \left \| \int_{(b,T]} \left [ \chi_{\{W_t > \alpha \}} -  \chi_{\{W_t^{(a,b]} > \alpha \}}
                     \right ] dW_t \right \|_p \\
&\le& \beta_p \left [ 2 \sqrt{b-a} + 
      \left \| \left ( \int_b^T \left| \chi_{\{W_t > \alpha \}} -  \chi_{\{W_t^{(a,b]} > \alpha \}}\right |^2 dt\right )^\frac{1}{2} \right \|_p 
      \right ] \\
& = & \beta_p \left [ 2 \sqrt{b-a} + 
      \left \| \int_b^T \chi_{I_\alpha(W_b,W_b^{(a,b]})}(W_t-W_b) dt \right \|_q ^\frac{1}{2} \right ] \\     
\tion
for $q:=p/2\in [1,\infty)$ and 
\[ I_\alpha(u,v) := (\alpha-u,\alpha-v] 
                        \cup  (\alpha-v,\alpha-u]
                  = (\alpha-\max\{u,v\},\alpha-\min\{u,v\}]. \]
Let $-\infty < A < B < \infty$ and define the function $f_{A,B}: \R\to \R$ by
\[ f_{A,B}(x) := \begin{cases}
                 0 & : x\le A \\
                 (x-A)^2 & : A < x  < B \\
                 (B-A)^2 + 2 (B-A)(x-B) & : B \le x
                 \end{cases}. \]
By the It\^o-Tanaka formula and the occupation times formula 
(see \cite[VI.1.5 and VI.1.6]{Revuz:Yor:99}) we get that, a.s.,
\[ f_{A,B}(W_T-W_b) = f_{A,B}(0) + \int_{(b,T]} f'_{A,B}(W_t-W_b) dW_t
                                   + \int_b^T \chi_{(A,B]}(W_t-W_b) dt. \] 
This gives that
\equa
&   & \left \|  \int_b^T \chi_{(A,B]}(W_t-W_b) dt \right \|_q \\
&\le& \left \| f_{A,B}(W_T-W_b) -  f_{A,B}(0)\right \|_q +  \left \| \int_{(b,T]} f'_{A,B}(W_t-W_b) dW_t \right \|_q\\
&\le& \| f'_{A,B} \|_\infty\left [  \left \| W_T-W_b\right \|_q + \beta_q\sqrt{T-b} \right ]\\
& = & 2(B-A) \left [  \left \| W_T-W_b\right \|_q + \beta_q\sqrt{T-b} \right ] \\
&\le& 4\beta_q\sqrt{T-b} (B-A).
\tion
Then
\equa
      \left \| \int_b^T \chi_{I_\alpha(W_b,W_b^{(a,b]})}(W_t-W_b) dt \right \|_q
&\le& 4\beta_q\sqrt{T-b} 
      \left \| W_b-W_b^{(a,b]} \right \|_q \\
&\le& 8 \beta_q\sqrt{T-b} 
      \left \| W_b-W_a \right \|_q \\ 
&\le& 8 \beta_q^2 \sqrt{T-b} \sqrt{b-a}.
\tion
Summarizing gives
\[     \|  \xi - \xi^{(a,b]} \|_p 
   \le \beta_p \left [ 2 \sqrt{b-a} + 
               (8 \beta_q^2 \sqrt{T-b} \sqrt{b-a})^\frac{1}{2} \right ]. \]
(b) Let us turn to the lower bound, where we assume $p\in (1,\infty)$. We obtain 
\equa
&   & \|  \xi - \xi^{(a,b]} \|_p \\
&\ge& -2 \left \| \int_{(a,b]} \chi_{\{W_t > \alpha \}} dW_t \right \|_p 
       + \left \| \int_{(b,T]} \left [ \chi_{\{W_t > \alpha \}} -  \chi_{\{W_t^{(a,b]} > \alpha \}}
                     \right ] dW_t \right \|_p \\
&\ge& - 2 \beta_p  \sqrt{b-a} + \frac{1}{\beta_p}
      \left \| \int_b^T \chi_{I_\alpha(W_b,W_b^{(a,b]})}(W_t-W_b) dt \right \|_q ^\frac{1}{2}.
\tion
Let $a=0$ and observe that on
$\{ W_b \le - \sqrt{b}, W_b' \ge \sqrt{b} \}$ one has that
\[ I_\alpha(W_b,W_b^{(0,b]}) = (\alpha-W'_b,\alpha-W_b] 
   \supseteq (\alpha-\sqrt{b},\alpha+\sqrt{b}). \]
Therefore, for $b\in (0,T/2)$,
\equa
&   & \|  \xi - \xi^{(0,b]} \|_p \\
&\ge& - 2 \beta_p  \sqrt{b}  
      + \frac{1}{\beta_p} \overline{\P} (W_b \le - \sqrt{b}, W_b' \ge \sqrt{b})^\frac{1}{2q}
      \left \| \int_b^T \chi_{(\alpha - \sqrt{b},\alpha +\sqrt{b})}(W_t-W_b) dt \right \|_q ^\frac{1}{2}\\
& = & - 2 \beta_p  \sqrt{b} 
      + \frac{1}{\beta_p} \overline{\P} (W_1 \le - 1, W_1' \ge 1)^\frac{1}{2q}
      \left \| \int_b^T \chi_{(\alpha - \sqrt{b},\alpha +\sqrt{b})}(W_t-W_b) dt \right \|_q ^\frac{1}{2} \\
&\ge& - 2 \beta_p  \sqrt{b}  
      + \frac{1}{\beta_p} \overline{\P} (W_1 \le - 1, W_1' \ge 1)^\frac{1}{2q}
      \left \| \int_b^{\frac{T}{2}+b} \chi_{(\alpha - \sqrt{b},\alpha +\sqrt{b})}(W_t-W_b) dt \right \|_q ^\frac{1}{2} \\
& = & - 2 \beta_p  \sqrt{b}  
      + \frac{1}{\beta_p} \overline{\P} (W_1 \le - 1, W_1' \ge 1)^\frac{1}{2q}
      \left \| \int_0^{\frac{T}{2}} \chi_{(\alpha - \sqrt{b},\alpha +\sqrt{b})}(W_t) dt \right \|_q ^\frac{1}{2}.
\tion
For the local time of the Brownian motion one has 
(see \cite[Corollary VI.1.9]{Revuz:Yor:99})
\[ L_t^\alpha = \lim_{\varepsilon \downarrow 0} \frac{1}{2\varepsilon} \int_0^t \chi_{(\alpha-\varepsilon,\alpha+\varepsilon)}(W_s) ds 
   \mbox{ a.s.} \]
Therefore, by Fatou's Lemma,
\[ \liminf_{b\downarrow 0} \frac{1}{\sqrt[4]{b}} 
   \left \| \int_0^{\frac{T}{2}} \chi_{(\alpha - \sqrt{b},\alpha +\sqrt{b})}(W_t) dt \right \|_q^\frac{1}{2}
   \ge \sqrt{2 \| L^\alpha_{\frac{T}{2}} \|_q} > 0. \]
\end{proof}

Because the local time $L_t^\alpha$ 
\index{local time ! {$L_T^\alpha$}}
can be expressed by Tanaka's formula by
\[ \frac{1}{2} L_T^\alpha = (W_T-\alpha)^+ - (W_0-\alpha)^+ - N_T^\alpha, \] 
see \cite[Theorem VI.1.2]{Revuz:Yor:99}, and because $(W_T-\alpha)^+\in \B_p^{\Phi_2}$ for all
$p\in (0,\infty)$ we immediately get the following corollary:

\begin{cor}\label{cor:smoothness_local_time}
For all $\alpha\in \R$ and $p\in (1,\infty)$ one has that
\[ L_T^\alpha \in \B_p^{\Phi_4}\setminus \left [ \bigcup_{r\in [2,4)} \B_p^{\Phi_r} \right ]. \]
\end{cor}


\chapter{\st{Continuous} BMO-Martingales}
\label{chapter:BMO}

The theory of BMO-martingales has become an important tool in the investigation of BSDEs.
For an account on this topic the reader is referred, for example, to \cite[p. 298]{Delbaen:Tang:10} 
and \cite[p. 2922]{Briand:Elie:13}. In particular, there are two key ingredients
that we will use as well: Fefferman's inequality and their generalizations, and the notion of reverse H\"older inequalities.  
In addition to these two ingredients, we exploit the concept of sliceable 
BMO-martingales which can be seen as a natural enhancement for the previous techniques.  
Sliceable BMO-martingales were used by Emery \cite{Emery:78,Emery:79} and Schachermayer \cite{Schachermayer:96}, and 
in the context of backward stochastic differential equations by 
Delbaen and Tang \cite{Delbaen:Tang:10} and Frei \cite{Frei:14}.
\smallskip

Throughout this chapter we assume a stochastic basis 
$(A,\cA,\Q,(\cA_t)_{t\in [0,T]})$, $T>0$, where $(A,\cA,\Q)$ is complete, $(\cA_t)_{t\in [0,T]}$ is 
right-continuous, $\cA_0$ contains all null-sets, and $\cA=\cA_T$. 


\section{\st{Continuous BMO-martingales and sliceable numbers}}

First we recall the notion of a BMO-martingale.

\begin{definition}\label{def:bmo}
\index{BMO!$\bmo_2$}
\index{BMO!$\bmo$}
A \st{\em continuous} martingale $M=(M_t)_{t\in [0,T]}$ is of {\it bounded mean oscillation} 
(we write $M\in \bmo$) provided that 
$M_0\equiv 0$ and there
is constant $c>0$ such that for all stopping times $\tau:A\to [0,T]$
one has that
\[ \E (|M_T-M_\tau|^2 | \cA_\tau) \le c^2 \mbox{ a.s.}. \]
We let $\| M\|_{\bmo_2} := \inf c$ where the infimum is taken over all $c>0$ as
above.
\end{definition}
\medskip

Next we introduce the  {\em sliceable numbers}. Without being defined 
explicitly, these numbers have their origin in an article of Schachermayer \cite{Schachermayer:96}
and will be used via Theorem \ref{theorem:scliceable_rh} below in our article.
Before giving the definition let us recall the
notation
\[ ^\sigma M^\tau := (M_{\tau \wedge t} - M_{\sigma \wedge t})_{t\in [0,T]} \]
\st{for random times $\sigma,\tau:A\to [0,T]$ with $0\le \sigma \le \tau\le T$.}
\index{$^\sigma M^\tau$}

\pagebreak

\begin{definition}
\label{definition:sliceable}
\index{sliceable}
\index{BMO!$\sli_N(M)$}
\index{BMO!$\sli_\infty(M)$}
For a \st{(continuous)} \bmo-martingale $M=(M_t)_{t\in [0,T]}$ and $N\ge 1$ we let
\[ \sli_N(M) := \inf \vare, \]
where the infimum is taken over all $\vare>0$ such that there are stopping times
$0=\tau_0\le \tau_1 \le \cdots \le \tau_N = T$ with
\[  \sup_{k=1,...,N} \| ^{\tau_{k-1}}M^{\tau_k} \|_{\bmo_2} \le \varepsilon. \]
Moreover, we let
\[ \sli_\infty(M):= \lim_N \sli_N(M). \]
We call $\sli_N(M)$ the {\em $N$-sliceable number} of $M$.
The \st{(continuous)} \bmo-martingale $M$ is called {\em sliceable} provided that
$\sli_\infty(M) =0$.
\end{definition}
\bigskip

Before we summarise some simple properties of the sliceable numbers we need 
the following lemma:

\begin{lemma}\label{lemma:refinement_bmo}
Let $0\le \sigma \le \tau \le T$ be stopping times and
$0=\tau_0\le \tau_1 \le \cdots \le \tau_N=T$ be a net of stopping times such that for all
$\omega\in A$ there is a $k\in \{1,...,N\}$ such that
\[ (\sigma(\omega),\tau(\omega)] \subseteq (\tau_{k-1}(\omega),\tau_k(\omega)]. \]
Then, for a \st{(continuous)} BMO-martingale $N$, one has that
\[ 
\| ^\sigma N^\tau \|_{\bmo_2} \le \sup_{k=1,...,N}  \| ^{\tau_{k-1}} N^{\tau_k} \|_{\bmo_2}.
\]
\end{lemma}
\begin{proof}
Let $\rho:A\to [0,T]$ be a stopping time. Then
\equa
      \E (|^\sigma N^\tau_T  - ^\sigma N^\tau_\rho|^2 | \cA_\rho )
& = & \E (|N_{\tau\vee \rho} - N_{\sigma\vee \rho}|^2 | \cA_\rho ) \\
& = & \E \left ( \E (|N_{\tau\vee \rho} - N_{\sigma\vee \rho}|^2 | \cA_{\sigma\vee \rho} )
         | \cA_\rho \right ).
\tion
Now we observe that $(\bar\sigma,\bar\tau)$ with 
$\bar\sigma:=\sigma \vee \rho$ and $\bar\tau:=\tau\vee \rho$ shares the same property as
$(\sigma,\tau)$.
We let $A_{N+1} := \{ \bar\sigma =T \}$, and  for $k=1,...,N$, 
\[ A_k:= \{ \bar\sigma \in [\tau_{k-1},\tau_k) \}. \]
This gives a partition $A = \bigcup_{k=1}^{N+1} A_k$ with $A_k \in \cA_{\bar\sigma}$ and
we have that
\equa
      \E (|N_{\bar\tau} - N_{\bar\sigma}|^2 | \cA_{\bar\sigma} ) 
& = & \sum_{k=1}^N  \E (\chi_{A_k}|N_{\bar\tau} - N_{\bar\sigma}|^2 | \cA_{\bar\sigma} ) \\
& = & \sum_{k=1}^N\E 
      (  \chi_{A_k}  |N_{\bar\tau\wedge \tau_k} - N_{\bar\sigma\vee \tau_{k-1}}|^2 |\cA_{\bar\sigma}) \\
& = & \sum_{k=1}^N \E 
      (  \chi_{A_k} \E (   |N_{\bar\tau\wedge \tau_k} - N_{\bar\sigma\vee \tau_{k-1}}|^2 | \cA_{\bar\sigma\vee\tau_{k-1}})|\cA_{\bar\sigma}) \\
&\le& \sup_{k=1,...,N}  \| ^{\tau_{k-1}} N^{\tau_k} \|^2_{\bmo_2}.
\tion
\end{proof}

To formulate the next result we recall the space $\H_\infty$:

\begin{definition}
\label{definition:H_infty}
\index{space!$\H_\infty$}
\index{$(\langle N \rangle_t)_{t\in [0,T]}$}
We let $\H_\infty$ be the space of all \st{continuous} martingales $N=(N_t)_{t\in [0,T]}$ such that $N_0\equiv 0$
and 
\[ \| N \|_{\H_\infty} := \esssup_{\omega\in A} \st{\langle N \rangle}_T(\omega)  < \infty, \]
\st{where $(\langle N \rangle_t)_{t\in [0,T]}$ denotes the quadratic variation of $(N_t)_{t\in [0,T]}$
(see, for example, \cite[Section IV.1]{Revuz:Yor:99}).}

\end{definition}

It follows directly from the definition that $\H_\infty \subseteq \bmo$.
\medskip

\begin{lemma}
\label{lemma:properties_sliceable_numbers}
\index{BMO!$d_{\bmo_2}(M,\H_\infty)$}
For \st{(continuous)} BMO-martingales $M$, $M_1$, and $M_2$ one has the following:
\begin{enumerate}
\item $\sli_1(M) = \| M \|_{\bmo_2}$.
\item $\sli_1(M) \ge \sli_2(M) \ge \cdots \ge 0$.
\item $\sli_{N_1+N_2-1}(M_1+M_2) \le \sli_{N_1}(M_1) + \sli_{N_2}(M_2)$.
\item $\sli_\infty(M) = d_{\bmo_2}(M,\H_\infty)$, where
      \[ d_{\bmo_2}(M,\H_\infty) := \inf \{ \| M-N \|_{\bmo_2} : N\in \H_\infty \}. \]
\end{enumerate}
\end{lemma}

\begin{proof}
(1) and (2) are obvious. To prove (3), we assume $\eta>0$ and find nets
$0=\tau_0^i\le \cdots \le \tau_{N_i}^i = T$ such that
\[  \sup_{k=1,...,N_i} \left \| ^{\tau_{k-1}^i}M_i^{\tau_k^i} \right \|_{\bmo_2} \le \sli_{N_i}(M_i)+\eta. \] 
Now we let $(\sigma_k)_{k=0}^{N_1+N_2-1}$ be the union of 
$(\tau_k^1)_{k=0}^{N_1}$ and $(\tau_k^2)_{k=0}^{N_2}$ and 
define the new net $(\tau_k)_{k=0}^{N_1+N_2-1}$ to be the order statistics
of $(\sigma_k)_{k=0}^{N_1+N_2-1}$, i.e.
\equa
\tau_0 &:=& \min_k \sigma_k = 0, \\
\tau_{N_1+N_2-1} &:=& \max_k \sigma_k =T,\\
\tau_k &:=& \min_{\genfrac[]{0pt}{0}{I\subseteq \{1,...,N_1+N_2-2\}}{{\rm card}(I)=k}}
            \max_{l\in I} \sigma_l.
\tion
With this definition and Lemma \ref{lemma:refinement_bmo} we  get for
$k=1,...,N_1+N_2-1$ that
\equa
&   & \left \| ^{\tau_{k-1}}(M_1+M_2)^{\tau_k} \right \|_{\bmo_2} \\
&\le& \left \|^{\tau_{k-1}} M_1^{\tau_k} \right \|_{\bmo_2} 
       + \left \| ^{\tau_{k-1}}M_2^{\tau_k} \right \|_{\bmo_2} \\
&\le& \sup_{k_1=1,...,N_1} 
      \left \|^{\tau^1_{k_1-1}} M_1^{\tau^1_{k_1}} \right \|_{\bmo_2} 
       + \sup_{k_2=1,...,N_2}
         \left \| ^{\tau^2_{k_2-1}}M_2^{\tau^2_{k_2}} \right \|_{\bmo_2} \\
&\le& \sli_{N_1}(M_1)+ \sli_{N_2}(M_2)+2\eta.
\tion
By $\eta \downarrow 0$ the assertion follows.
\medskip

(4) This part \st{is exactly \cite[Theorem 1.1, Corollary 1.2]{Schachermayer:96},} where we have to observe that our setting of a bounded 
time interval $[0,T]$ does not make a difference compared to $[0,\infty)$ from \cite{Schachermayer:96}.
\end{proof}

The next example will be used later:

\begin{example}
\label{example:abstract_2variation}
\st{For a continuous martingale $M=(M_t)_{t\in [0,T]}$} assume that
\[ \langle M \rangle_t = \int_0^t c_s^2 ds, \quad t\in [0,T], \quad \mbox{a.s.} \]
for some predictable process $c=(c_t)_{t\in [0,T]}$ and that 
there is a $\delta>0$ and some $\kappa \in [0,\infty)$ such that
\[ \left [ \E \left ( \int_\tau^T |c_s|^{2+\delta} ds | \cA_\tau \right ) \right ]^\frac{1}{2+\delta}
   \le \kappa \mbox{ a.s.} \]
for all stopping times $\tau:A\to [0,T]$. Then, for 
$\alpha := \frac{1}{2} - \frac{1}{2+\delta} >0$, and $N\ge 1$,
\[ \sli_N(M) \le \kappa \left ( \frac{T}{N} \right )^\alpha. \]
\end{example}
\begin{proof}
For $0\le a < b \le T$ we simply get a.s. that
\equa
      \left [ \E \left ( \int_\tau^T \chi_{(a,b]}(s)|c_s|^2 ds | \cA_\tau \right ) \right ]^\frac{1}{2}
& = & \left [ \E \left ( \int_{\tau\vee a}^{\tau\vee b} |c_s|^2 ds | \cA_\tau \right ) \right ]^\frac{1}{2} \\
&\le& \left [ \E \left ( \int_{\tau\vee a}^{\tau\vee b} |c_s|^{2+\delta} ds | \cA_\tau \right ) \right ]^\frac{1}{2+\delta} 
      (b-a)^{\alpha} \\
&\le& \left [ \E \left ( \int_\tau^T |c_s|^{2+\delta} ds | \cA_\tau \right ) \right ]^\frac{1}{2+\delta} 
      (b-a)^{\alpha} \\
&\le& \kappa (b-a)^{\alpha}. 
\tion
Choosing an equidistant partition of $[0,T]$ consisting of $N$ intervals
concludes the proof.
\end{proof}


\section{Fefferman's inequality and $\bmo(S_{2\theta})$ spaces}
\label{sec:BMO_S2eta}

In this section we slightly change the point of view: Instead of considering  martingales we think in terms of the
quadratic variation which is more convenient in the sequel for us.
The BMO-spaces, related to backward stochastic differential equations with generators
satisfying condition (B3) of Section \ref{sec:setting_bsdes} below, are defined 
as follows:

\begin{definition}
\label{definition:bmo(S_2theta)}
\index{BMO!$\bmo(S_{2\theta})$}
For $\theta\in (0,\infty)$  and an $\R$-valued progressively measurable process $Z=(Z_t)_{t\in [0,T]}$
with $\E \int_0^T |Z_s|^{2\theta} ds <\infty$ we let
$Z\in \bmo(S_{2\theta})$ provided that
\equa
      \| Z\|_{\bmo(S_{2\theta})}
&:= &  \sup_{t\in [0,T]} \left \| \E \left (  \int_t^T |Z_s|^{2\theta} ds | \cA_t \right ) 
                        \right \|_\infty^\frac{1}{2\theta} <  \infty.
\tion
\end{definition}

\st{
Before we continue  we rephrase Definition \ref{definition:sliceable}
in terms of $\bmo(S_2)$ for the usage in Theorem \ref{theorem:comparison_psi_phi} below:

\begin{definition}
\label{definition:sliceable_S2}
\index{BMO!$\sli_N^{S_2,\A}(c)$}
\index{BMO!$\sli_N^{S_2}(c)$}
For an $\R$-valued progressively measurable process $c=(c_t)_{t\in [0,T]}$, and $N\ge 1$, we let
\[ \sli_N^{S_2}(c) = \sli_N^{S_2,\A}(c) := \inf \vare, \]
where the infimum is taken over all $\vare>0$ such that there are stopping times
$0=\tau_0\le \tau_1 \le \cdots \le \tau_N = T$ with
\[  \sup_{k=1,...,N} \| (\chi_{(\tau_{k-1},\tau_k]}(t) c_t)_{t\in [0,T]} \|_{\bmo(S_2)} \le \varepsilon. \]
\end{definition}}
\medskip

The notation $S_{2\theta}$ \st{in Definition \ref{definition:bmo(S_2theta)}} is chosen to indicate that $\bmo(S_{2\theta})$ deals with
a modified square function.
For $\theta\in (1,\infty)$ we obtain a condition that is stronger than the classical 
\bmo-condition $\| Z\|_{\bmo(S_2)}$, whereas for $\theta\in (0,1)$ the condition gets weaker.
If we define
\[ \st{Y}_t := \int_0^t |Z_s|^{2\theta} ds, \]
then $Z\in \bmo(S_{2\theta})$ if and only if
\[ \sup_\tau \| \E (\st{Y}_T - \st{Y}_\tau | \cF_\tau) \|_\infty < \infty \]
with the supremum taken over all stopping times $\tau:A\to [0,T]$. This opens the path to apply known results  
about \bmo-spaces to the $\bmo(S_{2\theta})$-spaces. Therefore, by the John-Nirenberg Theorem we get that 
$Z\in \bmo(S_{2\theta})$ implies that 
\begin{equation}\label{eqn:JN_BMO(2theta)}
\int_0^T |Z_s|^{2\theta} ds \in L_{\exp}, 
\end{equation}
where the Orlicz space $L_{\exp}$ \index{space!$L_{\exp}$} is given by
\[ \| F\|_{L_{\exp}} := \inf \left \{ \lambda >0 : \E e^{\frac{|F|}{\lambda}} \le 2 \right \} \]
for a random variable $F$ taking values in $\R$, see \cite{Stroock:73,Garsia:73,Kazamaki:94} 
and \cite[Corollary 1]{Geiss:05}.
\medskip

For the next example the notion of a Banach function space is convenient:

\begin{definition}
\index{space!Banach function space}
\index{space!$E_\rho$}
A map $\rho:\cL_0^+(A,\cA,\Q) \to [0,\infty]$ defined on the non-negative random variables 
of $\cL_0(A,\cA,\Q)$ is a Banach function norm \index{space!Banach function norm} 
provided that the following conditions
are satisfied:
\begin{enumerate}
\item $\rho(X)=0$ if and only if $X=0$ a.s.
\item $\rho(X+Y) \le \rho(X) + \rho(Y)$.
\item $\rho(\alpha X) = \alpha \rho(X)$ for $\alpha\ge 0$.
\item $0\le X \le Y $ a.s. implies $\rho(X) \le \rho(Y)$.
\item $0\le X_n \uparrow X$ a.s. implies $\rho(X_n) \uparrow \rho(X)$.
\item $\rho(1) < \infty$.
\item There is a $c>0$ such that $\| X \|_1 \le c \rho(X)$
      for all $X\in \cL_0^+(A,\cA,\Q)$.
\end{enumerate}
The function $\rho$ is extended to 
$\|\cdot\|_{\st{E_\rho}}:L_0(A,\cA,\Q) \to [0,\infty]$ by $\|X\|_{\st{E_\rho}} := \rho(|X|)$ and we let
\[ E_\rho := \{ X\in L_0(A,\cA,\Q) : \| X \|_{\st{E_\rho}} < \infty \}. \]
\end{definition}
\medskip

The spaces $[E_\rho,\|\cdot\|_{\st{E_\rho}}]$ are Banach spaces having the Fatou property,
see \cite[Theorem 1.1.7]{Bennett:Sharpley:88}.

\begin{example}\label{example:weaker_bmo_is_weaker_new}
Let $T=1$ and assume that $\rho:L_0(A,\cA,\Q) \to [0,\infty]$ is a Banach function norm such that
for all $t\in (0,1]$ one has that 
\[ \sup \{ \| X \|_\infty : \|X\|_{\st{E_\rho}} \le 1, X\in L_0(A,\cA_t,\Q) \} = \infty. \]
Then for all $0< \theta < \eta \le 1$ there is a progressively measurable process $Z=(Z_t)_{t\in [0,T]}$ 
such that
\begin{enumerate}
\item $\int_0^T |Z_t|^{2\eta} dt \in E_\rho$,
\item $Z\in  \bmo(S_{2\theta})\setminus \bmo(S_{2\eta})$.
\end{enumerate}
\end{example}
\begin{proof}
Let $t_n := 1 -\frac{1}{2^n}$ for $n\ge 0$, take
\[ 0< \varepsilon < \frac{1}{2\theta} - \frac{1}{2\eta}, \]
and choose, for $n \ge 1$ random variables  $v_n: A\to\R$
that are  $\cA_{t_n}$-measurable and satisfy 
\[ 
   \| v_n \|_\infty = 2^{(n+1) \left [ \frac{1}{2\eta} + \varepsilon\right ]}
   \sptext{1}{but}{1}
   \| |v_n|^{2\eta} \|_{E_\rho} \le 1.
\]
Define the stochastic process $Z=(Z_t)_{t\in [0,1]}$ by
\[ Z_t := \sum_{n=2}^\infty \chi_{\left ( t_{n-1},t_n \right ]}(t) v_{n-1}. \]
Then we get the following three estimates:
\begin{enumerate}
\item For $n\ge 2$ we have 
      \[  \| Z \|_{\bmo(S_{2\eta})} \ge \| v_{n-1} \|_{\infty} (t_n-t_{n-1})^\frac{1}{2\eta}
                                     = 2^{n\left [ \frac{1}{2\eta} + \varepsilon \right ]} 2^{- \frac{n}{2\eta}} \to \infty
                    \]
      as $n\to \infty$, so that $\| Z \|_{\bmo(S_{2\eta})} = \infty$.
\item We have that
      \equa
            \| Z \|_{\bmo(S_{2\theta})}^{2\theta}
      &\le& \sum_{n=2}^\infty  \| v_{n-1}\|_{{\infty}}^{2\theta} (t_n-t_{n-1})  \\
      & = & \sum_{n=2}^\infty 2^{n \left [ \frac{\theta}{\eta} + 2 \vare \theta - 1 \right ]}
        <   \infty. \\
      \tion
\item On the other side, we have that
      \equa
            \left \| \int_0^T |Z_t|^{2\eta} dt \right \|_{E_\rho}
      &\le& \sum_{n=2}^\infty \| |v_{n-1}|^{2\eta} \|_{E_\rho} (t_n-t_{n-1}) \\
      &\le& \sum_{n=2}^\infty 2^{-n}
        <   \infty.
      \tion
\end{enumerate}
\end{proof}
\medskip

In the following we give a version of the generalized Fefferman's inequality 
that can be found in \cite[Lemma 1.6]{Delbaen:Tang:10}, see also \cite[Theorem 1.1]{Banuelos:Bennett:88}. 
Our contribution in Theorem \ref{thm:generalized_fefferman_measure} below 
consists in improving the asymptotic behavior of the constant from $p$ to $\sqrt{p}$
in Corollary \ref{cor:AB_generalized_Fefferman_inequality_new_constant} and
that the left-hand side in \eqref{eqn:thm:generalized_fefferman_measure}
is stronger than the left-hand side in \eqref{eqn:cor:AB_generalized_Fefferman_inequality_new_constant}.
\bigskip

We start with the definition of the $\H_p(S_2)$-spaces and continue by some elementary lemmas.
 
\begin{definition}
\index{space!$\H_p(S_2)$}
\label{definition:H_p(S_2)}
For $p\in (0,\infty]$ we define 
$\H_p(S_2)$ to be the space of all progressively measurable $\R$-valued process $\st{Z}=(\st{Z}_t)_{t\in [0,T]}$ such that
\[ \| \st{Z} \|_{\H_p(S_2)} := \left \| \left ( \int_0^T |\st{Z}_s|^2 ds \right )^\frac{1}{2} \right \|_p < \infty. \]
\end{definition}

\begin{lemma}
\label{lemma:partial_theta_integration}
Let $\mu$ be a finite measure on $\cB([0,T])$ with $\mu([0,T])>0$, $\theta \in (0,1)$, and let
\[ t_0 := \inf \{ t\in [0,T] : \mu ([0,t])>0 \}. \]
Then one has that
\[ \int_{[t_0,T]} \mu ([0,t])^{\theta-1} d\mu(t)
   \le \frac{1}{\theta} \mu ([0,T])^\theta. \]
\end{lemma}
The proof is standard and we leave it to the reader.
\bigskip

\begin{lemma}
\label{lemma:upper_bound_p-th_moment}
Let $p\in (1,\infty)$, $\nu$ be a finite measure on $\cB([0,T])$, and $f:[0,T]\to [0,\infty)$ be non-decreasing
and right-hand side continuous. Then
\[ \left | \int_{[0,T]} f(s) d\nu(s) \right |^p
   \le p \int_{[0,T]} \left | \int_{[0,t]} f(s) d\nu(s) \right |^{p-1} f(t) d\nu(t). \]
\end{lemma}

\begin{proof}
For $n\ge 1$ take the equi-spaced grid
\[ 0=t_0^n < t_1^n < \cdots < t_{2^n}^n =T. \]
By dominated  convergence it is enough to show that
\equa
&   & \left | f(0)  \nu(\{0\})
      + \sum_{i=1}^{2^n}  f(t_i^n) \nu ((t_{i-1}^n,t_i^n]) \right |^p \\
&\le& p \left | f(0) \nu(\{0\}) \right |^{p-1} f(0) \nu(\{0\}) + \\
&   & p \sum_{i=1}^{2^n}
      \left ( f(0) \nu(\{0\}) + \sum_{j=1}^i f(t_j^n) \nu ((t_{j-1}^n,t_j^n]) \right )^{p-1}f(t_i^n) \nu ((t_{i-1}^n,t_i^n]). 
\tion
Setting $a_0 := f(0)  \nu(\{0\}) $ and $a_i:= f(t_i^n)\nu ((t_{i-1}^n,t_i^n]) $ for $i=1,...,2^n$, this reads as
\[    \left | \sum_{i=0}^{2^n} a_i \right |^p 
  \le p \sum_{i=0}^{2^n}
      \left ( \sum_{j=0}^i a_j \right )^{p-1} a_i \]
which follows by writing the left-hand side as telescoping sum and applying 
the mean-value theorem from calculus.
\end{proof}

\begin{remark}
In Lemmas \ref{lemma:partial_theta_integration} and \ref{lemma:upper_bound_p-th_moment}
the factors $1/\theta$ and $p$ are sharp, but one does not have equalities in general
(one can check the cases where $\mu$ and $\nu$ are either the Lebesgue measure or the Dirac measure at (say) 
$T$, and $f\equiv 1$).
\end{remark}

\begin{definition}
\index{adapted random measure}
We call a map 
\[ \nu:A \times \cB([0,T])\to [0,\infty) \]
adapted random measure provided that 
\begin{enumerate}
\item the map $\nu(\omega,\cdot):\cB([0,T])\to [0,\infty)$ is a measure for all $\omega\in A$,
\item the map $\nu(\cdot,[0,t]):A \to \R$ is $\cA_t$-measurable for all
      $t\in [0,T]$.
\end{enumerate}
Moreover, we let
\[ \| \nu \|_{\bmo} := \sup_{t\in [0,T]} \| \E(\nu([t,T])|\cA_t) \|_\infty. \]
\end{definition}
Given any non-negative, non-decreasing, left-hand side continuous, and adapted process
$(f(s))_{s\in [0,T]}$, the process
$(\int_{[0,t]} f(s) d\nu(s))_{t\in [0,T]}$ is well defined, non-decrea\-sing, right-hand side continuous, and adapted.

\begin{lemma}
\label{lemma:fefferman_for_measures}
Let $\nu$ be an adapted random measure,
$(f(t))_{t\in [0,T]}$ be non-decreasing, adapted, non-negative, and left-hand side continuous. Then,
one has that
\[ \E \int_{[0,T]} f(s) d\nu(s) \le \E f(T) \| \nu \|_{\bmo}. \]
\end{lemma}

\begin{proof}
We can assume that $ \E f(T) \| \nu \|_{\bmo}<\infty$, otherwise there is nothing to prove.
Assuming the equi-spaced net 
\[ 0=t_0^n < \cdots < t_{2^n}^n=T, \]
it is sufficient to show that 
\[ \E \sum_{i=0}^{2^n-1} f(t_i^n)\nu ([t_i^n,t_{i+1}^n)) + \E f(T)\nu (\{T\})
   \le \E f(T) \sup_{j=0,...,2^n} \| \E( \nu ([t_j^n,T]) | \cA_{t_j^n} ) \|_{\infty}. \]
Letting $q_i^n:= \nu ([t_i^n,t_{i+1}^n))$ for $i=0,...,2^n-1$, $q_{2^n}^n := \nu (\{T\})$, and 
$a_0^n + \cdots + a_i^n = f(t_{i}^n)$, we get that
\equa
&   & \hspace*{-8em}
      \E \left [ \sum_{i=0}^{2^n-1} f(t_i^n)\nu ([t_i^n,t_{i+1}^n)) + f(T)\nu (\{T\}) \right ] \\
& = & \E \left [ \sum_{0\le j \le i \le 2^n} a_j^n q_i^n \right ] \\
& = & \sum_{j=0}^{2^n} \E \left [ a_j^n \E (q_j^n+\cdots+q_{2^n}^n|\cA_{t_j^n}) \right ] \\     
&\le& \E f(T) \sup_{j=0,...,2^n} \| \E( \nu ([t_j^n,T]) | \cA_{t_j^n} ) \|_{\infty}.
\tion      
\end{proof}

\begin{lemma}
\label{lemma:snd_term}
Let $\mu$ and $\nu$ be adapted random measures such that 
$(\mu(\cdot,[0,t]))_{t\in [0,T]}$ and $(\nu(\cdot,[0,t]))_{t\in [0,T]}$ are continuous processes.
Let $\eta\in (0,1)$, $p\in (1,\infty)$, and assume that
\[ \E  \left | \int_{[0,T]} \mu([0,t])^{\eta} d\nu (t) 
       \right |^p < \infty. \]
Then we have that 
\[  \left \| \int_{[0,T]} \mu([0,t])^{\eta} d\nu(t)
 \right \|_p 
    \le  p  \| \mu([0,T])^{\eta} \|_p
      \| \nu \|_{\bmo}. \]
\end{lemma}

\begin{proof}
For $p\in (1,\infty)$ we use 
Lemma \ref{lemma:upper_bound_p-th_moment} and 
Lemma \ref{lemma:fefferman_for_measures}
to get that
\equa
&   & \E \left |\int_{[0,T]} \mu([0,t])^{\eta} d\nu(t) \right |^p \\
&\le& p \E \int_{[0,T]} \left | \int_{[0,t]} \mu([0,s])^{\eta} d\nu(s)  \right |^{p-1}
      \mu([0,t])^{\eta}   d\nu(t) \\
&\le& p \E \left [ \left | \int_{[0,T]} \mu([0,s])^{\eta} d\nu(s)  \right |^{p-1}
                   \mu([0,T])^{\eta} \right ]
      \| \nu \|_{\bmo} \\
&\le& p \left [ \E  \left | \int_{[0,T]} \mu([0,t])^{\eta} d\nu (t) 
      \right |^p\right ]^\frac{p-1}{p}  \| \mu([0,T])^{\eta} \|_p
      \| \nu \|_{\bmo}.
\tion
Dividing by $\left [ \E  \left | \int_{[0,T]} \mu([0,t])^{\eta} d\nu (t) \right |^p\right ]^\frac{p-1}{p}$
in the case this expression is positive (otherwise there is nothing to prove),
gives the desired inequality.
\end{proof}

\begin{theorem}
\label{thm:generalized_fefferman_measure} 
Let $\mu$, $\nu$ be adapted random measures such that 
$(\mu(\cdot,[0,t]))_{t\in [0,T]}$ and $(\nu(\cdot,[0,t]))_{t\in [0,T]}$ are continuous processes
and $\mu(\omega,\{0\})>0$ for all $\omega\in A$.
Let $p\in (1,\infty)$ and assume that
\[ \E  \left | \int_{[0,T]} \mu([0,t])^{\frac{1}{2}} d\nu (t) 
       \right |^p < \infty. \]
Then we have that 
\begin{equation}\label{eqn:thm:generalized_fefferman_measure}
 \left \| \int_{[0,T]} \mu ([0,t])^{-\frac{1}{2}} d\mu(t) \right \|_p
 \left \| \int_{[0,T]} \mu([0,t])^{\frac{1}{2}} d\nu(t)
 \right \|_p 
\le  2p \| \mu ([0,T]) \|_{\frac{p}{2}}      \| \nu \|_{\bmo}.
\end{equation}
\end{theorem}
\smallskip

\begin{proof}
For $\theta = 1/2$ Lemma \ref{lemma:partial_theta_integration} gives that
\[ \left \| \int_{[0,T]} \mu ([0,t])^{-\frac{1}{2}} d\mu(t) \right \|_p
   \le 2 \left [ \E \mu ([0,T])^{\frac{p}{2}}\right ]^\frac{1}{p}
    =  2 \sqrt{ \| \mu ([0,T]) \|_{\frac{p}{2}}}. \]
Moreover, by Lemma \ref{lemma:snd_term} applied to $\eta=1/2$,
\[      \left \| \int_{[0,T]}\mu([0,t])^{\frac{1}{2}} d\nu(t) \right \|_p 
   \le  p  \| \mu([0,T])^{\frac{1}{2}} \|_p
        \| \nu \|_{\bmo} 
    =   p \sqrt{ \| \mu ([0,T]) \|_{\frac{p}{2}}}  \| \nu \|_{\bmo}. \qedhere\]
\end{proof}
\medskip

\begin{cor}
\label{cor:AB_generalized_Fefferman_inequality_new_constant}
\index{inequality!Fefferman's inequality}
Let $(A_t)_{t\in [0,T]}$ and $(B_t)_{t\in [0,T]}$ be progressively measurable $\R$-valued processes such that
$\E \int_0^T |B_t|^2 dt < \infty$ and $p\in [1,\infty)$. Then one has that
\begin{equation}\label{eqn:cor:AB_generalized_Fefferman_inequality_new_constant}
         \left \| \int_0^T | A_t B_t | dt \right \|_p 
     \le c_{\eqref{cor:AB_generalized_Fefferman_inequality_new_constant},p} \| A \|_{\H_p(S_2)} \| B \|_{\bmo(S_2)}
\end{equation}
with $c_{\eqref{cor:AB_generalized_Fefferman_inequality_new_constant},p}:=\sqrt{2p}$.
\st{If the optimal constant in \eqref{eqn:cor:AB_generalized_Fefferman_inequality_new_constant} is denoted by
$c^{\rm opt}_{\eqref{cor:AB_generalized_Fefferman_inequality_new_constant},p}$, then 
\begin{equation}\label{eqn:cor:AB_generalized_Fefferman_inequality_new_constant_opt}
\inf_{p\in [2,\infty)} \frac{c_{\eqref{cor:AB_generalized_Fefferman_inequality_new_constant},p}^{\rm opt}}{\sqrt{p}}>0,
\end{equation}
i.e. the order of magnitude $\sqrt{p}$ of $c_{\eqref{cor:AB_generalized_Fefferman_inequality_new_constant},p}$ as $p\to \infty$
is optimal.
}
\end{cor}
\medskip

\begin{proof}
\st{(1) We verify the inequality \eqref{eqn:cor:AB_generalized_Fefferman_inequality_new_constant}.}
We first assume that there is a $c>0$ such that $|A_s(\omega)| \le c$ and $|B_s(\omega)| \le c$ 
for all $(s,\omega)\in [0,T]\times A$.
For $\vare >0$ and the Dirac measure $\delta_0$ in $0$ define
\[ d\mu_\vare(t) := \vare d\delta_0(t) + A_t^2 dt 
   \sptext{1}{and}{1}
   d\nu (t) := B_t^2 dt. \]
Then, by Theorem \ref{thm:generalized_fefferman_measure},
\equa
 &   & \left \| \int_0^T | A_t B_t | dt \right \|_p \\
 &\le& \left \| \sqrt{\int_0^T\left |\vare + \int_0^t A_s^2ds\right |^{-\frac{1}{2}}  
                      | A_t |^2 dt} \right \|_{2p}
       \left \| \sqrt{ \int_0^T\left | \vare + \int_0^t A_s^2 ds \right |^\frac{1}{2} | B_t |^2 dt} \right \|_{2p} \\
 &\le& 
       \left \| \int_{[0,T]}|\mu_\vare([0,t])|^{-\frac{1}{2}}  d\mu_\vare(t) \right \|_{p}^\frac{1}{2}
       \left \| \int_{[0,T]} |\mu_\vare([0,t])|^\frac{1}{2} d\nu(t) \right \|_{p}^\frac{1}{2} \\
 &\le& \sqrt{ 2p \| \mu_\vare ([0,T]) \|_{\frac{p}{2}}      \| \nu \|_{\bmo}} \\
 & = & \sqrt{2p} \left \| \left (\vare + \int_0^T |A_t|^2 dt \right )^\frac{1}{2} \right \|_p 
        \sup_{t\in [0,T]} \left \| \E \left ( \int_t^T |B_s|^2 ds | \cA_t \right ) \right \|_\infty^\frac{1}{2}.
 \tion
By $\vare\downarrow 0$ we get that
\[ \left \| \int_0^T | A_s B_s | ds \right \|_p \le \sqrt{2p} \left \| \left (\int_0^T |A_t|^2 dt \right )^\frac{1}{2} \right \|_p 
        \sup_{t\in [0,T]} \left \| \E \left ( \int_t^T |B_s|^2 ds | \cA_t \right ) \right \|_\infty^\frac{1}{2} \]
whenever $|A_s(\omega)| \le c$ and $|B_s(\omega)| \le c$ for all $(s,\omega)\in [0,T]\times A$.
By monotone convergence we can omit the restriction on $A$ first, and finally we can do so for $B$ as well.
\medskip

\st{(2) We verify the inequality \eqref{eqn:cor:AB_generalized_Fefferman_inequality_new_constant_opt} and assume w.lo.g. that $T=1$
(otherwise we apply a re-scaling).
Let $A=B\in \bmo(S_2)$ and set $C_t := A_t^2$. Then  
\[       \left \| \int_0^T C_t dt \right \|_p 
     \le c^{\rm opt}_{\eqref{cor:AB_generalized_Fefferman_inequality_new_constant},p} \left \| \int_0^T C_t dt \right \|_{\frac{p}{2}}^\frac{1}{2}   
         \sup_{t\in [0,T]} \left \| \E \left ( \int_t^T C_s ds | \cA_t \right ) \right \|_\infty^\frac{1}{2}. \] 
Assume that we can choose $A=B\in \bmo(S_2)$ with
\begin{enumerate}
\item $M:=\sup_{t\in [0,T]} \left \| \E \left ( \int_t^T C_s ds | \cA_t \right ) \right \|_\infty < \infty$
\item and such that there exists a $c\in [1,\infty)$ such that for all $p\in [1,\infty)$ one has
      \[  \frac{p}{c}  \le  \left \| \int_0^T C_t dt \right \|_p \le c p. \]
\end{enumerate}
Then we would get that $\frac{p}{c} \le c^{\rm opt}_{\eqref{cor:AB_generalized_Fefferman_inequality_new_constant},p} \sqrt{c \frac{p}{2}}\sqrt{M}$ for $p\in [2,\infty)$ and therefore
\[ c_{\eqref{cor:AB_generalized_Fefferman_inequality_new_constant},p}^{\rm opt} \ge \frac{\sqrt{p}}{c \sqrt{c M}} \sqrt{2}. \]
\smallskip

Now we construct the process $C$. The probability space $(A,\cA,\Q)$ we define by
$A:=\{1,2,3,\ldots\}$ and $\Q(\{k\}) := 2^{-k}$ for $k\ge 1$,
where $\cA$ is the system of all subsets of $A$.
The right continuous filtration is constructed    in two steps. First we set
$\cA_{t_0} := \{\emptyset,A \}$ and
$\cA_{t_l} := \sigma (\{1\},\ldots,\{l\})$ for $l\ge 1$,
where $t_l := 1-2^{-l}$  for $l=0,1,2,\ldots$. Then this is extended 
to $(\cA_t)_{t\in [0,1]}$ by
$\cA_t     := \cA_{t_l}$ if $t\in [t_l,t_{l+1})$.
Finally we define the 
progressively measurable process $(C_t)_{t\in [0,1]}$ by
\[ C_t := \sum_{l=1}^\infty 2^l 1_{(t_{l-1},t_l]}(t) 1_{\{l,l+1,\ldots\}}
   \sptext{1}{for}{1}
   t\in [0,1]. \] 
For $\xi:= \int_0^1 C_t dt$ one gets 
$\Q (|\xi| = k) = \frac{1}{2^k}$ for $k=1,2,\ldots$ so that (by a standard computation using the Gamma function and Stirling's formula) one has
the two-sided estimate
$\frac{p}{c}  \le  \| \xi \|_p \le c p$ for all $p\in [1,\infty)$ and some $c\in [1,\infty)$.
On the other hand,
\[ \int_{\{l,l+1,\ldots\}} \left [ \int_{t_{l-1}}^1 C_s ds \right ] d\Q \le 2 \Q(\{l,l+1,\ldots\}) 
   \sptext{1}{for}{1}
   l\ge 1 \]
which implies that 
$\sup_{t\in [0,T]} \left \| \E \left ( \int_t^T C_s ds | \cA_t \right ) \right \|_\infty \le 2$.}
\end{proof}
\medskip

\begin{remark}
There is a connection to the  {\em Bhattacharyya coefficient} (also called 
{\em Hellinger coefficient}) of two measures, see \cite{Bhattacharyya:43}.
Assume two Borel measures $\mu,\nu$ on $\cB([0,T])$ and a reference measure 
$\sigma$ such that $\mu$ and $\nu$ are absolutely continuous with respect to $\sigma$. Then 
\[ B(\mu,\nu):= \int_{[0,T]} \sqrt{ \frac{d\mu}{d\sigma}\frac{d\nu}{d\sigma}} d\sigma, \]
which is independent from the particular choice of the reference measure,
is called {\rm Bhattacharyya coefficient}.
Under the assumptions of Corollary \ref{cor:AB_generalized_Fefferman_inequality_new_constant}, with 
$d\mu(t):= A_t^2 dt$ and $d\nu(t) = B_t^2 dt$, we have
\[  B(\mu(\omega,\cdot),\nu(\omega,\cdot)) = \int_0^T | A_s(\omega) B_s(\omega) | ds. \]
\end{remark}
\smallskip

\begin{cor}
\label{cor:Z_generalized_Fefferman_inequality_new_constant}
For $\theta \in (0,1]$, $p\in [1,\infty)$, and $Z\in \H_p(S_2) \cap \bmo(S_{2\theta})$one has
\[  \E \left | \int_0^T |Z_t|^{1+\theta} dt \right |^p < \infty \]
with 
\[
       \left \| \int_0^T |Z_t|^{1+\theta} dt \right \|_p 
   \le c_{\eqref{cor:AB_generalized_Fefferman_inequality_new_constant},p}  \| Z \|_{\H_p(S_2)} 
                   \| Z \|_{\bmo(S_{2\theta})}^\theta. 
\]
\end{cor}
\medskip

\begin{remark}
\begin{enumerate}
\item For $\theta=1$ we have that $\bmo(S_{2\theta})\subseteq  \H_p(S_2)$ because of 
      relation \eqref{eqn:JN_BMO(2theta)}.
\item In general, for $\theta \in (0,1)$ we do not have 
      \st{$\bmo(S_{2\theta})\subseteq \H_p(S_2)$ (here one can take deterministic processes)
      nor 
      $\H_p(S_2) \subseteq \bmo(S_{2\theta})$ (see Example \ref{example:weaker_bmo_is_weaker_new})}.
\item In general, neither the condition $Z\in \H_p(S_2)$ implies $\E | \int_0^T |Z_s|^{1+\theta} ds|^p <\infty$
      for $\theta \in (0,1]$, nor $Z\in \bmo(S_{2\theta})$  does for $\theta\in (0,1)$.
\end{enumerate}
\end{remark}


\section{Reverse H\"older inequalities}
\label{sec:reverse_hoelder_inequalities}

\st{So far, we assumed a stochastic basis 
$(A,\cA,\Q,(\cA_t)_{t\in [0,T]})$, $T>0$, where $(A,\cA,\Q)$ is complete, $(\cA_t)_{t\in [0,T]}$ is 
right-continuous, $\cA_0$ contains all null-sets, and $\cA=\cA_T$. To be in accordance with  \cite{Kazamaki:94}, we 
additionally assume now
that all local martingales are continuous. As we work on a closed time-interval we have to explain our 
understanding of a local martingale: we require that the localizing sequence of stopping times 
$0\le \tau_1 \le \tau_2 \le \cdots \le T$ satisfies $\lim_n \P(\tau_n=T)=1$. 
So we extend the filtration by $\cA_T$ to $(T,\infty)$, i.e. $\cA_t:= \cA_T$ for $t\in (T,\infty)$,
and extend all local martingales
$(N_t)_{t\in [0,T]}$  (in our setting) by $N_T$ to $(T,\infty)$. This yields the
standard notion of a local martingale.}
\smallskip

The probabilistic Muckenhoupt weights provide a natural way to verify various martingale inequalities 
after a change of measure, see  exemplary \cite{Izum:Kaza:77,Bona:Lepi:79,Kazamaki:94}. This change 
of measure will appear in our setting in terms of a Girsanov transformation that removes
a sub-quadratic or quadratic drift term in $Z$ that originates 
from the generator of our BSDE, \st{see Section \ref {sec:a_priori_estimate_new}}.
\smallskip

\begin{definition}
\index{inequality!reverse H\"older inequality}
\index{inequality!$\rh_\beta(\lambda)$}
\label{definition:A_p}
Assume a martingale $M\!=\!(M_t)_{t\in [0,T]}$ with $M_0\equiv 0$ such that
$\mathcal{E}(M)$ with
\[ \mathcal{E}(M)_t = e^{M_t-\frac{1}{2}\langle M \rangle_t} \]
for $t\in [0,T]$ is a martingale as well.
For $\beta\in (1,\infty)$ we let  $\mathcal{E}(M) \in \rh_\beta$ provided that there is a 
constant $c>0$ such that for all stopping times $\tau:A\to [0,T]$ one has that
\[          \E (|\mathcal{E}(M)_T|^\beta |\cA_\tau)^\frac{1}{\beta}
         \le c \mathcal{E}(M)_\tau \mbox{ a.s. } \]
The smallest possible $c\ge 0$ is denoted by $\rh_\beta(\mathcal{E}(M))$.
\end{definition}

It is known \cite[Theorem 2.3]{Kazamaki:94} that 
$\mathcal{E}(M)$ is a martingale for $M\in \bmo$. Moreover, we have the following result:
\bigskip

\begin{proposition}[{\cite[Theorems 2.4 and 3.4]{Kazamaki:94}}]
\label{proposition:RH-A-BMO}
Let $M$ be a martin\-gale  with $M_0\equiv 0$ such that $\mathcal{E}(M)$ is a martingale.
Then $M\in \bmo$ if and only if 
$\mathcal{E}(M)\in \bigcup_{\beta\in (1,\infty)} \rh_\beta$.
\end{proposition}
\bigskip

Later in our application we need to know whether a 
certain martingale $M$ generates a Dol\'ean-Dade exponential that
satisfies a reverse H\"older inequality. Here the $\bmo_2$-distance to $L_\infty$ would be a natural candidate
for the extreme case that the reverse H\"older inequality is satisfied for all parameters $\beta \in (1,\infty)$, as
Kazamaki
\cite[Theorem 3.8]{Kazamaki:94} provides the characterization $M\in \overline{L_\infty}^{[\bmo,\|\cdot\|_{\bmo_2}]}$ 
for this case. On the other hand, Grandits \cite{Grandits:96} has shown that a positive $\bmo_2$-distance to 
$L_\infty$ does not provide a reasonable estimate for the critical value of $\beta$ such that one has a reverse H\"older inequality
(see also the {\em Note added in Proof} of \cite{Schachermayer:96}). This is our reason to use the concept {\em sliceable}
(which describes the $\bmo_2$-distance to $\H_\infty$ due to the result of Schachermayer \cite{Schachermayer:96})
because the following observation yields explicit estimates for the critical exponent $\beta$ and the corresponding multiplicative constants 
in the reverse H\"older inequalities:
\medskip

\begin{theorem}\label{theorem:scliceable_rh}
Let $\Phi:(1,\infty)\to (0,\infty)$ be a non-increasing function and let
\[ \Psi: \Big \{(\gamma,\beta)\in [0,\infty)\times (1,\infty): 0\le\gamma < \Phi(\beta)<\infty  \Big \}\to [0,\infty) \]
be right-continuous in its first argument and such that
\[ \Psi(\gamma_1,\beta) \le \Psi(\gamma_2,\beta) 
   \sptext{1}{for}{1}
   0\le \gamma_1 \le \gamma_2 < \Phi(\beta), \]
with the property that
\[ \| M \|_{\bmo_2}  < \Phi(\beta) 
   \sptext{1}{implies}{1}
   \rh_\beta(\mathcal{E}(M)) \le \Psi(\| M \|_{\bmo_2},\beta). \]
Then, for $\sli_N(M) < \Phi(\beta)$ we have that
$\rh_\beta (\mathcal{E}(M)) \le \big [\Psi(\sli_N(M),\beta)\big ]^N$.
\end{theorem}
\bigskip
\begin{proof}
The proof is based on a simple recursion argument that uses the concept of a sliceable \bmo-martingale.
For $\sli_N(M) < \Phi(\beta)$ we choose $0=\tau_0 \le \cdots \le \tau_N =T$ such that
\[   \| ^{\tau_{k-1}}M^{\tau_k}\|_{\bmo_2}  
   < \sli_N(M) + \eta 
   < \Phi(\beta) \]
for some $\eta>0$ and all $k=1,\ldots,N$.
Therefore, 
\[     \rh_\beta(\mathcal{E}(^{\tau_{k-1}}M^{\tau_k}))
   \le \Psi( \| ^{\tau_{k-1}}M^{\tau_k}\|_{\bmo_2},\beta)
   \le \Psi(\sli_N(M)+\eta,\beta). \]
Letting $\tau:A\to [0,T]$ be a stopping time and 
$\sigma_k := \tau_k \vee \tau$
gives that
\equa
&   & \E_{\cA_\tau} \left ( e^{\beta \left ( M_T - \frac{1}{2} \langle M \rangle _T 
      \right )} \right ) \\
& = & \left ( e^{\beta \left ( M_\tau - \frac{1}{2} \langle M \rangle _\tau \right )}
      \right )
       \E_{\cA_\tau} \left ( e^{\beta \left ( [M_T-M_\tau] - 
       \frac{1}{2}[\langle M \rangle _T - \langle M \rangle _\tau] \right )} \right )\\
& = & \left ( e^{\beta \left ( M_\tau - \frac{1}{2} \langle M \rangle _\tau \right )}
      \right )
       \E_{\cA_\tau} \left ( \prod_{k=1}^N e^{\beta \left ( [M_{\sigma_k}-M_{\sigma_{k-1}}] - 
       \frac{1}{2}[\langle M \rangle _{\sigma_k} - \langle M \rangle _{\sigma_{k-1}}] \right )} \right ).
\tion
Next we observe that
\begin{equation}\label{eqn:one_step_rh_estimate}
  \E_{\cA_{\sigma_{k-1}}} \left ( e^{\beta \left ( [M_{\sigma_k}-M_{\sigma_{k-1}}] - 
       \frac{1}{2}[\langle M \rangle _{\sigma_k} - \langle M \rangle _{\sigma_{k-1}}] \right )} \right ) \le \big [ \Psi\left ( \sli_N(M) +\eta,\beta \right ) \big ]^{\beta}
\end{equation}
for $k=1,...,N$ which follows from 
\[     \| ^{\sigma_{k-1}}M^{\sigma_k} \|_{\bmo_2}
    =  \| ^{\tau\vee \tau_{k-1}}M^{\tau\vee\tau _k} \|_{\bmo_2}
  \le  \sup_{l=1,...,N} \| ^{\tau_{l-1}}M^{\tau_l} \|_{\bmo_2}
  < \sli_N(M) + \eta, \]
where we use Lemma \ref{lemma:refinement_bmo}.
Applying \eqref{eqn:one_step_rh_estimate} inductively backwards beginning with $k=N$
and using the projection property of the conditional expectation gives that
\[     \rh_\beta(\mathcal{E}(M))^\beta 
      \le \big [ \Psi\left ( \sli_N(M) +\eta,\beta \right ) \big ]^{\beta N}.
   \]
We conclude by $\eta\downarrow 0$.
\end{proof}
\bigskip

According to \cite[Proof of Theorem 3.1]{Kazamaki:94} possible choices of $(\Phi,\Psi)$ are
\begin{eqnarray}
        \Phi(\beta) 
&:= & \left ( 1+ \frac{1}{\beta^2} \log \left ( 1+ \frac{1}{2\beta-2} \right ) \right )^\frac{1}{2} - 1,
      \label{eqn:kazamaki_function}  \\
      \Psi(\gamma,\beta) 
&:= & \left ( \frac{2}{1-\frac{2\beta-2}{2\beta-1} e^{\beta^2 [\gamma^2 + 2\gamma]}}\right )^\frac{1}{\beta},
      \label{eqn:kazamaki_function_b}
\end{eqnarray}
where $\Phi$ is decreasing with
$\lim_{\beta\to \infty} \Phi(\beta) = 0$ and
$\lim_{\beta\to 1     } \Phi(\beta) = \infty$.


\section{\st{An application to BSDEs}}
\label{sec:a_priori_estimate_new}

In this section we follow the ideas of \cite[Proof of Proposition 2.3]{Briand:Elie:13} but adapt and extend the ideas for our purpose. 
Let $B=(B_t)_{t\in [0,T]}$ be an $n$-dimensional 
standard Brownian motion (where all paths are continuous)
on a basis $(A,\cA,\Q,(\cA_t)_{t\in [0,T]})$, where 
$(A,\cA,\Q)$ is complete, $(\cA_t)_{t\in [0,T]}$ \st{is} the augmentation of the
natural filtration of $B$, and $\cA_T=\cA$. 
\st{It is known (see \cite[Section IV.3]{Protter:04})
that the conditions of Section \ref{sec:reverse_hoelder_inequalities} are satisfied.}
We consider the two backward equations
\equa
Y_t^0 & = & \xi^0 + \int_t^T f^0(s,Y_s^0,Z_s^0) ds - \int_t^T Z_s^0 d B_s, \\
Y_t^1 & = & \xi^1 + \int_t^T f^1(s) ds - \int_t^T Z_s^1 d B_s,
\tion
where we assume the following conditions:
\medskip

\index{coditions!(D1), (D2),\ldots}
\begin{enumerate}[(D1)]
\item The processes $f^1$, $Z^0$ and $Z^1$ are predictable and 
      the processes $Y^0$ and $Y^1$ continuous and adapted,
\item $\E |\xi^i|^2<\infty$ and $\E \int_0^T |Z_s^i|^2 ds < \infty$ for $i=0,1$,
\item $\E \left | \int_0^T |f^0(s,Y_s^0,Z_s^0)| ds\right |^2 < \infty$ and
      $\E \left | \int_0^T |f^1(s)| ds            \right |^2 < \infty$,
\item the generator $f^0:\Om_T\times\R\times\R^{\st{n}}\to\R$ is such that 
      $(t,\om)\mapsto f^0(t,\om,y,z)$ is predictable for all $(y,z)$,
      $(y,z)\to f^0(t,\om,y,z)$ is continuous for all $(t,\om)$,
      and there is an $L_Y\ge 0$ such that,  for all $(t,\om,y_0,y_1,z)$,
      \[     |f^0(t,\om,y_0,z)-f^0(t,\om,y_1,z)| \\
         \le L_Y |y_0-y_1|. \]
\end{enumerate}
\medskip

We let $\Delta \xi :=  \xi^1 - \xi^0$, and for $s\in [0,T]$,
\equa
\Delta Y_s &:= & Y_s^1 - Y_s^0, \\
\Delta Z_s &:= & Z_s^1 - Z_s^0, \\
a_s &:= & f^1(s) - f^0(s,Y_s^1,Z_s^1), \\
c_s &:= & \frac{f^0(s,Y_s^0,Z_s^1) - f^0(s,Y_s^0,Z_s^0)}{|\Delta Z_s|^2} 
          \chi_{\{ \Delta Z_s \not = 0\}}  \Delta Z_s, \\
\Xi_s &:= & |\Delta \xi| +\int_s^T |a_r| dr.
\tion

\begin{lemma}
\label{lemma:briand:elie_new}
Assume that $c=(c_t)_{t\in [0,T]}\in \bmo(S_2)$ with $\| |c| \|_{\bmo(S_2)}\le \gamma < \infty$,
$\lambda_t :=\exp(\int_0^t c_s d B_s - \frac{1}{2} \int_0^t |c_s|^2 ds)$ and $p_0\in (1,\infty)$ such that
$\rh_{p_0'}(\lambda) \le \rho <\infty$ with
$1=(1/p_0)+(1/p_0')$. Assume $p\in [2,\infty)$ with $p>p_0$ such that 
\[ \E \left | \int_0^T |\Delta Z_s |^2 ds \right |^\frac{p}{2} <\infty. \]
Then there is a $c_\eqref{lemma:briand:elie_new}\in (0,\infty)$, depending at most on 
$(T,L_Y,p,p_0,\gamma,\rho,n)$, 
such that for all $t \in [0,T]$ one has that
\[     \left \| \sup_{s\in [t,T]} |\Delta Y_s| \right \|_p +  \left \|  \left ( \int_t^T |\Delta Z_s |^2 ds \right )^\frac{1}{2} \right \|_p 
   \le c_\eqref{lemma:briand:elie_new} \|  \Xi_t \|_p. \]
\end{lemma}

\begin{remark}
\label{remark:lemma:briand:elie}
As already mentioned before, Lemma \ref{lemma:briand:elie_new} continues the work done in \cite[Proof of Proposition 2.3]{Briand:Elie:13},
but also the work done in \cite[Theorem 5.1]{Ankirchner:Imkeller:DosReis:07}. The main new contribution consist in the fact that 
using the extension of Fefferman's inequality (Corollary \ref{cor:AB_generalized_Fefferman_inequality_new_constant})
we are able to get an $L_p$-$L_p$-estimate in contrast to a weaker $L_p$-$L_r$-estimate for $r>p$.
\end{remark}

\begin{proof}[Proof of lemma \ref{lemma:briand:elie_new}]
Let $d\Q^*:=  \lambda_T d\Q$. To distinguish between the integration with respect to $\Q$
and $\Q^*$, but not to overload the notation, we agree that $\|\cdot\|_p$ always means that we integrate with
respect to $\Q$. 
By Girsanov's theorem, $(B_s^*)_{s\in [0,T]}$ with 
$B_s^* := B_s - \int_0^s c_r dr$ is a standard $\Q^*$-Brownian motion. Now let us fix
$t\in [0,T]$ and assume that $\| \Xi_t \|_p<\infty$, otherwise there is nothing 
to prove. Additionally introducing
\[ b _s :=  \frac{f^0(s,Y_s^1,Z_s^1) - f^0(s,Y_s^0,Z_s^1)}{\Delta Y_s} 
          \chi_{\{ \Delta Y_s \not = 0\}}, \]
we get that 
\equa
&   & \Delta Y_t \\
& = & \Delta \xi + \int_t^T a_s ds + \int_t^T b_s \Delta Y_s ds +
      \int_t^T \langle c_s , \Delta Z_s  \rangle ds - \int_t^T \Delta Z_s dB_s \\
& = & \Delta \xi + \int_t^T a_s ds + \int_t^T b_s \Delta Y_s ds 
       - \int_t^T \Delta Z_s dB_s^*
\tion
where our conditions assure that all terms are well-defined.
Because of 
\[ \E_{\Q^*} \left ( \int_0^T |\Delta Z_s|^2 ds \right )^\frac{1}{2} 
   \le \left ( \E_\Q \lambda_T^{p'} \right )^\frac{1}{p'} 
       \left ( \E_\Q \left ( \int_0^T |\Delta Z_s|^2 ds \right )^\frac{p}{2} \right )^\frac{1}{p} < \infty \]
and the Burkholder-Davis-Gundy inequalities $(\int_0^t \Delta Z_s dB_s^*)_{t\in [0,T]}$ is of
class DL and therefore a $\Q^*$-martingale (see \cite[IV.1.7]{Revuz:Yor:99}). Applying
It\^o's formula implies that
\[    e^{\int_0^t b_s ds} \Delta  Y_t 
   =  e^{\int_0^T b_s ds} \Delta \xi + \int_t^T e^{\int_0^s b_r dr}a_s ds 
       - \int_t^T  e^{\int_0^s b_r dr}\Delta Z_s dB_s^* \]
and
\[   \Delta  Y_t
   = \E_{\Q^*} \left ( e^{\int_t^T b_s ds} \Delta \xi+\int_t^T e^{\int_t^s b_r dr}a_s ds  | \cA_t \right ). \]
Using $p_0\in (1,p)$ we continue with
\[
         |\Delta  Y_t |
    \le e^{(T-t) L_Y} \E_{\Q^*} (\Xi_t|\cA_t) 
    \le e^{(T-t) L_Y} \rho 
        \left ( \E_\Q \left ( \Xi_t^{p_0} | \cA_t \right ) \right )^\frac{1}{p_0}
        \mbox{ a.s.}
\]
By Doob's maximal inequality,
\begin{equation}\label{eqn:upper_bound_Y}
       \left \| \sup_{s\in [t,T]} |\Delta Y_s| \right \|_p 
   \le c_{\eqref{eqn:upper_bound_Y}} \|  \Xi_t \|_p
\end{equation}
with $c_{\eqref{eqn:upper_bound_Y}} :=  e^{(T-t) L_Y}  \rho \left (\frac{p}{p-p_0}\right )^\frac{1}{p_0}$.
Letting
\[ \Delta f_s :=  f^1(s)-f^0(s,Y_s^0,Z_s^0),\]
we also have that
\[     |\Delta f_s| 
   \le |a_s| + |b_s| |\Delta Y_s| + |c_s| |\Delta Z_s| 
   \le |a_s| + L_Y |\Delta Y_s| + |c_s| |\Delta Z_s| 
\]
and
\equa
&   & \int_t^T | \Delta Y_s \Delta f_s| ds \\
&\le& \int_t^T | \Delta Y_s | [|a_s| + L_Y |\Delta Y_s| + |c_s| |\Delta Z_s|] ds \\
&\le& \sup_{s\in [t,T]} |\Delta Y_s| \int_t^T  |a_s| ds
      + L_Y \int_t^T |\Delta Y_s|^2 ds + \int_t^T [|c_s| | \Delta Y_s | |\Delta Z_s|] ds \\
&\le& \frac{1}{2} \sup_{s\in [t,T]} |\Delta Y_s|^2 + \frac{1}{2} \left [ \int_t^T  |a_s| ds\right ]^2
      \\
&   & \hspace*{5em} + L_Y \int_t^T |\Delta Y_s|^2 ds + \int_t^T [|c_s| | \Delta Y_s | |\Delta Z_s|] ds \\
&\le& \Gamma^2 \sup_{s\in [t,T]} |\Delta Y_s|^2 
      + \frac{1}{2} \left [ \int_t^T |a_s|  ds \right ]^2
      + \int_t^T [|c_s| | \Delta Y_s | |\Delta Z_s|] ds 
\tion
with
$\Gamma^2 := \frac{1}{2} + T L_Y$.
Now for
$S_t(Z)^2 := \int_t^T |\Delta Z_s |^2 ds$
and
$^*Y_t := \sup_{s\in [t,T]} |\Delta Y_s|$
using 
It\^o's formula, 
the Burkholder-Davis-Gundy inequalities \eqref{eqn:BDG},
and Corollary \ref{cor:AB_generalized_Fefferman_inequality_new_constant},
we get that
\equa
&   & \| S_t(Z) \|_p \\
&\le& \left \| \left ( |\Delta \xi |^2 
       + 2 \left | \int_t^T \Delta Y_s \Delta Z_s dB_s \right |
               + 2 \int_t^T |\Delta Y_s \Delta f_s | ds \right )^\frac{1}{2} \right \|_p \\
&\le& \bigg \| \bigg ( |\Delta \xi |^2 
       + 2 \left | \int_t^T \Delta Y_s \Delta Z_s dB_s \right |
               + 2 \Gamma^2 {^*Y_t^2} 
      + \left [ \int_t^T |a_s|  ds \right ]^2 \\
&   & \hspace*{10em}  + 
                        2 \int_t^T [|c_s| | \Delta Y_s | |\Delta Z_s|] ds  \bigg )^\frac{1}{2} \bigg \|_p \\
&\le& \| \Xi_t \|_p + \sqrt{2}\left  \| \int_t^T [|c_s| | \Delta Y_s | |\Delta Z_s|] ds \right \|_{\frac{p}{2}}^\frac{1}{2}
      + \sqrt{2} \bigg \| \int_t^T \Delta Y_s \Delta Z_s dB_s \bigg \|_\frac{p}{2}^\frac{1}{2} \\
&   &  +  \sqrt{2}\Gamma \bigg \| {^*Y_t} \bigg \|_p \\
&\le& \| \Xi_t \|_p +  \sqrt{2 c_{\eqref{cor:AB_generalized_Fefferman_inequality_new_constant},\frac{p}{2}}}  
      \| |c| \|_{\bmo(S_2)}^\frac{1}{2} 
      \left  \| \left ( \int_t^T [ | \Delta Y_s | |\Delta Z_s|]^2 ds \right )^\frac{1}{2} \right \|_\frac{p}{2}^\frac{1}{2} \\
&   & + \sqrt{2\beta_{p/2}} \bigg \|   
        \left ( \int_t^T [|\Delta Y_s|| \Delta Z_s|]^2 ds \right )^\frac{1}{2} \bigg \|_\frac{p}{2}^\frac{1}{2} 
          +  \sqrt{2}  \Gamma \bigg \| {^*Y_t} \bigg \|_p \\
& = &  \| \Xi_t \|_p 
      + \left [ \sqrt{2 c_{\eqref{cor:AB_generalized_Fefferman_inequality_new_constant},\frac{p}{2}}}  
      \| |c| \|_{\bmo(S_2)}^\frac{1}{2} 
      +  \sqrt{2 \beta_{p/2}} \right ] \times \\
&   & \times \left  \| \left ( \int_t^T [ | \Delta Y_s | |\Delta Z_s|]^2 ds \right )^\frac{1}{2} 
                 \right \|_\frac{p}{2}^\frac{1}{2}
       +  \sqrt{2} \Gamma \bigg \| {^*Y_t} \bigg \|_p.
\tion
Therefore, for 
$\kappa:= \sqrt{2 c_{\eqref{cor:AB_generalized_Fefferman_inequality_new_constant},\frac{p}{2}} \gamma} +  \sqrt{2\beta_{p/2}}$ 
and $\lambda>0$ we obtained that
\equa
      \| S_t(Z) \|_p 
&\le& \| \Xi_t \|_p + \kappa  \left  \| \left ( \int_t^T [ | \Delta Y_s | |\Delta Z_s|]^2 ds \right )^\frac{1}{2} \right \|_\frac{p}{2}^\frac{1}{2} 
        + \sqrt{2} \Gamma \bigg \| {^*Y_t} \bigg \|_p \\
&\le& \| \Xi_t \|_p + \kappa  \left  \| {^*Y_t} S_t(Z) \right \|_\frac{p}{2}^\frac{1}{2} 
        + \sqrt{2} \Gamma \bigg \| {^*Y_t} \bigg \|_p \\    
&\le& \| \Xi_t \|_p + \kappa  \left  \| \frac{\lambda}{2} {^*Y_t}^2 + \frac{1}{2\lambda} S_t(Z)^2 \right \|_\frac{p}{2}^\frac{1}{2} 
        + \sqrt{2} \Gamma \bigg \| {^*Y_t} \bigg \|_p \\    
&\le& \| \Xi_t \|_p + \kappa   \sqrt{\frac{\lambda}{2}} \|^*Y_t\|_p  + 
      \kappa\sqrt{\frac{1}{2\lambda}} \|  S_t(Z) \|_p 
        + \sqrt{2} \Gamma \bigg \| {^*Y_t} \bigg \|_p.
\tion
Choosing $\lambda:= 2\kappa^2$ and using \eqref{eqn:upper_bound_Y} gives that
\equa  
      \| S_t(Z) \|_p 
&\le& 2 \| \Xi_t \|_p + \left [ 2 \kappa^2 + 2 \sqrt{2} \Gamma \right ]  \bigg \| {^*Y_t} \bigg \|_p \\
&\le& 2 \| \Xi_t \|_p + \left [ 2 \kappa^2 + 2 \sqrt{2} \Gamma \right ]   c_{\eqref{eqn:upper_bound_Y}} 
      \|  \Xi_t \|_p 
\tion
which concludes the proof.
\end{proof}


\chapter{Applications to BSDEs}
\label{chapter:BSDE}

In this chapter we consider a solution to the BSDE
\begin{equation}\label{equation:BSDE1_intro}
Y_t = \xi + \int_t^T f(s,Y_s,Z_s)ds - \int_t^T Z_s dW_s, \qquad t \in [0,T], \mbox{ a.s.},
\end{equation}
and will proceed as follows: Firstly, we extend equation \eqref{equation:BSDE1_intro} from 
$\probsp$ to  $(\overline{\Omega},\cF^0,\overline{\P})$ and follow 
Chapter \ref{chapter:transference_sde} to transform this extended BSDE from 
$(\overline{\Omega},\cF^0,\overline{\P})$
to $(\overline{\Omega},\cF^\varphi,\overline{\P})$ and 
$(\overline{\Omega},\cF^\psi,\overline{\P})$, respectively, and consider  for $\rho\in \{\varphi,\psi\}$ the two solutions
\begin{equation}\label{eqn:BSDE-rotated_intro}
Y_t^\rho = \xi^\rho + \int_t^T f^\rho(s,Y_s^\rho,Z_s^\rho) ds - \int_t^T Z_s^\rho d W_s^\rho, \qquad t \in [0,T], \mbox{ a.s.}
\end{equation}
Therefore  \eqref{eqn:BSDE-rotated_intro} describes two copies of \eqref{equation:BSDE1_intro}, parametrised with 
$\varphi$ and $\psi$, by transforming the underlying Gaussian structure.
Secondly, we interpret \eqref{eqn:BSDE-rotated_intro} as equations driven by the 
{\em joint Brownian motion} $\overline{W}=(\overline{W_t})_{t\in [0,T]}$ and apply an a priori estimate to
      obtain Theorem \ref{theorem:comparison_psi_phi} to describe the 
      stability of \eqref{equation:BSDE1_intro}.
From the stability we obtain non-linear embeddings for Besov spaces in Section \ref{sec:formulation_anisotropic_Besov}
and upper bounds for the $L_p$-variation of solution processes $(Y,Z)$ to our BSDE
\eqref{equation:BSDE1_intro} in Section \ref{sec:L_p-variation}.
\st{To explain by means of Section \ref{sec:L_p-variation} the usage of our general framework,
let us assume for the moment that the generator $f$ in \eqref{equation:BSDE1_intro} depends only on $(s,y,z)$. For $p\in [2,\infty)$ and
$0\le s<t \le T$ Theorem \ref{thm:L_p-variation} provides an upper bound for $\left \|\sup_{r\in [s,t]}|Y_r-Y_s|\right \|_p$ 
that mainly depends on $\|\xi-\xi^{(s,t]} \|_p$. In other words,
local estimates on $\xi$ imply local estimates for the variation of the process $Y$ , if local is understood as {\em local in time}.
To illustrate this further, assume a partition $0=r_0<r_1<\cdots<r_L=T$, again $p\in [2,\infty)$, and suppose
for $l=1,\ldots,L$ that $\xi_l\in \cL_p$ is a measurable functional of finitely many increments 
$W_b-W_a$ with $(a,b]\subseteq (r_{l-1},r_l]$. Consider  
\[ \xi := g(\xi_1,\ldots,\xi_L), \]
where $g:\R^L\to \R$ is a Lipschitz function with constant $L\ge 0$. Then
\[ \| \xi - \xi^{(s,t]}\|_p \le L \| \xi_l - \xi_l^{(s,t]} \|_p 
   \sptext{1}{whenever}{1}
  (s,t]\subseteq (r_{l-1},r_l]. \]
Therefore, the variation of $Y$ on $[r_{l-1},r_l]$ is mainly determined by properties of $\xi_l$.
This idea was first developed in \cite{GGG:12} and then extended to the framework of L\'evy processes in
\cite{CGeiss:Steinicke:16}.
}


\section{The setting}
\label{sec:setting_bsdes}

In this section we assume a stochastic basis $(\Omega,\cF,\P,(\cF_t)_{t\in [0,T]})$ with $\cF=\cF_T$
satisfying the usual conditions,
where $\F=(\cF_t)_{t\in [0,T]}$ is the augmentation of the natural filtration of the $d$-dimensional Brownian motion 
$(W_t)_{t\in [0,T]}$. We consider a solution to the BSDE \eqref{equation:BSDE1_intro} under the following set 
of assumptions, that describe the generators we will use and ensure that all expressions do exist:

\begin{assumption}
\index{coditions!(B1), (B2),\ldots}
\label{assumption:BSDE} 
\hspace*{1em}
\begin{enumerate}[{(B1)}]
\item The process $Z$ is predictable such that
      \[ \P \left ( \int_0^T |Z_s|^2 ds < \infty \right ) = 1. \]
\item The process $Y$ is adapted and path-wise continuous.
\item The generator $f:\Om_T\times\R\times\R^{d}\to\R$ is such that 
      $(t,\om)\mapsto f(t,\om,y,z)$ is predictable for all $(y,z)$ and there are
      $L_Y,L_Z\ge 0$ and $\theta\in [0,1]$ such that
      \[         |f(t,\om,y_0,z_0)-f(t,\om,y_1,z_1)| \\
         \le L_Y |y_0-y_1| + L_Z [1+|z_0|+|z_1|]^\theta |z_0-z_1| \]
   for all $(t,\om,y_0,y_1,z_0,z_1)$.
\item $\P \left ( \int_0^T | f(s,Y_s,Z_s)|ds < \infty \right ) = 1$.
\end{enumerate}
\end{assumption}

The case $\theta=0$ is the standard Lipschitz case, the case $\theta=1$ the standard quadratic 
case, and $\theta\in (0,1)$ can be seen as sub-quadratic case (see for example \cite{Cheridito:Nam:14_b}).
Our strategy for the first step is to impose in Lemma \ref{lemma:Y_in_Lp} below conditions on
the gradient process $Z$ and $f(s,0,0)$, only, but not on $\xi$, in order to verify that we deal with an
$L_p$-solution to our BSDE. This might also help to find more general conditions on $(\xi,f)$
that ensure the existence of $L_p$-solutions (see Section \ref{subsec:forward_setting} below).
Our conditions on $Z$ can be verified by results from 
Section \ref{sec:classes_quadratic_subquadratic_BSDEs}
below. In the following we assume that $p \in [2,\infty)$ because 
this assumption will be used in some steps of the proofs and because this case is more interesting 
with respect to the tail-behavior of $|Y_t-Y_s|$ than the case $p<2$.

\begin{lemma}
\label{lemma:Y_in_Lp}
In addition to the conditions {\rm (B1)-(B4)} we assume for $p\in [2,\infty)$ that
\begin{enumerate}
\item [{\rm (B5)}] $\int_0^T |f(s,0,0)| ds \in \cL_p$,
\item [{\rm (B6)}] $  \left ( \int_0^T |Z_s|^2 ds \right )^\frac{1}{2} \in \cL_p$,
\item [{\rm (B7)}] $  \int_0^T |Z_s|^{1+\theta} ds  \in \cL_p$.
\end{enumerate}
Then
$\int_0^T | f(s,Y_s,Z_s)|ds + \sup_{t\in [0,T]} | Y_t|  \in \cL_p$.
\end{lemma} 

\begin{proof}
We rewrite \eqref{equation:BSDE1_intro} as
\[ Y_t = Y_0 - \int_0^t f(s,Y_s,Z_s)ds + \int_0^t Z_s dW_s \]
for $t \in [0,T]$. For an integer $N\ge 1$ let
\[ \tau_N := \inf \{ t\in [0,T] : | Y_t - Y_0| = N \} \wedge T \]
with $\inf\emptyset := \infty$. Then
\[ Y_{t\wedge \tau_N}  = Y_0 - \int_0^{t\wedge \tau_N}  f(s,Y_s,Z_s)ds 
                             + \int_0^{t\wedge \tau_N} Z_s dW_s. \]
Because of
\begin{equation} \label{eqn:upper_bound_generator}
|f(s,y,z)| \le |f(s,0,0)| + L_y |y| + L_z [1+|z|]^\theta |z|
\end{equation}
we conclude that
\equa
      |Y_{t\wedge\tau_N}|
&\le& \bigg [ |Y_0| + \int_0^T |f(s,0,0)|ds + L_z \int_0^T [1+|Z_s|]^\theta |Z_s| ds \\
&   & + \sup_{r\in [0,T]} \left | \int_0^r Z_s dW_s \right | \bigg ] 
      + L_y \int_0^{t\wedge \tau_N} |Y_{s\wedge \tau_N}| ds \\
& =:& A + L_y \int_0^{t\wedge \tau_N} |Y_{s\wedge \tau_N}| ds
\tion
and
\[ M_t^N \le A + L_y \int_0^t M_s^N ds \]
with
\[ M_s^N := \sup_{r\in [0,s]} |Y_{r\wedge \tau_N}| = \sup_{r\in [0,s\wedge \tau_N]} |Y_r|. \]
The process $(M_t^N)_{t\in [0,T]}$ is continuous, adapted and bounded by $|Y_0|+N$. 
The inequality
\[ \| M_t^N\|_p \le \|A\|_p + L_y \int_0^t \| M_s^N \|_p ds \]
implies by Gronwall's lemma that
\[ \| M_T^N \|_p \le e^{L_yT} \| A \|_p. \]
Letting $N\to \infty$ gives $\sup_{t\in [0,T]} | Y_t|  \in \cL_p$ because $A\in \cL_p$ which follows from 
conditions (B5), (B6), and (B7). Finally, using \eqref{eqn:upper_bound_generator} the part
$\int_0^T | f(s,Y_s,Z_s)|ds \in \cL_p$ follows.
\end{proof}
\bigskip

Condition (B5) is a condition on the initial data of the BSDE, whereas (B6) and (B7) are
implicit conditions on the solution. For $\theta=0$ condition (B6) implies (B7).
Conversely, for $\theta=1$ condition (B7) implies (B6).
A sufficient condition for both, (B6) and (B7), is
$\left ( \int_0^T |Z_s|^2 ds \right )^{{1}/{2}}  \in \cL_{(1+\theta)p}$.
\bigskip


\section[Stability of BSDEs with respect to perturbations]
        {Stability of BSDEs with respect to perturbations of the Gaussian structure}

Now we substantiate the procedure explained in the beginning of this chapter:
we assume the setting of Section \ref{sec:Besov:setting} and follow
Convention \ref{convention:extension_new}(1) to extend 
(\ref{equation:BSDE1_intro}) to $\overline{\Omega}$ and find
\begin{equation}\label{equation:BSDE1-extended}
  \tilde Y_t = \tilde \xi + \int_t^T \tilde f(s,\tilde Y_s,\tilde Z_s)ds - \int_t^T \tilde Z_s dW_s^0, \qquad t \in [0,T].
\end{equation}
We remark that for a  $(\cP,\cB(C(M)))$-measurable $h:[0,T]\times \Omega\to C(M)$
the extension $\widetilde{h}:[0,T]\times \overline{\Omega}\to C(M)$ is $(\cP^0,\cB(C(M)))$-measu\-ra\-ble, and that 
there is a $\overline{\Omega}_0\in \overline{\cF}$ with $\overline{\P}(\overline{\Omega}_0)=1$, 
such that $(\int_0^t \widetilde{Z}_s dW_s^0)(\omega,\omega') = (\int_0^t Z_s dW_s)(\omega)$ for $t\in [0,T]$ and
$(\omega,\omega')\in \overline{\Omega}_0$.
Moreover, it is clear that the inequality from (B3) transfers directly.
Therefore we assume that  (\ref{equation:BSDE1_intro}) is extended to 
(\ref{equation:BSDE1-extended}) where we simplify the notation by denoting $(\tilde \xi,\tilde f,\tilde Y, \tilde Z)$ again by
$(\xi,f,Y,Z)$.
Using Theorem \ref{theorem:change_multi_SDE} in the setting of Section \ref{sec:Besov:setting}  
we obtain \eqref{eqn:BSDE-rotated_intro}. 
We also know that the transformed generator $f^\rho$ can be taken such that (B3) is satisfied, i.e.
\begin{multline*}
  | f^\rho(t,\bar\omega,y_0,z_0) - f^\rho(t,\bar\omega,y_1,z_1)| \\
   \le L_Y |y_0-y_1| + L_Z[1+|z_0|+|z_1|]^\theta |z_0-z_1| 
    =:H((y_0,z_0),(y_1,z_1)),
\end{multline*}
which follows from Remark \ref{rem:quantitative_bounds_transformed_generator}.
\medskip

Now let us turn to our basic result. Our strategy is to impose the conditions (B1)-(B6) and an
extra condition on $Z$ on equation \eqref{equation:BSDE1_intro} in the context of the
stochastic basis $(\Omega,\cF,\P,(\cF_t)_{t\in [0,T]})$ we did start from, and then to deduce 
by \st{Lemma \ref{lemma:briand:elie_new}} the
moment estimates in the extended setting of $(\overline \Omega,\overline \P)$. 
\st{
For the following we remind the reader that the number 
$\sli_N^{S_2,\A}(c)$ for an $\R$-valued progressively measurable process $c$, $N\ge 1$, 
and a filtration $\A$ was defined in Definition \ref{definition:sliceable_S2}.}
\medskip

\begin{theorem}
\label{theorem:comparison_psi_phi}
\index{$s_\infty$}
Assume $\theta\in [0,1]$, for equation \eqref{equation:BSDE1_intro} 
conditions {\rm (B1)-(B4)}, and additionally $|Z|\in \bmo(S_{2\theta})$ in the case $\theta\in (0,1]$.
Suppose that there is a non-increasing 
sequence $\st{(s_N)}_{N\ge 1} \subseteq [0,\infty)$, where
$s_\infty := \lim_N s_N$, such that
\[ \sli_N^{S_2,\F} (|Z|^\theta) \le s_N. \]
Suppose that conditions {\rm (B5)-(B6)} are satisfied for $p \in [2,\infty)$
where in the case $s_\infty>0$ we additionally assume that 
\[ p>p_0:= \frac{\Phi^{-1}(2\sqrt{2} L_Z s_\infty)  }
                {\Phi^{-1}(2\sqrt{2} L_Z s_\infty)-1}\in (1,\infty) \]
with the function $\Phi$ defined in \eqref{eqn:kazamaki_function}.
Then, one has for the extended equations for all $t \in [0,T]$ that
\equa
&&    \left \| \sup_{s\in [t,T]} | Y_s^\varphi - Y_s^\psi| \right \|_p 
   +  \left \| \left ( \int_t^T D\big [\varphi(s),\psi(s)\big ] \, |Z_s|^2 ds \right )^\frac{1}{2} 
       \right \|_p \label{eqn:theorem:comparison_psi_phi} \\
&&  \hspace*{9em} +  \left \| \left ( \int_t^T |Z_s^\varphi - Z_s^\psi |^2 ds \right )^\frac{1}{2} 
                \right \|_p  \\  
&\le& c_\eqref{theorem:comparison_psi_phi}  
      \left [ \|\xi^\vph-\xi^\psi \|_p 
    + \left \| \int_t^T | f^\varphi(s,Y^\psi_s,Z^\psi_s) - f^\psi(s,Y^\psi_s,Z^\psi_s)| ds 
            \right \|_p \right ], 
\tion
where $\varphi,\psi\in \Delta$,
$D[\eta_1,\eta_2] := 1- \sqrt{1-\eta_1^2} \sqrt{1-\eta_2^2} - \eta_1\eta_2$,
and $c_\eqref{theorem:comparison_psi_phi} >0$ depends at most on $(L_Y,L_Z,T,(s_N)_{N=1}^\infty,p,d)$.
\end{theorem}
\bigskip

The applications of Theorem \ref{theorem:comparison_psi_phi} are at least two-fold: 
Firstly, we obtain a non-linear embedding theorem for Besov spaces in Section \ref{sec:formulation_anisotropic_Besov}
(Corollary \ref{cor:theorem:comparison_psi_phi}).
Secondly, we deduce in Section \ref{sec:L_p-variation} upper bounds for the $L_p$-variation of solution processes $(Y,Z)$ to our BSDE
\eqref{equation:BSDE1_intro}.
\pagebreak

\begin{remark}\label{remark:theorem:comparison_psi_phi}
\hspace{0em}
\begin{enumerate}
\item The function $D[\eta_1,\eta_2]:[0,1]^2\to [0,1]$  measures the distance between
      $\eta_1$ and $\eta_2$, by projecting the vector $(\eta_1,\sqrt{1-\eta_1^2})$ onto 
      the linear subspace generated by
      $(\eta_2,\sqrt{1-\eta^2_2})$, and by comparing the projection to 
      $(\eta_2,\sqrt{1-\eta^2_2})$. In particular, $D[\eta_1,\eta_2]=0$
      if and only if $\eta_1=\eta_2$.

\item Because the case $\lim_N s_N=0$ is of particular importance
      in Theorem \ref{theorem:comparison_psi_phi}, as it enables us to use 
      the full range $p\in [2,\infty)$, we give some examples for this situation: 

      \begin{enumerate}
      \item For $\theta = 0$  we have that 
            \[    \sli_N^{S_2,\F}(|Z|^\theta) \le \sqrt{\frac{T}{N}} \]
            if we take equidistant time-nets.

      \item Let $0 < \theta < \eta \le 1$ and assume that $\| |Z| \|_{\bmo(S_{2\eta})} <\infty$. Then, similarly 
            to Example \ref{example:abstract_2variation}, we obtain
            \equa
            &   &  \| (\chi_{(a,b]}(t)|Z_t|)_{t\in [0,T]} \|_{\bmo(S_{2\theta})} \\
            &\le& (b-a)^{\frac{1}{2\theta}-\frac{1}{2\eta}} 
                    \| (\chi_{(a,b]}(t)|Z_t|)_{t\in [0,T]} \|_{\bmo(S_{2\eta})}
            \tion
            and, by using equidistant \st{time-nets}, that
            \[ \sli_N^{S_2,\F}(|Z|^\theta) \le \left ( \frac{T}{N} \right )^{\frac{1}{2}\left (1-\frac{\theta}{\eta}\right ) }
               \| |Z| \|_{\bmo(S_{2\eta})}^\theta. \]

\end{enumerate}
\item The usage of $(\sli_N^{S_2,\F} (|Z|^\theta))_{N\ge 1}$ might not be optimal in extremal cases
      as we mainly need the reverse H\"older inequality for the Dol\'ean-Dade exponential \eqref{eqn:lambda_rh_sufficient}
      in the proof of Theorem \ref{theorem:comparison_psi_phi} below: 
      If one would have $\int_0^\cdot  c_s  d\overline{W}_s \in \overline{L_\infty}^{\bmo_2}$, then according to the 
      remarks following Proposition \ref{proposition:RH-A-BMO} the reverse H\"older inequality for all
      exponents would be satisfied.
      It is part of future work to check conditions on the gradient $Z$ which guarantee this.
      On the other hand, if $\int_0^\cdot  c_s  d\overline{W}_s \not\in \overline{L_\infty}^{\bmo_2}$, then
      our approach yields explicit bounds for $c_\eqref{theorem:comparison_psi_phi}>0$ and the threshold $p_0$ 
      in terms of $(s_N)_{N\ge 1}$ which is implicitly a novelty of this statement. As shown in 
      Section \ref{sec:classes_quadratic_subquadratic_BSDEs} below, the usage 
      of the sliceable numbers gives $s_\infty=0$ in our relevant cases.

\item \st{In \cite{Frei:14} the sliceability condition is applied directly to $\xi$, instead of to $|Z|^\theta$ as 
      in our Theorem \ref{theorem:comparison_psi_phi}. This is done to consider a new concept of a solution to a BSDE, 
      called {\em split solution}, to solve multidimensional quadratic BSDEs.}

\end{enumerate}
\end{remark}
\medskip

\begin{proof}[Proof of Theorem \ref{theorem:comparison_psi_phi}]
(a) By Corollary \ref{cor:Z_generalized_Fefferman_inequality_new_constant} 
the assumptions (B6) and \linebreak $\| |Z| \|_{\bmo(S_{2\theta})}<\infty$
imply (B7) in the case $\theta>0$, whereas
for $\theta=0$ condition (B6) implies (B7) directly. Therefore we have 
\begin{equation}\label{eqn:Lp_boundedness_bsde}
\int_0^T | f(s,Y_s,Z_s)|ds + \sup_{t\in [0,T]} | Y_t|  \in \cL_p
\end{equation} by Lemma \ref{lemma:Y_in_Lp} for equation \eqref{equation:BSDE1_intro}. 
This yields the validity of conditions (B1)-(B7) and \eqref{eqn:Lp_boundedness_bsde} for 
the canonical extension to $\overline \Omega$. 
\medskip

(b) Now we define $h_1,h_2:[0,1]^2\to[0,1]$ by
 \equa
  h_1(x,z) &:=& \frac{x\sqrt{1-z^2}+z\sqrt{1-x^2}}{x+z}, \\
  h_2(x,z) &:=& \frac{x\sqrt{1-z^2}+z\sqrt{1-x^2}}{\sqrt{1-z^2}+\sqrt{1-x^2}},
 \tion
where for $x=z=0$ we set $h_1:=1$ and $h_2:=0$, analogously  
for $x=z=1$ we set $h_1:=0$ and $h_2:=1$, so that
\[ \left (
   \begin{matrix}
   \sqrt{1-x^2} & x \\
   \sqrt{1-z^2} & z \\
   \end{matrix}
   \right ) \binom{h_1(x,z)}{h_2(x,z)} = \binom{1}{1}
    \]
for all $x,z \in [0,1]$.
For $\rho\in \{ \varphi,\psi \}$ we let
\equa
\overline{Z}^\rho_s &:= & (Z_s^\rho \sqrt{1-\rho^2(s)}, Z_s^\rho \rho(s)),\\
 \overline{f}^\rho(s,y,(z,z')) &:= & f^\rho\left(s,y,
     h_1(\varphi(s),\psi(s))z
   + h_2(\varphi(s),\psi(s))z'\right), 
\tion
which leads to 
$ \overline{f}^\rho(s,Y_s^\rho,\overline{Z}_s^\rho)
=           f ^\rho(s,Y_s^\rho,          Z _s^\rho)$
and 
\begin{equation}\label{eqn:bsde_rho_bar}
 Y_t^\rho = \xi^\rho + \int_t^T \overline{f}^\rho(s,Y_s^\rho,\overline{Z}_s^\rho) ds
         - \int_t^T \overline{Z}_s^\rho d\overline{W}_s.
\end{equation}
Observe that
\begin{eqnarray}
&   & |\overline{Z}_s^\varphi - \overline{Z}_s^\psi|^2 \nonumber \\
& = & D[\varphi(s),\psi(s)]  [|Z_s^{\psi}|^2 + |Z_s^\varphi|^2] + [1-D[\varphi(s),\psi(s)]]
      |Z_s^\psi-Z_s^\varphi|^2 \label{eqn:zbar_1}
      \\
&\ge& \frac{2-D[\varphi(s),\psi(s)]}{2} |Z_s^\psi-Z_s^\varphi|^2 \nonumber \\
&\ge& \frac{|Z_s^\psi-Z_s^\varphi|^2}{2} \label{eqn:zbar_2}
\end{eqnarray}
and therefore we get for
\equa
      c_s 
&:= & \frac{\overline{f}^\varphi(s,Y_s^\varphi,\overline{Z}_s^\varphi) - 
                \overline{f}^\varphi(s,Y_s^\varphi,\overline{Z}_s^\psi)}
               {|\overline{Z}_s^\varphi - \overline{Z}_s^\psi|^2}
               \chi_{\{ \overline{Z}_s^\varphi \not = \overline{Z}_s^\psi\}} 
               [\overline{Z}_s^\varphi - \overline{Z}_s^\psi] \\
& = & \frac{f^\varphi(s,Y_s^\varphi,Z_s^\varphi) - 
            f^\varphi(s,Y_s^\varphi,Z_s^\psi)}
               {|\overline{Z}_s^\varphi - \overline{Z}_s^\psi|^2}
               \chi_{\{\overline{Z}_s^\varphi \not = \overline{Z}_s^\psi\}} 
               [\overline{Z}_s^\varphi - \overline{Z}_s^\psi] \\
\tion
that
\equa
      |c_s|
&\le& \sqrt{2} 
      \frac{|f^\varphi(s,Y_s^\varphi,Z_s^\varphi) - f^\varphi(s,Y_s^\varphi,Z_s^\psi)|}
           {|Z_s^\varphi-Z_s^\psi|}  \chi_{\{Z_s^\varphi \not = Z_s^\psi\}} \\
&\le& \sqrt{2} L_Z \big [1+|Z_s^\psi|+|Z_s^\varphi| \big ]^\theta \\
&\le& \sqrt{2} L_Z \big [1+|Z_s^\psi|^\theta+|Z_s^\varphi|^\theta \big ]. 
\tion
Lemma \ref{lemma:properties_sliceable_numbers} 
(to come into the setting of Lemma \ref{lemma:properties_sliceable_numbers} one can pass
from an $\R$-valued progressively measurable process $\alpha=(\alpha_t)_{t\in [0,T]}$ with
$\E\int_0^T |\alpha_t|^2 dt < \infty$ to a martingale by, for example,
$M_t:= \int_0^t \alpha_s d\overline{W}_{s,1}$) gives that
\begin{equation}\label{eqn:sli3}
       \sli_{3N-2}^{S_2,\overline \F}(|c|) 
  \le  \sqrt{2} L_Z [ \sli_N^{S_2,\overline\F}(1) + \sli_N^{S_2,\overline \F}(|Z^\psi|^\theta) 
       + \sli_N^{S_2,\overline\F}(|Z^\varphi|^\theta) ].
\end{equation}

\bigskip
(c) We return to the stochastic basis $(\Omega,\cF,\P,(\cF_t)_{t\in [0,T]})$,
take $\eta>0$ and find a sequence of stopping times $0=\tau_0 \le \cdots \le \tau_N = T$
such that
\[
       \sup_{k=1,...,N} \| (\chi_{(\tau_{k-1},\tau_k]}(t) |Z_t|^\theta)_{t\in [0,T]} \|_{\bmo(S_2)} 
   \le \sli_N^{S_2,\F}(|Z|^\theta) + \eta
   \le s_N + \eta.
\]
Letting 
\[ Z_t^k := \chi_{(\tau_{k-1},\tau_k]}(t) Z_t, \]
one can quickly check that
\[ \E \left ( \int_t^T |Z_s^k|^{2\theta} ds | \cF_t^0 \right ) \le (s_N+\eta)^2 \]
for all deterministic $t\in [0,T]$, where $Z^k$ is canonically extended to $\overline \Omega$. Assuming an 
$(\cF_t^0)_{t\in [0,T]}$-stopping time $\tau:\overline \Omega\to [0,T]$, and 
using the decomposition
\[    \E \left ( \int_\tau^T |Z_s^k|^{2\theta} ds | \cF_\tau^0 \right )
   =  \E \left ( \int_0^T    |Z_s^k|^{2\theta} ds | \cF_\tau^0 \right ) 
      -          \int_0^\tau |Z_s^k|^{2\theta} ds \]
and the optional stopping theorem, we may deduce that
\[ \E \left ( \int_\tau^T |Z_s^k|^{2\theta} ds | \cF_\tau^0 \right ) \le (s_N+\eta)^2. \]
Consequently,
\[   \sup_{k=1,...,N} \| (\chi_{(\tau_{k-1},\tau_k]}(t) |Z_t|^\theta)_{t\in [0,T]} \|_{\bmo(S_2)} 
     \le s_N + \eta \]
also after extending $Z$ and $(\tau_k)_{k=0}^N$ to $\overline \Omega$ where the filtration
$\st{\F^0=}(\cF^0_t)_{t\in [0,T]}$ is used.
This means that 
\begin{equation}
\label{eqn:sli_Ztheta}
\sli_N^{S_2,\st{\F^0}}(|Z|^\theta)\le s_N.
\end{equation}

(d) For any stopping time $\tau:\Omega\to [0,T]$ relative to $(\Omega,\cF,\P,(\cF_t)_{t\in [0,T]})$ and
for $\rho \in \{ \psi,\varphi \}$ consider $\tau^\rho:\overline\Omega\to \R$ and
take a representative such that $\tau^\rho:\overline\Omega\to [0,T]$. It is easy to check that 
$\tau^\rho$ is a stopping time with respect to the filtration $(\cF^\rho_t)_{t\in [0,T]}$.
Using 
$\E \left ( A^\rho | \overline{\cF}_t \right ) 
  = \left ( \E \left ( A | {\cF}_t^0 \right ) \right )^\rho$
$\overline{\P}$-a.s.
for $A\in \cL_1(\overline \Omega,\cF^0,\overline \P)$ 
(which can be checked by taking simple $A$ that depend only on finitely many increments of 
 the Brownian motion $W$ and then passing in $L_1$ to the limit), \st{Proposition \ref{proposition:properties_C_T}},
and Remark \ref{remark:consistent_new}(2) yield that
\[     \int_t^T \chi_{(\tau_{k-1}^\rho,\tau_k^\rho]}(s) |Z_s^\rho|^{2\theta} ds 
   =  \left ( \int_t^T \chi_{(\tau_{k-1},\tau_k]}(s) |Z_s|^{2\theta} ds \right )^\rho
   \mbox{ $\overline{\P}$-a.s.} \]
and
\equa
      \E \left ( \int_t^T \chi_{(\tau_{k-1}^\rho,\tau_k^\rho]}(s) |Z_s^\rho|^{2\theta} ds |
      \overline \cF_t \right )    
& = & \left ( \E \left ( \int_t^T \chi_{(\tau_{k-1},\tau_k]}(s) |Z_s|^{2\theta} ds |
      \cF_t^0 \right )   \right )^\rho \\
&\le& (s_N+\eta)^2.
\tion
Therefore,  we obtain
$\sli_N^{S_2,\overline\F}(|Z^\rho|^\theta)\le s_N +\eta$ as a complement of \eqref{eqn:sli_Ztheta}
(where we use the same optional stopping argument
as in step (c)) and can continue from 
\eqref{eqn:sli3} to 
\[     \sli_{3N-2}^{S_2,\overline{\F} }(|c|) 
   \le \sqrt{2} L_Z \left [ \sqrt{\frac{T}{N}} + 2 s_N + 2\eta \right ] \]
and
\[     \sli_{3N-2}^{S_2,\overline{\F}}(|c|) 
   \le \sqrt{2} L_Z \left [ \sqrt{\frac{T}{N}} + 2 s_N \right ] \]
by $\eta \downarrow 0$.
In the case $s_\infty = 0$ take $p_0 \in (1,2)$, say $p_0:= 3/2$, and in the case
$ s_\infty > 0$, define
\[ p_0:= \frac{ \Phi^{-1} (2\sqrt{2} L_Z s_{\infty} )     }
              { \Phi^{-1} (2\sqrt{2} L_Z s_{\infty} ) - 1 } \in (1,\infty) \]
and $p_1:= (p+p_0)/2$ so that
\[ 1< p_0<p_1<p<\infty. \]
Let 
\begin{equation}\label{eqn:lambda_rh_sufficient}
 \lambda_t :=\exp\left (\int_0^t c_s d \overline{W}_s - \frac{1}{2} \int_0^t |c_s|^2 ds \right ).
\end{equation}
We find an $N\ge 1$ such that
\[ \sli_{3N-2}^{S_2,\overline\F} (|c|) \le \sqrt{2} L_Z \left [ \sqrt{\frac{T}{N}} + 2 s_N \right ] 
                       < \Phi(p_1'). \]
This $N$ depends at most on $((s_N)_{N=1}^\infty, L_Z,T,p)$. 
Theorem \ref{theorem:scliceable_rh} implies that
\[ \rh_{p_1'}(\lambda)
   \le \left [\Psi \left (\sqrt{2} L_Z \left [ \sqrt{\frac{T}{N}} + 2 s_N \right ],p_1'
                   \right ) \right ]^{3N-2}<\infty \]
with $\Psi$ taken from \eqref{eqn:kazamaki_function_b}.
By assumption (B6) we have that 
\[ \left ( \int_0^T |Z_s|^2 ds \right )^\frac{1}{2} \in \cL_{p}. \]
Finally, fixing $t\in [0,T]$, we can assume for this $t$ that
\[  \left \| \int_t^T | f^\varphi(s,Y^\psi_s,Z^\psi_s)-f^\psi(s,Y^\psi_s,Z^\psi_s)| ds 
            \right \|_p < \infty, \]
otherwise there is nothing to prove. So we can apply \st{Lemma \ref{lemma:briand:elie_new}} to
the equations \eqref{eqn:bsde_rho_bar}  for $\rho \in \{ \varphi,\psi\}$
and conclude by using \eqref{eqn:zbar_1}
and \eqref{eqn:zbar_2}.
\end{proof}


\section{On classes of quadratic and sub-quadratic BSDEs}
\label{sec:classes_quadratic_subquadratic_BSDEs}

In this section we present results about particular classes of quadratic and sub-quadratic BSDEs that might be of independent 
interest. At the same time we check whether we may apply  Theorem \ref{theorem:comparison_psi_phi} to 
these BSDEs and what we can say about the critical value $s_\infty$.
\medskip

There are various articles that describe the existence and quantitative properties of solutions 
to BSDEs and provide comparison results. For the case $\theta=0$ the reader is referred to
\cite{Briand:Del:Hu:Par:Sto:03} and the references therein, and for the quadratic case 
we refer to 
\cite{Kobylanski:00,
      Lepel:Martin:98,
      Lepel:Martin:02,
      Hu:Imkeller:Mueller:05,
      Briand:Hu:06,
      Ankirchner:Imkeller:DosReis:07,
      Briand:Hu:08,
      Briand:Lepel:Martin:07,
      Morlais:09,
      Imke:Reis:10,
      Delbaen:Hu:Richou:11,
      Mocha:Westray:12,
      Barrieu:Karoui:13,
      Delbaen:Hu:Richou:15}.
We are mainly interested in the sub-quadratic and quadratic case, i.e. the case when $\theta\in (0,1]$. 
In Table 1 below we describe how we will embed these cases in the framework of this
article. Table 1 should be read in the way that we first choose $(\xi,\theta,f)$, then 
we obtain the integrability of the gradient process $Z=(Z_t)_{t\in [0,T]}$ and 
the conclusion for $s_\infty= \lim_N s_N$ \st{that are required for} Theorem \ref{theorem:comparison_psi_phi}. In the cases where the uniqueness of 
the solution is not known there exists a solution with the stated properties. \st{In (IV)-(V)} we 
leave out the range for $s_\infty$ as we do not have general results for these cases
\st{(see Remark \ref{remark:s_infty_in_(IV-V)} below).
     Moreover, for (II)-(V) we need the following additional condition:}
\begin{enumerate}
\item [{\rm (B8)}] \st{One has} $\sup_{(t,\omega)\in \st{[0,T]\times \Omega}} |f(t,\omega,0,0)| <\infty$, $L_Y>0$, 
                   and $L_Z>0$,
\end{enumerate}
\st{where the constants $L_Y,L_Z\ge 0$ were introduced in condition (B3) of
Section \ref{sec:setting_bsdes}.}

\begin{table}[H]\caption*{Table 1}
\label{table}
\begin{tabular}{|r||l|l|l||l|l|}\hline 
   & $\xi$                                        & $\theta$& $f$    & $|Z|$       & $s_\infty$  \\ \hline\hline
 \st{(I)} & $\xi \in L_p$                                & $0$     & (B3), (B5) & $\H_p(S_2)$ & $0$  \\  
   & for some $p\in [2,\infty)$                   &         &               &              &              \\ \hline \hline
\st{(II)} & $\xi \in \bh$                                & $(0,1)$ & (B3), (B8) & $\H_2(S_2) \cap\bigcap\limits_{\eta \in (0,1)} $ & $0$ \\
   &                                              &         &               & $ \bmo(S_{2\eta})$  &     \\ \hline 
\st{(III)}& $|\xi|_{\bh(\eta,\mu)}\!< \!\infty$ for some & 1       & (B3), (B8)  & $\H_2(S_2)\cap$ & $[0,\infty)$\\
   & $\eta\in (0,1]$, $\mu>\gamma e^{\beta T}$          &         &                  &  $\bmo(S_{2\eta})$ & if $\eta=1$ \\ \hline \hline
\st{(IV)} & $\E e^{\mu |\xi|}< \infty$                   & (0,1)   & (B3), (B8)& $\bmo^{\sqrt{\Psi}}(S_2)$  &     \\
   & for some  $\mu > 0$                          &         &               &     &        \\ \hline
\st{(V)}  & $\E e^{\mu |\xi|}< \infty$                   & 1       & (B3), (B8) & $\bmo^{\sqrt{\Psi}}(S_2)$ &     \\
   & for some $\mu > \gamma e^{\beta T}$          &         &               &     &        \\ \hline
\end{tabular}
\end{table}

\st{The spaces $\bmo^{\sqrt{\Psi}}(S_2)$ used in (IV) and (V) are explained in Theorem  \ref{theorem:Z_BMO_Psi} and the remark following it.}
We note that $|Z|\in \bmo^{\sqrt{\Psi}}(S_2)$ also implies $|Z|\in \H_2(S_2)$.
The case (I) follows from \cite[Theorem 4.2]{Briand:Del:Hu:Par:Sto:03} that gives (B6) and 
Remark \ref{remark:theorem:comparison_psi_phi}(2a) yields to $s_\infty=0$. In the following we verify our
contribution (II)-(V).
\medskip

\paragraph{\bf Notation and setting}
There is a series of papers dealing with the quadratic case where the terminal condition is unbounded, see 
\cite{Briand:Hu:06,Briand:Hu:08,Delbaen:Hu:Richou:11,Delbaen:Hu:Richou:15}. Below we use the setting of the initial article 
\cite{Briand:Hu:06}. For future work some extensions of \cite{Briand:Hu:06} done 
in \cite{Morlais:09} might be of interest for our context
as well.
To use the setting of \cite{Briand:Hu:06} we introduce constants $\alpha\ge 0$ and 
$\beta,\gamma>0$ such that, for all \st{$(s,\omega)\in [0,T]\times \Omega$},
\begin{equation}\label{eqn:BH:H1}
|f(s,\omega,y,z)| \le \alpha + \beta |y| + \frac{\gamma}{2} |z|^2
\sptext{1}{and}{1}
\alpha\ge \frac{\beta}{\gamma}.
\end{equation}
In our framework \st{we suppose, for the remainder of this section, 
{\bf that condition (B8) is satisfied}. Moreover, we choose $(\alpha,\beta,\gamma)$ to be}
\begin{eqnarray}
\alpha & := & \max\left \{\sup_{(t,\omega)\in \Omega\times [0,T]} |f(t,\omega,0,0)| + L_Z,
                    \frac{L_Y}{4 L_Z}\right \},\label{eqn:alpha} \\
\beta  & := & L_Y, \label{eqn:beta}\\
\gamma & := & 4 L_Z\label{eqn:gamma}.
\end{eqnarray}
As in \cite{Briand:Hu:06} we use the function $\Phi_t:[0,\infty)\to (0,\infty)$ given by
\[ \Phi_t(y) := e^{\gamma \alpha \frac{e^{\beta (T-t)}-1}{\beta}}
                e^{y\gamma e^{\beta (T-t)}}. \]
Moreover, we set
\[ \mu_T:=  \gamma e^{\beta T} > \gamma \]
which plays the role of a critical exponent in the case $\theta=1$.
Applying \cite[Theorem 2]{Briand:Hu:06} and inspecting its proof gives the following statement:

\begin{theorem}[\cite{Briand:Hu:06}]
\label{theorem:BH}
If there exists a $\mu > \mu_T$ such that 
\[ \E e^{\mu |\xi|} < \infty, \]
then there is a solution to the BSDE  \eqref{equation:BSDE1_intro} such that
\begin{enumerate}
\item $e^{\gamma |Y_t|} \le \E \left ( \Phi_t(|\xi|)|\cF_t \right )$ a.s. for $t\in [0,T]$,
\item $|Z| \in \H_2(S_2)$,
\item for $0\le s < t \le T$ and $\vare>0$ with $\gamma+\vare<\mu$ one has
      \[ \E \left ( \int_s^t |Z_r|^2 dr | \cF_s \right )
          \le c_{\eqref{theorem:BH}}^2 \E \left ( \sup_{r\in [s,t]} e^{(\gamma+\vare)|Y_r|} | \cF_s \right )
              \mbox{ a.s.} \]
       for $c_{\eqref{theorem:BH}}^2 := 2 \left [ \frac{1}{\gamma^2} + \frac{T}{\gamma} \max \{ \alpha,\frac{\beta}{\vare} \} \right ]$.
\end{enumerate}
\end{theorem}
\medskip

\paragraph{\bf Verification of \st{(IV)-(V)}} Here our main observation consists in
\medskip

\begin{theorem}
\label{theorem:Z_BMO_Psi}
Let $\theta \in (0,1]$ and assume that $\mu>\mu_T$ if $\theta=1$ and $\mu>0$ if $\theta\in (0,1)$. 
If $\E e^{\mu |\xi|} < \infty$, then there is a solution to the BSDE \eqref{equation:BSDE1_intro} such that
\begin{equation}\label{eqn:Z-bmo_Psi}
         \E \left (\int_s^T |Z_r|^2 dr| \cF_s \right)  
     \le  c^2_{\eqref{theorem:Z_BMO_Psi}} \Psi_s 
     \sptext{1}{with}{1}  
     \Psi_s := \E \left ( e^{\mu |\xi|}  | \cF_s \right )
\end{equation}
for all $s\in [0,T]$ and  $  c_{\eqref{theorem:Z_BMO_Psi}} = c(\alpha,\beta,\gamma,T,\theta,\mu)\in (0,\infty)$,
where we may assume  $(\Psi_s)_{s\in [0,T]}$ to be path-wise continuous.
Moreover, for all stopping times $\tau:\Omega\to [0,T]$, $B\in \cF_\tau$ of positive measure, and
      $\lambda,\nu>0$, one has 
      \[
      \P_B \left (\int_\tau^T |Z_r|^2 dr > \lambda \nu \right ) 
      \le e^{1-\lambda} + \delta \P_B \left ( \sup_{s\in [\tau,T]} \Psi_s > \frac{\nu}{D} \right ), 
      \]
      where $\P_B$ is the normalized restriction of $\P$ to $B$,  
      $D=D(\alpha,\beta,\gamma,T,\theta,\mu)>0$,
      and $\delta>0$ is an absolute constant.
\end{theorem}
\smallskip

In the spirit of \cite[Definition 1]{Geiss:05} the inequality \eqref{eqn:Z-bmo_Psi} could be abbreviated by
\index{BMO!$\bmo^{\sqrt{\Psi}}(S_2)$}
\[ \| |Z| \|_{\bmo^{\sqrt{\Psi}}(S_2)} \le c_{\eqref{theorem:Z_BMO_Psi}}. \]

\begin{proof}[Proof of Theorem \ref{theorem:Z_BMO_Psi}]
\underline{Case $\theta=1$:}
We choose $\vare>0$ and $p\in (1,\infty)$ such that
\[ \mu = p \mu_T = \frac{\gamma+\vare}{\gamma}\mu_T  \]
which implies by $\beta>0$ that $\gamma + \vare < \mu$.
Assuming $0\le s \le T$ and applying Theorem \ref{theorem:BH} gives, a.s., that
\equa
      \E \left (\int_s^T |Z_r|^2 dr| \cF_s \right) 
&\le& c^2_{\eqref{theorem:BH}}  \E \left ( \sup_{r\in [s,T]} e^{(\gamma + \vare)|Y_r|} | \cF_s \right ) \\
& = &  c^2_{\eqref{theorem:BH}}
      \E \left ( \sup_{r\in [s,T]} e^{p \gamma|Y_r|} | \cF_s \right ) \\
&\le& c^2_{\eqref{theorem:BH}} \E \left ( \sup_{r\in [s,T]} \left [ \E\Big (\Phi_s(|\xi|) |\cF_r\Big ) \right ]^p  | \cF_s \right ) \\
&\le& c^2_{\eqref{theorem:BH}} \left | \frac{p}{p-1} \right |^p 
      \E \left ( \Phi_s(|\xi|)^p  | \cF_s \right ) \\
&\le&  c^2_{\eqref{theorem:BH}} \left | \frac{p}{p-1} \right |^p 
      \kappa_T^p \E \left ( e^{\mu |\xi|}  | \cF_s \right ),
\tion
where $\kappa_T :=  e^{\gamma \alpha \frac{e^{\beta T}-1}{\beta}}$ and for 
$(\E(\Phi_s(|\xi|)|\cF_r))_{r\in [0,T]}$ a continuous modification is taken.
Therefore, letting 
\[ c^2 
   = c^2(\alpha,\beta,\gamma,T,\mu)
  := c^2_{\eqref{theorem:BH}} \left |\kappa_T \frac{p}{p-1} \right |^p, \]
we proved
\[  \E \left (\int_s^T |Z_r|^2 dr| \cF_s \right)  \le c^2 \Psi_s
    \mbox{ a.s.} \]
Using an optional stopping argument, this can be extended to
\[  \E \left (\int_\tau^T |Z_r|^2 dr| \cF_\tau \right)  \le c^2 \Psi_\tau
    \mbox{ a.s.} \]
for any stopping time $\tau:\Omega \to [0,T]$.
Given $\nu>0$ we get 
\equa
       \P_B \left (\int_\tau^T |Z_r|^2 dr > 3 \nu \right ) 
&\le&  \P_B \left (\int_\tau^T |Z_r|^2 dr > 3 c^2 \Psi_\tau \right ) 
       + \P_B \left (c^2 \Psi_\tau > \nu \right ) \\
&\le& \frac{1}{3} + \P_B \left (c^2 \Psi_\tau > \nu \right ).
\tion
If we define
\[    W(B,\nu;\tau) 
   := \P \left ( B \cap \left  \{ \sup_{r\in [\tau,T]} 
      3 c^2 \Psi_r > \nu \right \} \right ), \]
then we can directly apply \cite[Theorem 1]{Geiss:05}. 
\medskip

\underline{Case $\theta\in (0,1)$:}
This case can be  considered exactly as the case $\theta=1$.
In fact, with our choice of parameters $(\alpha,\beta,\gamma)$ in
\eqref{eqn:alpha}, \eqref{eqn:beta}, and \eqref{eqn:gamma} we obtain  the estimate
\[ |f(s,\omega,y,z)| \le \alpha + \beta |y| + \frac{\gamma}{2} |z|^{1+\theta}. \]
But now, for any given $\tilde \gamma>0$ we find an $\tilde \alpha\ge 0$ such that
\[ \alpha + \frac{\gamma}{2} |z|^{1+\theta} \le  \tilde \alpha + \frac{\tilde \gamma}{2} |z|^2 \]
for all $z\in \R^d$. In other words, we can arrange the parameters such that
$\mu > \tilde \gamma e^{\beta T}$ (and have an additional dependence of the constants
on $\theta$).
\end{proof}
\smallskip

\st{
\begin{remark}
\label{remark:s_infty_in_(IV-V)}
Assume equation \eqref{equation:BSDE1_intro} with $T=d=1$, $f\equiv 0$, and suppose
that $\E e^{\mu |\xi|} < \infty$ for all $\mu >0$. Then there is a unique solution $(Y,Z)$ 
under the assumption $Z\in \H_2(S_2)$. As 
we may choose any $\theta\in (0,1]$, we are in the setting of (IV) and (V).
Given $\eta\in (0,1]$, we will construct a $\xi$ as above with $Z\not\in \bmo(S_{2\eta})$.
This means, without any additional assumptions one cannot expect results about finite $s_\infty$
in (IV) and (V) of Table 1. The construction is as follows: 
For $\alpha \in [1,\infty)$ we recall the definition of the Orlicz spaces 
\index{space!$L_{\exp_\alpha}$}
$L_{\exp_\alpha}$ (see \cite{Bennett:Sharpley:88}),
\[   L_{\exp_\alpha}(\Omega,\cF,\P) 
  := \left \{ F \in L_0(\Omega,\cF,\P) :
         \| F\|_{L_{\exp_\alpha}} := \inf \{ \lambda>0 : \E e^{\left (\frac{|F|}{\lambda} \right )^\alpha} \le 2 \} 
    \right \}. \]
We fix $0<\eta\le 1 < \gamma < 2$, determine $\alpha\in (2,\infty)$
by $\frac{1}{\gamma}=\frac{1}{\alpha} + \frac{1}{2}$, and let $t_n:= 1-\frac{1}{2^n}$ for $n\ge 0$.
For $\vare>0$, $n\ge 1$, and $c_n\in (0,\infty)$ we set
\[ v_n(\omega) := 2^{(n+1) \left [\frac{1}{2\eta} + \vare \right ]}\chi_{\{|W_{t_n}(\omega)|\ge c_n \}} \]
so that $\|v_n\|_\infty = 2^{(n+1) \left [\frac{1}{2\eta} + \vare \right ]}$. We choose 
$c_n$ such that $\|v_n\|_{L_{\exp_\alpha}} \le 1$ and define, as in Example \ref{example:weaker_bmo_is_weaker_new},
the process
\[ Z_t := \sum_{n=2}^\infty \chi_{(t_{n-1},t_n]}(t) v_{n-1}. \]
The proof of Example \ref{example:weaker_bmo_is_weaker_new} confirms that 
$Z\not \in \bmo(S_{2\eta})$. On the other hand,
\equa
      \left \| \int_0^1 Z_s dW_s \right \|_{L_{\exp_\gamma}}
&\le& \sum_{n=2}^\infty  \left \| v_{n-1}(W_{t_n}-W_{t_{n-1}}) \right \|_{L_{\exp_\gamma}} \\
&\le& \sum_{n=2}^\infty  \left \| v_{n-1}\right \|_{L_{\exp_\alpha}} \left \| W_{t_n}-W_{t_{n-1}} \right \|_{L_{\exp_2}} \\
& = & \left \| W_1 \right \|_{L_{\exp_2}} \sum_{n=2}^\infty  \left \| v_{n-1}\right \|_{L_{\exp_\alpha}} \sqrt{\frac{1}{2^n}}\\
& < & \infty.
\tion
Therefore it holds that $\xi\in L_{\exp_\gamma}$ with $\gamma>1$, so that 
$\E e^{\mu |\xi|}<\infty$ for all $\mu>0$.
\end{remark}
}

\paragraph{\bf Verification of \st{(II)-(III)}}
The next definition will allow us to deduce that the gradient process $Z$ belongs to $\bmo(S_{2\eta})$:

\begin{definition}
\label{definition:bh_xi}
\index{space!$\bh(\eta,\mu)$}
For $\eta \in (0,1]$ and $\mu \in (0,\infty)$ we let 
\[
       |\xi|_{\bh(\eta,\mu)} 
    := \sup_{t\in [0,T)} (T-t)^{\frac{1}{\eta}-1} \left \| \E(e^{\mu |\xi|}|\cF_t) \right \|_\infty. 
 \]
\end{definition}
In the notation $\bh$ above, 'c' stands for {\em conditional} and 'Exp' for {\em \st{exponential}.}

\begin{remark}
\label{remark:properties_class_bh}
\hspace*{0em}
\begin{enumerate}
\item For $\eta=1$ we have that $|\xi|_{\bh(1,\mu)} = e^{\mu \|\xi\|_\infty}$.
\item For $\xi \in L_2$, $0<\eta<\tilde\eta<1$, and $0<\tilde\mu<\mu<\infty$ with
      $\mu (\frac{1}{\tilde \eta} - 1) = \tilde \mu \left ( \frac{1}{\eta} - 1 \right )$
      one has
      $     |\xi|_{\bh(\tilde \eta,\tilde \mu)}^\mu
        \le |\xi|_{\bh(\eta,\mu)}^{\tilde \mu}$.
\item For $\xi \in L_2$ and $\eta_0,\eta_1\in (0,1)$ one has 
      $|\xi|_{\bh(\eta_0,\mu_0)}<\infty$ for some $\mu_0 \in (0,\infty)$ if and only if 
      $|\xi|_{\bh(\eta_1,\mu_1)}<\infty$ for some $\mu_1\in (0,\infty)$.
\end{enumerate}
\end{remark} 
\smallskip

\begin{proof} Part (1) is obvious, (3) follows directly from (2). The assertion (2) is a consequence of
\equa
      |\xi|_{\bh(\tilde \eta,\tilde \mu)} 
& = & \sup_{t\in [0,T]} (T-t)^{\frac{1}{\tilde \eta}-1} \left \| \E(e^{\tilde \mu |\xi|}|\cF_t) \right \|_\infty \\
& = & \sup_{t\in [0,T]} (T-t)^{\frac{1}{\tilde \eta}-1} \left \| \E(e^{\mu \frac{\tilde \mu}{\mu}
      |\xi|}|\cF_t) \right \|_\infty \\
&\le& \sup_{t\in [0,T]} (T-t)^{\frac{1}{\tilde \eta}-1} \left \| \E(e^{\mu |\xi|}|\cF_t) \right \|_\infty^\frac{\tilde \mu}{\mu}\\
& = & \left [
      \sup_{t\in [0,T]} (T-t)^{\frac{1}{\eta}-1} \left \| \E(e^{\mu |\xi|}|\cF_t) \right \|_\infty
      \right ]^\frac{\tilde \mu}{\mu} \\
& = & |\xi|_{\bh(\eta,\mu)}^\frac{\tilde \mu}{\mu}.
\tion
\end{proof}

Directly from Theorem \ref{theorem:Z_BMO_Psi} we deduce
\smallskip

\begin{cor}
Assume $\theta =1$, $\eta\in (0,1]$, and in addition to the assumptions made in
Theorem \ref{theorem:Z_BMO_Psi} that $|\xi|_{\bh(\eta,\mu)} <\infty$ for some $\mu \in (0,\infty)$.
Then $|Z|\in \bmo(S_{2\eta})$ with
   \[      \| |Z| \|_{\bmo(S_{2\eta})} 
         \le  c_{\eqref{theorem:Z_BMO_Psi}} |\xi|^{\frac{1}{2}}_{\bh(\eta,\mu)}. \] 
\end{cor}
\smallskip

\begin{proof}
We simply have that
\[       
         \E \left (\int_s^T |Z_r|^2 dr| \cF_s \right)  
    \le  c^2_{\eqref{theorem:Z_BMO_Psi}} \E \left ( e^{\mu |\xi|}  | \cF_s \right )
    \le  c^2_{\eqref{theorem:Z_BMO_Psi}} |\xi|_{\bh(\eta,\mu)}  (T-s)^{1-\frac{1}{\eta}}
    \mbox{ a.s.}
\] 
for all $s\in [0,T]$ and therefore, a.s.,
\[
      \E \left ( \left ( \int_s^T |Z_r|^{2\eta} dr \right )^\frac{1}{\eta} | \cF_s \right ) 
 \le  (T-s)^{\frac{1}{\eta}-1} \E \left (\int_s^T |Z_r|^2 dr| \cF_s \right) 
 \le  c^2_{\eqref{theorem:Z_BMO_Psi}} |\xi|_{\bh(\eta,\mu)}.
\]
\end{proof}
\bigskip
The above corollary explains the case (III) from Table 1. 
It turns out that in the remaining case (II) the particular choice of parameter 
$\eta$ in $|\cdot|_{\bh(\eta,\mu)}$ does not have an impact. This is reflected by the following notation:

\begin{definition}
\label{definition:BH}
\index{space!$\bh(\eta,\mu)$}
\index{space!$\bh$}
\hspace*{0em}
\begin{enumerate}
\item For a c\`adl\`ag process $Y=(Y_t)_{t\in [0,T]}$ and 
      $(\eta,\mu) \in (0,1)\times (0,\infty)$ we let 
      \[ 
             |Y|_{\bh(\eta,\mu)} 
         :=  \sup_{t\in [0,T)} (T-t)^{\frac{1}{\eta}-1} \left \| \E(e^{\mu \sup_{s\in [t,T]} |Y_s|}|\cF_t) \right \|_\infty. 
      \]
We say that $Y\in \bh$ provided that $|Y|_{\bh(\eta,\mu)} < \infty$ for some $(\eta,\mu) \in (0,1)\times (0,\infty)$.
\item We say  
$\xi\in \bh$ provided that $|\xi|_{\bh(\eta,\mu)} < \infty$ for some $(\eta,\mu) \in (0,1)\times (0,\infty)$.
\end{enumerate}
\end{definition}
The definition of $|Y|_{\bh(\eta,\mu)}$ is consistent with Definition \ref{definition:bh_xi} as for
a random variable $\xi$ we may let $Y_t:=\xi$ and get $|Y|_{\bh(\eta,\mu)}=|\xi|_{\bh(\eta,\mu)}$. 

\begin{remark}
Exactly as in Remark \ref{remark:properties_class_bh} one can show that for $\eta_0,\eta_1\in (0,1)$ one has 
      $|Y|_{\bh(\eta_0,\mu_0)}<\infty$ for some $\mu_0 \in (0,\infty)$ if and only if 
      $|Y|_{\bh(\eta_1,\mu_1)}<\infty$ for some $\mu_1\in (0,\infty)$.
Therefore, $Y\in \bh$ if and only if there is some $\mu\in (0,\infty)$ such that
\[
    \sup_{t\in [0,T)} (T-t) \left \| \E(e^{\mu \sup_{s\in [t,T]} |Y_s|}|\cF_t) \right \|_\infty < \infty.
 \]
\end{remark}
\medskip

\begin{theorem}
\label{theorem:existence_uniqueness_bh_subquadratic}
Assume that $\theta\in (0,1)$ and $\xi \in \bh$. Then there is a unique solution $(Y,Z)$ to the 
{\rm BSDE} \eqref{equation:BSDE1_intro} in the class where $Y\in \bh$ and $|Z| \in \H_2(S_2)$.
Moreover, for this solution we have that
\begin{enumerate}
\item $s_\infty=0$ for $s_\infty$ defined as in Theorem \ref{theorem:comparison_psi_phi},
\item $|Z|\in \bmo(S_{2\eta})$ for all $\eta\in (0,1)$.
\end{enumerate}
\end{theorem}
\smallskip 

For the uniqueness in the above theorem we do not assume convexity properties of the generator. Instead of that, we
use $|Z| \in \bmo(S_{2\theta})$ and follow the methodology that BMO-properties of the $Z$ process give uniqueness, see
for example \cite{Hu:Imkeller:Mueller:05}. The difference to previous settings is that we exploit that the generator is
sub-quadratic and get therefore a weaker condition than the standard BMO-condition $|Z| \in \bmo(S_2)$.
Note that according to Example \ref{example:weaker_bmo_is_weaker_new} the spaces $\bmo(S_{2\eta})$ do not
coincide for different $\eta \in (0,1]$ in general.

\begin{proof}[Proof of Theorem \ref{theorem:existence_uniqueness_bh_subquadratic}]
\underline{Existence:}
The condition $\xi\in \bh$ implies that there are $(\eta,\mu)\in (0,1)\times (0,\infty)$ such that
\[ |\xi|_{\bh(\eta,\mu)} = \sup_{t\in [0,T)} (T-t)^{\frac{1}{\eta}-1} \left \| \E(e^{\mu |\xi|}|\cF_t )\right \|_\infty < \infty. \]
Because of $\theta<1$ we use the argument for the case $\theta \in (0,1)$ 
from the proof of Theorem \ref{theorem:Z_BMO_Psi} to replace $(\alpha,\beta,\gamma)$ by $(\tilde \alpha,\beta,\tilde \gamma)$ such that
\[ \mu > \tilde \mu_T := \tilde \gamma e^{\beta T}> \tilde \gamma. \]
We apply Theorem \ref{theorem:BH} and obtain a solution with
\begin{enumerate}
\item $e^{\tilde \gamma |Y_t|} \le \E(\tilde \Phi_t (|\xi|)|\cF_t)$ a.s. for $t\in [0,T]$,
\item $|Z|\in \H_2(S_2)$,
\end{enumerate}
where $\tilde\Phi _t$ is defined as $\Phi_t$ with $(\alpha,\beta,\gamma)$ replaced by $(\tilde \alpha,\beta,\tilde \gamma)$.
Let $\tilde p:= \mu/\tilde \mu_T\in (1,\infty)$ and assume $\tilde \gamma + \vare < \mu$ for some
$\vare>0$. Assuming $s \in [0,T)$, the arguments from the proof of Theorem \ref{theorem:Z_BMO_Psi} give, a.s., that
\equa
      \E \left ( \sup_{r\in [s,T]} e^{(\tilde \gamma + \vare)|Y_r|} | \cF_s \right ) 
&\le& \left | \frac{\tilde p}{\tilde p-1} \right |^{\tilde p}
      \tilde \kappa_T^{\tilde p} \E \left ( e^{\mu |\xi|}  | \cF_s \right ) \\
&\le& \left | \frac{\tilde p}{\tilde p-1} \right |^{\tilde p}
      \tilde \kappa_T^{\tilde p} 
      |\xi|_{\bh(\eta,\mu)} (T-s)^{1-\frac{1}{\eta}} 
\tion
where $\tilde \kappa_T :=  e^{\tilde \gamma \tilde \alpha \frac{e^{\beta T}-1}{\beta}}$.
Therefore, $|Y|_{\bh(\eta,\tilde \gamma +\vare)}< \infty$ and $Y\in \bh$.
\smallskip

\underline{Uniqueness:} Assume two solutions 
$(Y^0,Z^0)$ and $(Y^1,Z^1)$ with $Y^0,Y^1\in \bh$ and $Z^0,Z^1\in \H_2(S_2)$.
Let us fix $\eta \in (0,1)$ and find 
$\mu_0,\mu_1\in (0,\infty)$ such that 
\[
 |Y^i|_{\bh(\eta,\mu_i)} 
 =  \sup_{t\in [0,T)} (T-t)^{\frac{1}{\eta}-1} \left \| \E(e^{\mu_i \sup_{s\in [t,T]} |Y_s^i|}|\cF_t) \right \|_\infty
 < \infty.
 \]
Again exploiting $\theta<1$, we change in \eqref{eqn:BH:H1} the parameters $(\alpha,\beta,\gamma)$ 
to  $(\tilde \alpha,\beta,\tilde \gamma)$ such that 
\[ \mu:=\min \{ \mu_0,\mu_1\} > \tilde \gamma e^{\beta T}. \]
Analyzing the proof of \cite[Theorem 2, pp. 609-610]{Briand:Hu:06} gives 
for $0\le s < T$ and $\vare>0$ with $\tilde \gamma+\vare<\mu$ that
      \[ \E \left ( \int_s^T |Z_r^i|^2 dr | \cF_s \right )
          \le 2 \left [ \frac{1}{\tilde \gamma^2} + \frac{T}{\tilde \gamma} \max \left \{ \tilde \alpha,\frac{\tilde \beta}{\vare} 
              \right \} 
              \right ] \E \left ( \sup_{r\in [s,T]} e^{(\tilde \gamma+\vare)|Y_r^i|} | \cF_s \right )
              \mbox{ a.s.} \]
We continue with
\[
      \E \left ( \sup_{r\in [s,t]} e^{(\tilde \gamma+\vare)|Y_r^i|} | \cF_s \right ) 
  \le \E \left ( \sup_{r\in [s,t]} e^{\mu_i|Y_r^i|} | \cF_s \right ) \\
  \le |Y^i|_{\bh(\eta,\mu_i)}  (T-s)^{1-\frac{1}{\eta}}.
\]
Therefore, for 
$\tilde c^2 := 2 \left [ \frac{1}{\tilde \gamma^2} + \frac{T}{\tilde \gamma} 
               \max \left \{ \tilde \alpha,\frac{\tilde \beta}{\vare} \right \} 
               \right ]$, a.s.,
\[
      \E \left ( \left ( \int_s^T |Z_r^i|^{2\eta} dr \right )^{\frac{1}{\eta}} | \cF_s \right ) 
  \le (T-s)^{\frac{1}{\eta}-1} \E \left (\int_s^T |Z_r^i|^2 dr| \cF_s \right)
  \le \tilde c^2  |Y^i|_{\bh(\eta,\mu_i)}. 
\]
This implies that $Z^0,Z^1\in \bmo(S_{2\eta})$ for all $\eta \in (0,1)$. 
In particular, we have that $Z^0,Z^1\in \bmo(S_{2\theta})$ and this enables us to apply 
Lemma \ref{lemma:briand:elie_new}. Here we set
\equa
f^0(s,y,z) &:=& f(s,y,z), \\
f^1(s)     &:=& f(s,Y_s^1,Z_s^1).
\tion
The assumptions (D1), (D2), and (D4) are obviously satisfied, for (D3) we use that
\[ \left ( \E \left | \int_0^T |Z_s^i|^{1+\theta} ds \right |^2 \right )^\frac{1}{2}
   \le c_{\eqref{cor:AB_generalized_Fefferman_inequality_new_constant},2} \| |Z^i|\|_{\H_2(S_2)} \| |Z^i|^\theta \|_{\bmo(S_2)} \]
where $\| |Z^i|^\theta \|_{\bmo(S_2)}<\infty$ because of $|Z^i|\in \bmo(S_{2\theta})$.
The above definitions guarantee that $\Xi_s\equiv 0$. A straightforward computation
gives also that
\[     \E \left ( \int_t^T |c_s|^2 ds |\cF_t\right )
   \le L_Z^2 3^{2\theta} \left [ T + \| |Z^0| \|_{\bmo(S_{2\theta})}^{2\theta} + \| |Z^1| \|_{\bmo(S_{2\theta})}^{2\theta} \right ] 
       \mbox{ a.s.} \]
so that $\|c\|_{\bmo(S_2)} < \infty$. It remains to show that $p_0$ can be chosen such that $p_0\in (1,2)$.
Here we repeat the above argument and check, for $0\le a < b \le T$ and $\eta\in (\theta,1)$, that
\begin{multline*}
       \E \left ( \int_a^b |c_s|^2 ds |\cF_a\right ) \\
   \le L_Z^2 3^{2\theta} \left [ (b-a) 
       + (b-a)^{1-\frac{\theta}{\eta}}[ \| |Z^0| \|_{\bmo(S_{2\eta})}^{2\theta} + \| |Z^1| \|_{\bmo(S_{2\eta})}^{2\theta} ]\right ]
       \mbox{ a.s.}
\end{multline*}
This yields $\lim_N \sli_N^{S_2}(c) =0$ and we can choose $p_0\in (1,2)$. Therefore we may apply 
Lemma \ref{lemma:briand:elie_new} with $p=2$ and this yields uniqueness.
\medskip

The \underline{conclusion $s_\infty=0$} follows by Remark \ref{remark:theorem:comparison_psi_phi} (2b),
which is the same reasoning as used for $\lim_N \sli_N^{S_2}(c) =0$ above.
\end{proof}
\medskip

\begin{remark}
\label{remark:existence_uniqueness_bh_quadratic}
Theorem \ref{theorem:existence_uniqueness_bh_subquadratic} is an extension of the known case $\theta=1$
(cf. \cite{Hu:Imkeller:Mueller:05,Morlais:09}).
For $\theta=1$ and $\xi\in L_\infty$ Theorem \ref{theorem:BH} gives a solution $(Y,Z)$ with 
$\sup_{t\in [0,T]}\| Y_t \|_\infty < \infty$ and $|Z|\in \bmo(S_2)$. Assuming two such solutions, we may follow 
the (second half of the) part about uniqueness in the proof of Theorem \ref{theorem:existence_uniqueness_bh_subquadratic}. Here the 
difference is that we only get {\em some} $p_0\in (1,\infty)$ for applying Lemma \ref{lemma:briand:elie_new}. 
However, $|Z^0-Z^1|\in \bmo(S_2)$ implies that all moments of $\int_0^T |Z^0_s - Z^1_s|^2 ds$ exist and
Lemma \ref{lemma:briand:elie_new} is applicable for any $p\in (p_0,\infty)\cap [2,\infty)$. 
Therefore, in the case $\theta=1$ and $\xi\in L_\infty$ the solution 
$(Y,Z)$ is unique when $\sup_{t\in [0,T]}\| Y_t \|_\infty < \infty$ and $|Z|\in \bmo(S_2)$.
\end{remark}
\medskip

We finish by an example illustrating $\xi\in\bh$.

\begin{example}
Let $d=1$, $\eta\in (0,1)$, 
\[ \varphi_\eta(t) := \log \left ( 1 + (T-t)^{1-\frac{1}{\eta}} \right )
   \sptext{1}{for}{1}
   t\in [0,T), \]
so that $\varphi_\eta(t) \uparrow \infty$ as $t\to T$
and define the stopping time
\[ \tau_\eta := \inf \left \{ t\in [0,T) :
                W_t = \varphi_\eta(t) \right \} \wedge T.\]
Let 
\[ e^\xi := 1 + e^{W_{\tau_\eta}-\frac{\tau_\eta}{2}} \]
so that $\xi(\omega)\in(0,\infty)$ and
\[ \E \left ( e^\xi | \cF_t\right ) =  1 + e^{W_{\tau_\eta\wedge t}-\frac{\tau_\eta\wedge t}{2}}
                            \le 2 + (T-t)^{1-\frac{1}{\eta}} \mbox{ a.s.} \]
for $t\in [0,T)$. On the other hand, $\xi\not \in L_\infty$ because for all $c>0$ one has that
$\P(W_{\tau_\eta}>c)>0$. The latter fact can be checked by taking any 
$0<\vare < \varphi_\eta(0)<c<\infty$ and $S\in (0,T)$ with $c<\varphi_\eta(S)$
and using the known fact that $\P(\sup_{t\in [0,S]} |W_t| \le \vare )>0$ so that the probability that the Brownian motion 
exceeds $\varphi_\eta$ on $[S,(S+T)/2]$ is positive.
\end{example}


\section{Settings for the stability theorem}
\label{sec:settings_for_theorem:comparison_psi_phi}

The aim of this section is to discuss some settings for the stability Theorem \ref{theorem:comparison_psi_phi}.

\subsection{Forward setting} 
\label{subsec:forward_setting}

This setting corresponds to the setting of stochastic integration. If the generator $f$ does not depend on $Y$, then the process $Y$ computes directly as
\[ Y_t = Y_0 - \int_0^t f(s,Z_s) ds + \int_0^t Z_s dW_s. \]
This enables us to construct examples to understand what the correct conditions on $Z$ in the
quadratic case might be. Let us mention two cases:

\begin{enumerate}[(a)]
\item  Taking $Z$ from Example \ref{example:weaker_bmo_is_weaker_new} for $0<\theta < \eta =1$,
       we have examples where the $Z$-process fails to be in $\bmo(S_2)$ but satisfies $Z\in \bmo(S_{2\theta})$ and
       $\int_0^T |Z_s|^2 ds \in L_{\exp}$. The latter enables us to apply
       Lemma \ref{lemma:Y_in_Lp} under suitable integrability conditions on 
       $\int_0^T |f(s,0)| ds$ (note that $L_{\exp} \subseteq L_p$ for all $p\in (0,\infty$)).
\item Similarly, for $\theta=1$ we obtain an $L_p$-solution of our BSDE
      under (B3), (B5), and $\left ( \int_0^T |Z_t|^2 dt \right )^\frac{1}{2} \in L_{2p}$
      (see the arguments at the end of Section \ref{sec:setting_bsdes}).
      Therefore we can take any $Z\in \bmo(S_2)$, in particular, $Z$ can be an unbounded BMO-process 
      in the quadratic setting.
\end{enumerate}
\bigskip


\subsection{Potential estimates for the generator}
\label{sec:potential_for_generator}

In applications of Theorem \ref{theorem:comparison_psi_phi} one might need to estimate
\[
 \left \| \int_t^T | f^\varphi(s,Y^\psi_s,Z^\psi_s) - f^\psi(s,Y^\psi_s,Z^\psi_s)| ds 
            \right \|_p
\]
from above. One way to do this (we do not consider the remaining assumptions for
Theorem \ref{theorem:comparison_psi_phi}) is to find a potential estimate 
\[ |f^\varphi(s,y,z) - f^\psi(s,y,z)| 
 \le | \langle  (1,|y|,|z|,|z|^{1+\theta}), V_s^\varphi - V_s^\psi \rangle | \]
for all $(s,y,z)$ where the potential $(V_s)_{s\in [0,T]}$ is a predictable process
\[ V_s : \Omega \to \R^4.\]
Below we illustrate some special cases for $V$.
The general construction is as follows: We consider a continuous 
\[ h : [0,T]\times \R^N \times \R \times \R^d \to \R, \]
where $N\ge 1$,
and a predictable $\R^N$-valued process $A=(A_t)_{t\in [0,T]}$ on $\Omega$ to let
\[ f(t,\omega,y,z) := h(t,A_t(\omega),y,z). \]
Then $f$ is $(\cP, \cB(C(\R^{1+d})))$-measurable.
Assume that $A^\vph=(A_t^\varphi)_{t\in [0,T]}$ is a $\cP^\varphi$-measurable
\st{representative of $\widetilde{A}^\varphi$}, where
$\widetilde{A}$ is the canonical extension of $A$ to $\overline{\Omega}$. We get that 
\[ f^\varphi(t,\overline{\omega},y,z) := h(t,A_t^\varphi(\overline{\omega}),y,z) \]
is $(\cP^\varphi, \cB(C(\R^{1+d})))$-measurable and, for any fixed $(y,z)\in \R^{1+d}$, that
$f^\varphi(\cdot,\cdot,y,z\st{)}: [0,T]\times \overline{\Omega}\to \R$ is a 
\st{representative of $\widetilde{f}^\varphi$, where $\widetilde{f}$ is}
the canonical extension $\widetilde{f}(\cdot,\cdot,y,z\st{)}: [0,T]\times \overline{\Omega}\to \R$
(see \st{Proposition \ref{proposition:properties_C_T}(4)} applied to $\st{X}_{t,1}=t$ and 
$(\st{X}_{t,2}(\overline{\omega}),...,\st{X}_{t,{N+1}}(\overline{\omega}))=\widetilde{A}_t(\overline{\omega})$).
Therefore we will take in the sequel as transformed generator the map
$f^\varphi$ as defined above.

\begin{example}
\label{example:driver_that_depends_on_X}
Let
\[ f(s,\omega,y,z):= h(s,A_s(\omega),y,z), \]
where $h: [0,T]\times \R \times \R \times \R^d\to \R$ is continuous with
\[ |h(t,x_0,y_0,z_0)-h(t,x_1,y_1,z_1)|
   \le L_X |x_0-x_1| + L_Y|y_0-y_1| + L_Z [1+|z_0|+|z_1|] |z_0-z_1| \] 
for all $(t,x_0,x_1,y_0,y_1,z_0,z_1)$ and $(A_t)_{t\in [0,T]}$ is a predictable process. Then we get
\[ |f^\varphi(s,y,z) - f^\psi(s,y,z)|  \le L_X|A_s^\varphi - A_s^\psi |
   \sptext{1}{and}{1} V_s:= (L_X A_s,0,0,0) \]
and  
\[ 
\left \| \int_t^T | f^\varphi(s,Y^\psi_s,Z^\psi_s) - f^\psi(s,Y^\psi_s,Z^\psi_s)| ds 
            \right \|_p \le 
L_X \left \| \int_t^T | A_s^\varphi - A^\psi_s| ds 
            \right \|_p. \]
\end{example}
\bigskip

The next example indicates the case of random Lipschitz constants for $y$:
\medskip

\begin{example}
Assume that
\[ f(s,\omega,y,z) := A_s(\omega) g(y) \]
where $g:\R\to \R$ is a Lipschitz function and $(A_s)_{s\in [0,T]}$ is predictable
and uniformly bounded in $(s,\omega)$.
Then
\[ |f^\varphi(s,y,z) - f^\psi(s,y,z)|  \le |g(y)| |A_s^\varphi - A_s^\psi |
   \le [ |g(0)| + \Lip(g) |y| ] |A_s^\varphi - A_s^\psi | \]
and $V_s:= (|g(0)|A_s,\Lip(g) A_s,0,0)$.
Here we get (for example) that
\equa
&    & \left \| \int_t^T | f^\varphi(s,Y^\psi_s,Z^\psi_s) - f^\psi(s,Y^\psi_s,Z^\psi_s)| ds 
            \right \|_p \\
&\le&  |g(0)| \left \| \int_t^T | A_s^\varphi - A^\psi_s| ds 
            \right \|_p +
  \Lip(g)  \left \| \int_t^T |Y_s^\psi| | A_s^\varphi - A^\psi_s| ds 
            \right \|_p  \\
&\le& 
|g(0)| \left \| \int_t^T | A_s^\varphi - A^\psi_s| ds 
            \right \|_p +
  \Lip(g)  \left \| {^* (Y^\psi)}_t \int_t^T |A_s^\varphi - A^\psi_s| ds 
            \right \|_p  \\
&\le& 
|g(0)| \left \| \int_t^T | A_s^\varphi - A^\psi_s| ds 
            \right \|_p +
  \Lip(g)  \| {^* (Y^\psi)}_t \|_{p_0} \left \| \int_t^T |A_s^\varphi - A^\psi_s| ds 
            \right \|_{p_1} \\
&\le& 
[ |g(0)| +  \Lip(g)  \| {^* Y}_t \|_{p_0}] 
  \left \| \int_t^T |A_s^\varphi - A^\psi_s| ds 
            \right \|_{p_1}             
\tion
for any $\frac{1}{p}=\frac{1}{p_0}+\frac{1}{p_1}$ with $p<p_0,p_1<\infty$ and 
\index{${^* C}_t$}
\[ {^* C}_t := \sup_{s\in [t,T]} |C_s|, \]
where we used
that ${^* (Y^\psi)}_t$ and ${^* Y}_t$ have the same distribution which follows from \st{Proposition \ref{proposition:properties_C_T}(3)}.
\end{example}
\bigskip

The last example concerns the $Z$ component.
\medskip

\begin{example}
Assume that
\[ f(s,\omega,y,z) := A_s(\omega) |z|^{1+\theta} \]
with $\theta\in (0,1)$,
where $(A_s)_{s\in [0,T]}$ is predictable 
and uniformly bounded in $(s,\omega)$.
Then
\[ |f^\varphi(s,y,z) - f^\psi(s,y,z)|  \le |z|^{1+\theta} |A_s^\varphi - A_s^\psi | \]
and $V_s:= (0,0,0,A_s)$. Because \st{of}
\equa
      |f(s,\omega,y_0,z_0)- f(s,\omega,y_1,z_1)|
&\le& |A_s(\omega)| \Big | | z_0|^{1+\theta} - | z_1|^{1+\theta} \Big | \\
&\le& [1+\theta] |A_s(\omega)| \Big | | z_0| - | z_1| \Big | [1+|z_0|+|z_1|]^\theta \\
&\le& [1+\theta] |A_s(\omega)|  | z_0 -  z_1  | [1+|z_0|+|z_1|]^\theta\st{,} 
\tion
\st{the} condition (B3) is satisfied. Then an upper bound is obtained by
\equa
&    & \left \| \int_t^T | f^\varphi(s,Y^\psi_s,Z^\psi_s) - f^\psi(s,Y^\psi_s,Z^\psi_s)| ds 
            \right \|_p \\
&\le& \left \| \int_t^T |Z_s^\psi|^{1+\theta} |A_s^\varphi - A_s^\psi|  ds 
            \right \|_p \\
&\le& \left \| \left ( \int_t^T |Z_s^\psi|^2 ds \right )^{\frac{1+\theta}{2}}
               \left ( \int_t^T |A_s^\varphi - A_s^\psi|^\frac{2}{1-\theta}  ds  \right )^{\frac{1-\theta}{2}}
            \right \|_p \\
&\le& \left \| \left ( \int_t^T |Z_s|^2 ds \right )^{\frac{1}{2}}
                \right \|_{(1+\theta)p_0}^{1+\theta}
            \left \| 
               \left ( \int_t^T |A_s^\varphi - A_s^\psi|^\frac{2}{1-\theta}  ds  \right )^{\frac{1-\theta}{2}}
            \right \|_{p_1}
\tion
for any $\frac{1}{p}=\frac{1}{p_0}+\frac{1}{p_1}$ with $p<p_0,p_1<\infty$, where we use Remark \ref{remark:consistent_new}(2)
in the last step.
\end{example}


\subsection{Theorem \ref{theorem:comparison_psi_phi}
         for the perturbation $(\varphi,\psi)=(\chi_{(a,b]},0)$}

The importance of the pair $(\varphi,\psi)=(\chi_{(a,b]},0)$ follows from the fact 
that
\[ \| Y_t -Y_t^{(t-\vare,t]} \|_p \sim_2 
   \| Y_t -\E(Y_t|\cF_{t-\vare}) \|_p \]
for $p\in [1,\infty]$, i.e. the fractional smoothness of $Y_t$ is measured in terms of the speed of convergence
of the conditional expectations. In the case $(\varphi,\psi)=(\chi_{(a,b]},0)$ we have that 
\eqref{eqn:theorem:comparison_psi_phi} implies two inequalities that give different
information about the $L_p$-variation of the processes $Y=(Y_t)_{t\in [0,T]}$ and $Z=(Z_t)_{t\in [0,T]}$: Firstly, for $0<\vare<t$
we have that
\begin{multline}\label{eqn:comparison_inequality_1}
      \left \|  \st{Y_t - Y_t^{(t-\vare,t]}}  \right \|_p 
   +  \left \| \left ( \int_t^T |Z_s - Z_s^{(t-\vare,t]} |^2 ds \right )^\frac{1}{2} 
                \right \|_p \\
\le c_\eqref{theorem:comparison_psi_phi}  \left [ \|\xi-\xi^{(t-\vare,t]} \|_p + 
            \left \| \int_t^T | f(s,Y_s,Z_s) - f^{(t-\vare,t]}(s,Y_s,Z_s)| ds 
            \right \|_p \right ]
\end{multline}
and, secondly,
\begin{multline}\label{eqn:comparison_inequality_2}
       \left \| \left ( \int_{t-\vare}^t |Z_s|^2 ds \right )^\frac{1}{2} 
                \right \|_p \\
\le c_\eqref{theorem:comparison_psi_phi}  \left [ \|\xi-\xi^{(t-\vare,t]} \|_p + 
            \left \| \int_{t-\vare}^T | f(s,Y_s,Z_s) - f^{(t-\vare,t]}(s,Y_s,Z_s)| ds 
            \right \|_p \right ].
\end{multline}


\subsection{Theorem \ref{theorem:comparison_psi_phi} and Besov spaces}
\label{sec:formulation_anisotropic_Besov}

We want to transform Theorem \ref{theorem:comparison_psi_phi} into an embedding theorem
for the Besov spaces $\B^\Phi_p$. As the BSDEs we consider might be even quadratic we have - in some sense -
a non-linear embedding theorem. To handle the assumption on the generator we need a slight 
extension of our anisotropic Besov spaces:
\smallskip

\begin{definition}
\index{functional!$\indexnorm \cdot \indexnorm_{\Phi,q}^{r,t}$}
For $q,r\in [1,\infty)$, a predictable process $(A_t)_{t\in [0,T]}$ with
\[ \left \| \left ( \int_0^T |A_s|^r ds \right )^\frac{1}{r} \right \|_q < \infty, \]
for $t\in [0,T]$, and for an admissible functional $\Phi$ we let
\[ \| A \|_{\Phi,q}^{r,t}
   := \Phi \left ( \psi \to \left \| \left ( \int_t^T |A_s-A^\psi_s|^r ds \right )^\frac{1}{r} \right \|_q\right ).\]
\end{definition}
\bigskip

First we show that this definition is possible:

\begin{lemma}
The map 
\[  \psi \to \left \| \left ( \int_t^T |A_s-A^\psi_s|^r ds \right )^\frac{1}{r} \right \|_q \]
is continuous as a map from $\ws$ into $[0,\infty)$.
\end{lemma}

\begin{proof}
We fix an $N\ge 1$ and consider the truncation  $A_t^N := (-N) \vee (A_t \wedge N)$. 
For $u:= q\vee r$ and $\psi_n,\psi\in \ws$ we get
\equa
&   & \left \| \left ( \int_t^T |A_s^{\psi_n}-A^\psi_s|^r ds \right )^\frac{1}{r} \right \|_q \\
&\le& \left \| \left ( \int_t^T |A_s^{\psi_n}-(A^N)^{\psi_n}_s|^r ds \right )^\frac{1}{r} \right \|_q 
      + \left \| \left ( \int_t^T |(A^N)_s^{\psi_n}-(A^N)^\psi_s|^r ds \right )^\frac{1}{r} \right \|_q \\
&   & + \left \| \left ( \int_t^T |(A^N)_s^{\psi}-A^\psi_s|^r ds \right )^\frac{1}{r} \right \|_q \\
& = & 2 \left \| \left ( \int_t^T |A_s-(A^N)_s|^r ds \right )^\frac{1}{r} \right \|_q 
      + \left \| \left ( \int_t^T |(A^N)_s^{\psi_n}-(A^N)^\psi_s|^r ds \right )^\frac{1}{r} \right \|_q \\
&\le& 2    
      \left \| \left ( \int_t^T |A_s-(A^N)_s|^r ds \right )^\frac{1}{r} \right \|_q 
      +c_{q,r,T}   \left ( \int_t^T \|(A^N)_s^{\psi_n}-(A^N)^\psi_s\|^u_u ds \right )^\frac{1}{u}  \\  
\tion
where we used for the equality Remark \ref{remark:consistent_new}(2). Applying dominated convergence twice
we get that
\[ \lim_N  \left \| \left ( \int_t^T |A_s-(A^N)_s|^r ds \right )^\frac{1}{r} \right \|_q = 0. \]
Moreover, using \st{Proposition \ref{proposition:properties_C_T}(7)} we find a Borel set $B\subseteq [0,T]$ of Lebesgue measure 
$T$ such that
$A_t^\rho$ is the transformation of $A_t$ for any $t\in B$ and $\rho \in \{ \psi,\psi_1,\psi_2,...\}$.
In case $\psi_n\to\psi$ we can therefore apply Lemma \ref{lemma:exchange_continuity} to conclude the proof because this implies that
\[ \lim_n  \left ( \int_t^T \|(A^N)_s^{\psi_n}-(A^N)^\psi_s\|^u_u ds \right )^\frac{1}{u}  = 0. \qedhere \]
\end{proof}

Now we obtain the following embedding theorem:

\begin{cor}\label{cor:theorem:comparison_psi_phi}
Assume that the assumptions of Theorem \ref{theorem:comparison_psi_phi} are satisfied,
$t\in [0,T]$, and that there are predictable processes $(V_s^l)_{s\in [t,T]}$ 
such that, for all $\psi\in \ws$,
\[ \left \| \int_t^T | f(s,Y_s^\psi,Z_s^\psi) - f^\psi(s,Y_s^\psi,Z_s^\psi)| ds 
            \right \|_p
   \le \sum_{l=1}^L \| V_\cdot^l - (V_\cdot^l)^\psi \|_{L_{q_l}(L_{r_l}([t,T]))} \] 
for some $q_l\in [p,\infty)$ and $r_l\in [1,\infty)$
\footnote{The $V^l$ may depend on $(\xi,f,Y,Z,p,q_l,r_l)$.}. 
Let $\Phi: C^+(\Delta)\to [0,\infty]$ be admissible in the sense of 
Definition \ref{definition:admissible}. Then we have that
\[  \| Y_t \|_{\Phi,p} + \| Z \|_{\Phi,p}^{2,t}
                      \le 2 c_{\eqref{theorem:comparison_psi_phi}}
    \left [ \| \xi \|_{\Phi,p}  + \sum_{l=1}^L \| V^l \|_{\Phi,q_l}^{r_l,t}
    \right ]. \]
\end{cor}

\begin{proof}
The statement follows directly from Theorem \ref{theorem:comparison_psi_phi} applied to the pair
$(0,\psi)$.
\end{proof}

Examples, how to obtain processes $(V_s^l)_{s\in [t,T]}$, can be found in Section 
\ref{sec:potential_for_generator}.
For the sake of illustration we first combine Corollary \ref{cor:theorem:comparison_psi_phi} with 
Theorem \ref{theorem:existence_uniqueness_bh_subquadratic} (note that we use conditions (B3) and (B8)) so 
that the assumptions of Theorem \ref{theorem:comparison_psi_phi} are automatically satisfied with $p=2$ and $s_\infty=0$:
\medskip

\begin{cor}\label{cor:theorem:comparison_psi_phi_2}
Assume that $\theta\in (0,1)$, $t\in [0,T]$, $\xi \in \bh$, and
that $(Y,Z)$ is the unique solution to the BSDE \eqref{equation:BSDE1_intro} obtained in
Theorem \ref{theorem:existence_uniqueness_bh_subquadratic}.
Suppose a predictable process $(V_s)_{s\in [t,T]}$ such that
\[ \left \| \int_t^T \sup_{y,z} | f(s,y,z) - f^\psi(s,y,z)| ds 
            \right \|_2
   \le \| V_\cdot - (V_\cdot)^\psi \|_{L_2(L_{1}([t,T]))} \] 
for all $\psi\in \ws$. Let $\Phi: C^+(\Delta)\to [0,\infty]$ be admissible in the sense of 
Definition \ref{definition:admissible}. Then we have that
\[  \| Y_t \|_{\Phi,2} + \| Z \|_{\Phi,2}^{2,t}
                      \le 2 c_{\eqref{theorem:comparison_psi_phi}}
    \left [ \| \xi \|_{\Phi,2}  + \| V \|_{\Phi,2}^{1,t}
    \right ]. \]
\end{cor}
\bigskip
Taking also Theorem \ref{theorem:Phi_2} into the account we obtain another version of
Corollary \ref{cor:theorem:comparison_psi_phi_2} that only uses that $\xi$ is locally in $\D_{1,2}$ in the sense 
to check perturbations of the Gaussian structure up to time $t$ only. This confirms the smoothing effect of a BSDE 
as this already implies the smoothness of $Y_t$. More precisely we get:
\smallskip

\begin{cor}\label{cor:theorem:comparison_psi_phi_3}
Assume that $\theta\in (0,1)$, $t\in [0,T]$, $\xi \in \bh$, and
that $(Y,Z)$ is the unique solution to the {\rm BSDE} \eqref{equation:BSDE1_intro} 
obtained in Theorem \ref{theorem:existence_uniqueness_bh_subquadratic}.
Then we have
\begin{multline}\label{eqn:cor:theorem:comparison_psi_phi_3}
        \esssup_{s\in [0,t]} \| D_s Y_t \|_2 \\
    \le c \sup_{0\le a < b \le t} \frac{1}{\sqrt{b-a}} 
        \left [  \| \xi - \xi^{(a,b]} \|_2  
               + \left \| \int_t^T \sup_{y,z} | f(s,y,z) - f^{(a,b]}(s,y,z)| ds 
                 \right \|_2 \right ]
\end{multline}
with $c:=  c_{\eqref{theorem:comparison_psi_phi}} c_{\eqref{theorem:Phi_2}(1),2}$
and $c_{\eqref{theorem:Phi_2}(1),2}\ge 1$ taken from Theorem \ref{theorem:Phi_2} in the sense that
if the right-hand side is finite, then $Y_t\in \D_{1,2}$ and \eqref{eqn:cor:theorem:comparison_psi_phi_3} holds.
\end{cor}
\bigskip

\begin{proof}
We apply Theorems \ref{theorem:comparison_psi_phi} and \ref{theorem:Phi_2}, where for the latter we use 
\[ Y_t^{(a,b]} = Y_t^{(a\wedge t,b\wedge t]} \mbox{ a.s.} \]
because $Y_t$ is $\cF_t$-measurable.
\end{proof}

If $\xi\in \D_{1,2}$, then we have that
\[ \sup_{0\le a < b \le t} \frac{\| \xi - \xi^{(a,b]} \|_2}{\sqrt{b-a}} 
   \le  2 c_{(\ref{lemma:PDE-Stein})} \esssup_{s\in [0,t]} \| D_s \xi \|_2 \]
by Corollary \ref{cor:decoupling_upper_bounded_D12}, but for Corollary \ref{cor:theorem:comparison_psi_phi_3} the assumption 
$\xi\in \D_{1,2}$ is not necessary. 


\section{On the $L_p$-variation of BSDEs}
\label{sec:L_p-variation}

In this section we show how Theorem \ref{theorem:comparison_psi_phi} can be applied in order to
obtain information about the $L_p$-variation of our BSDE. The link between the
$L_p$-variation of the $Y$-process and the stability result Theorem \ref{theorem:comparison_psi_phi}
consists in the observation
\[     \| A_t - A_s \|_p 
   \le \| A_t - \E(A_t|\cF_s)\|_p + \|  \E(A_t|\cF_s) - A_s \|_p
   \le 3 \| A_t - A_s \|_p, \]
where $p\in [1,\infty]$, $(A_t)_{t\in [0,T]}\subseteq \cL_p$ is adapted,
and  $0\le s \le t \le T$. Our estimate for the $Z$-process will follow directly 
from Theorem \ref{theorem:comparison_psi_phi}.
\medskip

In Remark \ref{remark:L_p-variation}(1) below we show that under the conditions
$\int_0^T \| Z_r\|_p^2 dr < \infty$ and $\int_0^T \| f(r,Y_r,Z_r)\|_p dr < \infty$, 
{\em and under the a-priori} knowledge of the behaviour of the functions $r\to\| Z_r\|_p$ and 
$r\to \| f(r,Y_r,Z_r)\|_p$ one gets a rate of $1/\sqrt{n}$ for the $L_p$-variation of $Y$ and $Z$
by adapted time-nets. In \st{Corollary \ref{cor:thm:L_p-variation:cor_2_new}} 
below we will deduce estimates with explicit adapted time-nets where we only assume conditions on the initial data $(\xi,f)$. 
Regarding the case $p\in (2,\infty)$ there is another aspect:
In Remark \ref{remark:L_p-variation}(2) we show that even for the zero generator case
one might have situations where one cannot achieve the rate $1/\sqrt{n}$ for the variation
of $Y$, i.e. the variation of $Y$ is asymptotically higher. Our sufficient conditions give 
cases where one gets the rate $1/\sqrt{n}$ for the case $p\in (2,\infty)$.
\medskip

In the following the random variables are considered on the product space 
$\overline \Omega$ if necessary. In particular, random variables defined on $\Omega$ are automatically 
extended to $\overline \Omega$ in the natural way when needed.
\bigskip

\begin{theorem}\label{thm:L_p-variation}
Suppose that the assumptions of Theorem \ref{theorem:comparison_psi_phi}
are satisfied. Then,
for $c_{\eqref{thm:L_p-variation}}:= c_\eqref{theorem:comparison_psi_phi}[1 +  
    c_{\eqref{cor:AB_generalized_Fefferman_inequality_new_constant},p} L_Z (\sqrt{T}+1)]$ and
$0\le s < t \le T$, one has 
\equa
&   & \left \|\sup_{r\in [s,t]}|Y_r-Y_s|\right \|_p 
      + \left \| \left ( \int_s^t |Z_r|^2 dr \right )^\frac{1}{2}\right \|_p  \\
&\le& \left\| \int_s^t |f(r,0,0)| dr\right\|_p + L_Y (t-s) \sup_{r\in [0,T]} \|Y_r\|_p  
      + c_{\eqref{thm:L_p-variation}} \left [ 1 + \| |Z|^\theta \|_{\bmo(S_2)} \right ] 
      \times \\
&   & \times \left [ \|\xi-\xi^{(s,t]} \|_p + 
      \left \| \int_s^T | f(r,Y_r,Z_r) - f^{(s,t]}(r,Y_r,Z_r)| dr \right \|_p \right ].
\tion      
\end{theorem}
\smallskip

\begin{proof}
We fix $0\le s < t \le T$ and remark that  
$\left\| \int_s^t |f(r,0,0)| dr\right\|_p<\infty$ according to condition (B5).
We let $(q_l)_{l=1}^\infty$ be an enumeration of the rational numbers from $(s,t]$ so that
\[ \left \|\sup_{r\in [s,t]}|Y_r-Y_s|\right \|_p
   = \sup_{m=1,2,\ldots} \left \|\sup_{q\in \{q_1,\ldots,q_m\}}|Y_q-Y_s|\right \|_p \]
by monotone convergence. 
Using Lemma \ref{lemma:exchange-conditional_expectation_new}, the fact that 
the $Y_{q_l}$ are $\cF_t$-mea\-surable, and Theorem \ref{theorem:comparison_psi_phi} we obtain that
\pagebreak

\equa
&   & \left \|\sup_{q\in \{q_1,\ldots,q_m\}}|Y_q-Y_s|\right \|_p 
      + \left \| \left ( \int_s^t |Z_r|^2 dr \right )^\frac{1}{2}\right \|_p  \\
&\le& \left \|\sup_{q\in \{q_1,\ldots,q_m\}} |Y_q-\E^{\cF_s}Y_q| \right \|_p 
      + \left \|\sup_{q\in \{q_1,\ldots,q_m\}} |\E^{\cF_s}Y_q-Y_s|\right \|_p \\
&   & \hspace*{18em}  + \left \| \left ( \int_s^t |Z_r|^2 dr \right )^\frac{1}{2}\right \|_p \\
&\le& \left \|\sup_{q\in \{q_1,\ldots,q_m\}} |Y_q-Y_q^{(s,t]}| \right \|_p 
      +  \left \| \left ( \int_s^t |Z_r|^2 dr \right )^\frac{1}{2}\right \|_p \\
&   & \hspace*{16em} + \left \|\sup_{q\in \{q_1,\ldots,q_m\}} |\E^{\cF_s}Y_q-Y_s|\right \|_p \\
&\le& c_\eqref{theorem:comparison_psi_phi}  \left [ \|\xi-\xi^{(s,t]} \|_p + 
      \left \| \int_s^T | f(r,Y_r,Z_r) - f^{(s,t]}(r,Y_r,Z_r)| dr \right \|_p \right ] \\
&   & \hspace*{16em} + \left \|\sup_{q\in \{q_1,\ldots,q_m\}} |\E^{\cF_s}Y_q-Y_s|\right \|_p.
\tion
By Corollary \ref{cor:AB_generalized_Fefferman_inequality_new_constant}
and \eqref{eqn:comparison_inequality_2}
we bound the last term by

\equa 
&   & \left \|\sup_{q\in \{q_1,\ldots,q_m\}} |\E^{\cF_s}Y_q-Y_s|\right \|_p \\
& = & \left \|\sup_{q\in \{q_1,\ldots,q_m\}} \left |\E^{\cF_s} \int_s^q f(r,Y_r,Z_r) dr \right |\right \|_p \\
&\le& \left \| \int_s^t |f(r,Y_r,Z_r)| dr \right \|_p \\ 
&\le& \left\| \int_s^t |f(r,0,0)| dr\right\|_p
      + L_Y \left\|\int_s^t |Y_r| dr\right\|_p
      + L_Z \left\|\int_s^t [1+|Z_r|]^\theta|Z_r| dr\right\|_p \\
&\le& \left\| \int_s^t |f(r,0,0)| dr\right\|_p
      +  L_Y (t-s) \sup_{r\in [s,t]} \|Y_r\|_p  \\
&   & + L_Z c_{\eqref{cor:AB_generalized_Fefferman_inequality_new_constant},p} 
      \left\|\left( \int_s^t |Z_r|^2 dr\right )^\frac{1}{2} \right\|_p 
      \left\| (\chi_{(s,t]}(r) [1+|Z_r|]^\theta)_{r\in [0,T]} \right\|_{\bmo(S_2)} \\
&\le& \left\| \int_s^t |f(r,0,0)| dr\right\|_p
      +  L_Y (t-s) \sup_{r\in [s,t]} \|Y_r\|_p  \\
&   &      + L_Z c_{\eqref{cor:AB_generalized_Fefferman_inequality_new_constant},p}
                 c_\eqref{theorem:comparison_psi_phi}  \left [ \|\xi-\xi^{(s,t]} \|_p + 
            \left \| \int_s^T | f(r,Y_r,Z_r) - f^{(s,t]}(r,Y_r,Z_r)| dr 
            \right \|_p \right ]  \\
&   & \times [\sqrt{t-s} +  \| |Z|^\theta \|_{\bmo(S_2)}].
\tion
As remarked in the beginning of the proof of Theorem \ref{theorem:comparison_psi_phi} we have 
$\sup_{r\in [0,T]} |Y_r| \in \cL_p$ so that $\alpha_p:=L_Y \sup_{r\in [0,T]} \|Y_r\|_p < \infty$. Therefore,
\begin{multline*}
      \left \|\sup_{q\in \{q_1,\ldots,q_m\}} |\E^{\cF_s}Y_q-Y_s|\right \|_p
 \le  \left\| \int_s^t |f(r,0,0)| dr\right\|_p
      + \alpha_p (t-s) \\
      + \beta_p \left [ \|\xi-\xi^{(s,t]} \|_p + 
            \left \| \int_s^T | f(r,Y_r,Z_r) - f^{(s,t]}(r,Y_r,Z_r)| dr 
            \right \|_p
\right ]
\end{multline*}
for
$\beta_p :=  L_Z c_{\eqref{cor:AB_generalized_Fefferman_inequality_new_constant},p}
             c_\eqref{theorem:comparison_psi_phi} 
             [\sqrt{T} + \| |Z|^\theta \|_{\bmo(S_2)}]$.
\end{proof}
\medskip

The variation of our BSDE we measure by the following quantity:

\begin{definition}
\label{definition:var_p_tau}
\index{$\var_p([A,C] \vline \, \tau)$}
Let $p\in [1,\infty)$, $A=(A_t)_{t\in [0,T]}$ be a measurable c\`adl\`ag process
$A:[0,T]\times\Omega\to \R$, and $C=(C_t)_{t\in [0,T]}$ be a measurable
process $C:[0,T]\times \Omega \to \R^d$. For a deterministic time-net $\tau=(t_i)_{i=0}^n$ with $0=t_0\le t_1 \le \cdots \le t_n=T$ we let
\[ \var_p([A,C]|\tau):= 
   \sup_{i=1,...,n}  \left \| \sup_{t_{i-1}\le s \le t \le t_i} |A_t - A_s| \right \|_p +
   \sup_{i=1,...,n} \left \| \left ( \int_{t_{i-1}}^{t_i} |C_r|^2 dr \right )^\frac{1}{2} \right \|_p. \]
\end{definition}
\medskip

The variation  $\var_p([A,C]|\tau)$ behaves sub-additive as expected:

\begin{lemma}
For $p\in [1,\infty)$, families $((A_t^j,C_t^j))_{t\in [0,T]}$ and time-nets $\tau^j$, $j=0,1$, 
as in Definition \ref{definition:var_p_tau}, one has that 
\[ \var_p([A^0+A^1,C^0+C^1]|\tau^0\cup\tau^1 )\le \var_p([A^0,C^0]|\tau^0) + \var_p([A^1,C^1]|\tau^1). \]
\end{lemma}

\begin{proof}
Assume that $\tau=(t_i)_{i=0}^{n_0+n_1-1}$ is an ordering of the union of
$\tau^0=(t_i^0)_{i=0}^{n_0}$ and $\tau^1=(t_i^1)_{i=0}^{n_1}$. Then
one has that the interval $[t_{i-1},t_i]$ is contained in a closed interval of $\tau^0$ and, at the same time, in a
closed interval of $\tau^1$, so that
\equa
&   & \left \| \sup_{t_{i-1}\le s \le t \le t_i} |(A_t^0+A_t^1) - (A_s^0+A_s^1)| \right \|_p +
      \left \| \left ( \int_{t_{i-1}}^{t_i} |C_r^0+C_r^1|^2 dr \right )^\frac{1}{2} \right \|_p \\
&\le& \left \| \sup_{t_{i-1}\le s \le t \le t_i} | A_t^0- A_s^0 |\right \|_p +
      \left \| \left ( \int_{t_{i-1}}^{t_i} |C_r^0|^2 dr \right )^\frac{1}{2} \right \|_p \\
&   & \hspace*{8em} + \left \| \sup_{t_{i-1}\le s \le t \le t_i} | A_t^1 - A_s^1| \right \|_p +
      \left \| \left ( \int_{t_{i-1}}^{t_i} |C_r^1|^2 dr \right )^\frac{1}{2} 
      \right \|_p \\
&\le& \var_p([A^0,C^0]|\tau^0) + \var_p([A^1,C^1]|\tau^1). 
\tion
\end{proof}

Now we formulate consequences of Theorem \ref{thm:L_p-variation} in two different scenarios:
The first \st{Corollary \ref{cor:thm:L_p-variation:cor_1}} still relies on the assumptions of 
Theorem \ref{theorem:comparison_psi_phi}.
In the next step \st{Corollary \ref{cor:thm:L_p-variation:cor_1}} will be combined with the results from 
Section \ref{sec:classes_quadratic_subquadratic_BSDEs}
to guarantee the validity of the assumptions 
of Theorem \ref{theorem:comparison_psi_phi}. This yields to \st{Corollary \ref{cor:thm:L_p-variation:cor_2_new}}.
\medskip

To shorten the formulation of the statements we 
work with the following two definitions. 
\medskip

The first definition extends the spaces $\B_p^\Phi$ to the initial data $(\xi,f)$ of the 
BSDE:

\begin{definition}
\label{definition:potential_BSDE}
\index{space!$\B^{\Phi_{\gamma,\Gamma}}_p$}
We say that $(\xi,f)\in \B^{\Phi_{\gamma,\Gamma}}_p$,
where 
$p\in (0,\infty)$,
$\gamma\in [2,\infty)$, and 
$\Gamma:[0,T]\to [0,\infty)$ is integrable,
provided that $\xi \in L_p$ and 
for all $0\le a < b \le T$, 
\[ \| \xi - \xi^{(a,b]} \|_p + \left \| \int_a^T  \sup_{(y,z) \in \R^{d+1}} 
   |f(r,y,z)-f^{(a,b]}(r,y,z)| dr \right \|_p 
   \le \left ( \int_a^b \Gamma(r)dr \right )^\frac{1}{\gamma}. \] 
\end{definition}
\medskip
The term  
$\int_a^T \sup_{(y,z) \in \R^{d+1}} |f(r,y,z)-f^{(a,b]}(r,y,z)| dr$ is an extended random
variable on $(\overline{\Omega},\overline{\cF},\overline{\P})$.
Concerning the generator, the above definition  reflects the situation described in 
Example \ref{example:driver_that_depends_on_X}, where the
generator $f$ is obtained from some appropriate $h$ with 
\[ f(r,\omega,y,z) := h(r,A_r(\omega),y,z). \]

The second definition recalls a \st{well-known} principle to generate adapted time-nets:
\medskip

\begin{definition}   
\index{$\tau_n^\Lambda$}
Letting $\Lambda:[0,T]\to (0,\infty)$ be integrable and $n\ge 1$, the time-net
$\tau_n^\Lambda$ consists of $0=t_0 < \cdots < t_n=T$ such that, for all $i=1,...,n$,
\[ \int_{t_{i-1}}^{t_i} \Lambda(r) dr = \frac{1}{n} \int_0^T \Lambda(r) dr. \]
\end{definition}
\bigskip

The following corollary, which follows directly from 
Lemma \ref{lemma:exchange-conditional_expectation_new} and \eqref{eqn:variation_upper_bounded_integral_Ds},
yields  to {\em the} fundamental example for $\gamma=2$ 
concerning the part $\| \xi - \xi^{(a,b]} \|_p$ in 
Definition \ref{definition:potential_BSDE} above:
\medskip

\begin{cor}
\label{cor:decoupling_upper_bounded_D12}
For $p\in [2,\infty)$ and $\xi\in \D_{1,2}\cap L_p$ with $\int_{(0,T]} \| D_r \xi\|_p^2 dr <\infty$ one has for all $0\le a < b \le T$ that
\[ \|\xi - \xi^{(a,b]} \|_p  \le 
       2 
      c_{(\ref{lemma:PDE-Stein})}  \left ( \int_{(a,b]} \|D_r\xi\|^2_p dr \right )^\frac{1}{2}.       \]
\end{cor}
\medskip

\begin{remark}
\index{functional!$\Phi_{\gamma,\Gamma}$}
For $\Gamma >0$ Definition \ref{definition:potential_BSDE} and 
Example \ref{example:admissible_weight} 
yield to the admissible functional
\[    \Phi_{\gamma,\Gamma}(F) 
   := \sup_{0\le a < b \le T} \frac{F(\chi_{(a,b]})}{\sqrt[\gamma]{\int_a^b \Gamma(r) dr}} \]
that recovers the functional $\Phi_\gamma$ from \eqref{eqn:Phi_r} by $\Gamma\equiv 1$.
\end{remark}

Our first corollary of Theorem \ref{thm:L_p-variation} is
\medskip

\begin{cor}
\label{cor:thm:L_p-variation:cor_1}
Let $p,\gamma\in [2,\infty)$ and $\Gamma:[0,T]\to [0,\infty)$ be integrable. 
Suppose the assumptions of Theorem \ref{theorem:comparison_psi_phi}, 
$(\xi,f) \in\B^{\Phi_{\gamma,\Gamma}}_p$ and 
$\int_0^T \|f(r,0,0)\|_p dr<\infty$. Then
\[     \var_p([Y,Z]|\tau_n^\Lambda) 
   \le   \frac{c_\eqref{cor:thm:L_p-variation:cor_1}}{n}
       + \frac{d_\eqref{cor:thm:L_p-variation:cor_1}}{\sqrt[\gamma]{n}} \]
for $\Lambda(r) := 1+  \|f(r,0,0)\|_p + \Gamma(r)$ and
\equa  
      c_\eqref{cor:thm:L_p-variation:cor_1} 
&:= & 2 \| \Lambda \|_{L_1([0,T])} [1+ L_Y \sup_{t\in [0,T]} \|Y_t\|_p], \\
      d_\eqref{cor:thm:L_p-variation:cor_1} 
&:= & 2 c_\eqref{thm:L_p-variation} \| \Lambda\|_{L_1([0,T])}^\frac{1}{\gamma} 
      \left [ 1 + \| |Z|^\theta \|_{\bmo(S_2)} \right ].
\tion
\end{cor}

\begin{proof}
For $0\le s < t \le T$ Theorem \ref{thm:L_p-variation} implies that
\equa
&   & \hspace*{-3em}
      \left \|\sup_{r\in [s,t]}|Y_r-Y_s|\right \|_p + \left \| \left ( \int_s^t |Z_r|^2 dr \right )^\frac{1}{2}\right \|_p  \\
&\le& \left\| \int_s^t |f(r,0,0)| dr\right\|_p + L_Y (t-s) \sup_{r\in [0,T]} \|Y_r\|_p  \\
&   & \hspace*{2em} 
      + c_{\eqref{thm:L_p-variation}} \left [ 1 + \| |Z|^\theta \|_{\bmo(S_2)} \right ] \times\\
&   & \hspace*{3em} \times \left [ \|\xi-\xi^{(s,t]} \|_p + 
        \left \| \int_s^T | f(r,Y_r,Z_r) - f^{(s,t]}(r,Y_r,Z_r)| dr \right \|_p \right ] \\
&\le& \int_s^t \|f(r,0,0)\|_p dr + L_Y (t-s) \sup_{r\in [0,T]} \|Y_r\|_p  \\
&   & \hspace*{10em}
      + c_{\eqref{thm:L_p-variation}} \left [ 1 + \| |Z|^\theta \|_{\bmo(S_2)} \right ] 
      \left ( \int_s^t \Gamma(r) dr \right )^\frac{1}{\gamma}.
\tion
Assuming $0=t_0 \le t_1 \le \cdots \le t_n=T$ we conclude by
\equa
&   &   \left \|\sup_{t_{t-1}\le s \le t \le t_i}|Y_t-Y_s|\right \|_p 
      + \left \| \left ( \int_{t_{i-1}}^{t_i} |Z_r|^2 dr \right )^\frac{1}{2}\right \|_p  \\
&\le&  2 \left \|\sup_{r\in [t_{t-1},t_i]}|Y_r-Y_{t_{i-1}}|\right \|_p 
      + \left \| \left ( \int_{t_{i-1}}^{t_i} |Z_r|^2 dr \right )^\frac{1}{2}\right \|_p  \\
&\le& 2 \Bigg [
      \int_{t_{i-1}}^{t_i} \|f(r,0,0)\|_p dr + L_Y (t_i-t_{i-1}) \sup_{r\in [0,T]} \|Y_r\|_p  \\
&   & \hspace*{10em}
      + c_{\eqref{thm:L_p-variation}} \left [ 1 + \| |Z|^\theta \|_{\bmo(S_2)} \right ] 
      \left ( \int_{t_{i-1}}^{t_i} \Gamma(r) dr \right )^\frac{1}{\gamma} \Bigg ].
\tion
\end{proof}

\begin{cor}
\label{cor:thm:L_p-variation:cor_2_new}
Assume $\gamma \in [2,\infty)$, \st{an integrable $\Gamma:[0,T]\to [0,\infty)$,} and that one of the following sets of conditions is satisfied:
\begin{enumerate}
\item $\theta=0$, $p\in [2,\infty)$, $\xi\in L_p$, (B3), $\int_0^T \| f(r,0,0)\|_p dr <\infty$,       $(\xi,f) \in\B^{\Phi_{\gamma,\Gamma}}_p$.
\item $\theta\in (0,1)$, $\xi\in \bh$, (B3), (B8), and
      $(\xi,f) \in\B^{\Phi_{\gamma,\Gamma}}_2$.
\item $\theta=1$, $\xi\in L_\infty$, (B3), (B8), and
      $(\xi,f) \in\bigcap_{q\in [2,\infty)} \B^{\Phi_{\gamma,\Gamma}}_q$.
\end{enumerate}
Define the weight function 
\[ \Lambda(r) :=  1+ \|f(r,0,0)\|_u + \Gamma(r) \]
where 
$u=p$ for $\theta=0$,
$u=2$ for $\theta \in (0,1)$, and
$u=\infty$ for $\theta=1$.
Then one has that
\[ \sup_{n\ge 1} \sqrt[\gamma]{n} \var_v([Y,Z]|\tau_n^{\Lambda}) < \infty \]
for 
$v=p$ if $\theta = 0$,
$v=2$ if $\theta \in (0,1)$, and
for all $v\in (0,\infty)$ if $\theta=1$,
where 
for $\theta =0$       the solution is taken from \cite[Theorem 4.2]{Briand:Del:Hu:Par:Sto:03},
for $\theta\in (0,1)$ from Theorem \ref{theorem:existence_uniqueness_bh_subquadratic}, and 
for $\theta=1$        from Remark \ref{remark:existence_uniqueness_bh_quadratic}.
\end{cor}

\begin{proof}
The statement follows by a combination of Table 1 (cases \st{(I), (II), and (III)}) and \st{Corollary \ref{cor:thm:L_p-variation:cor_1}}.
For part (3) we remark that we first deduce our statement for $v\in [2,\infty) \cap (p_0,\infty)$ with $p_0$ taken from
Theorem \ref{theorem:comparison_psi_phi}, and then (obviously) the conclusion follows for all $v\in (0,\infty)$.
\end{proof}

\begin{remark}
\label{remark:L_p-variation}
\hspace*{0em}
\begin{enumerate}
\item For $p\in [2,\infty)$ assume for our BSDE the conditions $\int_0^T \| Z_r\|_p^2 dr < \infty$ and 
      $\int_0^T \| f(r,Y_r,Z_r)\|_p dr < \infty$.
      Take a net $\tau^n =(t_i^n)_{i=1}^n$ that satisfies
      \[   \int_{t_{i-1}^n}^{t_i^n} 
           \left [ \| f(r,Y_r,Z_r) \|_p + \| Z_r\|_p^2 \right ] dr 
         = \frac{1}{n}  
           \int_0^T \left [ \| f(r,Y_r,Z_r) \|_p + \| Z_r\|_p^2 \right ] dr. \]
      Given $i\in \{1,\ldots,n\}$ we derive
      \equa
      &   & \left \| \sup_{t_{i-1}^n \le s \le t \le t_i^n}|Y_t - Y_s| \right \|_p 
            + \left \| \left ( \int_{t_{i-1}^n}^{t_i^n} |Z_r|^2 dr \right )^\frac{1}{2}
                          \right \|_p \\
      &\le& \left \| \int_{t_{i-1}^n}^{t_i^n} |f(r,Y_r,Z_r)|dr \right \|_p 
            + \left \| \sup_{q\in [t_{i-1}^n,t_i^n]}\left | \int_{t_{i-1}^n}^q Z_r dW_r \right |\right \|_p \\
      &   & \hspace{18em} + \left \| \left ( \int_{t_{i-1}^n}^{t_i^n} |Z_r|^2 dr \right )^\frac{1}{2} \right \|_p \\
      &\le& \int_{t_{i-1}}^{t_i} \|f(r,Y_r,Z_r)\|_pdr + [2 \beta_p +1]
            \left \| \left ( \int_{t_{i-1}}^{t_i}  |Z_r|^2 dr \right )^\frac{1}{2}
                          \right \|_p \\
      &\le& \int_{t_{i-1}}^{t_i} \|f(r,Y_r,Z_r)\|_pdr + [2 \beta_p +1]  \left ( \int_{t_{i-1}}^{t_i} \|Z_r\|^2_p 
            dr \right )^\frac{1}{2} \\
      &\le& \frac{1}{n} \int_0^T \left [ \| f(r,Y_r,Z_r) \|_p + \| Z_r\|_p^2 \right ] dr\\
      &   & \hspace*{8em}
            +  \frac{2 \beta_p +1}{\sqrt{n}} \left ( \int_0^T \left [ \| f(r,Y_r,Z_r) \|_p 
            + \| Z_r\|_p^2 \right ] dr \right )^\frac{1}{2}
      \tion
      where the Burkholder-Davis-Gundy inequalities \eqref{eqn:BDG} were exploited.  
      \st{Consequently,} we have a variation of $1/\sqrt{n}$ by taking the nets $\tau^n$. 

\item However, in general for $p\in (2,\infty)$ such an estimate is not always possible as
      shown by the following example for $d=1$: Take an infinite 
      time-net converging to $T$,
      \[ 0=t_0 < t_1 < t_2 < \cdots, \]
      and pair-wise disjoint $A_k \in \cF_{t_k}$ of positive measure for $k=1,2,...$.
      (Given pair-wise disjoint non-empty finite intervals $I_k=(a_k,b_k)$ one can choose
      $A_1:=\{ W_{t_1} \in I_1\}$ and
      $A_k:=\{ W_{t_1} \not \in I_1,...,W_{t_{k-1}} \not \in I_{k-1}, W_{t_k} \in I_k\}$
      for $k\ge 2$.)
      For $(\alpha_k)_{k=1}^\infty \subset (0,\infty)$ and $s\in [0,T]$ define 
      \[ \lambda_s := \sum_{k=2}^\infty \alpha_{k-1} \chi_{A_{k-1}} \chi_{(t_{k-1},t_k]}(s).  
          \]
      Let  $0<\alpha<\frac{p}{2}-1$ and arrange the $\alpha_k$ such that
      \[ \left \| \int_{(t_{k-1},t_k]} \lambda_s dW_s \right \|_p = k^{-\frac{1+\alpha}{p}} 
         \]
      which implies
      \[   \E \left | \int_{(0,T]} \lambda_s dW_s \right |^p 
         = \sum_{k=2}^\infty k^{-(1+\alpha)} < \infty. \]
      Let us assume $c>0$ and a sequence of time-nets $\tau^n$, 
      $0=t_0^n \le \cdots \le t_n^n =T$, such that 
      \[ \| Y_{t_i^n} - Y_{t_{i-1}^n} \|_p \le \frac{c}{\sqrt{n}}
         \sptext{1}{for}{1}
         Y_t := \int_0^t \lambda_s dW_s. \]
      Then $(t_{k-1},t_k)\cap \tau^n=\emptyset$ for $k\ge 2$ implies that 
      \[  \| Y_{t_k} - Y_{t_{k-1}} \|_p = k^{-\frac{1+\alpha}{p}} \le \frac{c}{\sqrt{n}} \]
      or, equivalently, the condition $k^{-\frac{1+\alpha}{p}} > \frac{c}{\sqrt{n}}$ gives 
      $(t_{k-1},t_k)\cap \tau^n\not =\emptyset$ for $k\ge 2$. In other words, all intervals
      $(t_{k-1},t_k)$ with
      \[ 2\le k < \left ( \frac{\sqrt{n}}{c} \right )^\frac{p}{1+\alpha} \]
      contain at least one element of the time-net $\tau^n$. This gives a contradiction to
      $\frac{p}{2(1+\alpha)}>1$.
\end{enumerate}
\end{remark}


\st{
\section{Applications to other types of BSDEs}

The decoupling techniques developed in this article rely only on the existence of solutions to BSDEs, not on their
uniqueness nor on special techniques to prove existence or uniqueness. This opens the possibility to apply the
results and techniques to other types of BSDEs as well. Let us list some potential examples: 

\begin{enumerate}
\item {\sc Multidimensional BSDEs \& coupled forward-backward SDEs}
      \smallskip

     Theorem \ref{theorem:change_multi_SDE} is flexible enough to treat in \eqref{eqn:gBSDE} an $\R^n$-valued process 
     $(L_t)_{t\in [0,T]}$ by considering its coordinates separately. This might be applied to BSDEs where the $Y$-process 
     is multi-dimensional. Moreover, if in coupled forward-backward SDEs (see for example \cite{Ma:Yong:07})
     the dependencies in the forward diffusion on the backward component can be handled 
     by Theorem \ref{theorem:change_multi_SDE}, then our decoupling approach can be directly examined as well.
     \medskip

\item {\sc BSDEs with singular terminal conditions}
      \smallskip
     
      Singular terminal conditions are considered for instance in \cite{Popier:06,Ankirchner:Jeanblanc:Kruse:14}. 
      The general idea behind this type of singular terminal condition for BSDEs consists in replacing the one-parametric family of
      equations from $t$ to $T$ by the two-parametric family
      \[ Y_t = Y_r + \int_t^r f(s,Y_s,Z_s) ds - \int_t^r Z_s dW_s 
         \sptext {1}{for}{1}  
         0 \le t < r < T \]
      and to look for solutions $(Y_t,Z_t)_{t\in [0,T)}$ where the process $Y$ is subject to constraints as
      $t\uparrow T$. Let us indicate how the process $Y$ from \eqref{equation:BSDE1_intro} might be mapped into $Y^\rho$ 
      like in \eqref{eqn:BSDE-rotated_intro}
      for constraints of the form $\lim_{t\uparrow T} Y_t = \infty$ on $\Omega^+$ or
      $\lim_{t\uparrow T} Y_t = - \infty$ on $\Omega^-$ for some $\Omega^\pm\in \cF_T$ of positive  measure.
      With $h:=\arctan:[-\infty,\infty]\to [-\pi/2,\pi/2]$ the 
      transformed process $Y^h:=(h(Y_t))_{t\in [0,T]}$ is continuous and takes values in $[-\pi/2,\pi/2]$. 
      This process $Y^h$ can be mapped into $(Y^h)^\rho$ as in \eqref{eqn:BSDE-rotated_intro}, and by changing 
      $(Y^h)^\rho$ on a set of measure zero we may assume as well that  $(Y^h)^\rho$ takes 
      values in $[-\pi/2,\pi/2]$ only. Applying $h^{-1}$ gives a candidate for $Y^\rho$.
      \medskip

\item {\sc Extension to L\'evy processes}
      \smallskip

      Our approach in Chapter \ref{chapter:general_factorization} is not restricted to particular distributions
      and its general presentation is intended to apply the results in other settings than the 
      Wiener space as well. A first natural candidate are BSDEs driven by L\'evy processes. Here first results 
      were obtained in \cite{CGeiss:Steinicke:16}, where a decoupling is used in $L_2$ as in \cite{GGG:12} for the 
      Brownian motion. Formally the approach in \cite{CGeiss:Steinicke:16} differs slightly from our approach, 
      as it directly uses It\^o's chaos expansion from \cite{Ito:56}. To generalize \cite{CGeiss:Steinicke:16} further 
      along the ideas of our notes, it might be also necessary to extend 
      Proposition \ref{proposition:cont_modification}
      to processes that have certain discontinuous trajectories. Moreover, generalizations beyond the setting of 
      L\'evy processes is left to future work.
\end{enumerate}
}


\appendix
\chapter{Technical Facts}
\label{chapter:appendix}

Let $M\not = \emptyset$ be a complete metric space that is {\em locally $\sigma$-compact},
\index{locally $\sigma$-compact} i.e. 
there exist compact subsets $\emptyset \not = K_1 \subseteq K_2 \subseteq \dots $, such that $\overline{\mathring K}_n=K_n$ and
$M=\cup_{n=1}^\infty \mathring{K}_n$.
By continuity of a stochastic process $(X_x)_{x \in M}:\Om\to \R$ we understand that
$x\mapsto X_x(\om)$ is continuous for all $\omega\in\Om$.
\medskip

\begin{proposition}\label{proposition:cont_modification}
Let $M\not = \emptyset$ be a complete locally $\sigma$-compact metric space and
$(X_x)_{x\in M}$ be a continuous process defined on a probability space
$(\Om^0,\cF^0,\P^0)$, and let $(\beta_x)_{x \in M}$ be a stochastic process on a
probability space $(\Om^1,\cF^1,\P^1)$ such that $X$ and $\beta$ have the same
finite-dimensional distributions. Then the following is satisfied:
\begin{enumerate}[{\rm (1)}]
\item There exists a continuous process $(Y_x)_{x \in M}$ on $(\Om^1,\cF^1,\P^1)$, which is a modification
      of $(\beta_x)_{x \in M}$, i.e. $\P^1(Y_x = \beta_x)=1$ for all $x \in M$.
\item If there is another process $Y'$ with this property, then
      $\P^1(Y_x = Y'_x, x\in M) =1$.
\item If $\cG^1\subseteq \cF^1$  is a sub-$\sigma$-algebra and $D\subseteq M$ dense, such that 
      $\beta_x$ is $\cG^1$-measurable for all $x\in D$, then the process $Y$ can be taken to be
      $\cG^1$-measurable.
\end{enumerate}
\end{proposition}
\medskip

\begin{proof} 
There is a countable set $D_0=\{ a_k : k\ge 1 \} \subseteq D$ such that $D_0\subseteq M$
is dense as well. Taking a sequence $(K_n)_{n=1}^\infty$ like in the definition of locally $\sigma$-compact
we have therefore that $D_0\cap  K_n$ is dense in $K_n$ for all $n=1,2,...$
\medskip

(1) and (3): We prove both parts at the same time as (1) is a special case of (3) by
taking $D=M$ and $\cG^1=\cF^1$. Let $K$ be one of the sets $K_n$ and $A:= D_0\cap K$.
Since $x \mapsto X_x$ is continuous on $M$, it is uniformly continuous on $K$ and $A$.
Hence the set
\[ \bigcap_{n=1}^\infty \bigcup_{m=1}^\infty \bigcap_{\stackrel{d(u,v)<\frac{1}{m}}{u,v \in A}}
   \left\{\om : |X_u(\om)-X_v(\om)| \le \frac{1}{n} \right\} \in\cF^0\]
is of $\P^0$-measure one. By the fact that $X \stackrel{d}{=} \beta$, there exists
$\Om_0^1\in \cG^1$
with $\P^1(\Om_0^1)=1$ such that $x \mapsto \beta_x(\om)$ is uniformly continuous on $A$ for all
$\om \in \Om_0^1$. Since $A$ is dense in $K$ we can define for all $x \in K$ the extension  \[ Y_x(\om) := \left \{ \begin{array}{rcl}
           \lim_{\stackrel{x_n \to x,}{x_n \in A}} \beta_{x_n}(\om) &:& \om \in \Om_0^1 \\
           0                                                        &:& \om \in \Om^1\setminus\Om_0^1
                         \end{array} \right ..  \]
We obtain a $\cG^1$-measurable continuous process
$(Y_x)_{x \in K}$. Take $d\ge 1$, $x_1,\dots,x_d \in K$, and
$a_{j,m}\in A$ with $a_{j,m} \to x_j$ as $m \to \infty$. Then, for $(t_1,...,t_d)\in \R^d$,
\equa
      \int_{\Omega^1} e^{ i \sum_{j=1}^d t_j Y_{x_j}} d\P^1
& = & \lim_{m\to\infty}
      \int_{\Omega^1} e^{ i \sum_{j=1}^d t_j \beta_{a_{j,m}}} d\P^1 \\
& = & \lim_{m\to\infty}
      \int_{\Omega^0} e^{ i \sum_{j=1}^d t_j     X_{a_{j,m}}} d\P^0 \\
& = & \int_{\Omega^0} e^{ i \sum_{j=1}^d t_j X_{x_j}} d\P^0
\tion
so the finite-dimensional distributions of $Y$ and $X$ coincide.
To prove $\P^1(Y_x = \beta_x)=1$ for all $x \in K$ we check
$\P^1\left( |Y_x - \beta_x| > \epsilon \right)=0$
for all $\epsilon>0$ and all $x \in K$. Let $\epsilon>0$, $x \in K$, and choose $(x_k)_{k\ge 1}\subseteq A$
such that $x_k \to_k x$. Then
\equa
      \P^1\left( |Y_x - \beta_x| > \epsilon \right)
&\le& \P^1\left( |Y_x - \beta_{x_k}| > \frac{\epsilon}{2} \right)
      + \P^1\left( |\beta_x - \beta_{x_k}| > \frac{\epsilon}{2} \right) \\
& = & 2\P^0\left( |X_x - X_{x_k}| > \frac{\epsilon}{2} \right)
       \to_k 0,
\tion
where we used the fact that $Y \stackrel{d}{=} X \stackrel{d}{=} \beta$ and the fact that $X$ is continuous.
Thus on any compact $K_n\subseteq M$ we have a continuous $\cG^1$-measurable process $(Y^n_x)_{x \in K_n}$,
that is a modification of $(\beta_x)_{x \in K_n}$. Up to $\cG^1$-measurable null-sets the construction is
consistent in $n$ so that we can construct a $\cG^1$-measurable continuous process $(Y_x)_{x\in M}$
(where we use $M=\bigcup_{n=1}^\infty \mathring{K}_n$) that is a modification of $\beta$.
\medskip

(2) follows from the separability of $M$.
\end{proof}

The following lemma is well-known.
\smallskip

\begin{lemma}\label{lemma:joint_measurability}
Let $(A,\cA)$ be a measurable space and $M$ be a separable metric space. Assume that $f:M\times A \to \R$ is
such that $f(x,\cdot)$ is $\cA$-measurable for all $x \in M$ and $x \to f(x,\om)$ is continuous for
all $\om \in A$. Then $f$ is $\mathcal{B}(M)\otimes\cA$-measurable, where $\mathcal{B}(M)$ is generated
by the open sets.
\end{lemma}
\smallskip

\begin{proof} Let $(x_j)_{j\ge 1} \subseteq M$ be a dense set. We define for all $n,j \ge 1$
\[ B_j^n := \left\{ x \in M: d(x,x_j) \le \frac{1}{n} \right\} \]
and obtain a sequence of disjoint sets as follows:
$A^n_1 := B^n_1$, and
$A^n_k := B^n_k \setminus (\bigcup_{j=1}^{k-1} A^n_{j})$ for $k=2,3,\dots$ Then
$M = \bigcup_{k=1}^\infty A^n_k$
for all $n \ge 1$. Now we define $f^n:M\times A \to \R$ as follows:
 \[ f^n(x,\om) := \sum_{j=1}^\infty f(x_j,\om) 1_{A^n_j}(x).  \]
Since $f(x,\cdot)$ is $\cA$-measurable for all $x \in M$ and $A^n_j \in \mathcal{B}(M)$ for all $j,n \ge 1$,
it follows that each $f^n$ is $\mathcal{B}(M)\otimes\cA$-measurable. Moreover, for any $(x,\om) \in M\times A$
we have the pointwise convergence $f^n(x,\om) \to f(x,\om)$ as $n \to \infty$. This follows from the facts
 \[ |f^n(x,\om)-f(x,\om)| = |f(x_{j(n,x)},\om)-f(x,\om)|, \]
and $d(x_{j(n,x)},x)\le \frac{1}{n} \to_n 0$, where $j(n,x)$ is the index such that
$x\in A_{j(n,x)}^n$.
\end{proof}

\begin{lemma}\label{lemma:Lebesgue_differentiation}
Let $f\in L_1([0,T])$ be non-negative. Then
\[ \sup_{0\le a < b \le T} \frac{1}{b-a} \int_a^b f(t) dt = {\rm esssup}_{t\in [0,T]} f. \]
\end{lemma}

\begin{proof}
The inequality
\[ \sup_{0\le a < b \le T} \frac{1}{b-a} \int_a^b f(t) dt \le {\rm esssup}_{t\in [0,T]} f \]
is obvious. According to \cite[Theorem 3.3.8]{Stroock:11} there exists a Borel set $A\subseteq [0,T]$ with $\lambda(A)=T$
and $0\le a_n^s\le s \le b_n^s\le T$ with $0<b_n^s-a_n^s\to_n 0$ for $s\in A$,
such that
\[ \lim_n \frac{1}{b_n^s-a_n^s} \int_{a_n^s}^{b_n^s} f(t) dt = f(s) \]
for all $s\in A$. Hence,
\[ f(s) \le \sup_{0\le a < b \le T} \frac{1}{b-a} \int_a^b f(t) dt \]
for all $s\in A$.
\end{proof}
\bigskip

Let $\probsp$ be a complete probability space, $H$ be a separable Hilbert space
with $H\not = \{0\}$, and $(g_h)_{h\in H}$ be an iso-normal family of Gaussian
random variables $g_h:\Omega\to \R$. Assume that
\[ \cF = \sigma(g_h : h\in H) \vee \cN \]
where $\cN$ are the null-sets from $\cF$. Let $(e_k)_{k\in I}$ be an orthonormal basis of $H$ with
$I=\{1,...,d\}$ or $I=\{1,2,...\}$. Then
\[ \cF = \sigma(g_{e_k} : k\in I) \vee \cN. \]
We recall that $D:\D_{1,2}\to L_2^H$ is a closed operator  (see \cite[Proposition 1.2.1]{Nualart:06}). Assume that $\varphi_n:\R\to [0,\infty)\in C_0^\infty$ such that $\varphi_n(x)=0$ for $x\le0$ and $x\ge 1/n$ and that $\int_\R \varphi_n(x)dx=1$. Defining
$\psi_n(y) := \int_{-\infty}^y \varphi_n(x)dx$, we get $\psi_n(x)=0$ for $x\le 0$,
$\psi_n(x)=1$ if $x\ge 1/n$, and $0\le \psi_n(x) \le 1$. Finally, set
\[ L_n(y) := \int_{-\infty}^y  \psi_n(x) dx \]
so that $L_n'(x) = \psi_n(x) \to_n \chi_{(0,\infty)}(x)$ and
\[ 0\le x - L_n(x) \le \frac{1}{n} \]
for $x\ge 0$ whereas $L_n(x)=0$ for $x\le 0$. Given $\xi\in \D_{1,2}$ we get that
$| \xi^+ - L_n (\xi) |\le 1/n$ and $L_n'(\xi) D\xi \to   \chi_{(0,\infty)}(\xi) D\xi$
in $L_2^H$. Hence $\xi^+\in \D_{1,2}$ with
\[ D\xi^+ =  \chi_{(0,\infty)}(\xi) D\xi \]
and, for $L>0$,
\st{\[
    D(\xi \vee (-L))
  = D( (\xi+L)^+ - L) 
  = \chi_{(0,\infty)}(\xi+L) D(\xi +L)
  = \chi_{(-L,\infty)}(\xi) D(\xi).
\]
From this we get
\[
      D(\xi \wedge L)
   =  -  D((- \xi) \vee (-L)) 
   =  - \chi_{(-L,\infty)}(-\xi) D(-\xi)
   =    \chi_{(-\infty,L)}(\xi) D(\xi).
\]}
Finally,
\equa
      D ((\xi \vee (-L))\wedge L )
& = & \chi_{(-\infty,L)}(\xi \vee (-L)) D((\xi\vee (-L)) \\
& = & \chi_{(-\infty,L)}(\xi \vee (-L))\chi_{(-L,\infty)}(\xi) D(\xi) \\
& = & \chi_{(-L,L)}(\xi) D(\xi).
\tion

\begin{proposition}\label{prop:approximation-D1p}
Let $H$ be a separable Hilbert space with an orthonormal basis $(e_k)_{k\in I}$,
where $I=\{1,...,d\}$ or $I=\{1,2,...\}$, let $(g_h)_{h\in H}$, $g_h:\Omega\to\R$,
be an iso-normal family of Gaussian random variables defined on a complete probability
space $(\Omega,\cF,\P)$ with $\cF=\sigma (g_h:h\in H)\vee \cN$ with $\cN$ being the
null-sets of  $(\Omega,\cF,\P)$.
Let $p\in [2,\infty)$, $\vare>0$, and $\xi\in \D_{1,2}\cap L_p$ such that
$D\xi \in L_p^H$. Then there exist $n\ge 1$ and a bounded $f_n\in C^\infty(\R^n)$
such that all derivatives are bounded (where the bound may depend on the order of
the derivative) such that for $\xi_0:= f_n(g_{e_1},...,g_{e_n})$ one has
\[ \| \xi-\xi_0\|_p^p + \| D\xi - D\xi_0 \|_{L_p^H}^p < \vare^p. \]
\end{proposition}
\bigskip

\begin{proof}
\underline{(a) Reduction to $\dim (H)<\infty$} in the case ${\rm dim}(H)=\infty$: Let
$\cH_n := \sigma (g_{e_1},...,g_{e_n})$.
By martingale convergence it follows that
\[ \lim_n \xi_n := \lim_n \E(\xi|\cH_n) = \xi
   \mbox{ a.s. and in } L_p. \]
For $n\in I$ let $P_n:H\to {\rm span}\{e_1,...,e_n\}\subseteq H$ be the orthogonal projection.
Then
\equa
      \| D\xi - D\xi_n \|_{L_p^H}
& = & \| P_n D \xi -  D \xi_n + (I-P_n)  D \xi \|_{L_p^H} \\
&\le&  \| P_n D \xi - D \xi_n\|_{L_p^H} + \| (I-P_n)  D \xi \|_{L_p^H}.
\tion
By dominated convergence,
\[ \lim_n  \| (I-P_n)  D \xi \|_{L_p^H} = 0.\]
On the other hand, using
$D\xi_n = P_n \E(D\xi |\cH_n)$
we get
\[      \| P_n D \xi - D \xi_n\|_{L_p^H}
    =   \|P_n D \xi - P_n \E(D\xi |\cH_n) \|_{L_p^H}
   \le  \|D \xi - \E(D\xi |\cH_n) \|_{L_p^H} \]
that converges to zero as $n\to\infty$ because $\cF=\vee_{n\ge 1} \cH_n \vee \cN$ and because of known facts about Banach space valued closable
martingales. Summing up, we obtain
\[ \lim_n \left [ \|\xi -\xi_n \|_p^p + \| D\xi - D\xi_n \|_{L_p^H}^p \right ] = 0. \]
\underline{(b) Reduction to a bounded $\xi$}:  For $L\ge 1$ define the truncation function $\psi_L:\R\to \R$ by
$\psi_L(x) := (x\vee (-L))\wedge L$. Then
\[ \lim_{L\to\infty}  \|\xi_n -\psi_L(\xi_n) \|_p = 0 \]
where $\xi_n$ is an approximation obtained by (a) or we take $\xi_n=\xi$ in case ${\rm dim}(H)<\infty$.
Moreover, $\chi_{(-L,L)}(\xi_n)D\xi_n$ is a representative of $D(\psi_L(\xi_n))$, so that
\[ \lim_{L\to\infty}   \|D \xi_n - D (\psi_L(\xi_n)) \|_{L_p^H} = 0 \]
as well. Consequently, for all $\vare>0$ there are $n,L\ge 1$ such that
\[ \|\xi -\psi_L(\xi_n) \|_p^p + \| D\xi - D(\psi_L(\xi_n)) \|_{L_p^H}^p < \vare^p. \]
\underline{(c) Reduction to the smooth case}: By the factorization theorem we can write
\[ \psi_L(\xi_n) = f_n(g_{e_1},...,g_{e_n}) \in \D_{1,2} \]
for a bounded Borel function $f_n:\R^n\to \R$ where we suppress $L$ in the following.
Let $F_n:[0,1)\times \R^n\to\R$ be the solution
of the backward heat equation with terminal condition $f_n$ so that
\[ \lim_{t\to 1} F_n(t,B_t^n) = f_n(B_1^n)
   \sptext{1}{and}{1}
   \lim_{t\to 1} \nabla F_n(t,B_t^n) = Df_n(B_1^n)  \]
in $L_p$ and $L_p^{\R^n}$, respectively, and a.s., where $(B_t^n)_{t\in [0,1]}$ is
an $n$-dimensio\-nal standard Brownian motion. But this implies also that
\[ \lim_{t\to 1} F_n(t,\sqrt{t} B_1^n) = f_n(B_1^n)
   \sptext{1}{and}{1}
   \lim_{t\to 1} \nabla F_n(t,\sqrt{t} B_1^n) = Df_n(B_1^n)  \]
in $L_p$ and $L_p^{\R^n}$, respectively. This can be seen from the estimate
\equa
       \| F_n(t,\sqrt{t} B_1^n) - f_n(B_1^n) \|^p_p
& = &  \E | \widetilde{\E} f_n(\sqrt{t} B_1^n+ \widetilde{B}_{1-t}^n) - f_n(B_1^n) |^p \\
&\le&  \E \widetilde{\E} | f_n(\sqrt{t} B_1^n+ \widetilde{B}_{1-t}^n) - f_n(B_1^n) |^p \\
& = &  \E \widetilde{\E} | f_n(B_{\sqrt{t}}^n+ \widetilde{B}_{1-\sqrt{t}}^n) - f_n(B_1^n) |^p
\tion
so that
\[  \| F_n(t,\sqrt{t} B_1^n) - f_n(B_1^n) \|_p
    \le 2 \| F_n(\sqrt{t}, B_{\sqrt{t}}^n) - f_n(B_1^n) \|_p\to 0 \]
as $t\to 1$. The fact we used here is that $(F_n(t,B_t^n))_{t\in [0,1]}$ is a
martingale. As $(\nabla F_n(t,B_t^n))_{t\in [0,1]}$ is a
martingale as well, where we agree about
$Df_n=:\nabla F_n(1,\cdot)$, the same computation yields to
\[
       \| \nabla F_n(t,\sqrt{t} B_1^n) - Df_n(B_1^n) \|_{L_p^H} 
 \le 2 \| \nabla F_n(\sqrt{t}, B_{\sqrt{t}}^n) - D f_n(B_1^n) \|_{L_p^H}\to 0
\]
as $t\to 1$. Letting $f_{n,t}:=F_n(t,\sqrt{t}\cdot)$ for $t\in [0,1)$, we get that
\[ D f_{n,t}(g_{e_1},...,g_{e_n}) = \sqrt{t} \sum_{k=1}^n \frac{\partial}{\partial x_k} F_n(t,\sqrt{t} (g_{e_1},...,g_{e_n}))e_k \]
because $f_{n,t}\in C_1^b(\R^n)\cap C^b(\R^n)$, and therefore
\equa
&   & \| D f_n (g_{e_1},...,g_{e_n}) - D f_{n,t}(g_{e_1},...,g_{e_n})  \|_{L_p^H} \\
&\le& \left \| D f_n(g_{e_1},...,g_{e_n}) - \sum_{k=1}^n \frac{\partial}{\partial x_k}
            F_n(t,\sqrt{t} (g_{e_1},...,g_{e_n}))e_k \right \|_{L_p^H} \\
&   & + (1-\sqrt{t}) \left \| \sum_{k=1}^n \frac{\partial}{\partial x_k}
            F_n(t,\sqrt{t} (g_{e_1},...,g_{e_n}))e_k \right \|_{L_p^H} \\
&\le& \left \| D f_n(g_{e_1},...,g_{e_n}) - \sum_{k=1}^n \frac{\partial}{\partial x_k}
            F_n(t,\sqrt{t} (g_{e_1},...,g_{e_n}))e_k \right \|_{L_p^H} \\
&   & + (1-\sqrt{t}) \| Df_n (g_{e_1},...,g_{e_n}) \|_{L_p^H}.
\tion
Summarizing,
\begin{multline*}
 \lim_{t\to 1} \big [
         \| f_{n,t}(g_{e_1},...,g_{e_n}) - f_n (g_{e_1},...,g_{e_n}) \|^p_p \\
      + \| D f_n (g_{e_1},...,g_{e_n}) - D f_{n,t}(g_{e_1},...,g_{e_n})  \|_{L_p^H}^p
                            \big ] = 0. \qedhere
\end{multline*}
\end{proof}

\begin{lemma}[Stein's martingale inequality, \cite{Meyer:78} and cf. {\cite[Theorem 3.2]{Qiu:11}}]
\label{lemma:stein-inequality}
\index{inequality!Stein's inequality}

Let $\probsp$ be a probability space, $p\in (1,\infty)$ and let $(\cG_k)_{k=1}^n$ be an increasing sequence of sub-$\sigma$-algebras of $\cF$. Then one has
\[ \left \| \left ( \sum_{k=1}^n |\E(f_k|\cG_k)|^2 \right )^\frac{1}{2} \right \|_p
   \le c_p \left \| \left ( \sum_{k=1}^n |f_k|^2 \right )^\frac{1}{2} \right \|_p \]
for all $f_1,...,f_n\in L_p$ where the constant $c_p>0$ depends at most on $p$.
\end{lemma}
\bigskip

Note that by grouping the random variables in an appropriate way in Stein's inequality, 
we can also assume that $f_1,...,f_n$ are random vectors with values in \st{$\R^N$}, 
whereas the constant $c_p>0$ does not enlarge.

\begin{lemma}\label{lemma:Stein_continuous}
For $p\in (1,\infty)$ assume a stochastic process 
$a=(a_t)_{t\in [0,1]}$ with values in \st{$\R^N$} that has 
left-continuous paths and satisfies $\E \sup_t |a_t|^p <\infty$. Suppose a 
filtration $(\cH_t)_{t\in [0,1]}$ and an  
$(\cH_t)_{t\in [0,1]}$-adapted process $(b_t)_{t\in [0,1]}$ 
with values in \st{$\R^N$} and $\E |b_t|^p<\infty$ for all $t\in [0,1]$
that has left-continuous paths and such that
$b_t =\E (a_t|\cH_t)$ a.s. for $t=k/2^n$ with $n=0,1,2,...$ and $k=0,...,2^n-1$.
Then one has that
\[     \left \| \left ( \int_0^1 |b_t|^2 dt \right )^\frac{1}{2} \right \|_p
   \le c_{(\ref{lemma:stein-inequality})}  \left \|  \left ( \int_0^1 |a_t|^2 dt \right )^\frac{1}{2}
   \right \|_p \]
where $c_{(\ref{lemma:stein-inequality})}>0$ is taken from Lemma \ref{lemma:stein-inequality}.
\end{lemma}
\smallskip

\begin{proof}
Let $t_k^n := \frac{k}{2^n}$ for $n\ge 0$ and $k=0,...,2^n-1$. Then it follows from Lemma \ref{lemma:stein-inequality} that
\[
  \left \| \left ( \sum_{k=0}^{2^n-1} (t_{k+1}^n-t_{k}^n) |\E (a_{t_{k}^n}|\cH_{t_{k}^n})|^2 \right )^\frac{1}{2} \right \|_p 
   \le c_{(\ref{lemma:stein-inequality})} \left \|  \left (  \sum_{k=0}^{2^n-1} (t_{k+1}^n-t_{k}^n) |a_{t_{k}^n}|^2
   \right )^\frac{1}{2}
   \right \|_p.
\]
Applying twice Fatou's lemma on the left-hand side, we derive
\[  \left \| \left ( \int_0^1 |b_t|^2 dt \right )^\frac{1}{2} \right \|_p
   \le c_{(\ref{lemma:stein-inequality})} \liminf_n \left \| \left ( \sum_{k=0}^{2^n-1} (t_{k+1}^n-t_{k}^n) |a_{t_{k}^n}|^2 \right )^\frac{1}{2}\right \|_p  \]
and we can conclude by dominated convergence.
\end{proof}

\begin{lemma}\label{lemma:PDE-Stein}
Let $p\in (1,\infty)$, $N\ge 1$ and $f:\R^N\to \R\in C^\infty$ where $\| D^\alpha f \|_\infty < \infty$ for all
multi-indices $\alpha$. Let $\gamma_N$ be the standard Gaussian measure on $\R^N$. Then one has
\[ \left \| f - \int_{\R^N} f d\gamma_N \right \|_{L_p(\gamma_N)}
   \le c_{\eqref{lemma:PDE-Stein}} \| |\nabla f| \|_{L_p(\gamma_N)} \]
where the constant $c_{\eqref{lemma:PDE-Stein}}>0$ depends on $p$ only.
\end{lemma}

\begin{proof}
Let $B=(B_t)_{t\in [0,1]}$ be an $N$-dimensional standard Brownian motion on a complete 
probability space $(M,\Sigma,\mu)$ with the augmented natural filtration $(\cG_t)_{t\in [0,1]}$ and that $\Sigma=\cG_1$.
Let
\[ F(t,x) := \E f(x+ B_{1-t})  \]
so that, by It\^o's formula,
\[ f(B_1) - \E f(B_1) = \int_0^1 \nabla F (t,B_t) dB_t, \]
and, by the Burkholder-Davis-Gundy inequalities,
\equa
       \| f(B_1) - \E f(B_1) \|_p
&\le& c_p \left \| \left ( \int_0^1 |\nabla F(t,B_t) |^2 dt \right )^\frac{1}{2} \right \|_p
\tion
and we can conclude with Lemma \ref{lemma:Stein_continuous}
by $a_t\equiv \nabla f(B_1)$, $b_t := \nabla F(t,B_t)$ and $\cH_t=\cG_t$.
\end{proof}

We call a function $h:\Omega\to \R$ a $\Pi$-step-function, where $\Pi\subseteq 2^\Omega$ is non-empty system of subsets, provided that
$h=\sum_{k=1}^n \alpha_k \chi_{A_k}$ for some $\alpha_1,...,\alpha_n\in \R$ and $A_1,...,A_n\in \Pi$.
\medskip

\begin{theorem}
\label{thm:PI-system_is_dense}
Let $\Omega$ be a non-empty set and $\Pi$ be a system of subsets of $\Omega$ such that
\begin{enumerate}[{\rm (i)}]
\item $A,B\in \Pi$ implies $A\cap B\in \Pi$,
\item $\Omega\in \Pi$.
\end{enumerate}
Let $p\in [1,\infty)$ and $\cF:=\sigma(\Pi)$. Then for all $f\in L_p(\Omega,\cF,\P)$ there are $\Pi$-step-functions $f_n:\Omega\to\R$ such that
$\lim_n \| f-f_n \|_p =0$.
\end{theorem}

\begin{proof}
Let
$\cM := \{ \chi_A : A \in \Pi \}$
so that
$\cF = \sigma(\Pi) = \sigma (\cM)$.
Let $\cH$ be the set of all bounded measurable $f: \Omega \to \R$ such that there exist $\Pi$-step-functions $h_k:\Omega\to\R$
with $\lim_k \| f-h_k \|_p = 0$.
Then $\cH$ and $\cM$ satisfy the assumptions of the monotone class theorem (see \cite[p. 7]{Protter:04}).
Hence any bounded
$\cF$-measurable function can be approximated in $L_p$ by $\Pi$-step-functions. Our assertion follows by one more approximation
obtained by truncation of a general element of $L_p$.
\end{proof}
\smallskip

\begin{theorem}
Let $X=(X_t)_{t\in [0,T]}$, $T>0$, $X_t:\Omega\to \R^d$, be a stochastic process such that all families
$(X_{t_i^k}^k-X_{t_{i-1}^k}^k)_{k=1,i=1}^{d,N_k}$ with
\[ 0=t_0^k < \cdots < t_{N_k}^k = T
  \sptext{1}{and}{1}
  N_k\ge 1 \]
are independent, $\cF:=\sigma(X)$, and $p\in [1,\infty)$.
Then the following holds:
\begin{enumerate}[{\rm (i)}]
\item The linear span of
      \[ \prod_{k=1}^d \prod_{i=1}^{N_k} \chi_{\left \{ X_{t_i^k}^k-X_{t_{i-1}^k}^k \in (a_i^k,b_i^k)\right \}}, \]
      where for $N_k=0$ the corresponding product is replaced by $1$ and for
      $N_k\ge 1$ we have $-\infty<a_i^k<b_i^k<\infty$ and $0\le t_{i-1}^k<t_i^k\le T$, is dense in $L_p(\Omega,\cF,\P)$.
\item If $X$ is the $d$-dimensional standard Brownian motion, then the linear span of
      \[ \prod_{k=1}^d \prod_{i=1}^{N_k} \left (X_{t_i^k}^k-X_{t_{i-1}^k}^k \right ) \]
      is dense in $L_2(\Omega,\cF,\P)$, where for $N_k=0$ the corresponding product is replaced by $1$ and for $N_k\ge 1$
      the intervals $(t_{i-1}^k,t_i^k]$, $i=1,..,N_k$, are pair-wise  disjoint for any fixed $k$.
\end{enumerate}
\end{theorem}

\begin{proof}
(i) The system $\Pi$ consisting of $\Omega$ and all possible finite intersections of
$\{ X_t^k - X_s^k \in (a,b) \}$
with $k\in \{ 1,...,d \}$, $0\le s < t \le T$, and $-\infty < a < b < \infty$, satisfies (i) and (ii) of
Theorem \ref{thm:PI-system_is_dense} and $\cF=\sigma(\Pi)$. Therefore assertion (i) follows from the same Theorem \ref{thm:PI-system_is_dense}.
\smallskip

(ii) By step (i) the random variables of form
\[ \xi=f\left (\frac{X_{t_1} - X_{t_0}}{\sqrt{t_1-t_0}},..., \frac{X_{t_n}-X_{t_{n-1}}}{\sqrt{t_n-t_{n-1}}} \right ), \]
where $n\ge 1$, $0=t_0<\cdots < t_n=T$ and $f:\R^{nd}\to \R$ is a bounded Borel function, are dense in $L_2(\Omega,\cF,\P)$.
Exploiting the orthonormal basis of Hermite functions of $L_2(\R^{nd},\gamma_{nd})$ we can approximate $\xi$ by polynomials in
$(X_{t_i}^k-X_{t_{i-1}}^k)$ where $k=1,...,d$ and $i=1,...,n$. It remains to approximate $(X_b^k-X_a^k)^l$ for $l\ge 2$, $k\in \{1,...,d\}$ and
$0\le a < b \le T$ by
\[ \sum_{\genfrac{}{}{0pt}{}{i_1,...,i_l\in \{1,...,N\}}{\mbox{\tiny distinct}}} \left (X^k_{a+ i_1 \frac{b-a}{N}} -  X^k_{a+ (i_1-1)\frac{b-a}{N}}\right )
                        \cdots  \left (X^k_{a+ i_l \frac{b-a}{N}} -  X^k_{a+ (i_l-1)\frac{b-a}{N}}\right ) \]
and $N\to \infty$.
\end{proof}

The following lemma can be proved by the generalized Clark-Ocone formula from 
\cite[Proposition A.1]{Nualart:Pardoux:88}. For completeness we include an argument based
on a periodic time-shift of the Brownian motion.
\smallskip

\begin{lemma}
\label{lemma:BDG-chaos}
Let $p\in [2,\infty)$, $\xi=\sum_{k=0}^\infty I_k(f_k) \in \D_{1,2} \cap L_p(\Omega,\cF,\P)$ 
with symmetric kernels $f_k$, and $b\in (0,T]$. Then
there are measurable processes $(\mu^b_t(i))_{t\in [0,b]}$, $i=1,...,d$, 
such that for all $a\in [0,b)$ one has 
\begin{enumerate}
\item $\| \xi - \E(\xi|\cG_a^b) \|_p
       \sim_{\kappa_p} \left\| \left ( \int_a^b |\mu^b_r|^2 dr \right )^\frac{1}{2} \right \|_p$,
       where $\kappa_p\ge 1$ depends on $p$ only,
\item and that
      \begin{multline*}
         \int_{(a,b]} \E|\mu_r^b(i) - D(r,i)\xi|^2 dr \\
       = \int_{(a,b]} \sum_{k=1}^\infty k^2(k-1)! \| f_k((r,i),\cdot)[\chi_{((0,r]\cup (b,T])^{k-1}}-1]\|^2_{L_2^{k-1}}dr.
       \end{multline*}
\end{enumerate}
\end{lemma}
\smallskip

\begin{proof}
We represent our Wiener space by a different Brownian motion, obtained by a permutation of the original one.
For this purpose we let
\[ W_t^b := \begin{cases}
                     W_{b+t} - W_b         & : t \in [0,T-b] \\
                     W_{t-T+b} + W_T- W_b  & : t \in [T-b,T]
            \end{cases}  \]
and obtain a standard Brownian motion (as Gaussian process).
We have that $\sigma(W_t^b: t\in [0,T]) = \sigma(W_t:t\in [0,T])$ and
$\cG_t^b= \cF^{W^b}_{T-b+t}$. The symmetric kernels $f_n$ for the chaos decompositions with respect to $W$ may be transformed to
$W^b$ as
\begin{equation} \label{eqn:shift_of_kernels}
  f_n^b ((t_1               ,i_1),...,(t_n                ,i_n))
= f_n   (((\vph^b)^{-1}(t_1),i_1),...,((\vph^b)^{-1}(t_n),i_n))
\end{equation}
where $\vph^b(t) := t+(T-b)$ for $t\in (0,b]$ and $\vph^b(t) := t-b$ for $t\in (b,T]$.
Now we get that
\[ \xi - \E(\xi|\cG_a^b) = \xi - \E(\xi|\cF^{W^b}_{T-b+a}). \]
Let $\xi=\sum_{n=0}^\infty I_n^b(f_n^b)$ the chaos decomposition with respect to $W^b$ where the kernels are obtained from
the representation in terms of $W$ by formula (\ref{eqn:shift_of_kernels}). Exploiting 
the representation property on the Wiener space, we find progressively measurable (with
respect to the augmentation of the natural filtration $(\cF_t^{W^b})_{t\in [0,T]}$ of $(W_t^b)_{t\in [0,T]}$) 
processes $(\lambda_t^b(i))_{t\in [0,T]}$, $i=1,\ldots,d$, satisfying 
$\E\int_0^T |\lambda_t^b|^2 dt < \infty$ and
\[ \xi = \E \xi + \int_{(0,T]} \lambda_t^b dW_t^b \mbox{ a.s.} \]
Then the processes  $(\mu_t^b(i))_{t\in [0,b]}$ are defined by
\[ \mu_r^b(i) := \lambda_{T-b+r}^b(i). \]
By the Burkholder-Davis-Gundy inequalities we get that
\begin{multline*}
        \left\|  \xi - \E(\xi|\cG_a^b) \right \|_p
     =             \left\|  \xi - \E(\xi|\cF^{W^b}_{T-b+a}) \right \|_p \\
   \sim_{\kappa_p} \left\| \left ( \int_{T-b+a}^T |\lambda^b_r|^2 dr \right )^\frac{1}{2} \right \|_p
     =             \left\| \left ( \int_a^b       |\mu    ^b_r|^2 dr \right )^\frac{1}{2} \right \|_p. 
\end{multline*}
This proves part (1). Regarding part (2) it is sufficient to prove the equality for $\xi$ from 
a dense subset of $\D_{1,2}$. So we may assume $\xi=\sum_{k=1}^N I_k(f_k)$, $N\ge 1$, with symmetric $f_k$ that are constant 
on dyadic cuboids of side-length $T/2^L$, $L\ge 1$, and vanish on diagonal cuboids (where at least 
two edges coincide). For those $\xi$ we have the explicit formula
\[ \lambda_t^b(i) = \sum_{k=1}^N k I_{k-1}^b (f_k^b((t,i),\cdot) \chi_{(0,t]^{k-1}}) \]
where we chose the canonical representatives on the right-hand side.
In this case one can directly check part (2).
\end{proof}


{\bf Acknowledgment:} We would like to thank 
Christian Bender,
Christel Geiss,
David Nualart,
Adrien Richou,
and  Alexander Steinicke for helpful discussions, 
and the referee for reading the manuscript and for his valuable comments.


\printindex

\end{document}